\documentclass[a4paper,11pt]{article}
\usepackage{LatexDefinitions}
\usepackage{rotating}

\renewcommand{\d}{{\mathrm{d}}}

\newcommand{\meas}{\mc{M}}
\newcommand{\measp}{\mc{M}_+}
\newcommand{\prob}{\mc{P}}
\newcommand{\restr}{{\mbox{\LARGE$\llcorner$}}}

\newcommand{\RadNik}[2]{{\tfrac{\d#1}{\d#2}}}

\newcommand{\CDynM}{c_D}%{f_m}

\newcommand{\CPath}{c_p}

\newcommand{\BDyn}{B}
\newcommand{\BDynM}{B_D}%{B_m}
\newcommand{\BDynMH}{h_D}%{q_m}
\newcommand{\BDynMHTild}{\tilde h_D}%{q_m}
\newcommand{\BStat}{B_S}%{B_{01}}
\newcommand{\BStatPre}{B_{S,\tn{pre}}}%{B_{01}}
\newcommand{\BStatFlip}{B_{S,\tn{flip}}}%{B_{01}}
\newcommand{\HStat}{h_S}%{h_{01}}
\newcommand{\HStatAux}{\hat{h}_{S}}
\newcommand{\HStatInv}{\overline h_S}%{h_{10}}
\newcommand{\HStatBoth}{h}

\newcommand{\zKink}{\ol{s}} % max{ z : q_m(z)=0 }
\newcommand{\zKinkN}{\underline{s}} % min{ z : q_m(z)=0 }
\newcommand{\BDynMHeps}{h_{D,\veps}} % smoothed version
\newcommand{\BDynMeps}{B_{D,\veps}} % smoothed version
\newcommand{\BDynMepsN}{B_{D,\veps_j}} % smoothed version
\newcommand{\BStatPreEpsN}{B_{S,\tn{pre},\veps_j}} % smoothed version

\newcommand{\Flow}{F}

\newcommand{\FlowInv}{F^{\tn{inv}}}

\newcommand{\SimLoc}{C_S}%{D_{01}}
\newcommand{\SimLocC}{c_S}%{c_{01}}

\newcommand{\SimFR}{\SimLoc^{\tn{H}}}
\newcommand{\SimTV}{\SimLoc^{\tn{TV}}}
\newcommand{\SimDisc}{\SimLoc^{\tn{d}}}
\newcommand{\SimJS}{\SimLoc^{\tn{JS}}}
\newcommand{\SimChi}{\SimLoc^{\chi^2}}
\newcommand{\SimE}[1]{\SimLoc^{\tn{E},#1}}
\newcommand{\SimEp}{\SimE{p}}

\newcommand{\SCc}{c_{\tn{sc}}}
\newcommand{\SCB}{B_{\tn{sc}}}

\newcommand{\StrongConvBoundFlow}{K}

\graphicspath{{./}}

\title{Dynamic Models of Wasserstein-$1$-Type Unbalanced Transport}
\author{Bernhard Schmitzer \and Benedikt Wirth}
\date{}

\begin{document}
\maketitle
\begin{abstract}
We consider a class of convex optimization problems modelling temporal mass transport and mass change between two given mass distributions
(the so-called dynamic formulation of unbalanced transport),
where we focus on those models for which transport costs are proportional to transport distance.
For those models we derive an equivalent, computationally more efficient static formulation,
we perform a detailed analysis of the model optimizers and the associated optimal mass change and transport,
and we examine which static models are generated by a corresponding equivalent dynamic one.
Alongside we discuss thoroughly how the employed model formulations relate to other formulations found in the literature.
\end{abstract}
\tableofcontents

% !TEX root = W1TypeDynamic.tex
\section{Introduction}
Optimal transport seeks the optimal way of transporting mass from a given initial distribution $\rho_0$ to a final distribution $\rho_1$, both on some domain $\Omega\subset\R^n$.
In Kantorovich's classical formulation, the transport is described by a transport plan or coupling $\pi$,
where $\pi(x,y)$ is the amount of mass transported from $x$ to $y$ so that, formally, $\rho_0=\pi(\cdot\times\Omega)$ and $\rho_1=\pi(\Omega\times\cdot)$.
Among all those couplings, the optimal one minimizes
\begin{equation*}
\int_{\Omega\times\Omega}c(x,y)\,\d\pi(x,y)\,,
\end{equation*}
where $c(x,y)$ denotes the cost per mass unit for transport from $x$ to $y$.
Solving the above minimization problem is computationally very costly due to the high dimensionality of $\pi$.
An equivalent convex formulation (originally for the case $c(x,y)=\|x-y\|^2$) in much lower dimensions is provided by the celebrated Benamou--Brenier formula \cite{BenamouBrenier2000},
which describes the transport via a material flow on $\Omega$ during a time interval $[0,1]$.
For the special case $c(x,y)=\|x-y\|$, in which the total transport cost is known as the Wasserstein-$1$ or $W_1$ distance between $\rho_0$ and $\rho_1$ and in which the transport cost is proportional to the transport distance, one can even eliminate the time coordinate, reducing the problem dimensionality yet further.

\subsection{Unbalanced optimal transport}
In applications that need to quantify how close two mass distributions $\rho_0$ and $\rho_1$ are to each other,
pure optimal transport does typically not suffice as a similarity measure since it requires $\rho_0$ and $\rho_1$ to have the same mass.
Therefore, optimal transport models have recently been extended to the case of so-called unbalanced transport, where the masses are allowed to change during the transport.
Early proposals for unbalanced transport problems can be found, for instance, in \cite{Benamou-Unbalanced-2003,PeleECCV2008}.

A more systematic investigation started from dynamic formulations based on the Benamou--Brenier formula for the Wasserstein-2 distance by adding a source term to the mass conservation constraint and a suitable corresponding penalty to the energy functional.
In \cite{ChizatOTFR2015,KMV-OTFisherRao-2015,LieroMielkeSavare-HellingerKantorovich-2015a} the source penalty was chosen to be the Fisher--Rao or Hellinger distance, which leads to the Wasserstein--Fisher--Rao (WFR) or Hellinger--Kantorovich (HK) distance. In \cite{PiccoliRossi-GeneralizedWasserstein2016,MaRuSi2017} variants of the total variation norm are studied as penalties. All models are careful to retain some form of 1-homogeneity (at least in space) to allow for spatially singular measures.

In \cite{ChizatDynamicStatic2015,LieroMielkeSavare-HellingerKantorovich-2015a} equivalent expressions for the WFR/HK distance are derived, based on an extension of the `static' Kantorovich formulation.
In \cite{LieroMielkeSavare-HellingerKantorovich-2015a} the mass change is modelled by relaxing the exact marginal constraints $\rho_0=\pi(\cdot\times\Omega)$ and $\rho_1=\pi(\Omega\times\cdot)$ to soft marginal constraints where one penalizes the deviation with suitable entropy functionals.
In \cite{ChizatDynamicStatic2015} transport is described by two so-called semi-couplings $(\gamma_0,\gamma_1)$ with $\rho_0=\gamma_0(\cdot\times\Omega)$ and $\rho_1=\gamma_1(\Omega\times\cdot)$.
Intuitively, $\gamma_0(x,y)$ describes the mass starting out at $x$ with destination $y$, while $\gamma_1(x,y)$ is the mass arriving in $y$ from $x$. The Kantorovich functional is adapted suitably. It is shown that for a family of dynamic unbalanced problems (beyond WFR/HK) one can find corresponding semi-coupling formulations.
For the WFR/HK distance the static formulas given in \cite{LieroMielkeSavare-HellingerKantorovich-2015a} and \cite{ChizatDynamicStatic2015} are related via dualization and a change of variables \cite[Corollary 5.9]{ChizatDynamicStatic2015}.

Due to its special structure, unbalanced extensions of the Wasserstein-1 distance have attracted particular attention.
Marginal constraint relaxations of Wasserstein-1 where the deviation of the $\pi$-marginals from $\rho_0$ and $\rho_1$ is penalized by the total variation norm are studied, for instance, in \cite{PiccoliRossi-GeneralizedWasserstein2014,LellmannKantorovichRubinstein2014}. This is closely related to the optimal partial transport problem studied in \cite{Caffarelli-McCann-FreeBoundariesOT-2010}. The article \cite{PiccoliRossi-GeneralizedWasserstein2016} cited above gives essentially a dynamic reformulation of this distance. It is observed that this extension leads to a modified form of the Kantorovich--Rubinstein formula. 
A family of more general unbalanced extensions of this formula (in a certain sense the most general family) is studied in \cite{SchmitzerWirthUnbalancedW1-2017}.

Roughly speaking, the above discussion mentions three types of formulations for unbalanced transport problems: `dynamic' formulations based on the Benamou--Brenier formula, `static' semi-coupling extensions of the Kantorovich formulation, and unbalanced Wasserstein-1-type extensions of the Kantorovich--Rubinstein formula.
	In the following, we refer to these families by the shorthands (Dyn), (SC), and (W1T).
	Via convex duality, each of these formulations can be expressed in a primal and a dual form, which we denote by a suffix (P) or (D). (By convention we refer to the measure formulation as primal, even though measures are identified with the topological dual of continuous functions.)
	
	As discussed, it was observed that various unbalanced transport distances can be expressed in more than one formulation.
	In this article we study systematically the correspondence between unbalanced extensions of the Wasserstein-1 distance in the formulations (Dyn), (SC), and (W1T).
	A schematic relation between different (primal and dual) formulations, established correspondences, and new correspondences established in this article is shown in Figure~\ref{fig:Diagram}. Precise definitions for all formulas are given throughout the article at the indicated positions. We return to a more in-depth discussion in Section~\ref{sec:FormulationRelation}, when the technical definitions have been established.
	
\begin{figure}[hbt]
	\centering
	\begin{tikzpicture}[x=1cm,y=1cm,
		probNode/.style={
		anchor=north,
		inner sep=1.5pt,shape=rectangle,draw=black,line width=1pt,
		text width=4.4cm,align=center,minimum height=2.4cm},
		relNote/.style={},
		numNode/.style={shape=rectangle,draw,fill=white,inner sep=2pt,anchor=west}
	]

%	\clip (-4,-3.5) rectangle (11,8.1);
%
	\node[probNode] (1P) at (0,0) []{
		{\small Dynamic Primal (Def.~\ref{def:Dynamic})}\\
		{\scriptsize $
		\inf \{ \int_{[0,1] \times \Omega} \|\omega\| + \CDynM(\rho,\zeta)\,\d x\,\d t | \allowbreak (\rho,\omega,\zeta) \in \mc{CE}(\rho_0,\rho_1) \}
		$}
		};
			
	\node[probNode] (1D) at ($(1P.south)+(0,-1)$) []{
		{\small Dynamic Dual (Prop.~\ref{prop:DynamicDual})}\\
		{\scriptsize $
		\sup \{ \int_{\Omega} \phi(1,\cdot)\,\d \rho_1 \! - \! \int_{\Omega} \phi(0,\cdot) \, \d \rho_0
		 | \allowbreak
		 \phi \in \! C^1([0,1]\!\times\!\Omega),\, \|\nabla \phi\| \leq 1, \allowbreak
		 (\partial_t \phi,\phi) \in \BDynM \} $}};
		
	\node [probNode] (2P) at ($(1P.east)+(.8,0)$) [anchor=west]{
		{\small Semi-Coupling \mbox{Primal} (Def.~\ref{def:SemiCouplingPrimal})}\\
		{\scriptsize $\inf \{ \int_{\Omega^2} \SCc(x_0,\gamma_0,x_1,\gamma_1)\,\d x_0\, \d x_1
			| \allowbreak (\gamma_0,\gamma_1) \in \Gamma(\rho_0,\rho_1) \}$} };

	\node[probNode] (2D) at ($(2P.south)+(0,-1)$) []{
		{\small Semi-Coupling Dual (Prop.~\ref{prop:SemiCouplingDual})}\\
		{\scriptsize $
		\sup \{ \int_{\Omega} \alpha\,\d \rho_0 \! + \! \int_{\Omega} \beta \, \d \rho_1
		 | \allowbreak
		 (\alpha,\beta) \in \! C(\Omega)^2,\, (\alpha(x),\beta(y)) \in \SCB(x,y) \} $}};

	\node [probNode] (3P) at ($(2P.east)+(.8,0)$) [anchor=west]{
		{\small $W_1$-type Static Primal (Def.~\ref{def:StaticPrimal})}\\
		{\scriptsize 
		$\inf \{
			W_1(\rho_0,\rho_0') + \SimLoc(\rho_0',\rho_1') + W_1(\rho_1',\rho_1)
			|
			(\rho_0',\rho_1') \in \measp(\Omega) \}
		$
		}
		};

	\node[probNode] (3D) at ($(3P.south)+(0,-1)$) []{
		{\small $W_1$-type Static Dual (Prop.~\ref{prop:StaticEquivalence})}\\
		{\scriptsize $
		\sup \{ \int_{\Omega} \alpha\,\d \rho_0 \! + \! \int_{\Omega} \beta \, \d \rho_1
		 | \allowbreak
		 (\alpha,\beta) \in \! \Lip(\Omega)^2,\, (\alpha(x),\beta(x)) \in \BStat(x) \} $}};

	\node[numNode] at ($(1P.north west)-(0.1,0)$) []{Dyn-P}; % A
	\node[numNode] at ($(1D.north west)-(0.1,0)$) []{Dyn-D}; % B
	\node[numNode] at ($(2P.north west)-(0.1,0)$) []{SC-P}; %  C
	\node[numNode] at ($(2D.north west)-(0.1,0)$) []{SC-D}; %  D
	\node[numNode] at ($(3P.north west)-(0.1,0)$) []{W1T-P}; % E
	\node[numNode] at ($(3D.north west)-(0.1,0)$) []{W1T-D}; % F
		
	% PD equivalences
	\draw[<->] (1P.south) -- (1D.north) node [midway, right] {\cite{ChizatDynamicStatic2015}};
	\draw[<->] (2P.south) -- (2D.north) node [midway, right] {\cite{ChizatDynamicStatic2015}};
	\draw[<->] (3P.south) -- (3D.north) node [midway, right] {\cite{SchmitzerWirthUnbalancedW1-2017}};

	%  inter problem class equivalences
	\draw[->] (1P.east) -- (2P.west) node [midway, above] {\cite{ChizatDynamicStatic2015}};
	
	\draw[->] (1D.south) -- ($(1D.south)+(0,-1)$) -- ($(3D.south)+(0,-1)$)
		node [midway,below] {\small{(Prop.~\ref{thm:equivalence})}} -- (3D.south);
	\draw[->] (3D.west) -- (2D.east) node [midway,above,anchor=west,rotate=90] {\small{(Prop.~\ref{prop:SCW1TEquivalence})}};
	\draw[->,dashed] ($(3D.south)+(-.5,0)$) -- ($(3D.south)+(-.5,-.7)$) -- ($(1D.south)+(.5,-.7)$)
                node [midway,above]{\small{(Prop.~\ref{thm:staticAsDynamic})}} -- ($(1D.south)+(.5,0)$);

	% Hopf-Lax type equivalences
	\draw[->,dashed] (1D.east) -- (2D.west)  node [midway,above]{\small{[$\dagger$]}};

\end{tikzpicture}
	\caption[]{%
	Relating different dynamic and static unbalanced transport formulations. A solid arrow $A \rightarrow B$ indicates that every problem of class (A) induces a corresponding formulation in class (B). %
	A dashed link indicates that a correspondence has been established for a subset of class (A).
	The dashed link marked with [$\dagger$] has been established for special cases: balanced transport for various cost functions, in particular for the Wasserstein-2 distance, is discussed in \cite{BrenierExtendedMoKa2003}, the Wasserstein--Fisher--Rao (Hellinger--Kantorovich) distance is treated in \cite{LieroMielkeSavare-HellingerKantorovich-2015a}.
	The links marked with round brackets are established in this article.
	A more detailed discussion of the relation between the different formulations is given in Section \ref{sec:FormulationRelation}.}
	\label{fig:Diagram}
\end{figure}
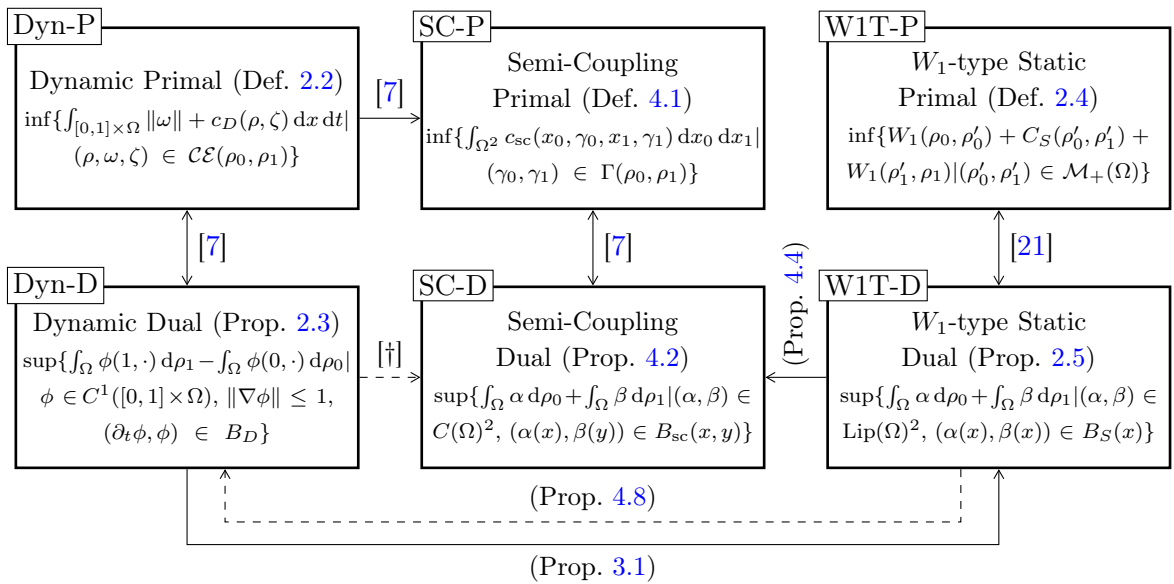

\subsection{Outline and contribution}

Restricting to a $W_1$-type penalization of transport, this article aims at augmenting the picture of unbalanced transport shown in Figure~\ref{fig:Diagram} by several relations.
The article is organized as follows.
\begin{itemize}
\item \emph{Section~\ref{sec:Reminder}}: In Section~\ref{sec:ReminderDynamic} we introduce a family of \emph{dynamic unbalanced transport problems}, (Dyn), where the penalty for transport is linear in its distance. This is the $W_1$-type subset of the more general family of transport problems studied in \cite{ChizatDynamicStatic2015}.
	In Section~\ref{sec:ReminderStatic} on the other hand, we introduce a family of \emph{static $W_1$-type unbalanced transport problems}, (W1T), based on generalizing the Kantorovich--Rubinstein formula. This is a subset of the family introduced in \cite{SchmitzerWirthUnbalancedW1-2017}.
\item \emph{Section~\ref{sec:DynStatEquivalence}: We establish for every (Dyn-D) problem a corresponding (W1T-D) problem such that the resulting optimal values are identical (\thref{thm:equivalence}).} This includes explicit relations between feasible candidates (e.g.~\thref{lem:DynamicDualConstruction}) and model parameters (\thref{cor:HStatFromDyn}).
\item \emph{Sections~\ref{sec:primalRelation} and \ref{sec:dynamicFromStaticTimeConc}: We examine the relation between primal optimizers of (W1T-P) and (Dyn-P).} For any optimizer of (W1T-P) we construct in Section~\ref{sec:primalRelation} an optimizer for (Dyn-P) (\thref{prop:DynamicPrimalOptimizers}). These dynamic optimizers exhibit a very particular structure which is characteristic for $W_1$-type transport problems: transport only occurs instantaneously at times 0 and 1, while in between only mass growth and shrinkage take place.
In Section~\ref{sec:dynamicFromStaticTimeConc} we give a sufficient condition for the dynamic model which implies that any optimizer of (Dyn-P) is of this particular form (\thref{cor:AllDynamicOptimizersStructure}).
Essentially, this temporal structure is the reason for (W1T-P) having a mass change penalty in between two Wasserstein-1 distances (as opposed to, for instance, a Wasserstein-1 distance between two mass change penalties, studied in \cite{PiccoliRossi-GeneralizedWasserstein2014}, which has no equivalent formulation in (Dyn-P))
\item \emph{Section~\ref{sec:characterization}: We characterize minimizers of (W1T-P) models concerning the spatial relation between mass growth, shrinkage, and transport (\thref{thm:transportChar}).}
This provides an intuition of how the unbalanced transport operates and automatically implies a corresponding characterization of (Dyn-P) model minimizers.
In particular, mass transport can neither occur into a region of previous or subsequent mass decrease nor out of a region of previous or subsequent mass increase.
Moreover, we derive a model-dependent distance threshold (which may be infinite) beyond which no transport occurs (\thref{prop:MaximalTransportDistance}).
\item \emph{Sections~\ref{sec:SemiCouplingFormulations}-\ref{sec:FormulationRelation}: We establish the equivalence of (W1T-D) models to corresponding (SC-D) models (\thref{prop:SCW1TEquivalence}).}
By the above results this implies a correspondence between (Dyn) and (SC) models.
Using the particular $W_1$-type structure, the relation between the corresponding model parameters is more explicit than the corresponding result in \cite{ChizatDynamicStatic2015}, but it should be noted that the latter covers more general transport problems.
\item \emph{Section~\ref{sec:staticAsDynamic}: We characterize precisely, which (W1T-D) models have an equivalent (Dyn-D) model (\thref{thm:staticAsDynamic}).}
In addition we provide corresponding simple sufficient as well as necessary conditions.
\item \emph{Section~\ref{sec:twoDiracs}: We perform a detailed analysis of the optimal unbalanced transport between two Dirac masses (\thref{exm:twoDiracs}).}
This provides information on maximum and minimum transport distances as well as on how the optimal mass changes depend on the previous or subsequent transport.
\item \emph{Section~\ref{sec:staticAsDynamicExamples}: We provide novel examples of (Dyn-D) models by seeking the dynamic formulation of known (W1T-D) models (Table~\ref{tab:discrepancies}).}
Though most of these induce the same topology on the space of nonnegative measures (\thref{thm:topology}), they penalize mass changes differently.
We also give a static counterexample for which no dynamic formulation exists (Remark \ref{rem:NoDynamicCounterexample}).
\end{itemize}

\subsection{Setting and notation}
Throughout the article, $\Omega$ denotes the closure of a fixed open bounded connected subset of $\R^n$
(note that the results would also hold for an arbitrary compact metric length space $\Omega$).
We will interpret $\Omega$ as a metric space with the metric $d:\Omega\times\Omega\to[0,\infty)$ induced by shortest paths in $\Omega$.
The Euclidean norm on $\R^n$ is indicated by $\|\cdot\|$.

For a metric space $X$ we denote the set of continuous functions $f:X\to\R$ by $C(X)$
and the space of Lipschitz continuous functions with Lipschitz constant no larger than $1$ by $\Lip(X)$  (for $X=\Omega$, the Lipschitz constant is with respect to the metric $d$).
If $X$ has the structure of a differentiable manifold, then $C^1(X)$ denotes the set of continuously differentiable functions $f:X\to\R$.
For the analogous function spaces of vector-valued functions into $\R^n$ we write $C(X;\R^n)$, $\Lip(X;\R^n)$, and $C^1(X;\R^n)$.

Now let $X$ be a Borel measurable subset of a Euclidean space.
By $\measp(X)$ and $\meas(X)$ we denote the space of nonnegative and of signed Radon measures (regular countably additive measures) on $X$, respectively.
The subset of probability measures on $X$ is denoted $\prob(X)$.
Given $\rho,\rho'\in\meas(X)$, we indicate that $\rho$ is absolutely continuous with respect to $\rho'$ by $\rho\ll\rho'$, and we denote the corresponding Radon--Nikodym derivative by $\RadNik\rho{\rho'}$.
Given a measurable subset $A\subset X$, the restriction of $\rho$ to $A$ is denoted $\rho\restr A$.
For a measure $\pi\in\meas(X\times X)$ its two marginals are denoted by $P_X \pi$ and $P_Y \pi$ and are defined for any Borel set $A$ via
\begin{equation*}
P_X\pi(A)=\pi(A\times X)\,,\qquad
P_Y\pi(A)=\pi(X\times A)\,.
\end{equation*}

Finally, the domain of a function $f:X\to[-\infty,\infty]$ is indicated as $\dom f=\{x\in X\ |\ f(x)\notin\{-\infty,\infty\}\}$,
the indicator function of a set $A\subset X$ is defined as $\iota_A(x)=0$ if $x\in A$ and $\iota_A(x)=\infty$ otherwise,
and the interior of a set $A\subset X$ is abbreviated as $\inter A$.
For a normed vector space $X$ we denote its topological dual by $X^\ast$.
If $f:X\to[-\infty,\infty]$ is a proper function, then the Legendre--Fenchel conjugate of $f$ is defined as $f^\ast:X^\ast\to(-\infty,\infty]$, $f^\ast(y)=\sup_{x\in X}\langle x,y\rangle_{X,X^\ast}-f(x)$.
For a linear operator $A:X\to Y$ we denote its adjoint by $A^\ast:Y^\ast\to X^\ast$.
As usual we identify the topological duals of $C(\Omega,\R^n)$ and $C([0,1] \times \Omega,\R^n)$ with $\meas(\Omega)^n$ and $\meas([0,1] \times \Omega)^n$.

% !TEX root = W1TypeDynamic.tex

\section{Reminder: models for unbalanced optimal transport}
\label{sec:Reminder}
Here we recall different extensions of the classical Wasserstein-$1$ metric that allow for mass changes during the mass transport.
In particular, we recapitulate the class of dynamic models from \cite{ChizatDynamicStatic2015} as well as a class of static unbalanced optimal transport models from \cite{SchmitzerWirthUnbalancedW1-2017}.
Establishing the equivalence of those two model classes belongs to the main aims of this work.
We will use subscripts $D$ and $S$ to indicate dynamic and static formulations throughout;
primal and dual energies are denoted $P$ and $D$, respectively.

\subsection[Dynamic W1-type models]{Dynamic $W_1$-type models}
\label{sec:ReminderDynamic}

In the dynamic model formulation we consider a time-varying measure $\rho$ which moves at a flux $\omega$ and simultaneously changes mass at rate $\zeta$.
The relation between $\rho$, $\omega$, and $\zeta$ is described by the following weak continuity equation.

\begin{definition}[Weak continuity equation with source \protect{\cite[Def.\,4.1]{ChizatDynamicStatic2015}}]
	\thlabel{def:ContinuityEquation}
	For $(\rho_0,\rho_1) \in \measp(\Omega)^2$ denote by $\mc{CE}(\rho_0,\rho_1)$ the affine subset of $\meas([0,1] \times \Omega)^{1+n+1}$ of triplets of measures $(\rho,\omega,\zeta)$ satisfying the continuity equation $\partial_t \rho + \nabla \cdot \omega = \zeta$ in the distributional sense, interpolating between $\rho_0$ and $\rho_0$ and satisfying homogeneous Neumann boundary conditions. More precisely, we require
	\begin{align}\label{eqn:ContinuityEquation}
		\int_{[0,1] \times \Omega} (\partial_t \phi)\,\d\rho
			+ \int_{[0,1] \times \Omega} (\nabla \phi) \cdot \d\omega
			+ \int_{[0,1] \times \Omega} \phi\,\d\zeta =
		\int_\Omega \phi(1,\cdot)\,\d\rho_1 - \int_{\Omega} \phi(0,\cdot)\,\d\rho_0
	\end{align}
	for all $\phi \in C^1([0,1] \times \Omega)$.
\end{definition}
This definition does not require that the map $t \mapsto \rho_t$, which takes $t$ to the corresponding time-disintegration of $\rho$, is (weakly) continuous. Instantaneous movement of mass via $\omega$ is characteristic for the $W_1$ distance and the unbalanced extensions that we study (cf.~Remark \ref{rem:InterpretationDynamicOptimizers}).
A Wasserstein-1 type model for unbalanced transport now penalizes the flux $\omega$ via its total variation as well as the mass change $\zeta$ via an infinitesimal cost $\CDynM$.

\begin{definition}[Dynamic unbalanced $W_1$-type model \protect{\cite[Def.\,4.2-3]{ChizatDynamicStatic2015}}]
	\thlabel{def:Dynamic}
	Let $\CDynM : \R^2 \to [0,\infty]$ be lower semi-continuous, convex, 1-homogeneous, and let it satisfy
	\begin{align}
		\label{eq:CDynM}
		\CDynM(\rho,\zeta) \begin{cases}
			= +\infty & \tn{if } \rho < 0, \\
			= 0 & \tn{if } \rho \geq 0,\, \zeta=0, \\
			> 0 & \tn{else.}
			\end{cases}
	\end{align}
	The associated \emph{dynamic cost functional} is defined as $P_D : \meas([0,1] \times \Omega)^{1+n+1} \to [0,\infty]$,
	\begin{align}
		\label{eq:DynamicEnergy}
    P_D(\rho,\omega,\zeta) = \int_{[0,1] \times \Omega} \left\|\RadNik{\omega}{\mu}\right\|+\CDynM\left(\RadNik{\rho}{\mu},\RadNik{\zeta}{\mu} \right)\,\d\mu\,,
	\end{align}
	where $\mu \in \measp([0,1] \times \Omega)$ is any measure such that $(\rho,\omega,\zeta) \ll \mu$. By the 1-homogeneity of $\|\cdot\|$ and $\CDynM$ this definition does not depend on the choice of $\mu$.
	The corresponding \emph{primal dynamic unbalanced $W_1$-type transport problem} for fixed marginals $(\rho_0,\rho_1) \in \measp(\Omega)^2$ reads
	\begin{align}
		\label{eq:DynamicProblem}
		W_D(\rho_0,\rho_1) = \inf \left\{ P_D(\rho,\omega,\zeta)\ \middle| \ (\rho,\omega,\zeta) \in \mc{CE}(\rho_0,\rho_1) \right\}\,.
	\end{align}
\end{definition}

The cost $W_D$ also admits an equivalent dual formulation.

\begin{proposition}[Dynamic dual problem]
	\thlabel{prop:DynamicDual}
	For a dynamic problem as in \thref{def:Dynamic} let $\BDynM \subset \R^2$ be the closed convex set characterized by $\iota_{\BDynM} = \CDynM^\ast$. Introduce $\BDyn = \{(\alpha,\beta,\gamma) \in \R^{1+n+1} \ | \ \allowbreak \|\beta\|
		\allowbreak \leq 1,\, (\alpha,\gamma) \in \BDynM \}$ and $D_D : C^1([0,1] \times \Omega) \to \RCupInf$,
	\begin{equation}
		\label{eq:DynamicDualFunctional}
		D_D(\phi) =\begin{cases}
			\int_\Omega \phi(1,\cdot)\,\d \rho_1 - \int_\Omega \phi(0,\cdot)\,\d \rho_0 & 
				\text{if } (\partial_t \phi(t,x), \nabla \phi(t,x), \phi(t,x))\in\BDyn \\
				& \qquad \text{ for all }(t,x)\in[0,1]\times\Omega,\\
			- \infty&\text{else.}
			\end{cases}
	\end{equation}
	Then
	\begin{align}
		\label{eq:DynamicDual}
		W_D(\rho_0,\rho_1) = \sup \{ D_D(\phi) \,|\, \phi \in C^1([0,1] \times \Omega) \}\,.
	\end{align}
	Furthermore, minimizers of problem \eqref{eq:DynamicProblem} exist if the infimum is finite.
\end{proposition}

\begin{proof}
	The proof is essentially identical to \cite[Thm.~2.1]{ChizatOTFR2015} and \cite[Prop.~4.2]{ChizatDynamicStatic2015}.
	Problem \eqref{eq:DynamicDual} can be rewritten as
	\begin{align*}
		\sup_{\phi \in C^1([0,1] \times \Omega)}
			-F(A\,\phi) - G(-\phi)
	\end{align*}
	with
	\begin{align*}
		F & : C([0,1] \times \Omega;\R^{1 + n +1}) \to \RCupInf, &
			(\alpha,\beta,\gamma) & \mapsto \int_{[0,1] \times \Omega} \iota_B(\alpha(t,x),\beta(t,x),\gamma(t,x))\,\d(t,x), \\
		G & : C^1([0,1] \times \Omega) \to \RCupInf, &
			\phi & \mapsto \int_\Omega \phi(1,\cdot)\,\d \rho_1 - \int_\Omega \phi(0,\cdot)\,\d \rho_0,\\
		A & : C^1([0,1] \times \Omega) \to C([0,1] \times \Omega; \R^{1+n+1}),\!\!\!\! &
			\phi & \mapsto (\partial_t \phi, \nabla \phi, \phi).
	\end{align*}
	Note that $F$ and $G$ are convex and lower semi-continuous. $A$ is linear and bounded.
	Further, there is some $\phi \in C^1([0,1] \times \Omega)$ such that $G(-\phi) < \infty$ and $F$ is continuous at $A\,\phi$:
	for instance, let $(-a,c) \in \inter \BDynM$ with $a,c>0$ and
	set $\phi(t,x) = c\exp(-\frac{a}ct)$, then $t \mapsto (\partial_t \phi(t,x), \phi(t,x))$ will move along a line from $(-a,c)$ towards $(0,0)$,
	which can easily be shown to remain in the interior of $\BDynM$ (the reader may refer to \thref{thm:DynConjugateSet}, which will show this calculation in detail).
	
	Using the above, by virtue of the Fenchel-Rockafellar Theorem \cite[Thm.~3]{Rockafellar-Duality1967} one obtains
	\begin{align*}
		\sup_{\phi \in C^1([0,1] \times \Omega)}
			-F(A\,\phi) - G(-\phi)
		= \min_{(\rho,\omega,\zeta) \in \measp([0,1]\times\Omega)^{1+n+1}}
			G^\ast(A^\ast (\rho,\omega,\zeta)) + F^\ast(\rho,\omega,\zeta)\,,
	\end{align*}
	where the minimizer on the right-hand side exists if it is finite.
	We find
	\begin{multline*}
		G^\ast(A^\ast(\rho,\omega,\zeta)) = \sup_{\phi \in C^1([0,1] \times \Omega)}
			\int_{[0,1] \times \Omega}
				\partial_t \phi \,\d \rho
			+ \int_{[0,1] \times \Omega}
				\nabla \phi \cdot \d \omega \\
			+ \int_{[0,1] \times \Omega}
				\phi \, \d \zeta
			- \int_\Omega \phi(1,\cdot)\,\d \rho_1 + \int_\Omega \phi(0,\cdot)\,\d \rho_0
	\end{multline*}
	which is the indicator function of $\mc{CE}(\rho_0,\rho_1)$.
	With \cite[Thm.\,5]{Rockafellar-IntegralConvexFunctionals71} one obtains $F^\ast=P_D$.
\end{proof}

\subsection[Static W1-type models]{Static $W_1$-type models}
\label{sec:ReminderStatic}
A different class of unbalanced optimal transport models was introduced in \cite{SchmitzerWirthUnbalancedW1-2017} of which we here consider a particular subclass.
It is based on an infimal convolution type extension of the Wasserstein-$1$ distance by a penalty $\SimLocC(m_0,m_1)$ for changing a mass $m_0$ into $m_1$.

\begin{definition}[Static unbalanced $W_1$-type model \protect{\cite[Def.\,2.21]{SchmitzerWirthUnbalancedW1-2017}}]
	\thlabel{def:StaticPrimal}
	A \emph{local discrepancy} is a function $\SimLocC : \R^2 \to [0,\infty]$ satisfying the following properties,
	\begin{enumerate}[(i)]
		\item $\SimLocC$ is convex, 1-homogeneous, and lower semi-continuous,
		\item $\SimLocC(m,m)=0$ if $m \in [0,\infty)$ and $\SimLocC(m_0,m_1) > 0$ if $m_0 \neq m_1$,
		\item $\SimLocC(m_0,m_1) = \infty$ if $m_0<0$ or $m_1<0$.
	\end{enumerate}
	A local discrepancy $\SimLocC$ induces a \emph{discrepancy} $\SimLoc : \measp(\Omega)^2 \to [0,\infty]$ via
	\begin{align}
		\SimLoc(\rho_0,\rho_1) & = \int_\Omega
			\SimLocC\left( \RadNik{\rho_0}{\gamma}(x), \RadNik{\rho_1}{\gamma}(x) \right) \,d\gamma(x)\,,
	\end{align}
	where $\gamma\in\measp(\Omega)$ is any measure with $\rho_0,\rho_1\ll\gamma$ (the functional is independent of the particular choice).
	$\SimLoc$ measures the cost of changing $\rho_0$ into $\rho_1$ by pointwise mass changes, where each mass change is penalized according to $\SimLocC$.
	The associated \emph{static cost functional} is defined as $P_S:\measp(\Omega\times\Omega)^2\to[0,\infty]$,
	\begin{equation}\label{eq:StaticPrimal}
	P_S(\pi_0,\pi_1)=\int_{\Omega\times\Omega} d(x,y)\,\d \pi_0(x,y) + \SimLoc(P_Y \pi_0, P_X \pi_1) + \int_{\Omega\times\Omega} d(x,y)\,\d \pi_1(x,y)\,,
	\end{equation}
	and the corresponding \emph{primal static unbalanced $W_1$-type transport problem} reads
	\begin{equation}
		\label{eq:StaticPrimalProblem}
		W_S(\rho_0,\rho_1) = \inf \left\{P_S(\pi_0,\pi_1) \ \middle|\
			(\pi_0, \pi_1) \in \measp(\Omega\times\Omega)^2,\,
			P_X \pi_0 = \rho_0,\,
			P_Y \pi_1 = \rho_1
			\right\}\,.
	\end{equation}
\end{definition}

The cost $W_S$ also admits an equivalent dual formulation.

\begin{proposition}[Static dual problem \protect{\cite[Corollary 2.32]{SchmitzerWirthUnbalancedW1-2017}}]\thlabel{prop:StaticEquivalence}
For a static problem as in \thref{def:StaticPrimal} let $\BStat\subset\R^2$ be the closed convex set characterized by $\iota_{\BStat}=\SimLocC^\ast$.
Introduce $D_S : C(\Omega)^2 \to \RCupInf$,
\begin{equation}\label{eq:StaticDualObjective}
D_S(\alpha,\beta) = \begin{cases}
\int_\Omega \alpha \, \d\rho_0 + \int_\Omega \beta \, \d\rho_1 & \tn{if } (\alpha,\beta) \in \Lip(\Omega)^2, \, (\alpha(x),\beta(x)) \in \BStat \tn{ for all } x \in \Omega, \\
- \infty & \tn{else,}
\end{cases}
\end{equation}
then
\begin{equation}\label{eq:StaticDual}
W_S(\rho_0,\rho_1) = \sup \left\{ D_S(\alpha,\beta) \,\middle|\,(\alpha,\beta) \in C(\Omega)^2 \right\}\,.
\end{equation}
Furthermore, minimizers of problem \eqref{eq:StaticPrimalProblem} exist if the infimum is finite.
Finally, the sets $\BStat$ characterized by $\iota_{\BStat}=\SimLocC^\ast$ for some local discrepancy $\SimLocC$ are exactly the sets of the form
	\begin{align}
		\label{eq:BStatHForm}
		\BStat = \{(\alpha,\beta) \in \R^2 \,:\,\beta \leq \HStat(-\alpha) \}
			= \{(\alpha,\beta) \in \R^2 \,:\, \alpha \leq \HStatInv(-\beta) \}
	\end{align}
	where $\HStat : \R \to \RCupMInf$ (or equivalently $\HStatInv : \R \to \RCupMInf$) satisfies (we drop the indices)
	\begin{enumerate}
		\item $\HStatBoth$ is concave, upper semi-continuous, and monotonically increasing,
		\item $\HStatBoth(z)\leq z$ for $z \in \R$, $\HStatBoth(0)=0$,
		\item $\HStatBoth$ is differentiable at $0$ and $\HStatBoth'(0)=1$.
	\end{enumerate}
	Note that on their respective domains, $\HStat=-\HStatInv^{-1}(-\cdot)$ and $\HStatInv=-\HStat^{-1}(-\cdot)$.
\end{proposition}

% !TEX root = W1TypeDynamic.tex
\section{Equivalence of dynamic and static problems}\label{sec:equivalence}
We first show the equivalence between dynamic and static problems in their dual formulations, after which we shall turn to the primal interpretation.

\subsection{Equivalence of optimal values}
\label{sec:DynStatEquivalence}

There are some structural similarities between the dynamic dual problem \eqref{eq:DynamicDual} and the static problem \eqref{eq:StaticDual}.
If one identifies $\alpha=-\phi(0,\cdot)$ and $\beta=\phi(1,\cdot)$, then the finite part of \eqref{eq:DynamicDualFunctional} corresponds to the finite part of the objective in \eqref{eq:StaticDualObjective}. The constraint $\|\nabla \phi(t,x)\| \leq 1$ guarantees that $\alpha$, $\beta \in \Lip(\Omega)$.
In this section we show that the constraint $(\partial_t \phi(t,x),\phi(t,x)) \in \BDynM$ for $(t,x) \in [0,1] \times \Omega$ can be translated into a constraint $(\alpha(x),\beta(x)) \in \BStat$ for $x \in \Omega$ for a suitable choice of $\BStat$, such that problems \eqref{eq:DynamicDual} and \eqref{eq:StaticDual} are in fact equivalent.

\begin{proposition}[Equivalence between dynamic and static problems]\thlabel{thm:equivalence}
	\thlabel{prop:W1DynamicStaticEquivalence}
	Consider the dynamic dual problem \eqref{eq:DynamicDual} and its characterizing set $\BDynM$.
	Set
	\begin{gather}
		\label{eq:BStatPre}
		\BStatPre = \left\{ (-\phi(0),\phi(1)) \,\middle|\,
			\phi \in C^1([0,1]) \,,
			(\partial_t \phi(t),\phi(t)) \in \BDynM \tn{ for all } t \in [0,1] 
			\right\}\,.
	\end{gather}
	Then, by choosing
	\begin{align}
		\label{eq:BStat}
		\BStat = \ol{\BStatPre}
			+ (-\infty,0]^2
	\end{align}
	for the static dual problem \eqref{eq:StaticDual} one finds
	\begin{align*}
		W_D(\rho_0,\rho_1) = W_S(\rho_0,\rho_1)\,.
	\end{align*}
\end{proposition}

The function $\HStat$ that characterizes $\BStat$ as described in \thref{prop:StaticEquivalence} will be given in \thref{cor:HStatFromDyn}.
A direct consequence of the proposition is that also the primal dynamic problem \eqref{eq:DynamicProblem}
is equivalent to a primal static problem \eqref{eq:StaticPrimalProblem}.
Here, the corresponding static mass change penalty $\SimLocC$ can be obtained from the dynamic mass change penalty $\CDynM$
by first calculating the set $\BDynM$ via \thref{prop:DynamicDual}, applying the above proposition to obtain the set $\BStat$, and finally calculating $\SimLocC$ via \thref{prop:StaticEquivalence}. This will be done explicitly in \thref{prop:W1DynamicStaticEquivalencePrimal}.

The proof is divided into several auxiliary lemmas. The strategy is as follows:
The constraints of \eqref{eq:DynamicDual} are $(\partial_t \phi(t,x), \phi(t,x)) \in \BDynM$ and $\|\nabla \phi(t,x)\| \leq 1$ for all $(t,x) \in [0,1] \times \Omega$. Let us ignore the constraint on $\nabla \phi$ for now. Since $\rho_1$ is nonnegative, for fixed $\phi(0,\cdot)$ the function $\phi(1,\cdot)$ in \eqref{eq:DynamicDualFunctional} will want to be as large as the $\BDynM$-constraint allows. With the identification $(\alpha,\beta)=(-\phi(0,\cdot),\phi(1,\cdot))$ this yields an upper bound on $\beta$ for fixed $\alpha$ and defines the set $\BStatPre$. This already implies $W_D \leq W_S$. Note that the set $\BStatPre$ does not necessarily satisfy the formal structural assumptions for $\BStat$, as specified by \eqref{eq:BStatHForm}, which is why some post-processing from the preliminary $\BStatPre$ to $\BStat$ is required. We show in the proof of \thref{prop:W1DynamicStaticEquivalence}, however, that replacing $\BStat$ with $\BStatPre$ does not change the optimal value of \eqref{eq:StaticDual}.
From feasible (in the sense of $\BStatPre$) static dual variables $(\alpha,\beta)$ for \eqref{eq:StaticDualObjective} we then reconstruct a feasible dynamic dual variable $\phi$ for \eqref{eq:DynamicDualFunctional} such that $\phi(0,\cdot)=-\alpha$, $\phi(1,\cdot)=\beta$. The Lipschitz constraint on $(\alpha,\beta)$ suffices to imply $\|\nabla\phi(t,x)\|\leq 1$, thus establishing the converse inequality $W_S \leq W_D$.
We now study the dynamic cost $\CDynM$ and the corresponding set $\BDynM$ in more detail.

\begin{lemma}[Properties of $\BDynM$]\thlabel{thm:DynConjugateSet}
The set $\BDynM$ from \thref{prop:DynamicDual} satisfies
\begin{equation*}
\BDynM = \left\{ (\psi, \phi) \in \R^2\;\colon\;\psi \leq \BDynMH(\phi) \right\}\,,
\end{equation*}
where $\BDynMH:\R\to\R\cup\{-\infty\}$ is concave, upper semi-continuous, nonpositive, increasing on $(-\infty,0]$ and decreasing on $[0,\infty)$ with $\BDynMH(0)=\BDynMH'(0)=0$
as well as $\overline{\dom\BDynMH}=[-\CDynM(0,-1),\CDynM(0,1)]$.
\end{lemma}
\begin{proof}
Since $\CDynM$ is one-homogeneous, $\BDynM$ must be a closed convex set.
Due to $\CDynM(\rho,\zeta) = \infty$ for $\rho<0$ we have
$$\iota_{\BDynM}(\psi_1,\phi)=\sup_{\rho\geq0,\zeta}\psi_1\rho+\phi\zeta-\CDynM(\rho,\zeta)\leq\sup_{\rho\geq0,\zeta}\psi_2\rho+\phi\zeta-\CDynM(\rho,\zeta)=\iota_{\BDynM}(\psi_2,\phi)$$
for all $\psi_1\leq\psi_2$ so that $(\psi,\phi) \in \BDynM$ implies $(-\infty,\psi] \times \{\phi\} \subset \BDynM$.
Thus there exists a function $\BDynMH$ such that $(\psi,\phi) \in \BDynM \Leftrightarrow \psi \leq \BDynMH(\phi)$. Concavity and upper semi-continuity of $\BDynMH$ follow from convexity and closedness of $\BDynM$.
Now $\CDynM(\rho,0)= 0$ for $\rho \geq 0$ implies
$$\iota_{\BDynM}(\psi,\phi)=\sup_{\rho\geq0,\zeta}\psi\rho+\phi\zeta-\CDynM(\rho,\zeta)\geq\sup_{\rho\geq0}\psi\rho+\phi\cdot0-\CDynM(\rho,0)=\begin{cases}0&\text{if }\psi\leq0\\\infty&\text{else}\end{cases}$$
so that $\BDynMH\leq 0$.
Furthermore, $\iota_{\BDynM}(0,0)=\sup_{\rho,\zeta}-\CDynM(\rho,\zeta)=0$ so that $\BDynMH(0)=0$.
The monotonicity properties now follow from $\BDynMH(0)=0$, $\BDynMH \leq 0$, and concavity.
Next, by the assumptions on $\CDynM$ it is its own convex envelope. Therefore $\CDynM = \iota_{\BDynM}^\ast$ and in particular
\begin{align*}
0 < \CDynM(0,1) & = \sup \{ \phi \ |\ \BDynMH(\phi) > -\infty \}, &
0 < \CDynM(0,-1) & = \sup \{ \phi \ |\ \BDynMH(-\phi) > -\infty \}.
\end{align*}
Finally, note that the left and the right derivative of $\BDynMH$ in $0$ exist due to concavity
and are given by the monotone limits $\BDynMH^{l/r}(0)=\lim_{\varepsilon\searrow0}\frac{\BDynMH(\mp\varepsilon)}{\mp\varepsilon}$.
We show $\BDynMH^{r}(0)=0$ (the other equality follows analogously).
For a contradiction assume the existence of $C>0$ such that for all $\varepsilon>0$ we have $\frac{\BDynMH(\varepsilon)}\varepsilon<-C$.
Thus, for any $\rho,\zeta>0$ with $\zeta<C\rho$ we have
\begin{multline*}
\CDynM(\rho,\zeta)
=\iota_{\BDynM}^\ast(\rho,\zeta)
=\sup_{\phi,\,\psi\leq\BDynMH(\phi)}\psi\rho+\phi\zeta
=\sup_{\phi\geq0,\,\psi\leq\BDynMH(\phi)}\psi\rho+\phi\zeta\\
\leq\sup_{\phi\geq0,\,\psi\leq-C\phi}\psi\rho+\phi\zeta
=\sup_{\phi\geq0}\phi(\zeta-C\rho)
=0\,,
\end{multline*}
which contradicts the positivity of $\CDynM$ for $\rho,\zeta>0$.
\end{proof}

For an illustration of the sets $\BDynM$ and $\BStat$ as well as the functions $\BDynMH$ and $\HStat$ see Figure~\ref{fig:involvedFunctions}.

Now we turn to the structure of $\BStat$. As described above, for fixed $\phi(0,\cdot)$ the value of $\phi(1,\cdot)$ will intuitively try to be as large as the constraint $(\partial_t \phi, \phi) \in \BDynM$ allows. So, ignoring regularity, from \thref{thm:DynConjugateSet} we infer that $\phi(1,x)$ will be given by the solution to the differential equation $\partial_t \phi(t,x)=\BDynMH(\phi(t,x))$ for initial value $\phi(0,x)$.
This upper bound is rigorously established in \thref{lem:Flow}.
We first consider the case where the function $\BDynMH$ introduced in \thref{thm:DynConjugateSet} satisfies an additional technical assumption. Functions not satisfying this assumption will be treated via an extra smoothing argument in \thref{lem:QKinkSmoothing}.

\begin{assumption}
	\thlabel{asp:QNoKink}
	Let $\zKink=\max\{ z \in \R \,|\, \BDynMH(z)=0 \}$. We assume that $\BDynMH$ is differentiable at $\zKink$ (left-differentiable if $\zKink\in\partial(\dom\BDynMH)$).
\end{assumption}
Note that functions $\BDynMH$ with $\zKink=0$ automatically satisfy \thref{asp:QNoKink} due to $\BDynMH'(0)=0$.

\begin{lemma}[Properties of flow]
	\thlabel{lem:Flow}
	For $s>0$ we define the flow of $\BDynM$ as $\Flow_s : \R \to [-\infty,\infty)$,
	\begin{align}
		\Flow_s(z) &= \sup \left\{ \phi(s) \ \middle|\ \phi \in C^1([0,s]),\, \phi(0)=z,\,(\partial_t \phi(t),\phi(t)) \in \BDynM \tn{ for } t \in [0,s] \right\}\nonumber\\
		&=\sup \left\{ \phi(s) \ \middle|\ \phi \in C^1([0,s]),\, \phi(0)=z,\,\partial_t \phi(t)\leq\BDynMH(\phi(t)) \tn{ for } t \in [0,s] \right\}\,,\label{eq:CMassFlow}
	\end{align}
	where the supremum of the empty set is $-\infty$.
	For $s=0$ we set $\Flow_0(z) = z-\iota_{\dom \BDynMH}(z)$.
	The flow satisfies the following properties for all $s\geq0$.
	\begin{enumerate}[(i)]
		\item If $\Flow_s(z) > -\infty$, the map $[0,s] \ni t \mapsto F_t(z)$ solves the initial value problem \label{item:FlowIVP}
		\begin{align}
			\label{eqn:FlowIVP}
			\partial_t \phi(t) & = \BDynMH(\phi(t)) \tn{ for } t \in [0,s], &
			\phi(0) & = z\,.
		\end{align}
		\item If $\Flow_s(z) > -\infty$, then for all $t \in [0,s]$ we have $\Flow_t(z) > -\infty$ and $\Flow_s(z)=\Flow_t(\Flow_{s-t}(z))$, and the map $[0,s] \ni t \mapsto \Flow_t(z)$ is contained in $\{ \phi \in C^1([0,s]) \,\colon\, (\partial_t \phi(t),\phi(t)) \in \BDynM \tn{ for } t \in [0,s]\}$.
			\label{item:FlowBasic}
		\item $\Flow_s(z) \leq z$.
			\label{item:FlowDiagonal}
		\item $\Flow_s$ is strictly increasing on its domain.
			\label{item:FlowIncreasing}
		\item $\Flow_s$ is concave.
			\label{item:FlowSpaceConcave}
		\item For $z \in \dom \BDynMH \cap [0,\infty)$, $\Flow_s(z) \geq 0$ with $\Flow_s(z)>0$ if $z>0$.
			\label{item:FlowPositive}
		\item For $z \in \dom \BDynMH \cap [0,\infty)$ the map $[0,\infty) \ni t \mapsto \Flow_t(z)$ is convex.
			\label{item:FlowTimeConvex}
		\item $\Flow_s$ is differentiable in $0$ with $\Flow_s'(0)=1$.
			\label{item:FlowZeroDerivative}
		\item $\Flow_s$ is non-expansive on $\dom \BDynMH \cap [0,\infty)$.
			\label{item:FlowContraction}
		\item $\Flow_s$ is locally Lipschitz differentiable on $\inter \dom \BDynMH \cap [0,\infty)\setminus\zKink$ with $\Flow_s'(z)=1$ for $z<\zKink$ and $\Flow_s'(z)=\frac{\BDynMH(\Flow_s(z))}{\BDynMH(z)}$ for $z>\zKink$.
			Under \thref{asp:QNoKink} it is differentiable on all of $\inter\dom \BDynMH \cap [0,\infty)$.
			\label{item:FlowDerivative}
	\end{enumerate}
\end{lemma}
\begin{proof}
	\textbf{(\ref{item:FlowIVP})}
	We first show that a solution to \eqref{eqn:FlowIVP} exists when $\Flow_s(z)>-\infty$. We will construct such a solution explicitly.
  Let
  \begin{align*}
    S_0 & = \BDynMH^{-1}(\{0\}), &
    S_+ & = (\dom \BDynMH \cap [0,\infty)) \setminus S_0, &
    S_- & = (\dom \BDynMH \cap (-\infty,0]) \setminus S_0
  \end{align*}
  be a partition of $\dom \BDynMH$ into three connected components.
  If $z\in S_0$ the solution to \eqref{eqn:FlowIVP} is given by $\phi(t)=z$ for all $t\geq0$. If $z\in S_+\cup S_-$,
  let $\psi \in C^1([0,s])$ be a suitable candidate in \eqref{eq:CMassFlow}. The function $\psi$ must be nonincreasing due to $\BDynMH\leq0$.
  Without loss of generality we may assume that $\psi([0,s])$ is contained in either $S_+$ or $S_-$.
  Indeed, when $z\in S_-$, then $\psi(t) \leq \psi(0)=z$ must be contained in $S_-$ for all $t\in[0,s]$ as well.
  On the other hand, when $z \in S_+$, we may pick ${\psi}(t) = \zKink + (z-\zKink) \cdot \exp\big(\tfrac{\BDynMH(z)}{z-\zKink} t\big)$ where $\zKink = \max S_0$. This is feasible in \eqref{eq:CMassFlow} since $\BDynMH(\zKink)=0$ by upper semi-continuity and $\BDynMH$ is concave.
  Now for $y \in \R$ consider the integral
  \begin{align*}
    I_z(y) & = \int_z^y \frac{1}{\BDynMH(x)} \,\d x\,.
  \end{align*}
  This is well-defined and finite when $z,y \in S_+$ or when $z,y \in S_-$.
  The function $I_z$ is strictly decreasing and thus invertible. Denote its inverse by $\phi_z$.
  Due to $I_z(z)=0$ and
  \begin{align*}
    \infty > I_z(\psi(s)) = \int_{\psi(0)}^{\psi(s)} \frac{1}{\BDynMH(x)} \,\d x
    = \int_0^s \frac{\psi'(t)}{\BDynMH(\psi(t))} \,\d t \geq s
  \end{align*}
  the inverse $\phi_z$ is well-defined at least on $[0,s]$.
  Furthermore, $\phi_z(0)=z$ and
  \begin{align*}
    \phi_z'(t) = \frac{1}{\partial_y I_z(\phi_z(t))} = \BDynMH(\phi_z(t))
  \end{align*}
  so that $\phi=\phi_z$ indeed solves the initial value problem \eqref{eqn:FlowIVP}.
  
  Now let $s>0$ and $z\in\R$ with $\Flow_s(z)>-\infty$ and denote the solution to \eqref{eqn:FlowIVP} by $\phi_z$.
  Since $\phi_z$ is feasible in \eqref{eq:CMassFlow}, $\Flow_s(z)\geq \phi_z(s)$.
  To show the reverse inequality, consider a competitor $\psi \in C^1([0,s])$ feasible for \eqref{eq:CMassFlow}.
  If $z\in S_0$, then $\psi(s)\leq z=\phi_z(s)$ since $\psi$ is decreasing.
  If on the other hand $z\in S_+\cup S_-$, then as above we may assume $\psi([0,s])$ to be contained in either $S_+$ or $S_-$.
  Due to $I_z(\phi_z(s))=s$ and $I_z(\psi(s))\geq s$, the monotonicity of $I_z$ implies $\psi(s)\leq I_z^{-1}(s)=\phi_z(s)$.
  Summarizing, $\Flow_s(z)\leq \phi_z(s)$.
	
  \textbf{(\ref{item:FlowBasic})}
  This is a direct consequence of (\ref{item:FlowIVP}).
  
  \textbf{(\ref{item:FlowDiagonal})}
  If $\Flow_s(z)=-\infty$ this holds trivially. Otherwise, $t \mapsto \Flow_t(z)$ is nonincreasing on $t \in [0,s]$ due to (\ref{item:FlowIVP}) and $\BDynMH\leq0$ so that $\Flow_s(z)\leq\Flow_0(z)=z$.
  
  \textbf{(\ref{item:FlowIncreasing})}
  This is a standard property of solutions to scalar ordinary differential equations.
  Indeed, for $z\in S_0$ this follows from $\Flow_s(z)=z$, and for $z_1,z_2\in S_-\cup S_+$ with $z_1<z_2$ assume $\Flow_s(z_1)\geq\Flow_s(z_2)$. Then with the above notation,
  \begin{equation*}
  s=I_{z_1}(\Flow_s(z_1))=\int_{z_1}^{\Flow_s(z_1)}\frac1{\BDynMH(x)}\,\d x<\int_{z_2}^{\Flow_s(z_2)}\frac1{\BDynMH(x)}\,\d x=I_{z_2}(\Flow_s(z_2))=s\,,
  \end{equation*}
  which yields a contradiction.
		
	\textbf{(\ref{item:FlowSpaceConcave})}
	Let $z_1, z_2 \in \dom \Flow_s$ and let $\phi_1$, $\phi_2$ be the corresponding solutions to \eqref{eqn:FlowIVP}. For any $\lambda \in [0,1]$ let $z_\lambda=\lambda \cdot z_1 + (1-\lambda) \cdot z_2$ and $\phi_\lambda = \lambda \cdot \phi_1 + (1-\lambda) \cdot \phi_2$. By convexity of $\BDynM$ one finds that $(\partial_t \phi_\lambda(t),\phi_\lambda(t)) \in \BDynM$ for $t \in [0,s]$. Hence $\phi_\lambda$ is feasible in \eqref{eq:CMassFlow} for $\Flow_s(z_\lambda)$ which implies $\Flow_s(z_\lambda) \geq \phi_\lambda(s) = \lambda \cdot \Flow_s(z_1) + (1-\lambda) \cdot \Flow_s(z_2)$.
			
	\textbf{(\ref{item:FlowPositive})}
	The function ${\psi}(t) = \zKink + (z-\zKink) \cdot \exp\big(\tfrac{\BDynMH(z)}{z-\zKink} t\big)$ from above is feasible in \eqref{eq:CMassFlow} so that $\Flow_s(z) \geq \psi(s)\geq 0$ and $\psi(s)>0$ if $z>0$.%

	By (\ref{item:FlowIVP}) and (\ref{item:FlowPositive}), $t\mapsto\Flow_t(z)$ is decreasing and nonnegative. Thus, since $\BDynMH$ is decreasing on $[0,\infty)$, the derivative $t\mapsto\partial_t\Flow_t(z)=\BDynMH(\Flow_t(z))$ is increasing.%
	
	\textbf{(\ref{item:FlowZeroDerivative})}
	We show that the right derivative of $\Flow_s$ in $0$ is $1$ (the argument for the left derivative is analogous).
  If $\inf S_+>0$ this follows directly from $\Flow_s(z)=z$ on $S_0$.
  If $\inf S_+=0$, since $\Flow_s$ is concave its right derivative in $0$ is given by the monotone limit $\lim_{\veps \searrow 0} \tfrac{\Flow_s(\veps)}{\veps}$, which cannot exceed $1$ due to $\Flow_s(\veps)\leq \veps$.
	For a contradiction we assume that the right derivative is strictly smaller than $1$, that is, there is some $C<1$ with $\Flow_s(\veps)\leq C \cdot \veps$ for all $\veps>0$. Since $\Flow_s(\veps)=\phi_\veps(s)$ this implies
	\begin{align*}
		s = I_\veps(\phi_\veps(s)) \geq I_\veps(C \cdot \veps) = \int_{\veps}^{C \veps}
			\frac{1}{\BDynMH(x)} \,\d x \geq \frac{(C-1) \cdot \veps}{\BDynMH(\veps)}
			= \frac{\veps}{\BDynMH(\veps)} \cdot (C-1)\,.
	\end{align*}
	Since $\BDynMH'(0)=0$ the right-hand side diverges to $\infty$ as $\veps \to 0$ which is a contradiction.
	
  \textbf{(\ref{item:FlowContraction})}
	This follows from the concavity and monotonicity of $\Flow_s$ together with $\Flow_s'(0)=1$.
	
	\textbf{(\ref{item:FlowDerivative})}
	For $z \in \inter S_0$ the statement is trivial.
	For $z \in \inter S_+$ fix some arbitrary $a\in S_+$ with $a>z$.
	Then by (\ref{item:FlowIVP}) we have $z=\phi_a(I_a(z))=\Flow_{I_a(z)}(a)$ and thus by (\ref{item:FlowBasic}) $\Flow_s(z) = \Flow_s(\Flow_{I_a(z)}(a)) = \Flow_{I_a(z)+s}(a) = \phi_a(I_a(z)+s)$.
	Taking the derivative we find
	\begin{align*}
		\Flow_s'(z) = \frac{\phi_a'(I_a(z)+s)}{\BDynMH(z)} = \frac{\BDynMH(\Flow_s(z))}{\BDynMH(z)}\,,
	\end{align*}
	which is locally Lipschitz on $\inter S_+$ since the concave functions $\BDynMH$ and $\Flow_s$ are locally Lipschitz on their domain interiors.
	Finally, at $\zKink=\max S_0$ the left derivative is $1$ due to $\Flow_s(z)=z$ on $S_0$.
	If $\zKink\in\inter\dom\BDynMH$, using \thref{asp:QNoKink} one shows that the right derivative is $1$ as well in the same way as for (\ref{item:FlowZeroDerivative}).
\end{proof}

\begin{lemma}[Inverse flow]\thlabel{thm:FlowInv}
	On $\Flow_s(\dom \Flow_s)$ let $\FlowInv_s$ be the inverse of $\Flow_s$ (which is well-defined by \thref{lem:Flow}\eqref{item:FlowIncreasing}).
	Let $\BDynMHTild : z \mapsto \BDynMH(-z)$ and let $\tilde{\Flow}_s$ be the flow \eqref{eq:CMassFlow} associated with $\BDynMHTild$. Then
	\begin{align}
		\label{eq:FlowInvRelation}
		\FlowInv_s(z)=-\tilde{\Flow}_s(-z)\,.
	\end{align}
\end{lemma}
\begin{proof}
	By \thref{lem:Flow}\eqref{item:FlowIVP} one finds for $z \in \Flow_s(\dom \Flow_s)$ that $z=\phi(s)$ where $\phi$ solves the initial value problem
	\begin{equation*}
		\partial_t \phi(t) = -\BDynMH(\phi(t)) \tn{ for } t \in [0,s], \qquad
		\phi(0)=\FlowInv_s(z)\,. \\
	\end{equation*}
	Consequently the function $\tilde{\phi}(t)=-\phi(s-t)$ solves the initial value problem
	\begin{equation*}
		\partial_t \tilde{\phi}(t) = \BDynMH(-\tilde{\phi}(t)) = \BDynMHTild(\tilde{\phi}(t)) \tn{ for } t \in [0,s], \qquad
		\tilde{\phi}(0)=-z\,,
	\end{equation*}
	so that $\FlowInv_s(z)=\phi(0)=-\tilde\phi(s)=-\tilde\Flow_s(-z)$.
\end{proof}
Relation \eqref{eq:FlowInvRelation} allows to translate results of \thref{lem:Flow} to $\FlowInv_s$. Throughout this section we will use that arguments involving $\Flow_s$ can be applied to $\FlowInv_s$ in an analogous way.

The flow $\Flow_s$ will be crucial in constructing feasible candidates for the dynamic dual problem \eqref{eq:DynamicDual}. For this we need differentiability of $F_s$ on $\dom \BDynMH \cap [0,\infty)$ which requires the additional \thref{asp:QNoKink} (see \thref{lem:Flow}\eqref{item:FlowDerivative}).
Since this assumption may not always hold we will have to approximate $\BDynMH$ by a suitable sequence of functions.

\begin{lemma}[Approximation of $\BDynMH$]
	\thlabel{lem:QKinkSmoothing}
	Assume that $\BDynMH$ does not satisfy \thref{asp:QNoKink}, that is, $0<\zKink\in\inter(\dom\BDynMH)$ and the right derivative $\BDynMH^r(\zKink)$ of $\BDynMH$ at $\zKink$ is finite but strictly negative, $\BDynMH^r(\zKink)=-C<0$ for some $C>0$.
	Then for $\veps \in (0,\zKink/2]$ the function
	\begin{align*}
		\BDynMHeps(z) = \begin{cases}
			\BDynMH(z) & \tn{if } z \leq \zKink-\veps \\
			-\frac{C}{2\,\veps} ( z-\zKink+\veps)^2 & \tn{if } z \in [\zKink-\veps,\zKink] \\
			\BDynMH(z) - \frac{C\,\veps}{2} & \tn{if } z \geq \zKink
			\end{cases}
	\end{align*}
	possesses the same properties from \thref{thm:DynConjugateSet} as $\BDynMH$ and satisfies \thref{asp:QNoKink}.
	In addition, the set
	\begin{align*}
		\BDynMeps = \left\{ (\psi,\phi) \in \R^2 \ \middle|\ \psi \leq \BDynMHeps(\phi) \right\}
	\end{align*}
	satisfies $(1-\veps/\zKink) \cdot \BDynM \subset \BDynMeps \subset \BDynM$.
\end{lemma}
\begin{proof}
	The properties of $\BDynMH^\veps$ are straightforward to check.
	Moreover, since $\BDynMH^\veps(z) \leq \BDynMH(z)$ one finds $\BDynM^\veps \subset \BDynM$.
	It remains to show $\lambda \cdot \BDynM \subset \BDynM^\veps$ for $\lambda=(1-\veps/\zKink) \in [1/2,1)$ or equivalently $\BDynMH(z/\lambda) \leq \BDynMH^\veps(z)/\lambda$ for $z \in \R$.
	This holds true for $z\in(-\infty,\zKink-\veps]$ since $\BDynMH^\veps(z)=\BDynMH(z)$ is nondecreasing and concave on this interval.
	For $z\in[\zKink-\veps,\zKink]=[\lambda\zKink,\zKink]$ one finds
	\begin{align*}
		\BDynMH(z/\lambda) & \leq
			-C \cdot (z/\lambda-\zKink) \leq -C\frac{\lambda\cdot (z/\lambda-\zKink)^2}{2 \cdot \veps}
			= \BDynMH^\veps(z)/\lambda\,.
	\end{align*}
	Finally, for $z \in \dom \BDynMH = \dom \BDynMH^\veps$, $z \geq \zKink$ one has
	\begin{multline*}
		\BDynMH(z/\lambda) \leq \BDynMH(z) + \BDynMH^r(z) \cdot (1/\lambda-1) \cdot z 
			= \frac{\BDynMH(z)}{\lambda} + (1/\lambda-1) \cdot \left(\BDynMH^r(z) \cdot z - \BDynMH(z) \right) \\
		\leq \frac{1}{\lambda} \cdot \left[
			\BDynMH(z) + (1-\lambda) \cdot \left(\BDynMH^r(\zKink) \cdot \zKink - \BDynMH(\zKink) \right)
				\right]
			= \frac{1}{\lambda} \cdot \left[ \BDynMH(z) - C \cdot \veps
				\right]\,,
	\end{multline*}
	where we have used that $z \mapsto \BDynMH^r(z) \cdot z - \BDynMH(z)$ is decreasing for $z \geq 0$.
	For $z \neq \dom \BDynMH$ the statement is trivial.
\end{proof}

\begin{remark}[Approximation of $\BDynMH$]\thlabel{rem:QKinkSmoothing}
If $\BDynMH(-\cdot)$ does not satisfy \thref{asp:QNoKink},
then with the same argument we can find some $\BDynMHeps$ possessing the same properties from \thref{thm:DynConjugateSet} as $\BDynMH$
such that $\BDynMHeps(-\cdot)$ satisfies \thref{asp:QNoKink} and $(1+\veps/\zKinkN) \cdot \BDynM \subset \BDynMeps \subset \BDynM$ for $\zKinkN=\min\BDynMH^{-1}(\{0\})$.
Likewise, if neither $\BDynMH(-\cdot)$ nor $\BDynMH$ satisfy \thref{asp:QNoKink}, then we can construct $\BDynMHeps$ possessing the same properties from \thref{thm:DynConjugateSet} as $\BDynMH$
as well as \thref{asp:QNoKink} for $\BDynMHeps$ and $\BDynMHeps(-\cdot)$ and $(1-C\veps) \cdot \BDynM \subset \BDynMeps \subset \BDynM$ for some fixed $C>0$.
\end{remark}

We can now use the flows $\Flow_1$ and $\FlowInv_1$ to reconstruct for suitable $(\alpha,\beta) \in \BStatPre$ a corresponding candidate function $\phi$ in \eqref{eq:BStatPre} with $\phi(0)=-\alpha$, $\phi(1)=\beta$ (\thref{lem:LocalInterpolant}). This allows a more explicit description of the set $\BStatPre$ via $\Flow_1$ (\thref{cor:BStatPreForm}), which we will later need to analyse the transition from $\BStatPre$ to $\BStat=\ol{\BStatPre} + (-\infty,0]^2$ in \thref{prop:W1DynamicStaticEquivalence}.
Then, in \thref{lem:DynamicDualConstruction} we establish that for $(\alpha,\beta) \in C^1(\Omega) \cap \Lip(\Omega)$ with $(\alpha(x),\beta(x)) \in \BStatPre$ for all $x \in \Omega$, applying for each $x \in \Omega$ the interpolation of \thref{lem:LocalInterpolant} yields a dynamic dual feasible candidate $\phi$ for \eqref{eq:DynamicDual} with $\phi(0,\cdot)=-\alpha$, $\phi(1,\cdot)=\beta$.
The main ingredient of this interpolation are the two auxiliary functions examined in the next \thnameref{lem:LocInterp}, which are careful combinations of the flow functions.
With these preliminaries in place we can then provide the proof of \thref{prop:W1DynamicStaticEquivalence}.

\newcommand{\LocInterp}{g}
\newcommand{\LocInterpInv}{g^{\tn{inv}}}
\begin{lemma}
	\thlabel{lem:LocInterp}
	We introduce the two auxiliary functions
	\begin{align*}
		\LocInterp & : [0,1] \times \dom \BDynMH  \ni (t,z) \mapsto \begin{cases}
			\Flow_t(z)-t \cdot \Flow_1(z) & \tn{if } z>0, \\
			0 & \tn{if } z=0, \\
			(1-t) \cdot z & \tn{if } z<0,
			\end{cases} \\
		\LocInterpInv & : [0,1] \times \dom \BDynMH \ni (t,z) \mapsto \begin{cases}
			t \cdot z & \tn{if } z>0, \\
			0 & \tn{if } z=0, \\
			\FlowInv_{1-t}(z)-(1-t) \cdot \FlowInv_1(z) & \tn{if } z<0.
			\end{cases}
	\end{align*}
	These have the following properties:
	\begin{enumerate}[(i)]
		\item $\LocInterp$ and $\LocInterpInv$ are continuous and differentiable in $t$.
			\label{item:LocalInterpDomain}
		\item $\LocInterp(0,z)=z$, $\LocInterp(1,z)=0$, $\LocInterpInv(0,z)=0$, $\LocInterpInv(1,z)=z$.
			\label{item:LocalInterpBoundary}
		\item Under \thref{asp:QNoKink} for $\BDynMH$ and $\BDynMH(-\cdot)$ (see \thref{rem:QKinkSmoothing}) $\LocInterp$ and $\LocInterpInv$ are differentiable in $z$.
			\label{item:LocalInterpZDerivative}
		\item Under the same assumptions, $\partial_z \LocInterp(t,z) \in [0,1-t]$ and $\partial_z \LocInterpInv(t,z) \in [0,t]$.
			\label{item:LocalInterpConvexity}
	\end{enumerate}
\end{lemma}%

\begin{proof}
	We prove the statements for $\LocInterp$. All arguments for $\LocInterpInv$ work in full analogy by using \thref{thm:FlowInv}.
	
	\textbf{(\ref{item:LocalInterpDomain})} By virtue of \thref{lem:Flow}\eqref{item:FlowDiagonal} and \eqref{item:FlowPositive} we have $\dom \Flow_t \supset \dom \BDynMH \cap [0,\infty)$. Therefore, $\dom \LocInterp = [0,1] \times \dom \BDynMH$.
	Since $\Flow_t(0)=0$ and $\Flow_t$ is continuous on its domain, $\LocInterp$ is continuous in $z=0$ and on all of $\dom \BDynMH$ for all $t \in [0,1]$.
	Finally, differentiability of $\LocInterp$ in $t$ for $z\leq 0$ is trivial, for $z>0$ it is provided by \thref{lem:Flow} \eqref{item:FlowIVP}.
	
	\textbf{(\ref{item:LocalInterpBoundary})} This follows immediately from the definition, using the convention $\Flow_0(z)=z$ on $\dom \BDynMH$ (\thref{lem:Flow}).
	
	\newcommand{\LocInterpAux}{k}
	\textbf{(\ref{item:LocalInterpZDerivative})}
	With \thref{asp:QNoKink}, due to \thref{lem:Flow}\eqref{item:FlowZeroDerivative} and \eqref{item:FlowDerivative} $\LocInterp$ is differentiable in $z$ for $z\neq 0$, and the one-sided derivatives coincide in $z=0$.
	
	\textbf{(\ref{item:LocalInterpConvexity})} For $z \leq 0$ the result is clear.
	The derivative $\partial_z \Flow_s(z)$ for $z\geq 0$ is given by \thref{lem:Flow}\eqref{item:FlowDerivative}.
	For $z \in (0,\zKink]$ (see \thref{asp:QNoKink}) we have $\partial_z \Flow_s(z)=1$ so that the result is also established.
	Now fix some $z>\zKink$ (thus $\BDynMH(z)<0$). Then $\partial_z \LocInterp(t,z)=\LocInterpAux(\Flow_t(z))-t \cdot \LocInterpAux(\Flow_1(z))$ with $\LocInterpAux(y) = \tfrac{\BDynMH(y)}{\BDynMH(z)}$.
	For $y \geq 0$, since $\BDynMH$ is negative, concave, and decreasing for positive arguments, $\LocInterpAux$ is positive, convex, and increasing. The map $t \mapsto \Flow_t(z)$ is convex and decreasing by \thref{lem:Flow}\eqref{item:FlowIVP} and \eqref{item:FlowTimeConvex}. Therefore, $\LocInterpAux(\Flow_t(z)) \geq \LocInterpAux(\Flow_1(z)) \geq t \cdot \LocInterpAux(\Flow_1(z))$ and thus $\partial_z \LocInterp(t,z) \geq 0$.
	In addition, the composition $t \mapsto \LocInterpAux(\Flow_t(z))$ is convex and thus $\LocInterpAux(\Flow_t(z)) \leq (1-t) \cdot \LocInterpAux(\Flow_0(z)) + t \cdot \LocInterpAux(\Flow_1(z))$ so that $\partial_z \LocInterp(t,z) \leq (1-t) \cdot \LocInterpAux(z)=(1-t)$.
\end{proof}

\begin{lemma}
	\thlabel{lem:LocalInterpolant}
	For $(\alpha,\beta) \in \R^2$ with $(-\alpha,\beta) \in (\dom \BDynMH)^2$ and $\beta \leq \Flow_1(-\alpha)$ the function%
	\begin{align}
		\label{eq:LocalInterpolantNew}
		\phi_{\alpha,\beta} : [0,1] \to \R, \qquad t \mapsto \LocInterp(t,-\alpha) + \LocInterpInv(t,\beta)
	\end{align}%
	satisfies $\alpha=-\phi_{\alpha,\beta}(0)$, $\beta=\phi_{\alpha,\beta}(1)$, $\phi_{\alpha,\beta} \in C^1([0,1])$, and $(\partial_t \phi_{\alpha,\beta}(t),\phi_{\alpha,\beta}(t)) \in \BDynM$ for $t \in [0,1]$.
\end{lemma}
\begin{proof}
		The first three properties follow directly from \thref{lem:LocInterp}.
		It remains to show $\partial_t \phi_{\alpha,\beta}(t)\leq\BDynMH(\phi_{\alpha,\beta}(t))$ for all $t \in [0,1]$. If $\alpha\leq0\leq\beta$, then by \thref{lem:Flow}\eqref{item:FlowIVP}
		\begin{equation*}
		\partial_t \phi_{\alpha,\beta}(t)
		\!=\BDynMH(\Flow_t(-\alpha))+\beta-\Flow_1(-\alpha)
		\!\leq\BDynMH(\Flow_t(-\alpha))
		\!\leq\BDynMH(\Flow_t(-\alpha)+t(\beta-\Flow_1(-\alpha)))
		\!=\BDynMH(\phi_{\alpha,\beta}(t)),
		\end{equation*}
		where we used that $\BDynMH$ is monotonously decreasing for positive arguments.
		The result follows analogously for $\beta\leq0\leq\alpha$ using \thref{thm:FlowInv}.
		Due to $\Flow_t(-\alpha)\leq -\alpha$ (\thref{lem:Flow}\eqref{item:FlowDiagonal}) one cannot have $\alpha,\beta>0$ so that it remains to consider the case $\alpha,\beta<0$.
		Note first that $\phi_\alpha : t \mapsto \LocInterp(t,-\alpha) = \Flow_t(-\alpha)-t \, \Flow_1(-\alpha)$ and $\phi_\beta : t \mapsto \LocInterpInv(t,\beta)=\FlowInv_{1-t}(\beta)-(1-t)\,\FlowInv_{1}(\beta)$
		both satisfy $\partial_t \phi\leq\BDynMH(\phi)$ by the same argument as above.
		Furthermore one has $\phi_\beta(t) \leq 0 \leq \phi_\alpha(t)$ by \thref{lem:LocInterp}\eqref{item:LocalInterpConvexity} and $\LocInterp(t,0)=\LocInterpInv(t,0)=0$ and thus $\phi_{\alpha,\beta}(t) \in [\phi_\beta(t),\phi_\alpha(t)]$.
		Now due to its concavity, $\BDynMH$ attains its minimum on the interval $[\phi_\beta(t),\phi_\alpha(t)]$ at the interval boundary so that with the negativity of $\BDynMH$ we obtain
		\begin{multline*}
			\partial_t \phi_{\alpha,\beta}(t) \leq \BDynMH(\phi_\beta(t)) + \BDynMH(\phi_\alpha(t)) \leq \min\{ \BDynMH(\phi_\beta(t)) , \BDynMH(\phi_\alpha(t)) \}\\
			= \min\{\BDynMH(z)\,|\,z\in[\phi_\beta(t),\phi_\alpha(t)]\}
				\leq \BDynMH(\phi_{\alpha,\beta}(t))\,.\qedhere
		\end{multline*}
\end{proof}

The previous result allows a more explicit description of the set $\BStatPre$ from \eqref{eq:BStatPre}.
\begin{proposition}
	\thlabel{cor:BStatPreForm}
	The set $\BStatPre$ can be characterized by
	\begin{align}
		\label{eq:BStatPreFlowForm}
		\BStatPre  & = \left\{ (\alpha,\beta) \in \R^2 \,\middle|\, 
			(-\alpha,\beta) \in (\dom \BDynMH)^2,\,
			\beta \leq \Flow_1(-\alpha) \right\}\,, \\
		\label{eq:BStatPreClosedFlowForm}
		\ol{\BStatPre}  & = \left\{ (\alpha,\beta) \in \R^2 \,\middle|\, 
			(-\alpha,\beta) \in (\ol{\dom \BDynMH})^2,\,
			\beta \leq \ol{\Flow_1}(-\alpha) \right\}\,, \\
		\intertext{where $\ol{\Flow_1}$ is the concave upper semi-continuous hull of $\Flow_1$ which is given by}
		\ol{\Flow_1}(z) & = \begin{cases}
			\Flow_1(z) & \tn{if } z \in \inter \dom \Flow_1, \\
			\sup \Flow_1(\dom \Flow_1) & \tn{if } z = \sup \dom \Flow_1, \\
			\inf \Flow_1(\dom \Flow_1) & \tn{if } z = \inf \dom \Flow_1.
			\end{cases}
	\end{align}
	Note that $\inf\Flow_1(\dom \Flow_1) = \inf \dom \BDynMH$.
\end{proposition}
\begin{proof}
	\textbf{(\ref{eq:BStatPre}) $\subset$ (\ref{eq:BStatPreFlowForm}):}
	For $(\alpha,\beta) \in \text{\eqref{eq:BStatPre}}$ by definition $(-\alpha,\beta) \in (\dom \BDynMH)^2$ and there is a feasible candidate $\phi$ in \eqref{eq:CMassFlow} for $\Flow_1(-\alpha)$ with $\phi(0)=-\alpha$ and $\phi(1)=\beta$, thus $\Flow_1(-\alpha)\geq \beta$.
	
	\textbf{(\ref{eq:BStatPreFlowForm}) $\subset$ (\ref{eq:BStatPre}):}	
	For $(\alpha,\beta) \in \text{\eqref{eq:BStatPreFlowForm}}$ the function $\phi_{\alpha,\beta}$ constructed in \thref{lem:LocalInterpolant} establishes that $(\alpha,\beta) \in \text{\eqref{eq:BStatPre}}$.
	
	Since $\Flow_1$ is concave it is continuous on its domain. Hence, $\ol{\Flow_1}$ only differs at the extreme points of the domain, and the values at those points follow from the monotonicity of $\Flow_1$.
\end{proof}

For $(\alpha,\beta) \in C^1(\Omega) \cap \Lip(\Omega)$ with $(\alpha(x),\beta(x)) \in \BStatPre$ for all $x \in \Omega$, the interpolation of \thref{lem:LocalInterpolant} at each point $x \in \Omega$ yields a function $\phi \in C^1([0,1] \times \Omega)$ such that $(\partial_t \phi(t,x),\phi(t,x)) \in \BDynM$ for all $(t,x)$. In the following Lemma we show that this pointwise interpolation also satisfies $\|\nabla\phi(t,x)\|\leq 1$ so that $\phi$ is feasible for the dynamic dual problem. This is a crucial part of proving $W_S(\rho_0,\rho_1) \leq W_D(\rho_0,\rho_1)$ in \thref{prop:W1DynamicStaticEquivalence}.

\begin{lemma}[Construction of feasible dynamic dual candidates]
	\thlabel{lem:DynamicDualConstruction}
	Assume $\BDynMH$ and $\BDynMHTild=\BDynMH(-\cdot)$ satisfy \thref{asp:QNoKink}.
	Let $\alpha, \beta \in C^1(\Omega) \cap \Lip(\Omega)$ with $(\alpha(x),\beta(x)) \in \BStatPre$ for all $x \in \Omega$.
	For given $(a,b) \in \BStatPre$ let $\phi_{a,b}$ be the corresponding function \eqref{eq:LocalInterpolantNew}.
	Then $\phi : [0,1] \times \Omega \to \R$,
	\begin{align*}
		\phi(t,x) & = \phi_{\alpha(x),\beta(x)}(t)
		= \LocInterp(t,-\alpha(x)) + \LocInterpInv(t,\beta(x))\,,
	\end{align*}
	satisfies $\phi \in C^1([0,1] \times \Omega)$, $\phi(0,\cdot) = -\alpha$ and $\phi(1,\cdot) = \beta$, $\|\nabla \phi(x,t)\| \leq 1$ and $(\partial_t \phi(t,x), \phi(t,x) ) \in \BDynM$ for all $(t,x) \in [0,1] \times \Omega$.
\end{lemma}

\begin{proof}
	By virtue of \eqref{eq:BStatPreFlowForm} $(\alpha(x),\beta(x))$ satisfy the requirements of \thref{lem:LocalInterpolant} for all $x \in \Omega$ so  $\phi$ and $\partial_t \phi$ are well-defined with $(\partial_t \phi(t,x),\phi(t,x)) \in \BDynM$ at each point $(t,x)$.
	With \thref{asp:QNoKink} for $\BDynMH$ and $\BDynMH(-\cdot)$ the function $\phi$ is differentiable in space with $\|\nabla\phi(t,x)\| \leq 1$ (\thref{lem:LocInterp} \eqref{item:LocalInterpZDerivative} and \eqref{item:LocalInterpConvexity}).
\end{proof}

\begin{proof}[Proof of Proposition \ref{prop:W1DynamicStaticEquivalence}]
	$\boldsymbol{W_D(\rho_0,\rho_1) \leq W_S(\rho_0,\rho_1)}$:
	Let $\phi \in C^1([0,1] \times \Omega)$ with $D_D(\phi) > -\infty$ (if no such $\phi$ exists, the inequality is trivial). Set $\alpha = -\phi(0,\cdot)$, $\beta = \phi(1,\cdot)$. Clearly, $\alpha$, $\beta \in \Lip(\Omega)$ and $(\alpha(x),\beta(x)) \in \BStat$ for all $x \in \Omega$. Hence, $D_S(\alpha,\beta) = D_D(\phi)$.
	
	$\boldsymbol{W_D(\rho_0,\rho_1) \geq W_S(\rho_0,\rho_1)}$:
	First we show that $\BStat$ from \eqref{eq:BStat} can be replaced by $\ol{\BStatPre}$ and subsequently $\BStatPre$ in \eqref{eq:StaticDualObjective} without changing the value of \eqref{eq:StaticDual}.
	
	\newcommand{\alphaMin}{\alpha_{\tn{min}}}
	\newcommand{\betaMin}{\beta_{\tn{min}}}
	Let $\alphaMin=-\sup \dom \Flow_1=-\sup \dom \BDynMH$ and $\betaMin=\inf \dom \FlowInv_1=\inf \dom \BDynMH$.
	Further, let $(\alpha,\beta)$ be static dual feasible with $D_D(\alpha,\beta)>-\infty$ for the set \eqref{eq:BStat}. Using the characterization \eqref{eq:BStatPreClosedFlowForm} of $\ol{\BStatPre}$, we deduce that for every $x \in \Omega$ the condition $(\alpha(x),\beta(x)) \in \BStat$ (as defined in \eqref{eq:BStat}) implies that there are some $(u,v) \in [0,\infty)^2$ such that $(-\alpha(x)-u,\beta(x)+v) \in \ol{\dom{\BDynMH}}^2$ and
	\begin{align*}
		\beta(x) \leq \beta(x) + v \leq \ol{\Flow_1}(-\alpha(x)-u)\,.
	\end{align*}
	Since $-\alpha(x) - u \leq \min\{-\alpha(x),-\alphaMin\}$ and $\ol{\Flow_1}$ is monotonically increasing on $\ol{\dom \BDynMH}$ this implies $\beta(x) \leq \ol{\Flow_1}(-\alpha(x))$ if $-\alpha(x) \in \ol{\dom \BDynMH}$ and $\beta(x) \leq \ol{\Flow_1}(-\alphaMin)$ otherwise. Now consider
	\begin{align*}
		\hat{\alpha}(x) & = \max\{ \alphaMin, \alpha(x)\}, &
		\hat{\beta}(x) & = \max\{ \betaMin, \beta(x)\}\,.		
	\end{align*}
	Then $\hat{\alpha}$ and $\hat{\beta}$ still lie in $\Lip(\Omega)$, they majorize $\alpha$ and $\beta$, and $(\hat{\alpha}(x),\hat{\beta}(x))\in \ol{\BStatPre}$ for all $x \in \Omega$. Thus replacing $\BStat$ by $\ol{\BStatPre}$ in \eqref{eq:StaticDualObjective} does not change the value of \eqref{eq:StaticDual}.
	Further, note that $\lambda \cdot \ol{\BStatPre} \subset \BStatPre$ for all $\lambda\in(0,1)$ (since $\Flow_1$ is concave and $\Flow_1(0)=0$, see \thref{lem:Flow}). Let $(\lambda_n)_{n=0}^\infty$ be a sequence in $(0,1)$ with $\lim_{n \to \infty} \lambda_n =1$. Then for any $(\hat{\alpha},\hat{\beta}) \in \Lip(\Omega)^2$ with $(\hat{\alpha}(x),\hat{\beta}(x)) \in \ol{\BStatPre}$ for all $x \in \Omega$, the sequence $(\lambda_n \cdot (\hat{\alpha},\hat{\beta}))_{n=0}^\infty$ in $\Lip(\Omega)^2$ lies pointwise in $\BStatPre$. Since $\Omega$ is compact the sequence converges uniformly to $(\hat{\alpha},\hat{\beta})$ and so the integral of $\lambda_n \cdot (\hat{\alpha},\hat{\beta})$ with respect to $\rho_0$ and $\rho_1$ converges to the one of $(\hat{\alpha},\hat{\beta})$. Hence, using $\BStatPre$ instead of $\BStat$ yields the same supremum.
	
	Next we show that in \eqref{eq:StaticDual} we may even restrict to $(\alpha,\beta)\in C^1(\Omega)^2$.
	Indeed, given $\alpha,\beta\in\Lip(\Omega)$ with $(\alpha(x),\beta(x))\in\BStatPre$ for all $x\in\Omega$, let
	\begin{equation*}
	\alpha_\delta=(1-\delta)\, \alpha -\delta\, M\,,\qquad
	\beta_\delta=(1-\delta) \, \beta -\delta\, M
	\end{equation*}
	for $M=\min(|\alphaMin|,|\betaMin|)/2$, $\delta \in (0,1)$,
	then also $\alpha_\delta,\beta_\delta\in\Lip(\Omega)$. By convexity of $\BStatPre$ and since $[-2 M,0]^2 \subset \BStatPre$ one finds $(\alpha_\delta(x),\beta_\delta(x))+[-\delta M,\delta M]^2=(1-\delta)\cdot(\alpha(x),\beta(x))+\delta\cdot[-2M,0]^2\subset\BStatPre$ for all $x\in\Omega$.
	By the density of $C^1(\Omega)\cap\Lip(\Omega)$ in $\Lip(\Omega)$ with respect to uniform convergence due to \thref{thm:StoneWeierstrass}
	there exist $\hat\alpha_\delta,\hat\beta_\delta\in C^1(\Omega)\cap\Lip(\Omega)$ with $|\hat\alpha_\delta-\alpha_\delta|,|\hat\beta_\delta-\beta_\delta|\leq\delta M$
	and thus $(\hat\alpha_\delta(x),\hat\beta_\delta(x))\in\BStatPre$ for all $x\in\Omega$.
	Furthermore, as $\delta\to0$ we have $(\hat\alpha_\delta,\hat\beta_\delta)\to(\alpha,\beta)$ uniformly so that $D_S(\hat\alpha_\delta,\hat\beta_\delta)\to D_S(\alpha,\beta)$.
	Hence the value of \eqref{eq:StaticDual} does not change if we restrict additionally to $(\alpha,\beta)\in C^1(\Omega)$.
	
	Now let $(\alpha,\beta) \in [C^1(\Omega) \cap \Lip(\Omega)]^2$ with $(\alpha(x),\beta(x)) \in \BStatPre$ for all $x \in \Omega$.
	If $\BDynMH$ and $\BDynMH(-\cdot)$ satisfy \thref{asp:QNoKink} then we can directly use \thref{lem:DynamicDualConstruction} to construct a $\phi \in C^1([0,1] \times \Omega)$ with $D_D(\phi)=D_S(\alpha,\beta)$ and thus establish equivalence.
	Otherwise we will now construct a sequence $(\phi_j)_{j=1}^\infty$ with $\phi_j \in C^1([0,1] \times \Omega)$ such that $\lim_{j \to \infty} D_D(\phi_j) = D_S(\alpha,\beta)$.
	Let $\veps_j>0$ be a sequence converging to zero as $j\to\infty$ and set $\lambda_j=1-C\veps_j$ for $C$ from \thref{rem:QKinkSmoothing}.
	Let $\Flow_{s,j}$ (and $\FlowInv_{s,j}$) be the flow (and its inverse) of the smoothed set $\BDynMepsN$ from \thref{rem:QKinkSmoothing}
	and let $\BStatPreEpsN$ be the corresponding set induced via \eqref{eq:BStatPre} or \eqref{eq:BStatPreFlowForm}.
	Set $(\alpha_j,\beta_j)=\lambda_j \cdot (\alpha,\beta)\in[C^1(\Omega)\cap\Lip(\Omega)]^2$, then $(\alpha_j(x),\beta_j(x))\in\BStat$ for all $x\in\Omega$.
	Hence, since $\lambda_j \to 1$ and $\Omega$ is compact, $(\alpha_j,\beta_j) \to (\alpha,\beta)$ uniformly and thus $D_S(\alpha_j,\beta_j) \to D_S(\alpha,\beta)$.
	As $\lambda_j \cdot \BDynM \subset \BDynMepsN$ one has $\lambda_j \cdot \BStatPre \subset \BStatPreEpsN$ and so $(\alpha_j(x),\beta_j(x)) \in \BStatPreEpsN$ for all $x \in \Omega$.
	Therefore, the function $\phi_j$, constructed from $\Flow_{s,j}$ and $\FlowInv_{s,j}$ via \thref{lem:DynamicDualConstruction} for $\BStatPreEpsN$, satisfies $D_D(\phi_j)=D_S(\alpha_j,\beta_j)$.
\end{proof}

\begin{proposition}
	\thlabel{cor:HStatFromDyn}
	The function $\HStat$ corresponding to $\BStat$ as given by \eqref{eq:BStat} is given by
	\begin{align*}
		\HStat(z) & = \begin{cases}
		\Flow_1(z) & \tn{if } z \in \inter \dom \Flow_1, \\
		\sup \Flow_1(\dom \Flow_1) & \tn{if } z \geq \sup \dom \Flow_1, \\
		\inf \Flow_1(\dom \Flow_1) & \tn{if } z = \inf \dom \Flow_1,\\
		-\infty & \tn{if } z < \inf \dom \Flow_1,
		\end{cases}
	\end{align*}
	where $F_1$ is the flow of $\BDynM$ at time $t=1$, as defined in \thref{lem:Flow}.
\end{proposition}
\begin{proof}
This follows from \eqref{eq:BStat} and \thref{cor:BStatPreForm}. The properties required for $\HStat$ in \thref{prop:StaticEquivalence} follow from \thref{lem:Flow}: $\HStat$ is increasing and concave by virtue of (\ref{item:FlowIncreasing},\ref{item:FlowSpaceConcave}), $\HStat$ is continuous on $\inter \dom \Flow_1$ by concavity of $\Flow_1$ and upper semi-continuous on $\R$ by construction due to monotonicity of $\Flow_1$, and $\HStat(z) \leq z$ follows from (\ref{item:FlowDiagonal}) and $\HStat(0)=0$, $\HStat'(0)=1$ from (\ref{item:FlowZeroDerivative}).
\end{proof}

\subsection{Dynamic primal optimizers from static primal optimizers}
\label{sec:primalRelation}
The equivalence of the dynamic and static models has been established in \thref{prop:W1DynamicStaticEquivalence} via equivalence of their dual formulations. In this section we clarify the primal interpretation.
\thref{prop:W1DynamicStaticEquivalencePrimal} provides the relation between the dynamic and the static mass change penalties $\CDynM$ and $\SimLocC$ from the primal formulations \thref{def:Dynamic} and \thref{def:StaticPrimal}.
Based on this, \thref{prop:DynamicPrimalOptimizers} shows how to construct optimzers of the primal dynamic problem \eqref{eq:DynamicProblem} from optimizers of the primal static problem \eqref{eq:StaticPrimalProblem}.

\begin{proposition}[Relation between dynamic and static mass change penalties]
	\thlabel{prop:W1DynamicStaticEquivalencePrimal}
	For a dynamic unbalanced transport problem \eqref{eq:DynamicProblem} with mass change penalty $\CDynM$, the mass change penalty $\SimLocC$ of the equivalent static problem \eqref{eq:StaticPrimalProblem} is given by
	\begin{equation}
		\label{eq:InducedSimLocC}
		\SimLocC(m_0,m_1) = \min \left\{ \int_0^1 \CDynM(\RadNik{\rho}{\mu},\RadNik{\zeta}{\mu})\,\d\mu
			\ \middle|\
			(\rho,\zeta) \in \mc{CES}(m_0,m_1) \right\}\,,
  \end{equation}
	where $\mu$ is any measure in $\measp([0,1])$ with $\rho$, $\zeta \ll \mu$ and
	\begin{multline*}
		\mc{CES}(m_0,m_1) = \left\{\vphantom{\int_0^1}
			(\rho,\zeta) \in \meas([0,1])^2 \ \middle|\right.\\
			\left.\int_0^1 (\partial_t \phi)\,\d\rho + \int_0^1 \phi\,\d\zeta = m_1 \cdot \phi(1) - m_0 \cdot \phi(0) \text{ for all }\phi\in C^1([0,1]) \right\}\,.
	\end{multline*}
	If \eqref{eq:InducedSimLocC} is finite, it admits a minimizer.
\end{proposition}
The set $\mc{CES}(m_0,m_1)$ encodes a weak `single point continuity equation' with no transport and only growth, in analogy to \thref{def:ContinuityEquation}.
Thus $\SimLocC$ represents the minimum cost of a time trajectory $(\rho,\zeta) \in \mc{CES}(m_0,m_1)$ along which mass change is penalized by $\CDynM$.
\begin{proof}
By \thref{prop:StaticEquivalence} we have $\iota_{\BStat}=\SimLocC^\ast$ and thus
  \begin{align*}
    \SimLocC(m_0,m_1) & = \sup \left\{ m_0\cdot \alpha + m_1 \cdot \beta - \iota_{\BStat}(\alpha,\beta) \right\} \\
      & = \sup \left\{
        m_1 \cdot \phi(1) - m_0 \cdot \phi(0) \ \middle|\
          \phi \in C^1([0,1]),\,
          (\partial_t \phi(t),\phi(t)) \in \BDynM \tn{ for all } t \in [0,1] \right\}\,,
  \end{align*}
where we have used the relation \eqref{eq:BStat} between the sets $\BStat$ and $\BDynM$ from \thref{prop:W1DynamicStaticEquivalence}.
The right-hand side is essentially a variant of the dynamic unbalanced transport problem \eqref{eq:DynamicDual} without transport, so what follows is a variant of the proof of \thref{prop:DynamicDual}. Abbreviating
  \begin{align*}
    G(\phi) & = m_1 \cdot \phi(1) - m_0 \cdot \phi(0), &
    F(\partial_t \phi, \phi) & = \int_0^1 \iota_{\BDynM}(\partial_t \phi(t),\phi(t))\,\d t, &
    A\phi & = (\partial_t \phi, \phi)\,,
  \end{align*}
  we just repeat the arguments from the proof of \thref{prop:DynamicDual} to obtain
  \begin{align*}
  \SimLocC(m_0,m_1) & = \sup \left\{ -F(A\phi)-G(-\phi)\ \middle|\ \phi \in C^1([0,1]) \right\}\\
  &=\min \left\{ F^\ast(\rho,\zeta) + G^\ast(A^\ast(\rho,\zeta)) \ \middle|\
      (\rho,\zeta) \in \meas([0,1])^2 \right\}\,,
  \end{align*}
  where $G^\ast(A^\ast(\rho,\zeta))=\iota_{\mc{CES}(m_0,m_1)}(\rho,\zeta)$ and $F^\ast(\rho,\zeta)=\int_0^1 \CDynM(\RadNik{\rho}{\mu},\RadNik{\zeta}{\mu})\,\d\mu$ due to $\iota_{\BDynM}=\CDynM^\ast$.
\end{proof}

In order to construct an optimal triple $(\rho,\omega,\zeta)$ for the dynamic problem \eqref{eq:DynamicProblem} from an optimal pair $(\pi_0,\pi_1)$ for the static problem \eqref{eq:StaticPrimalProblem},
we first introduce two basic building blocks, a pure mass flow between two points and a pure mass change at a single point.

\begin{definition}[Singular mass flow and change]
\thlabel{def:SingularMassFlowChange}
For given $x,y\in\Omega$ denote a shortest path between them by $p_{x,y}:[0,1]\to\Omega$ with $p_{x,y}(0)=x$ and $p_{x,y}(1)=y$.
By the \emph{singular mass flow from $x$ to $y$} we denote the measure $\omega_{x,y}\in\meas(\Omega)^n$ defined via
\begin{equation*}
\int_{\Omega}\psi\cdot\d\omega_{x,y}
=\int_{[0,1]}\psi(p_{x,y}(t))\cdot\dot p_{x,y}(t)\,\d t
\end{equation*}
for all $\psi\in C(\Omega;\R^n)$.
By the \emph{singular mass change from $m_0\geq 0$ to $m_1\geq 0$} we denote a pair $(\rho[m_0,m_1],\zeta[m_0,m_1])\in\meas([0,1])^2$ which minimizes \eqref{eq:InducedSimLocC}.
\end{definition}

\begin{proposition}[Construction of dynamic optimizers from static optimizers]\thlabel{prop:DynamicPrimalOptimizers}
Let $(\pi_0,\pi_1)$ be a pair of optimal couplings in \eqref{eq:StaticPrimalProblem}, abbreviate $\rho_0'=P_Y\,\pi_0$, $\rho_1'=P_X\,\pi_1$, and pick $\gamma\in\measp(\Omega)$ such that $\rho_0',\rho_1'\ll\gamma$.
For $x\in\Omega$ let $\big(\rho\big[\RadNik{\rho_0'}\gamma(x),\allowbreak\RadNik{\rho_1'}\gamma(x)\big],\allowbreak\zeta\big[\RadNik{\rho_0'}\gamma(x),\allowbreak\RadNik{\rho_1'}\gamma(x)\big]\big)\allowbreak\in\allowbreak\meas([0,1])^2$ be the singular mass change from \thref{def:SingularMassFlowChange} between $m_0=\RadNik{\rho_0'}\gamma(x)$ and $m_1=\RadNik{\rho_1'}\gamma(x)$.
An optimal triple $(\rho,\omega,\zeta)$ in \eqref{eq:DynamicProblem} is obtained by
\begin{align*}
\int_{[0,1]\times\Omega}\phi\,\d\rho
&=\int_\Omega\int_0^1\phi(t,x)\,\d\rho\big[\RadNik{\rho_0'}\gamma(x),\RadNik{\rho_1'}\gamma(x)\big](t)\d\gamma(x)\,,\\
\int_{[0,1]\times\Omega}\psi\cdot\d\omega
&=\int_\Omega\psi(0,\cdot)\cdot\d\omega_0+\int_\Omega\psi(1,\cdot)\cdot\d\omega_1\\
&\text{with }\int_\Omega\psi(i,\cdot)\cdot\d\omega_i=\int_{\Omega\times\Omega}\int_\Omega\psi(i,\cdot)\cdot\d\omega_{x,y}\d\pi_i(x,y)\,,\;i=0,1,\\
\int_{[0,1]\times\Omega}\phi\,\d\zeta
&=\int_\Omega\int_0^1\phi(t,x)\,\d\zeta\big[\RadNik{\rho_0'}\gamma(x),\RadNik{\rho_1'}\gamma(x)\big](t)\d\gamma(x)
\end{align*}
for all $\phi\in C([0,1]\times\Omega)$ and $\psi\in C([0,1]\times\Omega;\R^n)$.
\end{proposition}
\begin{remark}[Non-uniqueness of minimizers] \label{rem:NonUniqueDynamicOptimizers}
	Note that the dynamic minimizers $(\rho,\omega,\zeta)$ constructed above are not necessarily unique (the static minimizers $(\pi_0,\pi_1)$ need not be unique either).
	For instance, if no mass changes occur, the transport can also be performed over an extended interval in time. Vice versa, mass changes may sometimes also occur instantaneously. In the particular case $\SimLocC(m_0,m_1) = |m_0-m_1|$ for $\min\{m_0,m_1\}\geq 0$, which corresponds to $\CDynM(\rho,\zeta) = |\zeta|$ for $\rho \geq 0$, mass changes and transport may both happen during arbitrary subintervals or even time points of $[0,1]$, leading to a large degenerate set of minimizers.
	In Section~\ref{sec:dynamicFromStaticTimeConc} we give sufficient conditions for $\SimLocC$ such that all dynamic optimizers are of the above form.
\end{remark}
\begin{proof}
We first show $(\rho,\omega,\zeta)\in\mc{CE}(\rho_0,\rho_1)$.
Indeed, first note that for $\varphi\in C^1(\Omega)$ we have $\int_\Omega\nabla\varphi\cdot\d\omega_{x,y}=\int_{[0,1]}\nabla\varphi(p_{x,y}(t))\cdot\dot p_{x,y}(t)\,\d t=\varphi(y)-\varphi(x)$.
Therefore, for $\phi\in C^1([0,1]\times\Omega)$ we have
\begin{align*}
\int_{[0,1]\times\Omega} \nabla \phi \cdot \d \omega
&=\int_{\Omega\times\Omega}\int_\Omega\nabla\phi(0,\cdot)\cdot\d\omega_{x,y}\d\pi_0(x,y)+\int_{\Omega\times\Omega}\int_\Omega\nabla\phi(1,\cdot)\cdot\d\omega_{x,y}\d\pi_1(x,y)\\
&= \int_{\Omega\times\Omega} \phi(0,y)-\phi(0,x)\, \d\pi_0(x,y) + \int_{\Omega\times\Omega} \phi(1,y)-\phi(1,x)\, \d\pi_1(x,y)\\
&=\int_{\Omega} \phi(0,\cdot)\,\d(\rho_0'-\rho_0) - \int_{\Omega} \phi(1,\cdot)\,\d(\rho_1'-\rho_1)\,.
\end{align*}
In addition, due to $(\rho[m_0,m_1],\zeta[m_0,m_1]) \in \mc{CES}(m_0,m_1)$ for $(m_0,m_1) \in [0,\infty)^2$ we obtain
\begin{multline*}
  \int_{[0,1] \times \Omega} (\partial_t \phi)\,\d\rho +
  \int_{[0,1] \times \Omega} \phi\,\d\zeta\\
  = \int_\Omega \left[
      \int_0^1 \partial_t \phi(t,x) \,\d\rho\big[\RadNik{\rho_0'}\gamma(x),\RadNik{\rho_1'}\gamma(x)\big](t) +
      \int_0^1 \phi(t,x) \, \d\zeta\big[\RadNik{\rho_0'}\gamma(x),\RadNik{\rho_1'}\gamma(x)\big](t)
      \right] \d \gamma(x) \\
  = \int_\Omega \left[ \phi(1,x) \cdot \RadNik{\rho_1'}\gamma(x)- \phi(0,x) \cdot \RadNik{\rho_0'}\gamma(x) \right] \d \gamma(x)
  = \int_\Omega \phi(1,\cdot)\,\d\rho_1' - \int_\Omega \phi(0,\cdot)\,\d\rho_0'
\end{multline*}
which implies $(\rho,\omega,\zeta) \in \mc{CE}(\rho_0,\rho_1)$.

Next we calculate the cost $P_D(\rho,\omega,\zeta)$ from \eqref{eq:DynamicEnergy}.
Let us abbreviate
$$\overline C=\left\{\psi\in C([0,1]\times\Omega;\R^n)\ \middle|\ \|\psi(t,x)\|\leq1\text{ for all }(t,x)\in[0,1]\times\Omega\right\}\,,$$
and let us denote the total variation norm of $\omega$ by $\|\omega\|_{\meas([0,1] \times \Omega)^n}=\sup_{\psi\in\overline C}\int_{[0,1] \times \Omega}\psi\cdot\d\omega$,
then we have
\begin{multline*}
\int_{[0,1] \times \Omega} \|\RadNik{\omega}{\mu}\|\,\d\mu
=\|\omega\|_{\meas([0,1] \times \Omega)^n}
=\sup_{\psi\in\overline C}\sum_{i=0}^1\int_{\Omega\times\Omega}\int_0^1\psi(i,p_{x,y}(t))\cdot\dot p_{x,y}(t)\,\d t\d\pi_i(x,y)\\
\leq\sum_{i=0}^1\int_{\Omega\times\Omega}\int_0^1\|\dot p_{x,y}(t)\|\,\d t\d\pi_i(x,y)
=\int_{\Omega\times\Omega}d(x,y)\,\d\pi_0(x,y)+\int_{\Omega\times\Omega}d(x,y)\,\d\pi_1(x,y)\,.
\end{multline*}
For the growth term, by optimality of $\big(\rho\big[\RadNik{\rho_0'}\gamma(x),\RadNik{\rho_1'}\gamma(x)\big],\zeta\big[\RadNik{\rho_0'}\gamma(x),\RadNik{\rho_1'}\gamma(x)\big]\big)$ in \eqref{eq:InducedSimLocC}, we get
\begin{align*}
	\int_{[0,1] \times \Omega} \CDynM\left(
			\RadNik{\rho}{\mu},\RadNik{\zeta}{\mu} \right)\,\d\mu & =
	\int_\Omega \left[ \int_0^1 \CDynM\left(
			\frac{\d\rho\big[\RadNik{\rho_0'}\gamma(x),\RadNik{\rho_1'}\gamma(x)\big]}{\d\mu_x},\frac{\d\zeta\big[\RadNik{\rho_0'}\gamma(x),\RadNik{\rho_1'}\gamma(x)\big]}{\d\mu_x} \right)\,\d\mu_x \right] \d \gamma(x)
	\\
	& = \int_\Omega \SimLocC\big(\RadNik{\rho_0'}\gamma(x),\RadNik{\rho_1'}\gamma(x)\big)\,\d\gamma(x) = \SimLoc(\rho_0',\rho_1')\,.
\end{align*}
Here, $\{\mu_x\}_{x \in \Omega}$ with $\mu_x \in \measp([0,1])$ is a family of measures such that $\rho\big[\RadNik{\rho_0'}\gamma(x),\RadNik{\rho_1'}\gamma(x)\big],\zeta\big[\RadNik{\rho_0'}\gamma(x),\allowbreak\RadNik{\rho_1'}\gamma(x)\big] \ll \mu_x$,
and $\mu$ is chosen as $\int_{[0,1]\times\Omega}\phi\,\d\mu=\int_\Omega\int_0^1\phi(t,x)\,\d\mu_x(t)\d\gamma(x)$ for all $\phi\in C([0,1]\times\Omega)$.
Therefore, $P_D(\rho,\omega,\zeta)\leq P_S(\pi_0,\pi_1)=W_S(\rho_0,\rho_1)$, but $W_S(\rho_0,\rho_1)=W_D(\rho_0,\rho_1)\leq P_D(\rho,\omega,\zeta)$ by \thref{prop:W1DynamicStaticEquivalence}. Hence the triple $(\rho,\omega,\zeta)$ minimizes \eqref{eq:DynamicProblem}.
\end{proof}

\begin{remark}[Beckmann's problem]
	\label{rem:Beckmann}
	The $W_1$ distance on convex subsets $\Omega\subset\R^n$ can be rewritten as a minimization problem of a cost for mass flows, also known as Beckmann's problem. The measures $\omega_i$ constructed in \thref{prop:DynamicPrimalOptimizers} are the canonical cost minimizing flows corresponding to the optimal transport plans $\pi_i$, and one has $W_1(\rho_i,\rho_i') = \|\omega_i\|_{\meas(\Omega)^n}$ (for more details see \cite[Thm.\,4.6]{Santambrogio15}).
\end{remark}

\begin{remark}[Interpretation of dynamic optimizers]
	\label{rem:InterpretationDynamicOptimizers}
	The interpretation of the dynamic minimizer $(\rho,\omega,\zeta)$ constructed in \thref{prop:DynamicPrimalOptimizers} is as follows. At time $t=0$ mass is instantaneously transported from $\rho_0$ to $\rho_0'$, as encoded in $\pi_0$ or rather $\omega_0$. Then, throughout the interval $[0,1]$ the mass distribution $\rho_0'$ is gradually transformed into $\rho_1'$ via simultaneous mass change throughout $\Omega$. Finally, at $t=1$ mass is instantaneously transported from $\rho_1'$ to $\rho_1$, as encoded in $\pi_1$ and $\omega_1$.
	The temporal concentration of the transport into instantaneous shifts is possible since the transport part of the  cost, $(\rho,\omega,\zeta) \mapsto \|\omega\|_{\meas([0,1] \times \Omega)^n}$, is 1-homogeneous in $\omega$. This is different, for instance, from the Benamou--Brenier formula \cite{BenamouBrenier2000} which is a dynamic formulation of the $W_2$ distance, or from the Wasserstein--Fisher--Rao distance \cite{ChizatOTFR2015}.
\end{remark}

\subsection{A necessary condition for dynamic primal optimizers}\label{sec:dynamicFromStaticTimeConc}
In the previous section we have constructed primal dynamic optimizers from static optimizers. In this section we show that under suitable assumptions on the dynamic mass change cost $\CDynM$ any dynamic optimizer is of this form.
\begin{proposition}[Instantaneous mass transport]
	\thlabel{prop:DynamicNecessary}
	Assume $\BDynMH(z)<0$ for $z \neq 0$ ($\BDynMH$ determines the set $\BDynM$ via \thref{thm:DynConjugateSet}, and $\BDynM$ characterizes $\CDynM$ via \thref{prop:DynamicDual}).
	If $(\rho,\omega,\zeta) \in \mc{CE}(\rho_0,\rho_1)$ is a dynamic primal optimizer of $W_D(\rho_0,\rho_1)$ in \eqref{eq:DynamicProblem}, then $|\omega|((0,1) \times \Omega) = 0$,
	that is, mass transport only happens instantaneously at time $0$ and $1$.
\end{proposition}
\begin{remark}[Interpretation of assumption]
	Assumption $\BDynMH(z)<0$ for $z \neq 0$ is equivalent to $\zeta \to \CDynM(1,\zeta)$ being differentiable with derivative $0$ at $\zeta=0$.
	A prototypical counterexample is $\CDynM(\rho,\zeta)=|\zeta|$ for $\rho \geq 0$, see Remark \ref{rem:NonUniqueDynamicOptimizers}.
\end{remark}
\begin{proof}
With \thref{prop:DynamicDual} and \thref{thm:DynConjugateSet} we obtain $$\CDynM(1,\zeta)=\iota_{\BDynM}^\ast(1,\zeta)=\sup\{\BDynMH(\phi)+\zeta \cdot \phi\,|\,\phi \in \dom(\BDynMH)\} = (-\BDynMH)^\ast(\zeta)\,.$$
If $\BDynMH(z)<0$ for all $z \neq 0$, then by concavity and upper semicontinuity of $\BDynMH$ we have
\begin{equation*}
z=0 \quad \Leftrightarrow \quad 0 \in \partial(-\BDynMH)(z) \quad \Leftrightarrow \quad z \in \partial[\CDynM(1,\cdot)](0)
\end{equation*}
so that $\partial[\CDynM(1,\cdot)]=\{0\}$. A similar argument works for the reverse implication.
\end{proof}
\begin{corollary}[Structure of dynamic primal optimizers]\thlabel{cor:dynamicFromStatic}
	\thlabel{cor:AllDynamicOptimizersStructure}
	Under the above assumption on $\BDynMH$, any dynamic primal optimizer can be constructed from a static primal optimizer as in \thref{prop:DynamicPrimalOptimizers}.
\end{corollary}
\begin{proof}[Sketch of proof]
	By \thref{prop:DynamicNecessary} a dynamic primal optimal $\omega$ must have the form $\delta_0 \otimes \omega_0 + \delta_1 \otimes \omega_1$ for some $\omega_0, \omega_1 \in \meas(\Omega)^n$. Using the continuity equation \eqref{eqn:ContinuityEquation} one can show that these have weak divergences $\sigma_0, \sigma_1\in\meas(\Omega)$, that is, $\int_\Omega \nabla \psi \,\d \omega_i = -\int_\Omega \psi \,\d \sigma_i$ for $i=0,1$ and all $\psi \in C^1(\Omega)$,
	and that $(\rho,0,\zeta) \in \mc{CE}(\rho_0'=\rho_0-\sigma_0,\rho_1'=\rho_1+\sigma_1)$.
	Using arguments similar to those of the following section it is easy to show that $\rho_0'$ and $\rho_1'$ are nonnegative (if $\rho_0'$ were negative at some point, $\zeta$ must have a compensating singular positive component at $t=0$ which implies that mass is first created and then transported by $\omega_0$, which cannot be optimal).
	By optimality, $\omega_0$ and $\omega_1$ must then be optimal flows for $W_1(\rho_i,\rho_i')$ in Beckmann's formulation, corresponding to optimal couplings $\pi_i$ (cf.~Remark \ref{rem:Beckmann}). Similarly, $(\rho,\zeta)$ must be pointwise optimizers for \eqref{eq:InducedSimLocC} in \thref{prop:W1DynamicStaticEquivalencePrimal}.
\end{proof}

\thref{prop:DynamicNecessary} is proven via a primal-dual argument. On a formal level the primal-dual optimality conditions for the dynamic problems require that $\|\nabla \phi(t,x)\|=1$ on the support of $\omega$ and that the gradient of $\phi$ is aligned with the orientation of $\omega$. Essentially, we show that $\|\nabla \phi(t,x)\|<1$ for $t \in (0,1)$ and sufficiently `many' $x \in \Omega$ for `optimal' $\phi$. The remaining $x \in \Omega$ are shown to be harmless via the continuity equation. However, in general no dynamic dual optimizers exist in $C^1([0,1]\times \Omega)$, and thus the argument has to be made via a maximizing sequence, for which a more explicit bound on the gradient is required. This bound is established via some auxiliary lemmas in \thref{prop:DynamicDualGradBound}. The duality argument is then made rigorous at the end of the section.
\begin{proposition}[Lipschitz bound for constructed dynamic dual candidates]
	\thlabel{prop:DynamicDualGradBound}
	The feasible dynamic dual candidates constructed in \thref{lem:DynamicDualConstruction} satisfy
	\begin{align*}
		\|\nabla \phi(t,x)\| \leq 1- t(1-t) \cdot [\StrongConvBoundFlow(-\alpha(x)) + \StrongConvBoundFlow(-\beta(x))]
	\end{align*}
	for some function $\StrongConvBoundFlow$ with $\StrongConvBoundFlow(z)=0$ for $z\leq 0$ and $\StrongConvBoundFlow$ strictly positive and increasing for $z >0$.
\end{proposition}
The proof is distributed over several auxiliary lemmas.
\begin{lemma}
	\thlabel{lem:KStrongConvexity}
	Let $f : [0,1] \to \R$ be convex and decreasing with concave derivative $f' : [0,1] \to \R$. Then
	\begin{align*}
		f(t) \leq g(t)\qquad\text{for }g(t)= (1-t)\,f(0) + t\,f(1)- t\,(1-t)\,\hat K
	\end{align*}
	with $\hat K = f'(1)-f(1)+f(0)\geq0$.
\end{lemma}
\begin{proof}
The convexity of $f$ implies $f(0)\geq f(1)-f'(1)$ and thus $\hat K\geq0$.
Moreover, the concavity of $f'$ implies $f(1)-f(0)=\int_0^1 f'(t) \d t \geq \int_0^1 [(1-t)\,f'(0) + t\,f'(1)]\d t=\tfrac12 (f'(0)+f'(1))$ so that $2(f(1)-f(0))-f'(1)\geq f'(0)$.
Therefore, $g'(0)=f(1)-f(0)-\hat K\geq f'(0)$ and $g'(1)=f'(1)$. Now by concavity of $f'$ and thus also of $f'-g'$ there must be some $t_0 \in [0,1]$ such that
\begin{equation*}
f'(s)-g'(s) \leq 0 \text{ on }s \in [0,t_0]
\qquad\text{and}\qquad
f'(s)-g'(s) \geq 0 \text{ on }s \in [t_0,1]\,.
\end{equation*}
Together with $g(0)=f(0)$ and $g(1)=f(1)$ this implies $g(s) \geq f(s)$ for $s \in [0,1]$.
\end{proof}

\begin{lemma}
	\thlabel{lem:FlowStrongConvexity}
	If $\BDynMH(z)<0$ for $z>0$ then there is an increasing function $K: \dom(\BDynMH) \cap [0,\infty) \to [0,\infty)$, $K(z)>0$ for $z>0$ such that for $z \in \dom(\BDynMH) \cap [0,\infty)$ one has
	\begin{align*}
		\Flow_t(z) \leq (1-t)\,z+t\,\Flow_1(z) - t\,(1-t)\,K(z)\,.
	\end{align*}
\end{lemma}
\begin{proof}
  Let $z>0$. The map $t \mapsto \Flow_t(z)$ is convex and decreasing by \thref{lem:Flow}\eqref{item:FlowIVP} and \eqref{item:FlowTimeConvex}. Furthermore, the map $t \mapsto \partial_t \Flow_t(z)=\BDynMH(\Flow_t(z))$ is a composition of a concave decreasing function with a convex function and thus concave. Therefore, we can apply \thref{lem:KStrongConvexity} to obtain an upper bound on $\Flow_t(z)$, $t\in[0,1]$, with coefficient
  \begin{align*}
    \hat K(z) & = \partial_t \Flow_1(z) - \Flow_1(z)+\Flow_0(z)
    = \BDynMH(\Flow_1(z))-\Flow_1(z)+z
    =z\left[1+\frac{\Flow_1(z)}{z}\left(\frac{\BDynMH(\Flow_1(z))}{\Flow_1(z)}-1\right)\right]\,.
  \end{align*}
  Let us abbreviate $c(z)=\tfrac{|\BDynMH(\Flow_1(z))|}{\Flow_1(z)}\geq0$
  and note that by the properties of $\BDynMH$ from \thref{thm:DynConjugateSet} we have $\BDynMH(z)\leq-c(z)z$ for all $z\geq\Flow_1(z)$.
  Thus, by \thref{lem:Flow}\eqref{item:FlowIVP} and $\Flow_t(z)\geq\Flow_1(z)$ for $t\in[0,1]$, we have $\partial_t\Flow_t(z)=\BDynMH(\Flow_t(z))\leq-c(z)\Flow_t(z)$.
  Using $\Flow_0(z)=z$, Gr\"onwall's inequality now implies $\Flow_t(z)\leq z\exp(-c(z)t)$ and thus $\frac{\Flow_1(z)}{z}\leq\exp(-c(z))$.
  Summarizing, we obtain
  \begin{equation*}
  \hat K(z)\geq K(z)\quad\text{for }
  K(z)=z\left[1+\exp(-c(z))\left(-c(z)-1\right)\right]
  \end{equation*}
  so that, for $t\in[0,1]$,
  \begin{align*}
    \Flow_t(z) \leq (1-t)\,\Flow_0(z)+t\,\Flow_1(z)-t\,(1-t)\,{\hat K}(z) \leq (1-t)\,z+t\,\Flow_1(z)-t\,(1-t)\,K(z)\,.
  \end{align*}
  Note that $c(0)=0$, $c(z)>0$ for $z>0$, and that $c$ is increasing by concavity of $\BDynMH$ and monotonicity of $\Flow_1$.
  It is then straightforward to check that $K$ is increasing and strictly positive for $z>0$.
\end{proof}

\begin{lemma}
	\thlabel{lem:ConvexityDerivativeBound}
	Let $f : [0,\infty) \to \R$ be convex and increasing with $f(0)=0$. Let $x>y\geq0$, $t \in [0,1]$, and $\delta \geq 0$ such that $(1-t)\,x + t\,y-\delta \geq y$. Then
	\begin{align*}
		f\big((1-t)\,x+t\,y-\delta\big) \leq (1-t)\,f(x)+t\,f(y)-\delta \tfrac{f(x)}{x}\,.
	\end{align*}
\end{lemma}
\begin{proof}
	Since $z=(1-t)\,x + t\,y-\delta \in [y,x]$, we can bound $f(z)$ from above by the secant between $(y,f(y))$ and $(x,f(x))$. This can be written as
	\begin{align*}
		f(z) \leq (1-t)\,f(x)+t\,f(y)-\delta \tfrac{f(x)-f(y)}{x-y}.
	\end{align*}
	Since $f$ is convex and increasing, the slope $\tfrac{f(x)-f(y)}{x-y}$ of this secant is not smaller than the secant slope $\tfrac{f(x)}{x}$ between $(0,f(0))$ and $(x,f(x))$.
\end{proof}

\begin{proof}[Proof of \thref{prop:DynamicDualGradBound}]
For $z>0$ we find for $\LocInterp$ as defined in \thref{lem:LocInterp}
\begin{align*}
	\partial_z \LocInterp(t,z)=\frac{\BDynMH(\Flow_t(z))}{\BDynMH(z)}-t \frac{\BDynMH(\Flow_1(z))}{\BDynMH(z)}\,,
\end{align*}
where we used the formula for the derivative of the flow from \thref{lem:Flow}\eqref{item:FlowDerivative}.
Let $\StrongConvBoundFlow$ be the function appearing in \thref{lem:FlowStrongConvexity}. Then for $t \in [0,1]$ we get
\begin{align*}
	\Flow_1(z) \leq \Flow_t(z) \leq (1-t)\,z + t\,\Flow_1(z) - \delta
\end{align*}
where $\delta=t\,(1-t)\,\StrongConvBoundFlow(z)$.
The function $x \mapsto \tfrac{\BDynMH(x)}{\BDynMH(z)}$ is convex and increasing and $0$ for $x=0$. Therefore, by \thref{lem:ConvexityDerivativeBound} we obtain
\begin{align*}
	\frac{\BDynMH(\Flow_t(z))}{\BDynMH(z)}
	\leq \frac{\BDynMH((1-t)\,z + t\,\Flow_1(z) - \delta)}{\BDynMH(z)}
	\leq (1-t) + t \frac{\BDynMH(\Flow_1(z))}{\BDynMH(z)} - t\,(1-t)\,\StrongConvBoundFlow(z)
\end{align*}
so that $\partial_z \LocInterp(t,z)\leq (1-t)- t\,(1-t)\,\StrongConvBoundFlow(z)$ for $z>0$. By extending $\StrongConvBoundFlow(z)=0$ for $z\leq 0$ this bound is true for all $z \in \dom \BDynMH$ and also holds in absolute value due to \thref{lem:LocInterp}\eqref{item:LocalInterpConvexity}, $|\partial_z \LocInterp(t,z)|\leq (1-t)- t\,(1-t)\,\StrongConvBoundFlow(z)$. Note that Assumption \ref{asp:QNoKink} is satisfied due to $\zKink=0$.
In complete analogy we can show existence of an increasing function $\tilde{\StrongConvBoundFlow}$ with $\tilde{\StrongConvBoundFlow}(z)>0$ for $z>0$ such that $|\partial_z \LocInterpInv(t,z)|\leq t- t\,(1-t)\,\tilde{\StrongConvBoundFlow}(-z)$ for $t\in[0,1]$. From now on denote by $\StrongConvBoundFlow$ the pointwise minimum of the two functions, which is still increasing in $z$ with $\StrongConvBoundFlow(z)>0$ for $z>0$.
The result then follows from $\phi(t,x)=\LocInterp(t,-\alpha(x))+\LocInterpInv(t,\beta(x))$ and $\alpha,\beta\in\Lip(\Omega)$.
\end{proof}

\begin{proof}[Proof of \thref{prop:DynamicNecessary}]
Let $\phi$ be feasible for the dynamic dual problem \eqref{eq:DynamicDual}, set $\alpha=-\phi(0,\cdot)$, $\beta=\phi(1,\cdot)$ and finally $\hat{\phi}(t,x)=\LocInterp(t,-\alpha(x))+\LocInterpInv(t,\beta(x))$, as constructed in \thref{lem:DynamicDualConstruction}. Then $\hat{\phi}$ has the same score as $\phi$. Therefore, without loss of generality, we may in the following restrict the dynamic dual problem to candidates of the form constructed in \thref{lem:DynamicDualConstruction}.

Let $(\rho,\omega,\zeta)$ be a minimizer of the dynamic primal problem \eqref{eq:DynamicProblem} and let $(\phi_k)_{k \in \N}$ be a maximizing sequence of the dynamic dual problem \eqref{eq:DynamicDual}.
Using the notation from the proof of \thref{prop:DynamicDual} this implies that
\begin{align*}
	F^\ast(\rho,\omega,\zeta)+G^\ast(A^\ast(\rho,\omega,\zeta))+F(A\phi_k)+G(-\phi_k) \searrow 0
	\tn{ as } k \to \infty.
\end{align*}
By the Fenchel--Young inequality one has $F^\ast(\rho,\omega,\zeta)+F(A\phi_k) - \int_{[0,1]\times\Omega} A\phi_k \,\d(\rho,\omega,\zeta) \geq 0$ and $G^\ast(A^\ast(\rho,\omega,\zeta))+G(-\phi_k) + \int_{[0,1]\times\Omega} A\phi_k \,\d(\rho,\omega,\zeta) \geq 0$, therefore the above implies $F^\ast(\rho,\omega,\zeta)+F(A\phi_k) - \int_{[0,1]\times\Omega} A\phi_k \,\d(\rho,\omega,\zeta) \searrow 0$ as $k \to \infty$. Further, due to the $W_1$-structure of the problem, $F$ and $F^\ast$ separate into an $\omega$-related part and a $(\rho,\zeta)$-related part,
\begin{multline*}
F^\ast(\rho,\omega,\zeta)+F(A\phi_k) - \int_{[0,1]\times\Omega} A\phi_k \,\d(\rho,\omega,\zeta)\\
=\left(|\omega|([0,1]\times\Omega)+\int_{[0,1]\times\Omega}\iota_{B_1(0)}(\nabla\phi_k)\,\d(t,x)-\int_{[0,1]\times\Omega}\nabla\phi_k\,\d\omega\right)\\
\qquad+\left(\int_{[0,1] \times \Omega} \CDynM\left(\RadNik{\rho}{\mu},\RadNik{\zeta}{\mu} \right)\,\d\mu
+\int_{[0,1]\times\Omega}\iota_{\BDynM}(\partial_t\phi_k,\phi_k)\,\d(t,x)
-\int_{[0,1]\times\Omega}\partial_t\phi_k\,\d\rho
-\int_{[0,1]\times\Omega}\phi_k\,\d\zeta\right)\,,
\end{multline*}
where $\mu$ is any measure with $\mu\gg\rho,\zeta$, $|\omega|\in\measp(\Omega)$ denotes the total variation measure, and $B_1(0)$ the closed unit ball in $\R^n$.
Both terms in parentheses are nonnegative due to the Fenchel--Young inequality and thus converge to zero from above separately.
Since the $\phi_k$ are admissible, we thus obtain $|\omega|([0,1]\times\Omega) -\int_{[0,1]\times\Omega} \nabla \phi_k \d \omega \searrow 0$.

Let $v=\RadNik{\omega}{|\omega|}$ be the local orientation of $\omega$ and recall that $\|v\|=1$ $|\omega|$-almost everywhere. Then $|\omega|([0,1]\times\Omega)-\int_{[0,1]\times\Omega} \nabla \phi_k \,\d \omega=\int_{[0,1]\times\Omega} (1- \la \nabla \phi_k,v \ra) \,\d |\omega|$, where the integrand is non-negative $|\omega|$-almost everywhere since $\|\nabla \phi_k\| \leq 1$.
For arbitrary $\delta>0$ let now $k$ be such that $|\omega|([0,1]\times\Omega)-\int_{[0,1]\times\Omega} \nabla \phi_k\, \d \omega < \delta$ and set $\alpha=-\phi_k(0,\cdot)$, $\beta=\phi_k(1,\cdot)$.
Let $\veps \in (0,\tfrac1{16})$, set $T_\veps=[\sqrt{\veps},1-\sqrt{\veps}]$ and introduce the following subsets of $[0,1] \times \Omega$:
\begin{align*}
	S_1 & = \{ (t,x) \in T_\veps \times \Omega \,|\,
		|\alpha(x)| \geq \veps \text{ or } |\beta(x)| \geq \veps \}, \\
	S_2 & = \{ (t,x) \in (T_\veps \times \Omega) \setminus S_1 \,|\,
		-\nabla \alpha(x) \cdot v(t,x) \leq (1-\veps) \text{ or } 
		\nabla \beta(x) \cdot v(t,x) \leq (1-\veps) \}, \\
	S_3 & = (T_\veps \times \Omega) \setminus (S_1 \cup S_2), \\
	\hat{S}_3 & = ([2\sqrt{\veps},1-2\sqrt{\veps}] \times \Omega) \setminus (S_1 \cup S_2), \\
	S_4 & = ([0,1]\times\Omega)\setminus(S_1\cup S_2\cup S_3) = ([0,1] \setminus T_\veps) \times \Omega.
\end{align*}
We now estimate the mass of $|\omega|$ on the sets $S_1$, $S_2$, and $\hat S_3$.
Note that $t\,(1-t) \geq \veps$ for $t \in T_\veps$ so that $\|\nabla \phi_k(t,x)\| \leq 1 - \veps\,\StrongConvBoundFlow(\veps)$ for $(t,x) \in S_1$ by \thref{prop:DynamicDualGradBound}. Therefore
\begin{align*}
	\delta & \geq \int_{S_1} \left(1-\|\nabla \phi_k\|\right) \d |\omega| \geq \veps\,\StrongConvBoundFlow(\veps)\,|\omega|(S_1)\,.
\end{align*}
Using $\phi_k=\LocInterp(t,-\alpha(x))+\LocInterpInv(t,\beta(x))$ and \thref{lem:LocInterp}\eqref{item:LocalInterpConvexity} one finds $\nabla \phi_k \cdot v \leq 1-\veps$ on $S_2$ and so
\begin{align*}
	\delta & \geq \int_{S_2} \left(1-\nabla \phi_k \cdot v \right) \d |\omega| \geq \veps\,|\omega|(S_2)\,.
\end{align*}
Let now $h_\veps : [0,1] \to [0,1]$ be a smooth cutoff function with $h_\veps(t)=0$ for $t \in [0,\sqrt{\veps}] \cup [1-\sqrt{\veps},1]$, $h_\veps(t)=1$ on $t \in [2\sqrt{\veps},1-2\sqrt{\veps}]$ and $|h_\veps'(t)|\leq \tfrac{2}{\sqrt{\veps}}$ for $t \in [0,1]$. Further, let $g_\veps : \R \to [-2\veps,2\veps]$ be a smooth increasing function with $g_\veps(z)=z$ for $z \in [-\veps,\veps]$ and $g'_\veps(z) \in [0,1]$. Set $\psi_\veps(t,x)=h_\veps(t) \cdot g_\veps(-\alpha(x))$. We observe
\begin{align*}
	|\psi_\veps(t,x)| & \leq 2 \veps, &
	\|\nabla \psi_\veps(t,x)\| & \leq 1, &
	|\partial_t \psi_\veps(t,x)| & \leq 4 \sqrt{\veps}
\end{align*}
for $(t,x) \in [0,1] \times \Omega$ and
\begin{align*}
	\nabla \psi_\veps(t,x) \cdot v(t,x) & \geq 0 & & \tn{for } (t,x) \in S_3, &
	-\nabla \alpha(x) \cdot v(t,x) & > (1-\veps) & & \tn{for } (t,x) \in \hat{S}_3, \\
	\nabla \psi_\veps(t,x) & =-\nabla \alpha(x) & & \tn{for } (t,x) \in \hat{S}_3, &
	\psi_\veps(t,x) & = 0 & & \tn{for } (t,x) \in S_4.
\end{align*}
We therefore obtain
\begin{align*}
	\int_{S_3} \nabla \psi_\veps \, \d \omega = \int_{S_3} \nabla \psi_\veps \cdot v \,\d|\omega|
	\geq \int_{\hat{S}_3} -\nabla \alpha \cdot v \,\d|\omega|
	\geq \int_{\hat{S}_3} (1-\veps) \,\d|\omega| =(1-\veps) |\omega|(\hat{S}_3)\,.
\end{align*}
Choosing the test function $\psi_\veps$ in the continuity equation \eqref{eqn:ContinuityEquation} we also find
\begin{align*}
	\int_{S_1 \cup S_2 \cup S_3} \nabla \psi_\veps \, \d \omega =
	\int_{[0,1] \times \Omega} \nabla \psi_\veps \, \d \omega =
	-\int_{[0,1] \times \Omega}  \partial_t \psi_\veps \, \d \rho
	-\int_{[0,1] \times \Omega} \psi_\veps \, \d \zeta\,.
\end{align*}
Together with the above bounds on $|\omega|(S_1)$ and $|\omega|(S_2)$ this implies
\begin{multline*}
  (1-\veps) |\omega|(\hat{S}_3)
	\leq \int_{S_3} \nabla \psi_\veps \, \d \omega
	= - \int_{[0,1] \times \Omega}  \partial_t \psi_\veps \, \d \rho - \int_{[0,1] \times \Omega} \psi_\veps \, \d \zeta -\int_{S_1 \cup S_2}\nabla \psi_\veps \, \d \omega\\
  \leq 4\sqrt{\veps} |\rho|([0,1]\times\Omega) + 2\veps |\zeta|([0,1]\times\Omega) + |\omega|(S_1\cup S_2)
  \leq 4\sqrt{\veps} |\rho|([0,1]\times\Omega) + 2\veps |\zeta|([0,1]\times\Omega) + \tfrac{\delta}{\veps\,\StrongConvBoundFlow(\veps)} + \tfrac{\delta}{\veps}
\end{multline*}
and thus
\begin{align*}
 |\omega|(\hat{S}_3) \leq \tfrac{1}{1-\veps} \left(4\sqrt{\veps} |\rho|([0,1]\times\Omega) + 2\veps |\zeta|([0,1]\times\Omega) + \tfrac{\delta}{\veps\,\StrongConvBoundFlow(\veps)} + \tfrac{\delta}{\veps} \right).
\end{align*}
Now we send $\delta$ and $\veps$ jointly to $0$ in a way such that $\tfrac{\delta}{\veps\,\StrongConvBoundFlow(\veps)} + \tfrac{\delta}{\veps} \to 0$.
Note that then $|\omega|([2\sqrt\veps,1-2\sqrt\veps]\times\Omega)\leq|\omega|(S_1\cup S_2\cup \hat S_3)\to0$, which implies $|\omega|([\tau,1-\tau] \times \Omega)=0$ for all $\tau>0$. By inner regularity of Radon measures this implies $\omega((0,1) \times \Omega)=0$.
\end{proof}

\subsection{Characterization of static and dynamic primal optimal solutions}\label{sec:characterization}
In \thref{prop:GrowthShrinkDecomposition}, \thref{thm:massTptChng}, and \thref{thm:transportChar} we will characterize minimizers $\pi_0$ and $\pi_1$ of the static primal formulation \eqref{eq:StaticPrimalProblem} in more detail.
By the two previous sections, this immediately implies a characterization also of optimizers for the dynamic primal formulation \eqref{eq:DynamicProblem}.
Our characterization concerns the possible support of the couplings $\pi_0$ and $\pi_1$.

\begin{proposition}[Necessary optimality condition I for static primal]
	\thlabel{prop:GrowthShrinkDecomposition}
	Let $(\pi_0,\pi_1)$ be a feasible candidate in \eqref{eq:StaticPrimalProblem}.
	Denote $\rho_0' = P_Y \pi_0$, $\rho_1' = P_X \pi_1$, choose some $\gamma \in \measp(\Omega)$ with $\rho_0,\rho_1,\rho_0'$, $\rho_1' \ll \gamma$, and set
	\begin{align*}
		\Omega_+ & = \left\{
			x \in \Omega \,\middle|\, \RadNik{\rho_0'}{\gamma}(x) < \RadNik{\rho_1'}{\gamma}(x)
			\right\}, &
		\Omega_- & = \left\{
			x \in \Omega \,\middle|\, \RadNik{\rho_0'}{\gamma}(x) > \RadNik{\rho_1'}{\gamma}(x)
			\right\}, &
		\Omega_= & = \Omega\setminus(\Omega_+\cup\Omega_-).
	\end{align*}
	Clearly, $\Omega_+$ and $\Omega_-$ are well-defined up to $\gamma$-negligible (and thus $\rho_0'$, $\rho_1'$-negligible) sets.
	If $(\pi_0,\pi_1)$ are minimizers of \eqref{eq:StaticPrimalProblem}, then
	\begin{gather}
		\label{eq:GrowthShrinkDecompositionPi}
		\int_{\Omega \times \Omega_-} d(x,y)\,\d \pi_0(x,y) = 0
		\qquad \tn{and} \qquad
		\int_{\Omega_+ \times \Omega} d(x,y)\, \d \pi_1(x,y) = 0\,.
	\end{gather}
\end{proposition}
$\Omega_+$ and $\Omega_-$ are the areas where mass is grown or shrunk in the intermediate $\SimLoc$-term in \eqref{eq:StaticPrimal}. The above proposition states that in an optimal arrangement or during the optimal dynamics mass is not transported after it has been increased and not transported prior to being decreased.

\begin{figure}
	\centering
	\begin{tikzpicture}[x=2cm,y=3cm,
		xnode/.style={
		anchor=center,
		inner sep=1pt,shape=circle,draw=black,line width=1pt
		},
		tr/.style={->,black!50!white,line width=0.5pt},
		lblnode/.style={black,midway,above,sloped}
	]

	\node[xnode] (x0) at (0,0) [label=left:$x_0$]{};
	\node[xnode] (x1) at (0,-1) [label=left:$x_1$]{};
	\node[xnode] (z) at (1,-0.5) [label=above:$z$]{};
	\node[xnode] (y0) at (2,0) [label=right:$y_0$]{};
	\node[xnode] (y1) at (2,-1) [label=right:$y_1$]{};
	{\small
	\draw[tr] (x0) -- (z) node [lblnode]{$a_0$};
	\draw[tr] (x1) -- (z) node [lblnode]{$a_1$};
	\draw[tr] (z) -- (y0) node [lblnode]{$b_0$};
	\draw[tr] (z) -- (y1) node [lblnode]{$b_1$};
	}
	
	\begin{scope}[shift=({4,0})]
		\node[xnode] (x0) at (0,0) [label=left:$x_0$]{};
		\node[xnode] (x1) at (0,-1) [label=left:$x_1$]{};
		\node[xnode] (y0) at (2,0) [label=right:$y_0$]{};
		\node[xnode] (y1) at (2,-1) [label=right:$y_1$]{};
		{\small
		\draw[tr] (x0) -- (y0) node [lblnode]{$a_0 \cdot \tfrac{b_0}{b_0+b_1}$};
		\draw[tr] (x0) -- (y1) node [lblnode,pos=0.75]{$a_0 \cdot \tfrac{b_1}{b_0+b_1}$};
		\draw[tr] (x1) -- (y0) node [lblnode,pos=0.25]{$a_1 \cdot \tfrac{b_0}{b_0+b_1}$};
		\draw[tr] (x1) -- (y1) node [lblnode,below]{$a_1 \cdot \tfrac{b_1}{b_0+b_1}$};
		}
	\end{scope}
\end{tikzpicture}
	\caption[]{Illustration of \thref{prop:GrowthShrinkDecomposition}.
	For pairwise disjoint $x_0, x_1, z, y_0, y_1 \in \Omega$ consider initial and final measure $\rho_0=a_0 \cdot \delta_{x_0} + a_1 \cdot \delta_{x_1}$ and $\rho_1 = b_0 \cdot \delta_{y_0} + b_1 \cdot \delta_{y_1}$ with $a_0 + a_1 < b_0 + b_1$ (in particular, there is zero mass located at $z,y_0,y_1$ initially).
		\textit{Left:} Masses $a_i$ are transported from $x_i$ to intermediate point $z$ by the coupling $\pi_0$. The mass in $z$ is then increased from $a_0+a_1$ to $b_0+b_1$ with the $\SimLoc$ term. Finally, masses $b_i$ are transported to $y_i$ by the coupling $\pi_1$. (See \eqref{eq:StaticPrimal} for the various terms.) %
		\textit{Right:} The cost of the left configuration can always be improved by transporting mass directly from $x_i$ to $y_j$ in suitable fractions via the coupling $\pi_0$, then increasing mass at final points $y_j$ appropriately, and performing no more transport via $\pi_1$. (See \thref{exp:GrowthShrinkDecomposition} for details.)}
	\label{fig:ShortCut}
\end{figure}
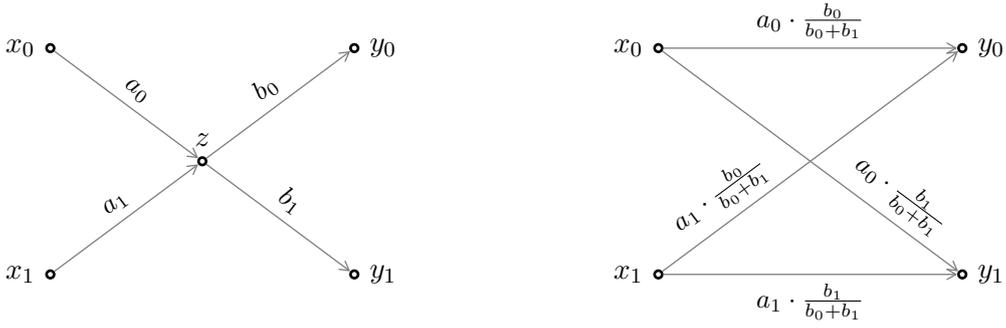

\begin{example}
\thlabel{exp:GrowthShrinkDecomposition}
	The intuition behind \thref{prop:GrowthShrinkDecomposition} is illustrated in Figure~\ref{fig:ShortCut}.
	For the left configuration the combined cost \eqref{eq:StaticPrimal} for transport and growth is $\sum_{i=0}^1 a_i \cdot d(x_i,z) + \SimLocC(a_0+a_1,b_0+b_1) + \sum_{i=0}^1 b_i \cdot d(z,y_i)$.
	For the configuration on the right it is $\sum_{i,j=0}^1 a_i \tfrac{b_j}{b_0+b_1} \cdot d(x_i,y_j) + \sum_{j=0}^1 \SimLocC(b_j \tfrac{a_0+a_1}{b_0+b_1}, b_j)$.
	By using $a_0+a_1 < b_0 + b_1$, $z \neq y_i$ and the triangle inequality one finds for the transport costs
	\begin{multline*}
		\sum_{i,j=0}^1 a_i \tfrac{b_j}{b_0+b_1} \cdot d(x_i,y_j)
		 \leq \sum_{i,j=0}^1 a_i \tfrac{b_j}{b_0+b_1} \cdot [ d(x_i,z) + d(z,y_j)] \\
		 =
		 \sum_{i=0}^1 a_i \cdot d(x_i,z) + \frac{a_0 +a_1}{b_0 + b_1} \sum_{j=0}^1 b_j \cdot d(z,y_j)
		 < \sum_{i=0}^1 a_i \cdot d(x_i,z) + \sum_{j=0}^1 b_j \cdot d(z,y_j)\,.
	\end{multline*}
	Moreover, by 1-homogeneity of $\SimLocC$ one obtains
	\begin{align*}
		\sum_{j=0}^1 \SimLocC(b_j \tfrac{a_0+a_1}{b_0+b_1}, b_j) = 
			\SimLocC(a_0+a_1,b_0+b_1)\,.	
	\end{align*}
	Thus, the second configuration has a strictly lower cost. Analogously, if $a_0+a_1$ were strictly smaller than $b_0+b_1$ and $x_i \neq z$, the left configuration can be strictly improved by letting the first coupling $\pi_0$ perform no transport at all, then decreasing masses at $x_i$ and finally transporting the decreased masses to the points $y_j$ in appropriate fractions via $\pi_1$.
\end{example}

\begin{proof}[Proof of Proposition \ref{prop:GrowthShrinkDecomposition}]
	The strategy of the proof is as follows: for any feasible candidate $(\pi_0,\pi_1)$ that does not satisfy condition \eqref{eq:GrowthShrinkDecompositionPi} we will construct a new candidate $(\hat{\pi}_0,\hat{\pi}_1)$ with a strictly better cost.
	To this end, \thref{exp:GrowthShrinkDecomposition} is generalized to arbitrary configurations.
	
	\textbf{(step i)} First we construct the improved candidate. Let the family $\{\pi_{0,z}\}_{z \in \Omega}$ be the disintegration of $\pi_0$ with respect to its second marginal $P_Y \pi_0 = \rho_0'$ (see for instance \cite[Thm.~5.3.1]{AmbrosioGradientFlows2005}). This means $\pi_{0,z} \in \prob(\Omega)$ for $\rho_0'$-a.e.~$z \in \Omega$ and, for any measurable $\phi : \Omega \to [0,\infty]$,
	\begin{align*}
		\int_{\Omega\times\Omega} \phi \, \d \pi_0 = \int_\Omega \left[ \int_\Omega
			\phi(x,z)\, \d\pi_{0,z}(x) \right] \d \rho_0'(z)\,.
	\end{align*}
	Intuitively, $\pi_{0,z}$ describes the distribution of mass before transport that arrives at $z$ according to the transport plan $\pi_0$. It corresponds to the normalized coefficients $a_i/(a_0+a_1)$ in \thref{exp:GrowthShrinkDecomposition}. Similarly, let $\{\pi_{1,z}\}_{z \in \Omega}$ be the disintegration of $\pi_1$ according to its first marginal $P_X \pi_1 = \rho_1'$ such that
	\begin{align*}
		\int_{\Omega\times\Omega} \phi \, \d \pi_1 = \int_\Omega \left[ \int_\Omega
			\phi(z,y)\, \d\pi_{1,z}(y) \right] \d \rho_1'(z)\,.	
	\end{align*}
	We now modify $\pi_0$ and $\pi_1$ as follows:
	All mass that is transported by $\pi_0$ into $\Omega_-$ (and then after a mass decrease is further transported by $\pi_1$) will no longer be moved by $\hat\pi_0$ and will instead be transported by $\hat\pi_1$ directly from its original to its final position.
	Likewise, all mass that is transported by $\pi_0$ into $\Omega_+$ (and then after a mass increase is further transported by $\pi_1$) will no longer be moved by $\hat\pi_1$ but will instead be transported by $\hat\pi_0$ directly from its original to its final position.
	In detail, we set $\hat{\pi}_i = \hat{\pi}_{i,+} + \hat{\pi}_{i,-} + \hat{\pi}_{i,=}$ for $i \in \{0,1\}$,
	where we define the measures $\hat{\pi}_{i,+},\hat{\pi}_{i,-},\hat{\pi}_{i,=}\in\measp(\Omega\times\Omega)$ for $i \in \{0,1\}$ by how they act on measurable functions $\phi : \Omega^2 \to [0,\infty]$, using the above disintegrations,
	\begin{align*}
		\int_{\Omega\times\Omega} \phi\, \d \hat{\pi}_{0,+} & = \int_{\Omega_+}
			\left[ \int_{\Omega\times\Omega} \phi(x,y)\, \d \pi_{0,z}(x) \, \d \pi_{1,z}(y) \right] \d \rho_0'(z) \,,\\
		\int_{\Omega\times\Omega} \phi\, \d \hat{\pi}_{1,+} & =
			\int_{\Omega_+\times\Omega}\phi(y,y)\, \d \pi_1(z,y) \,,\\
		\int_{\Omega\times\Omega} \phi\, \d \hat{\pi}_{0,-} & =
			\int_{\Omega\times\Omega_-}\phi(x,x)\, \d \pi_0(x,z) \,,\\
		\int_{\Omega\times\Omega} \phi\, \d \hat{\pi}_{1,-} & = \int_{\Omega_-}
			\left[ \int_{\Omega\times\Omega} \phi(x,y)\, \d \pi_{0,z}(x) \, \d \pi_{1,z}(y) \right] \d \rho_1'(z) \,,\\
		\int_{\Omega\times\Omega} \phi\, \d \hat{\pi}_{0,=} & = \int_{\Omega \times \Omega_=} \phi\,\d\pi_0 \,,\\
		\int_{\Omega\times\Omega} \phi\, \d \hat{\pi}_{1,=} & = \int_{\Omega_= \times \Omega} \phi\,\d\pi_1\,.
	\end{align*}
	The measure $\hat{\pi}_{0,+}$ corresponds to the coefficients $a_i \tfrac{b_j}{b_0+b_1}$ in \thref{exp:GrowthShrinkDecomposition}. $\hat{\pi}_{1,+}$ is supported on the diagonal which indicates that these particles do not move in the second transport term. $\hat{\pi}_{i,-}$ are the appropriate counterparts when mass is decreased and $\hat{\pi}_{i,=}$ are the corresponding parts of the original couplings in areas where the amount of mass is not changed.
	One can readily verify $P_X \hat{\pi}_0 = P_X \pi_0 = \rho_0$ and $P_Y \hat{\pi}_1 = P_Y \pi_1 = \rho_1$.
	
	\textbf{(step ii)} Now we assume that $(\pi_0,\pi_1)$ do not satisfy \eqref{eq:GrowthShrinkDecompositionPi} and then show that the transport costs implied by $(\hat{\pi}_0,\hat{\pi}_1)$ are strictly smaller than those of $(\pi_0,\pi_1)$.
	By definition we find
	\begin{align*}
		\int_{\Omega\times\Omega} d(x,y)\,\d \hat{\pi}_{0,=}(x,y) &= \int_{\Omega \times \Omega_=} d(x,y)\,\d \pi_0(x,y),&
                \int_{\Omega\times\Omega} d(x,y)\,\d \hat{\pi}_{0,-}(x,y) &= 0, \\
		\int_{\Omega\times\Omega} d(x,y)\,\d \hat{\pi}_{1,=}(x,y) &= \int_{\Omega_= \times \Omega} d(x,y)\,\d \pi_1(x,y),&
		\int_{\Omega\times\Omega} d(x,y)\, \d \hat{\pi}_{1,+}(x,y) &=0.
	\end{align*}
	Now assume $\int_{\Omega_+ \times \Omega} d(x,y)\, \d \pi_1(x,y) > 0$. Then
	\begin{align*}
		\int_{\Omega\times\Omega} d\,\d \hat{\pi}_{0,+} & = \int_{\Omega_+}
			\left[ \int_{\Omega\times\Omega} d(x,y)\,\d\pi_{0,z}(x)\,\d\pi_{1,z}(y) \right]
			\d \rho_0'(z) \\
		& \leq \int_{\Omega_+} \left[
			\int_{\Omega} d(x,z)\,\d\pi_{0,z}(x) + \int_\Omega d(z,y)\, \d \pi_{1,z}(y) \right]
			\d \rho_0'(z) \\
		& < \int_{\Omega \times \Omega_+} d\,\d\pi_0 + \int_{\Omega_+ \times \Omega} d \,\d \pi_1\,.
	\end{align*}
	The first inequality is obtained using $d(x,y) \leq d(x,z) + d(z,y)$ for any $z \in \Omega$. 
	The strictness of the second inequality follows from $\RadNik{\rho_0'}{\gamma}(z) < \RadNik{\rho_1'}{\gamma}(z)$ for $z \in \Omega_+$ and the assumption $\int_{\Omega_+ \times \Omega} d(x,y)\, \d \pi_1(x,y) > 0$.
	Analogously one shows that $\int_{\Omega \times \Omega_-} d(x,y)\, \d \pi_0(x,y) > 0$ implies $\int_{\Omega\times\Omega} d\,\d \hat{\pi}_{1,-} < \int_{\Omega \times \Omega_-} d\,\d\pi_0 + \int_{\Omega_- \times \Omega} d\,\d\pi_1$.
	Consequently, by assumption
	\begin{align*}
		\int_{\Omega\times\Omega} d\,\d\hat{\pi}_0 + \int_{\Omega\times\Omega} d\,\d\hat{\pi}_1 <
		\int_{\Omega\times\Omega} d\,\d\pi_0 + \int_{\Omega\times\Omega} d\,\d\pi_1\,.
	\end{align*}
	
	\textbf{(step iii)}
	Let $\hat{\rho}_0' = P_Y \hat{\pi}_0$, $\hat{\rho}_1' = P_X \hat{\pi}_1$. We now show $\SimLoc(\hat{\rho}_0',\hat{\rho}_1') \leq \SimLoc(\rho_0',\rho_1')$, which together with the previous step implies that $(\hat\pi_0,\hat\pi_1)$ has a strictly smaller cost. Let
	\begin{gather*}
		\hat{\rho}'_{0,\chi} = P_Y \hat{\pi}_{0,\chi} \qquad \tn{and} \qquad
		\hat{\rho}'_{1,\chi} = P_X \hat{\pi}_{1,\chi} \qquad \tn{for} \qquad
		\chi \in \{+,-,=\}\,, \\
		\intertext{such that}
		\hat{\rho}'_i = \sum_{\chi \in \{+,-,=\}} \hat{\rho}'_{i,\chi} \qquad \tn{for} \qquad
		i \in \{0,1\}\,.
	\end{gather*}
	Since $\{\Omega_+,\Omega_-,\Omega_=\}$ is a partition of $\Omega$, one has
	\begin{gather*}
		\SimLoc(\rho_0',\rho_1') = \SimLoc(\rho_0' \restr_{\Omega_+},\rho_1' \restr_{\Omega_+})
		+ \SimLoc(\rho_0' \restr_{\Omega_-},\rho_1' \restr_{\Omega_-})
		+ \SimLoc(\rho_0' \restr_{\Omega_=},\rho_1' \restr_{\Omega_=})\,,
		\intertext{and by joint subadditivity of $\SimLoc$ in its two arguments}
		\SimLoc(\hat{\rho}_0',\hat{\rho}_1') \leq
			\sum_{\chi \in \{+,-,=\}}
				\SimLoc(\hat{\rho}_{0,\chi}',\hat{\rho}_{1,\chi}'). 
	\end{gather*}
	One finds $\hat{\rho}_{i,=}' = \rho_i' \restr_{\Omega_=}$, $i \in \{0,1\}$ and thus $\SimLoc(\rho_0' \restr_{\Omega_=},\rho_1' \restr_{\Omega_=}) = \SimLoc(\hat{\rho}_{0,=}',\hat{\rho}_{1,=}')$.
	In the following let $\pi_\gamma \in \measp(\Omega\times\Omega)$ be defined by
	\begin{align*}
		\int_{\Omega\times\Omega} \phi \, \d \pi_{\gamma} = \int_{\Omega} \left[
			\int_\Omega \phi(z,y) \, \d \pi_{1,z}(y) \right] \d \gamma(z)\,,
	\end{align*}
	where we combine the disintegration $\{\pi_{1,z}\}_{z \in \Omega}$ with a different marginal $\gamma$ (as chosen in the statement of the result). Let $\hat{\gamma} = P_Y \pi_\gamma$ (note that by construction $\gamma = P_X \pi_\gamma$) and let $\{ \pi_{\gamma,y} \}_{y \in \Omega}$ be the disintegration of $\pi_\gamma$ with respect to the second marginal, that is
	\begin{align*}
		\int_{\Omega\times\Omega} \phi \, \d \pi_{\gamma} = \int_{\Omega} \left[
			\int_\Omega \phi(z,y) \, \d \pi_{\gamma,y}(z) \right] \d \hat{\gamma}(y)\,.
	\end{align*}
	Now consider
	\begin{align*}
		\SimLoc(\rho_0' \restr_{\Omega_+},\rho_1' \restr_{\Omega_+}) & =
			\int_{\Omega_+}
				\SimLocC\left(
					\RadNik{\rho_0'}{\gamma}(z), \RadNik{\rho_1'}{\gamma}(z) \right) \, \d \gamma(z) 
		= \int_{\Omega_+} \left[
			\int_\Omega \SimLocC(\ldots)\,\d \pi_{1,z}(y) \right] \d \gamma(z)\\
			&= \int_{\Omega_+ \times \Omega} \SimLocC(\ldots)\, \d \pi_\gamma(z,y)
		= \int_{\Omega} \left[
			\int_{\Omega_+} c(\ldots) \, \d\pi_{\gamma,y}(z) \right] \d \hat{\gamma}(y) \\
		& \geq \int_{\Omega} \SimLocC \left(
			\int_{\Omega_+} \RadNik{\rho_0'}{\gamma}(z)\, \d \pi_{\gamma,y}(z),
			\int_{\Omega_+} \RadNik{\rho_1'}{\gamma}(z)\, \d \pi_{\gamma,y}(z)
			\right) \d \hat{\gamma}(y) \\
		& = \SimLoc(\sigma_0,\sigma_1)\,,
	\end{align*}
	where the inequality is due to Jensen's inequality and
	$\sigma_0$ and $\sigma_1 \in \measp(\Omega)$ are defined for $i \in \{0,1\}$ by
	\begin{align*}
		\int_{\Omega} \phi\,\d\sigma_i & = \int_{\Omega}
			\left[ \int_{\Omega_+} \RadNik{\rho_i'}{\gamma}(z) \, \d \pi_{\gamma,y}(z) \right]
			 \phi(y)\, \d \hat{\gamma}(y)\,. \\
		\intertext{One finds}
		\int_{\Omega} \phi\,\d\sigma_0 & =
			\int_{\Omega_+} \left[ \int_{\Omega}
				\phi(y)\,\RadNik{\rho_0'}{\gamma}(z)\,\d \pi_{1,z}(y) \right] \d \gamma(z) =
		\int_{\Omega_+} \left[ \int_{\Omega} \phi(y)\,\d\pi_{1,z}(y) \right] \d \rho_0'(z) \\
		& = \int_{\Omega} \phi(y)\,\d(P_Y \hat{\pi}_{0,+})(y)
	\end{align*}
	and therefore $\sigma_0 = \hat{\rho}_{0,+}'$. Similarly one finds $\sigma_1 = \hat{\rho}_{1,+}'$ and thus $\SimLoc(\rho_0' \restr_{\Omega_+},\rho_1' \restr_{\Omega_+}) \geq \SimLoc(\hat{\rho}_{0,+}',\allowbreak\hat{\rho}_{1,+}')$. With analogous arguments one shows $\SimLoc(\rho_0' \restr_{\Omega_-},\rho_1' \restr_{\Omega_-}) \geq \SimLoc(\hat{\rho}_{0,-}',\hat{\rho}_{1,-}')$ and thus $\SimLoc(\rho_0',\rho_1') \geq \SimLoc(\hat{\rho}_{0}',\hat{\rho}_{1}')$.
\end{proof}

As a further consequence we have that, in an optimal arrangement or during the optimal dynamics,
not only is mass never transported after it has been increased or before it is decreased,
but also no material is transported out of a region of mass increase before this increase happens,
and no material is transported into a region of mass decrease after this decrease happened.

\begin{corollary}[Necessary optimality condition II for static primal]\thlabel{thm:massTptChng}
If $(\pi_0,\pi_1)$ are minimizers of \eqref{eq:StaticPrimalProblem}, then
\begin{equation*}
\pi_0\restr(\Omega_+\times(\Omega\setminus\Omega_+))=0\,,\qquad
\pi_1\restr((\Omega\setminus\Omega_-)\times\Omega_-)=0\,.
\end{equation*}
\end{corollary}
\begin{proof}
We show $\pi_0\restr(\Omega_+\times(\Omega\setminus\Omega_+))=0$, the other statement follows analogously.
We need to show $\pi_0\restr(\Omega_+\times\Omega_=)=0$ since $\pi_0\restr(\Omega_+\times\Omega_-)=0$ by the previous proposition.
Assume $\pi_0\restr(\Omega_+\times\Omega_=)\neq0$, then one can modify $\pi_0$ and $\pi_1$ to $\hat\pi_0$ and $\hat\pi_1$ such that their cost does not increase.
Indeed, consider all the mass being transported by $\pi_0$ from $\Omega_+$ to $\Omega_=$ and then further by $\pi_1$ from $\Omega_=$ to $\Omega$.
The modification of the couplings is to transport this mass not at all during the transport associated with $\hat\pi_0$
and to transport it instead via $\hat\pi_1$ directly from $\Omega_+$ to the final position in $\Omega$ without the intermediate deposition in $\Omega_=$.
In detail, using the disintegration notation from the previous proof, we choose $\hat\pi_0$ and $\hat\pi_1$ according to
\begin{align*}
\int_{\Omega\times\Omega}\phi\,\d\hat\pi_0
&=\int_{(\Omega\times\Omega)\setminus(\Omega_+\times\Omega_=)}\phi\,\d\pi_0
+\int_{\Omega_+\times\Omega_=}\phi(x,x)\,\d\pi_0(x,y)\,,\\
\int_{\Omega\times\Omega}\phi\,\d\hat\pi_1
&=\int_{\Omega\times\Omega}\phi\,\d\pi_1
-\int_{\Omega_=}\int_{\Omega}\phi(z,y)\,\d\pi_{1,z}(y)\d P_Y(\pi_0\restr(\Omega_+\times\Omega_=))(z)\\
&\hspace*{10em}+\int_{\Omega_+\times\Omega_=}\int_\Omega\phi(x,y)\,\d\pi_{1,z}(y)\d\pi_0(x,z)
\end{align*}
for all measurable $\phi:\Omega^2\to[0,\infty]$.
It is readily checked that $\hat\pi_0,\hat\pi_1\in\measp(\Omega\times\Omega)$ as well as $P_X\hat\pi_0=P_X\pi_0=\rho_0$ and $P_Y\hat\pi_1=P_Y\pi_1=\rho_1$ so that $\hat\pi_0$ and $\hat\pi_1$ are admissible.
Likewise it is straightforward to check $\Delta\rho \eqdef \hat\rho_1'-\hat\rho_0'=P_X\hat\pi_1-P_Y\hat\pi_0=P_X\pi_1-P_Y\pi_0=\rho_1'-\rho_0'$ so that the same mass change happens as before.
In particular, the domains $\Omega_+,\Omega_-,\Omega_=$ did not change, and $\hat\rho_0'\restr\Omega_+\geq\rho_0'\restr\Omega_+$ as well as $\hat\rho_0'\restr\Omega_-=\rho_0'\restr\Omega_-$.
Therefore,
\begin{multline*}
\SimLoc(\hat\rho_0',\hat\rho_1')
=\SimLoc(\hat\rho_0'\restr(\Omega\setminus\Omega_=),\hat\rho_0'\restr(\Omega\setminus\Omega_=)+\Delta\rho)
\\
\leq\SimLoc(\rho_0'\restr(\Omega\setminus\Omega_=),\rho_0'\restr(\Omega\setminus\Omega_=)+\Delta\rho)
=\SimLoc(\rho_0',\rho_1')
\end{multline*}
since $m \mapsto \SimLocC(m,m+\Delta m)$ is nonincreasing.
(Indeed, using the 1-homogeneity and convexity of $\SimLocC$ one has $\SimLocC(m+\delta,m+\delta+\Delta m) = \tfrac{m+\delta}{m} \SimLocC(m,m+\tfrac{m}{m+\delta} \Delta m) \leq \tfrac{m+\delta}{m} \big( (1-\tfrac{m}{m+\delta}) \cdot \SimLocC(m,m) + \tfrac{m}{m+\delta} \SimLocC(m,m+\Delta m) \big)=\SimLocC(m,m+\Delta m)$. In words, the mass increase in $\Omega_+$ is identical to before, but may have a cheaper cost since it starts from an already higher mass.)
Finally, the transported mass and the transport distance per particle (and thus the total transport cost) is no larger than before,
\begin{align}
 & \int_{\Omega\times\Omega}d\,\d\hat\pi_0+\int_{\Omega\times\Omega}d\,\d\hat\pi_1 \nonumber \\
=&\int_{(\Omega\times\Omega)\setminus(\Omega_+\times\Omega_=)}d\,\d\pi_0
+\int_{\Omega\times\Omega}d\,\d\pi_1
+\int_{\Omega_+\times\Omega_=}\int_\Omega d(x,y)\,\d\pi_{1,z}(y)\d\pi_0(x,z)\nonumber\\
&-\int_{\Omega_=}\int_{\Omega}d(z,y)\,\d\pi_{1,z}(y)\d P_Y(\pi_0\restr(\Omega_+\times\Omega_=))(z)\nonumber\\
\leq&\int_{(\Omega\times\Omega)\setminus(\Omega_+\times\Omega_=)}d\,\d\pi_0
+\int_{\Omega\times\Omega}d\,\d\pi_1
+\int_{\Omega_+\times\Omega_=}\int_\Omega d(x,z)+d(z,y)\,\d\pi_{1,z}(y)\d\pi_0(x,z)\nonumber\\
&-\int_{\Omega_=}\int_{\Omega}d(z,y)\,\d\pi_{1,z}(y)\d P_Y(\pi_0\restr(\Omega_+\times\Omega_=))(z)\nonumber\\
=&\int_{\Omega\times\Omega}d\,\d\pi_0+\int_{\Omega\times\Omega}d\,\d\pi_1\,.\label{eqn:noLargerTransportCost}
\end{align}
Summarizing, the modification does not increase the cost so that $(\hat\pi_0,\hat\pi_1)$ are optimal as well.
However, using the optimality of $(\pi_0,\pi_1)$ and $(\hat\pi_0,\hat\pi_1)$ and the previous proposition,
\begin{equation*}
0=\int_{\Omega_+\times\Omega}d\,\d\hat\pi_1
=\int_{\Omega_+\times\Omega_=}\int_\Omega d(x,y)\,\d\pi_{1,z}(y)\d\pi_0(x,z)
\end{equation*}
so that the inequality in \eqref{eqn:noLargerTransportCost} is actually strict and thus $(\pi_0,\pi_1)$ cannot be optimal.
\end{proof}

The characterization from \thref{prop:GrowthShrinkDecomposition} and \thref{thm:massTptChng} can be refined even further, in particular for special cases.

\begin{figure}
\centering
\begin{tabular}{r|c|c|c|c}
&general case&$\SimLocC(1,\cdot)$ smooth at $1$&$\rho_0\perp\rho_1$&$\SimLocC(1,\cdot)=|1-\cdot|$\\\hline
\setlength{\unitlength}{4.5ex}
\raisebox{1.5\unitlength}{$\pi_0$}&
\setlength{\unitlength}{4.5ex}
\begin{picture}(4,3.5)(-1,0)
\multiput(0,0)(0,1){4}{\line(1,0){3}}
\multiput(0,0)(1,0){4}{\line(0,1){3}}
\put(0,3){\makebox(1,.5){\small$\Omega_+$}}
\put(1,3){\makebox(1,.5){\small$\Omega_=$}}
\put(2,3){\makebox(1,.5){\small$\Omega_-$}}
\put(-.9,2){\makebox(1,1){\small$\Omega_+$}}
\put(-.9,1){\makebox(1,1){\small$\Omega_=$}}
\put(-.9,0){\makebox(1,1){\small$\Omega_-$}}
\put(0.1,0.1){\color{gray}\rule{.8\unitlength}{.8\unitlength}}
\put(0.1,1.1){\color{gray}\rule{.8\unitlength}{.8\unitlength}}
\put(0.1,2.1){\color{gray}\rule{.8\unitlength}{.8\unitlength}}
\put(1.1,1.1){\color{gray}\rule{.8\unitlength}{.8\unitlength}}
\thicklines
\put(2.1,.9){\line(1,-1){.8}}
\end{picture}
&
\setlength{\unitlength}{4.5ex}
\begin{picture}(4,3.5)(-1,0)
\multiput(0,0)(0,1){4}{\line(1,0){3}}
\multiput(0,0)(1,0){4}{\line(0,1){3}}
\put(0,3){\makebox(1,.5){\small$\Omega_+$}}
\put(1,3){\makebox(1,.5){\small$\Omega_=$}}
\put(2,3){\makebox(1,.5){\small$\Omega_-$}}
\put(-.9,2){\makebox(1,1){\small$\Omega_+$}}
\put(-.9,1){\makebox(1,1){\small$\Omega_=$}}
\put(-.9,0){\makebox(1,1){\small$\Omega_-$}}
\put(0.1,0.1){\color{gray}\rule{.8\unitlength}{.8\unitlength}}
\put(0.1,1.1){\color{gray}\rule{.8\unitlength}{.8\unitlength}}
\put(0.1,2.1){\color{gray}\rule{.8\unitlength}{.8\unitlength}}
\thicklines
\put(1.1,1.9){\line(1,-1){.8}}
\put(2.1,.9){\line(1,-1){.8}}
\end{picture}
&
\setlength{\unitlength}{4.5ex}
\begin{picture}(4,3.5)(-1,0)
\multiput(0,0)(0,1){4}{\line(1,0){3}}
\multiput(0,0)(1,0){4}{\line(0,1){3}}
\put(0,3){\makebox(1,.5){\small$\Omega_+$}}
\put(1,3){\makebox(1,.5){\small$\Omega_=$}}
\put(2,3){\makebox(1,.5){\small$\Omega_-$}}
\put(-.9,2){\makebox(1,1){\small$\Omega_+$}}
\put(-.9,1){\makebox(1,1){\small$\Omega_=$}}
\put(-.9,0){\makebox(1,1){\small$\Omega_-$}}
\put(0.1,0.1){\color{gray}\rule{.8\unitlength}{.8\unitlength}}
\put(0.1,1.1){\color{gray}\rule{.8\unitlength}{.8\unitlength}}
\put(1.1,1.1){\color{gray}\rule{.8\unitlength}{.8\unitlength}}
\thicklines
\put(2.1,.9){\line(1,-1){.8}}
\end{picture}
&
\setlength{\unitlength}{4.5ex}
\begin{picture}(4,3.5)(-1,0)
\multiput(0,0)(0,1){4}{\line(1,0){3}}
\multiput(0,0)(1,0){4}{\line(0,1){3}}
\put(0,3){\makebox(1,.5){\small$\Omega_+$}}
\put(1,3){\makebox(1,.5){\small$\Omega_=$}}
\put(2,3){\makebox(1,.5){\small$\Omega_-$}}
\put(-.9,2){\makebox(1,1){\small$\Omega_+$}}
\put(-.9,1){\makebox(1,1){\small$\Omega_=$}}
\put(-.9,0){\makebox(1,1){\small$\Omega_-$}}
\put(0.1,0.1){\color{gray}\rule{.8\unitlength}{.8\unitlength}}
\put(0.1,1.1){\color{gray}\rule{.8\unitlength}{.8\unitlength}}
\put(1.1,1.1){\color{gray}\rule{.8\unitlength}{.8\unitlength}}
\thicklines
\put(0.1,2.9){\line(1,-1){.8}}
\put(2.1,.9){\line(1,-1){.8}}
\end{picture}
\\
\setlength{\unitlength}{4.5ex}
\raisebox{1.5\unitlength}{$\pi_1$}&
\setlength{\unitlength}{4.5ex}
\begin{picture}(4,3.5)(-1,0)
\multiput(0,0)(0,1){4}{\line(1,0){3}}
\multiput(0,0)(1,0){4}{\line(0,1){3}}
\put(0,3){\makebox(1,.5){\small$\Omega_+$}}
\put(1,3){\makebox(1,.5){\small$\Omega_=$}}
\put(2,3){\makebox(1,.5){\small$\Omega_-$}}
\put(-.9,2){\makebox(1,1){\small$\Omega_+$}}
\put(-.9,1){\makebox(1,1){\small$\Omega_=$}}
\put(-.9,0){\makebox(1,1){\small$\Omega_-$}}
\put(0.1,0.1){\color{gray}\rule{.8\unitlength}{.8\unitlength}}
\put(1.1,0.1){\color{gray}\rule{.8\unitlength}{.8\unitlength}}
\put(2.1,0.1){\color{gray}\rule{.8\unitlength}{.8\unitlength}}
\thicklines
\put(.1,2.9){\line(1,-1){.8}}
\put(1.1,1.9){\line(1,-1){.8}}
\end{picture}
&
\setlength{\unitlength}{4.5ex}
\begin{picture}(4,3.5)(-1,0)
\multiput(0,0)(0,1){4}{\line(1,0){3}}
\multiput(0,0)(1,0){4}{\line(0,1){3}}
\put(0,3){\makebox(1,.5){\small$\Omega_+$}}
\put(1,3){\makebox(1,.5){\small$\Omega_=$}}
\put(2,3){\makebox(1,.5){\small$\Omega_-$}}
\put(-.9,2){\makebox(1,1){\small$\Omega_+$}}
\put(-.9,1){\makebox(1,1){\small$\Omega_=$}}
\put(-.9,0){\makebox(1,1){\small$\Omega_-$}}
\put(0.1,0.1){\color{gray}\rule{.8\unitlength}{.8\unitlength}}
\put(1.1,0.1){\color{gray}\rule{.8\unitlength}{.8\unitlength}}
\put(2.1,0.1){\color{gray}\rule{.8\unitlength}{.8\unitlength}}
\thicklines
\put(.1,2.9){\line(1,-1){.8}}
\put(1.1,1.9){\line(1,-1){.8}}
\end{picture}
&
\setlength{\unitlength}{4.5ex}
\begin{picture}(4,3.5)(-1,0)
\multiput(0,0)(0,1){4}{\line(1,0){3}}
\multiput(0,0)(1,0){4}{\line(0,1){3}}
\put(0,3){\makebox(1,.5){\small$\Omega_+$}}
\put(1,3){\makebox(1,.5){\small$\Omega_=$}}
\put(2,3){\makebox(1,.5){\small$\Omega_-$}}
\put(-.9,2){\makebox(1,1){\small$\Omega_+$}}
\put(-.9,1){\makebox(1,1){\small$\Omega_=$}}
\put(-.9,0){\makebox(1,1){\small$\Omega_-$}}
\put(0.1,0.1){\color{gray}\rule{.8\unitlength}{.8\unitlength}}
\put(1.1,0.1){\color{gray}\rule{.8\unitlength}{.8\unitlength}}
\thicklines
\put(.1,2.9){\line(1,-1){.8}}
\put(1.1,1.9){\line(1,-1){.8}}
\end{picture}
&
\setlength{\unitlength}{4.5ex}
\begin{picture}(4,3.5)(-1,0)
\multiput(0,0)(0,1){4}{\line(1,0){3}}
\multiput(0,0)(1,0){4}{\line(0,1){3}}
\put(0,3){\makebox(1,.5){\small$\Omega_+$}}
\put(1,3){\makebox(1,.5){\small$\Omega_=$}}
\put(2,3){\makebox(1,.5){\small$\Omega_-$}}
\put(-.9,2){\makebox(1,1){\small$\Omega_+$}}
\put(-.9,1){\makebox(1,1){\small$\Omega_=$}}
\put(-.9,0){\makebox(1,1){\small$\Omega_-$}}
\put(0.1,0.1){\color{gray}\rule{.8\unitlength}{.8\unitlength}}
\put(1.1,0.1){\color{gray}\rule{.8\unitlength}{.8\unitlength}}
\thicklines
\put(.1,2.9){\line(1,-1){.8}}
\put(1.1,1.9){\line(1,-1){.8}}
\put(2.1,0.9){\line(1,-1){.8}}
\end{picture}
\end{tabular}
\caption[]{Illustration of \thref{thm:transportChar}: There are optimal couplings $\pi_0,\pi_1$ that take the shown form for the different particular cases.
Each square field corresponds to the direct product of the regions indicated by the row and column.
An empty field corresponds to $\pi$ being zero, a grey field indicates that $\pi$ may be nonzero, and a field with a diagonal line corresponds to diagonal support of $\pi$ (that is, mass is not moved in that region).
If $\SimLocC(1,\cdot)$ is differentiable at $1$ and at the same time $\rho_0\perp\rho_1$, then in addition $\pi_0\restr(\Omega_=\times\Omega_=)=\pi_1\restr(\Omega_=\times\Omega_=)=0$.}
\label{fig:transportStructure}
\end{figure}

\begin{proposition}[Characterization of optimal unbalanced mass transport]\thlabel{thm:transportChar}
For all $(\rho_0,\rho_1)$ and among all possible optimizers of \eqref{eq:StaticPrimalProblem}, there is always a particular one $(\pi_0,\pi_1)$ such that its support is as shown in Figure~\ref{fig:transportStructure} left, that is, in addition to \thref{prop:GrowthShrinkDecomposition} and \thref{thm:massTptChng} we have
\begin{equation*}
\pi_0(\Omega_-\times\Omega_=)
=\pi_1(\Omega_=\times\Omega_+)
=0
\qquad\text{and}\qquad
\int_{\Omega\setminus\Omega_-\times\Omega}d(x,y)\,\d\pi_1(x,y)=0\,.
\end{equation*}
Furthermore, $\rho_0'+\rho_1'\ll\rho_0+\rho_1$.
If in addition
\begin{enumerate}
\item $\SimLocC(1,\cdot)$ is differentiable at $1$ or\label{enm:diff}
\item $\rho_0\perp\rho_1$ or\label{enm:sing}
\item $\SimLocC(1,\cdot)=|1-\cdot|$,
\end{enumerate}
then the supports of the optimal $(\pi_0,\pi_1)$ can be chosen as shown in Figure~\ref{fig:transportStructure} in the second, third, and fourth column, respectively.
If case \eqref{enm:diff} and \eqref{enm:sing} hold simultaneously, then in addition $\pi_0\restr(\Omega_=\times\Omega_=)=\pi_1\restr(\Omega_=\times\Omega_=)=0$.
\end{proposition}

\begin{remark}\thlabel{rem:CombineOptimalityCriteria}
Of course, the cases from Figure~\ref{fig:transportStructure} can be combined.
For instance, if $\SimLocC(1,\cdot)$ is differentiable at $1$ and $\rho_0\perp\rho_1$, then the transport can be interpreted as follows (see Figure~\ref{fig:dynamicSteps}):
From points contributing to $\rho_0$ there is first mass transported to $\rho_1$ (blue), then the mass is shrunk in the $\rho_0$-region (orange) and grown in the $\rho_1$-region, and finally the rest is transported to $\rho_1$ (yellow).
Those steps are essentially a manifestation of the nonconvexity of the so-called path cost and its convexification being a combination of transport with subsequent growth and shrinkage with subsequent transport, see Section~\ref{sec:FormulationRelation}.
\end{remark}

\begin{figure}
\centering
\setlength\unitlength{.074\linewidth}
\begin{picture}(13.5,1.8)
\put(0,0){\includegraphics[width=13.5\unitlength]{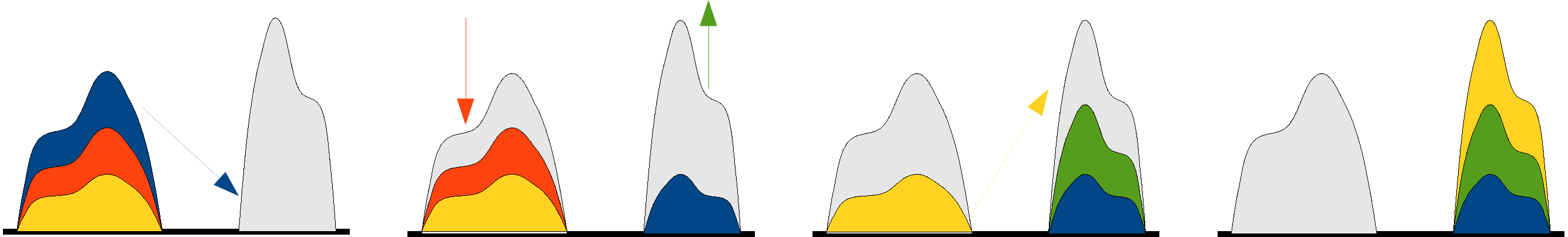}}
\put(.7,-.3){\small$\rho_0$}
\put(12.9,-.3){\small$\rho_1$}
\put(5.1,-.3){\small$\rho_0'$}
\put(8.6,-.3){\small$\rho_1'$}
\end{picture}
\caption[]{Illustration of the optimal dynamic mass transport as characterized by \thref{thm:transportChar}, consisting of an initial transport (blue), subsequent mass shrinkage (orange) and growth (green), and a final transport (yellow). See \thref{rem:CombineOptimalityCriteria} for more details.}
\label{fig:dynamicSteps}
\end{figure}

\begin{proof}[Proof of \thref{thm:transportChar}]
Note first that by \thref{prop:StaticEquivalence}, if the cost is finite then minimizing pairs $(\pi_0,\pi_1)$ exist.

\emph{General case.}
By \thref{prop:GrowthShrinkDecomposition} and \thref{thm:massTptChng} it only remains to show that optimal $(\pi_0,\pi_1)$ can be found such that $\pi_1$ is diagonal on $\Omega_=\times\Omega_=$ (that is, $\int_{\Omega_=\times\Omega_=}\phi(x,y)\,\d\pi_1(x,y)=\int_{\Omega_=\times\Omega_=}\phi(x,x)\,\d\pi_1(x,y)$ for all test functions $\phi$) and that $\pi_0\restr(\Omega_-\times\Omega_=)=0$ as well as $\pi_1\restr(\Omega_=\times\Omega_+)=0$
(of which we shall only show the first, the second follows analogously).
To this end we just modify any optimal couplings $(\pi_0,\pi_1)$ without increasing their cost.
In particular, consider the mass transported from $\Omega$ to $\Omega_=$ by $\pi_0$ and then transported further within $\Omega_=$ by $\pi_1$.
We modify $\pi_0$ and $\pi_1$ such that the transport of this mass from its initial distribution inside $\rho_0$ to its final distribution inside $\rho_1$ happens already in $\pi_0$ and thus $\pi_1$ no longer transports mass within $\Omega_=$
(corresponding formulas for the modified $\pi_0$ and $\pi_1$ are straightforward to obtain as in the two previous proofs).
Obviously, the mass change cost is identical after the modification, and the transport cost has not increased so that we may assume $\pi_1$ to be diagonal on $\Omega_=\times\Omega_=$.
Next assume $\pi_0\restr(\Omega_-\times\Omega_=)\neq0$ and modify $\pi_0$ and $\pi_1$ as follows.
Consider the mass transported from $\Omega_-$ to $\Omega_=$ by $\pi_0$ and then transported further by $\pi_1$.
We modify $\pi_0$ and $\pi_1$ such that $\pi_0$ no longer transports mass from $\Omega_-$ to $\Omega_=$, but instead the transport of this mass happens in $\pi_1$
(again, corresponding formulas for the modified $\pi_0$ and $\pi_1$ are straightforward to obtain as in the previous proofs).
Obviously, the transport cost has not increased by the modification.
Furthermore, the amount of mass decrease on $\Omega_-$ is identical to before, but may have a cheaper cost since it starts from a higher initial mass and $\SimLocC(m,m-\Delta m)$ is nonincreasing in $m$.
Finally, the amount and cost of mass increase on $\Omega_+$ stays unchanged.

The property $\rho_0'+\rho_1'\ll\rho_0+\rho_1$ can now be inferred as follows:
By definition of $\Omega_+$ we have $\rho_0'\restr\Omega_+ \leq \rho_1'\restr\Omega_+$ and therefore $\rho_0'\restr\Omega_+\ll\rho_1'\restr\Omega_+$, which is absolutely continuous with respect to $\rho_1\restr\Omega_+$ since $\pi_1$ is diagonal on $\Omega_+ \times \Omega_+$.
Likewise, $\rho_1'\restr\Omega_-\ll\rho_0'\restr\Omega_-\ll\rho_0\restr\Omega_-$ (mass decreases on $\Omega_-$ and $\pi_0$ is diagonal on $\Omega_- \times \Omega_-$).
Finally, $\rho_1'\restr\Omega_=\ll\rho_1\restr\Omega_=$ by the structure of $\pi_1$.
If we now use Lebesgue's decomposition theorem to decompose $\rho_0'\restr\Omega_==\nu_0+\nu_1$
with $\nu_0\ll\rho_0+\rho_1$ and $\nu_1(\Omega\setminus\Omega')=0$ for some $\Omega'\subset\Omega_=$ with $(\rho_0+\rho_1)(\Omega')=0$,
then $\nu_1\ll\rho_0'\restr\Omega'=\rho_1'\restr\Omega'\ll\rho_1\restr\Omega'=0$
(note that $\rho_0'\restr{\Omega_=}=\rho_1'\restr{\Omega_=}$ was not obvious a priori).
Thus, $\rho_0',\rho_1'\ll\rho_0+\rho_1$.

\emph{Case $\SimLocC(1,\cdot)$ differentiable at $1$.}
For a contradiction assume $\pi_0\restr(\Omega_=\times\Omega_=)$ is not diagonal.
Define $\rho_0^==P_X(\pi_0\restr(\Omega_=\times\Omega_=))$ as well as the diagonal coupling $\pi_0^d(A\times B)=\rho_0^=(A\cap B)$ for all Borel sets $A,B\subset\R^n$.
Next define the coupling $\pi_0^\varepsilon$ by setting $\pi_0^\varepsilon=\pi_0$ on $(\Omega\times\Omega)\setminus(\Omega_=\times\Omega_=)$
and $\pi_0^\varepsilon=(1-\varepsilon)\pi_0+\varepsilon\pi_0^d$ on $\Omega_=\times\Omega_=$ for some $\varepsilon\in[0,1]$.
Note that the pair $(\pi_0^\varepsilon,\pi_1)$ is an admissible candidate in \eqref{eq:StaticPrimalProblem}.
Denote its cost by $W^\varepsilon$ and note $W_S(\rho_0,\rho_1)=W^0$ due to the assumption of $(\pi_0,\pi_1)$ being optimal.
Then at $\varepsilon=0$ we have
\begin{multline*}
\frac{\d W^\varepsilon}{\d\varepsilon}
=-\varepsilon\int_{\Omega_=\times\Omega_=}d(x,y)\,\d\pi_0(x,y)+\frac\d{\d\varepsilon}\SimLoc(P_Y \pi_0^\varepsilon, P_X \pi_1)\\
=-\varepsilon\int_{\Omega_=\times\Omega_=}d(x,y)\,\d\pi_0(x,y)+\int_{\Omega_=}\frac\d{\d\varepsilon}\SimLocC\left(\RadNik{P_Y\pi_0^\varepsilon}{\rho_1'},1\right)\,\d\rho_1'
=-\varepsilon\int_{\Omega_=\times\Omega_=}d(x,y)\,\d\pi_0(x,y)
<0
\end{multline*}
since $\frac\d{\d\varepsilon}\SimLocC\big(\RadNik{P_Y\pi_0^\varepsilon}{\rho_1'},1\big)|_{\varepsilon=0}=0$.
This contradicts the optimality of $(\pi_0,\pi_1)$.

\emph{Case $\rho_0\perp\rho_1$.}
Using the general structure of $\pi_0$ and $\pi_1$, for $E\subset\Omega_+$ measurable we have
\begin{equation*}
\rho_1\restr E
=P_Y(\pi_1\restr(\Omega\times E))
\gg P_Y(\pi_1\restr(E\times E))
= P_X(\pi_1\restr(E\times E))
=\rho_1'\restr E
\gg\rho_0'\restr E\,.
\end{equation*}
Thus, $\rho_1\restr E=0$ would imply $\rho_1'\restr E=\rho_0'\restr E=0$ so that, due to $E\subset\Omega_+$, $E$ can only be a $\gamma$-nullset (where $\gamma$ is the reference measure used in \thref{prop:GrowthShrinkDecomposition} to define $\Omega_+$, $\Omega_-$ and $\Omega_=$).
Now assume $\rho_0\restr\Omega_+\neq0$, then due to the singularity with $\rho_1$ there is a measurable $E\subset\Omega_+$ with $\rho_0\restr E\neq0$ and $\rho_1\restr E=0$.
By the above, $E$ must be a $\gamma$-nullset, which contradicts $\rho_0\restr E\neq0$.
Therefore, $\rho_0\restr\Omega_+=0$ and thus also $\pi_0\restr(\Omega_+\times\Omega)=0$.

Likewise one shows that $\rho_0\restr E=0$ with $E\subset\Omega_-$ can only happen for $\gamma$-nullsets $E$ and that $\pi_1\restr(\Omega\times\Omega_-)=0$.

If in addition $\SimLocC(1,\cdot)$ is differentiable at $1$, then from the previous case we know
\begin{align*}
\rho_0\restr\Omega_=
&\gg P_X(\pi_0\restr(\Omega_=\times\Omega_=))
=P_Y(\pi_0\restr(\Omega_=\times\Omega_=))
=\rho_0'\restr\Omega_=\,,\\
\rho_1\restr\Omega_=
&\gg P_Y(\pi_1\restr(\Omega_=\times\Omega_=))
=P_X(\pi_1\restr(\Omega_=\times\Omega_=))
=\rho_1'\restr\Omega_=\,.
\end{align*}
Hence, since $\rho_0$ and $\rho_1$ are singular with respect to each other, so are $\rho_0'\restr\Omega_=$ and $\rho_1'\restr\Omega_=$.
However, due to $\rho_0'\restr\Omega_==\rho_1'\restr\Omega_=$ this is only possible if both are zero, which in turn implies $\pi_0\restr(\Omega\times\Omega_=)=\pi_1\restr(\Omega_=\times\Omega)=0$.

\emph{Case $\SimLocC(1,\cdot)=|1-\cdot|$.}
Note that by \thref{prop:StaticEquivalence} the corresponding function $\HStat$ is given by $\HStat(z)=\min(z,1)-\iota_{[-1,\infty)}(z)$ so that the dual formulation \eqref{eq:StaticDual} can be transformed into
\begin{equation*}
\textstyle W_S(\rho_0,\rho_1)=\sup\{\int_\Omega\alpha\,\d(\rho_0-\rho_1)\ |\ \alpha\in\Lip(\Omega),|\alpha|\leq1\}\,.
\end{equation*}
This implies that the cost actually only depends on $\rho_0-\rho_1$.
Therefore, defining
\begin{equation*}
\rho=\min\left(\frac{\d\rho_0}{\d(\rho_0+\rho_1)},\frac{\d\rho_1}{\d(\rho_0+\rho_1)}\right)(\rho_0+\rho_1)\,,\qquad
\tilde\rho_0=\rho_0-\rho\in\measp(\Omega)\,,\quad
\tilde\rho_1=\rho_1-\rho\in\measp(\Omega)\,,
\end{equation*}
we see $W_S(\tilde\rho_0,\tilde\rho_1)=W_S(\rho_0,\rho_1)$.
Now given optimal couplings $(\tilde\pi_0,\tilde\pi_1)$ for $W_S(\tilde\rho_0,\tilde\rho_1)$ we simply set $\pi_0=\tilde\pi_0+\pi$, $\pi_1=\tilde\pi_1+\pi$,
where $\pi$ is the diagonal coupling defined via $\int_{\Omega\times\Omega}\phi\,\d\pi=\int_\Omega\phi(x,x)\,\d\rho(x)$ for all measurable $\phi:\Omega^2\to[0,\infty]$.
It is straightforward to check that $(\pi_0,\pi_1)$ is a feasible candidate in \eqref{eq:StaticPrimalProblem} and that it has the same cost as $(\tilde\pi_0,\tilde\pi_1)$.
Therefore, $(\pi_0,\pi_1)$ must be optimal.
Note also that $\Omega_\pm$ are the same for $(\pi_0,\pi_1)$ as for $(\tilde\pi_0,\tilde\pi_1)$.
Now by the previous case, $\tilde\pi_0\restr(\Omega_+\times\Omega_+)=\tilde\pi_1\restr(\Omega_-\times\Omega_-)=0$
so that $\pi_0\restr(\Omega_+\times\Omega_+)=\pi\restr(\Omega_+\times\Omega_+)$ and $\pi_1\restr(\Omega_-\times\Omega_-)=\pi\restr(\Omega_-\times\Omega_-)$ are diagonal.
\end{proof}

The final result of this section shows that the maximal transport distance may be bounded depending on the model.
It implies for instance that, if the supports of $\rho_0$ and $\rho_1$ have distance larger than some threshold (which for some models is infinite), then no transport happens at all.
Since the transport component of the total cost depends on the transport distance while the mass change cost is independent of it and only depends on how much mass is created or removed, such a threshold effect is not unexpected.

\begin{proposition}[Maximal transport distance]\thlabel{thm:maxTransportDistance}%
\thlabel{prop:MaximalTransportDistance}
Any optimizer $(\pi_0,\pi_1)$ of \eqref{eq:StaticPrimalProblem} satisfies
\begin{equation*}
\pi_0\restr\{(x,y)\in\Omega\times\Omega\,|\,d(x,y)>L_0\}
=\pi_1\restr\{(x,y)\in\Omega\times\Omega\,|\,d(x,y)>L_1\}
=0
\end{equation*}
for $\displaystyle L_0=\lim_{a\to\infty}\partial_1\SimLocC(a,1)-\lim_{a\to0}\partial_1\SimLocC(a,1)\in(0,\infty]$,
$\displaystyle L_1=\lim_{a\to\infty}\partial_2\SimLocC(1,a)-\lim_{a\to0}\partial_2\SimLocC(1,a)\in(0,\infty]$,
where $\partial_i$ denotes the $i$th (left or right) partial derivative.
\end{proposition}
\begin{proof}
We only prove $\pi_1\restr\{(x,y)\in\Omega\times\Omega\,|\,d(x,y)>L_1\}=0$, the other statement follows analogously.
Let $\Delta\geq0$ arbitrary and abbreviate
\begin{equation*}
L^r=\lim_{a\to\infty}\partial_2\SimLocC(1,a)\,,\qquad
L^l=-\lim_{a\to0}\partial_2\SimLocC(1,a)\,.
\end{equation*}
By convexity of $\SimLocC(1,\cdot)$ we have for any $b\geq0$
\begin{equation*}
\SimLocC(1,b+\Delta)\leq\SimLocC(1,b)+\partial_2\SimLocC(1,b+\Delta)\Delta\leq\SimLocC(1,b)+L^r\Delta
\end{equation*}
and thus by the one-homogeneity $\SimLocC(a,b+\Delta)\leq\SimLocC(a,b)+L^r\Delta$ for all $a\geq0$.
Analogously one obtains $\SimLocC(a,b-\Delta)\leq\SimLocC(a,b)+L^l\Delta$ for $a\geq 0$, $b\geq \Delta$.
As a consequence, for any nonnegative measures $\rho',\tilde\rho',\hat\rho'$ with $\rho'-\tilde\rho'+\hat\rho'$ nonnegative we have
\begin{equation*}
\SimLoc(\rho_0',\rho'-\tilde\rho'+\hat\rho')
\leq\SimLoc(\rho_0',\rho')+L^l\tilde\rho'(\Omega)+L^r\hat\rho'\,.
\end{equation*}
Now assume $(\pi_0,\pi_1)$ optimizes \eqref{eq:StaticPrimalProblem} with $\pi_1^{\text{far}}\equiv\pi_1\restr\{(x,y)\in\Omega\times\Omega\,|\,d(x,y)>L_1\}\neq0$.
We will construct a competitor with strictly lower cost.
Define $\pi_1^{\text{cls}}$ and $\pi_1^{\text{dgl}}$ via
\begin{equation*}
\pi_1^{\text{cls}}=\pi_1-\pi_1^{\text{far}}
\qquad\text{and}\qquad
\int_{\Omega\times\Omega}\phi\,\d\pi_1^{\text{dgl}}=\int_{\Omega\times\Omega}\phi(y,y)\,\d\pi_1^{\text{far}}(x,y)
\end{equation*}
for all continuous $\phi:\Omega\times\Omega\to\R$,
that is, $\pi_1^{\text{far}}$ describes the transport of mass that moves farther than $L_1$ and $\pi_1^{\text{dgl}}$ is the coupling corresponding to not moving that mass at all.
Our competitor is $\tilde\pi_1=\pi_1^{\text{cls}}+\pi_1^{\text{dgl}}=\pi_1-\pi_1^{\text{far}}+\pi_1^{\text{dgl}}$, and we estimate
\begin{multline*}
P_S(\pi_0,\tilde\pi_1)
=\int_{\Omega\times\Omega} d\,\d \pi_0 + \SimLoc(P_Y \pi_0, P_X \pi_1-P_X\pi_1^{\text{far}}+P_X\pi_1^{\text{dgl}}) + \int_{\Omega\times\Omega} d\,\d (\pi_1-\pi_1^{\text{far}}+\pi_1^{\text{dgl}})\\
<\int_{\Omega\times\Omega} d\,\d \pi_0 + \SimLoc(P_Y \pi_0, P_X \pi_1) +L^l\pi_1^{\text{far}}(\Omega\times\Omega)+L^r\pi_1^{\text{dgl}}(\Omega\times\Omega) + \int_{\Omega\times\Omega} d\,\d\pi_1-L_1\pi_1^{\text{far}}(\Omega\times\Omega)\\
=P_S(\pi_0,\pi_1)+(L^l+L^r-L_1)\pi_1^{\text{far}}(\Omega\times\Omega)
=P_S(\pi_0,\pi_1)\,,
\end{multline*}
where we used $(P_X\pi_1^{\text{far}})(\Omega)=\pi_1^{\text{far}}(\Omega\times\Omega)=(P_X\pi_1^{\text{dgl}})(\Omega)=\pi_1^{\text{dgl}}(\Omega\times\Omega)$.
This contradicts the optimality of $(\pi_0,\pi_1)$ and thus concludes the proof.
\end{proof}

% !TEX root = W1TypeDynamic.tex
\section{Overview of and relation between static and dynamic formulations}

In this section we continue and extend the discussion from the introduction on different model formulations and their relations.
In particular, we provide a semi-coupling formulation of the static $W_1$-type problem, and we identify, which static $W_1$-type problems have an equivalent dynamic formulation.

\subsection{Semi-coupling formulation of unbalanced transport}
\label{sec:SemiCouplingFormulations}
In this section we establish the relation between the formulations for transport problems from Section \ref{sec:Reminder} with the semi-coupling formulation introduced in \cite{ChizatDynamicStatic2015}.
A main result of \cite{ChizatDynamicStatic2015}, Theorem 4.3, is the relation between dynamic problems and the semi-coupling formulation.
Using the results of Section \ref{sec:equivalence} we give an alternative proof for this statement in the special case of the $W_1$-type setting (\thref{prop:SCW1TEquivalence}).

\begin{definition}[Primal Semi-Coupling Formulation \protect{\cite[Definition 3.3]{ChizatDynamicStatic2015}}]
	\label{def:SemiCouplingPrimal}
	For $\rho_0$, $\rho_1 \in \measp(\Omega)$ the corresponding set of \emph{semi-couplings} is given by
	\begin{align}
		\Gamma(\rho_0,\rho_1) & = \left\{
			(\gamma_0,\gamma_1) \in (\measp(\Omega\times\Omega))^2
				\,\middle|\,
				P_X \gamma_0 = \rho_0,\,
				P_Y \gamma_1 = \rho_1
				\right\}\,.
	\end{align}
	Let further $\SCc : (\Omega \times \R)^2 \to \RCupInf$ be a lower semi-continuous function such that for all fixed $(x_0,x_1) \in \Omega\times\Omega$ the function
	$(m_0,m_1) \mapsto \SCc(x_0,m_0,x_1,m_1)$ is convex and 1-homogeneous with $\SCc(x_0,m_0,x_1,m_1)= \infty$ if $m_0<0$ or $m_1<0$.
	The \emph{primal semi-coupling problem} is given by
	\begin{equation}
		\label{eq:SCPrimal}
		W_{SC}(\rho_0,\rho_1) = \inf \left\{ \int_{\Omega^2}
			\SCc\left(x_0,\RadNik{\gamma_0}{\gamma}(x_0,x_1),x_1,\RadNik{\gamma_1}{\gamma}(x_0,x_1)\right)\,\d \gamma(x_0,x_1)
			\middle|
			(\gamma_0,\gamma_1) \in \Gamma(\rho_0,\rho_1)
			\right\}\,,
	\end{equation}
	where $\gamma \in \measp(\Omega\times\Omega)$ is such that $\gamma_0$, $\gamma_1 \ll \gamma$.
\end{definition}
The semi-coupling cost $\SCc(x_0,m_0,x_1,m_1)$ may be interpreted as the cost for starting out at $x_0$ with mass $m_0$ destined for $x_1$, but only arriving at $x_1$ with changed mass $m_1$.
Again, there is a corresponding dual problem.
\begin{proposition}[Dual Semi-Coupling Formulation \protect{\cite[Proposition 3.11]{ChizatDynamicStatic2015}}]
	\thlabel{prop:SemiCouplingDual}
	For a function $\SCc$ as in Definition \ref{def:SemiCouplingPrimal}, let
	\begin{align*}
		\SCB(x_0,x_1) & = \left\{(\alpha,\beta) \in \R^2
			\,\middle|\,
			[\SCc(x_0,\cdot,x_1,\cdot)]^\ast(\alpha,\beta) < \infty \right\}.
	\end{align*}
	Then
	\begin{multline}
		W_{SC}(\rho_0,\rho_1) = \sup \left\{
			\int_\Omega \alpha\,\d\rho_0 + \int_\Omega \beta\, \d \rho_1 \right|
			(\alpha,\beta) \in C(\Omega)^2 , \\
			\left. \vphantom{\int_\Omega}
			(\alpha(x_0),\beta(x_1)) \in \SCB(x_0,x_1) \tn{ for all } (x_0,x_1) \in \Omega\times\Omega \right\}\,.
			\label{eq:SCDual}
	\end{multline}
\end{proposition}

A main result of \cite{ChizatDynamicStatic2015} was to show that for suitable dynamic problems, including those of the form of Definition \ref{def:Dynamic}, there is an equivalent semi-coupling formulation.
\begin{proposition}[Equivalent semi-coupling formulation for dynamic problem \protect{\cite[Theorem 4.3]{ChizatDynamicStatic2015}}]\thlabel{thm:semiCoupling}
	For a dynamic primal problem as in Definition \ref{def:Dynamic}, the associated path cost $\CPath : (\Omega \times \R)^2 \mapsto [0,\infty]$ is given by
	\begin{multline*}
		\CPath(x_0,m_0,x_1,m_1) = \inf\left\{
			\int_0^1 m(t)\cdot \|\partial_t x(t)\| + \CDynM(m(t),\partial_t m(t)) \,\d t
			\right| \\
			(x,m) \in C^1([0,1],\Omega \times \R) ,
			\left. \vphantom{\int_0^1} (x(i),m(i)) = (x_i,m_i) \tn{ for } i \in \{0,1\}
			\right\}\,.
	\end{multline*}
	Assume that there exists some $C < \infty$ such that $\CDynM(\rho,2\zeta) \leq C \cdot \CDynM(\rho,\zeta)$ for all $(\rho,\zeta) \in \R^2$,
	and let $\SCc$ be the convex envelope of $\CPath$ with respect to the arguments $(m_0,m_1)$. Then $W_{D} = W_{SC}$.
\end{proposition}

A similar result can be established by finding an equivalent semi-coupling formulation for static $W_1$-type problems and then using the equivalence between dynamic and static $W_1$-type problems (\thref{prop:W1DynamicStaticEquivalence}, see also \thref{prop:W1DynamicStaticEquivalencePrimal}).
This strategy has the advantage that one avoids computing the path cost $\CPath$ and obtains a more explicit form of the semi-coupling cost function $\SCc$, involving only the static mass change cost $\SimLocC$ from \eqref{eq:InducedSimLocC}.
\newcommand{\cPre}{c_{\tn{pre}}}
\begin{proposition}
	\thlabel{prop:SCW1TEquivalence}
	For a given static $W_1$-type problem according to Definition \ref{def:StaticPrimal} and \thref{prop:StaticEquivalence}, characterized by a function $\SimLocC$ or equivalently a set $\BStat$, let
	\begin{equation*}
		\cPre(x_0,m_0,x_1,m_1) = \SimLocC(m_0,m_1) + \min\{m_0,m_1\} \cdot d(x_0,x_1)
  \end{equation*}
	and denote by $\cPre^{\ast \ast}(x_0,m_0,x_1,m_1)$ its convex envelope with respect to the arguments $(m_0,m_1)$. Then for the choice $\SCc = \cPre^{\ast \ast}$ or equivalently
	\begin{equation*}
		\SCB(x_0,x_1) = [ \BStat + \{(0,d(x_0,x_1))\} ] \cap
			[\BStat + \{(d(x_0,x_1),0)\}]
	\end{equation*}
	one has $W_S = W_{SC}$.
\end{proposition}

The proof is split into the two following lemmas. First, equivalence of the dual formulations is established.
\begin{lemma}
	For $\SCB(x_0,x_1)$ as given in \thref{prop:SCW1TEquivalence}, the dual static $W_1$-type problem from \thref{prop:StaticEquivalence} is equivalent to the dual semi-coupling problem from \thref{prop:SemiCouplingDual}.
\end{lemma}
\begin{proof}
	Throughout this proof we write $\Delta x = d(x_0,x_1)$.
	Note that the set $\SCB(x_0,x_1)$ can be written as
	\begin{align}
		\label{eq:SCB}
		\SCB(x_0,x_1) = \{(a,b) \in \R^2 \,|\, [b \leq \HStat(-a+\Delta x)] \text{ and }
			[b - \Delta x \leq \HStat(-a) ] \}\,.
	\end{align}
	We show first that every feasible candidate of \eqref{eq:StaticDual} is also feasible for \eqref{eq:SCDual}.
	Let $(\alpha,\beta) \in \Lip(\Omega)^2$, $(\alpha(x),\beta(x)) \in \BStat$ for all $x \in \Omega$.
	Then for all $(x_0,x_1) \in \Omega\times\Omega$ we have
	\begin{align*}
		\beta(x_1) & \leq \HStat(-\alpha(x_1)) \leq \HStat(-\alpha(x_0)+\Delta x), &
		\beta(x_1)-\Delta x & \leq \beta(x_0) \leq \HStat(-\alpha(x_0))\,.
	\end{align*}
	Hence, $(\alpha(x_0),\beta(x_1)) \in \SCB(x_0,x_1)$ and thus $W_{SC} \geq W_{S}$.
	
	Conversely, for every feasible candidate $(\alpha,\beta)$ of \eqref{eq:SCDual} we construct a feasible candidate with potentially better score for \eqref{eq:StaticDual}.
	Since $\alpha$ and $\beta$ might not be in $\Lip(\Omega)$ we may have to alter them. We first modify $\beta$ for fixed $\alpha$ and then vice versa.
	Recall that $\HStat$ is concave and $\HStat'(0)=1$. Hence, for some $a \geq \Delta x$ any super-gradient $\zeta$ of $\HStat$ at $-a+\Delta x$ satisfies $\zeta \geq 1$ and consequently
	\begin{align*}
		\HStat(-a+\Delta x) \geq \HStat(-a)+\Delta x \cdot \zeta
			\geq \HStat(-a)+\Delta x\,.
	\end{align*}
	We introduce the auxiliary function
	\begin{align*}
		\HStatAux(z) = \begin{cases}
			\HStat(z) & \tn{if } z\geq 0, \\
			z & \tn{if } z \leq 0.
			\end{cases}
	\end{align*}
	Note that $\HStatAux$ is Lipschitz and $\HStatAux(z) \geq \HStat(z)$ so that
	\begin{align*}
    \SCB(x_0,x_1) &= \{(a,b) \in \R^2 \,|\, b \leq \min\{\HStat(-a+\Delta x),\HStat(-a)+\Delta x\} \}\\
    &= \{(a,b) \in \R^2 \,|\, b \leq \min\{\HStatAux(-a+\Delta x),\HStat(-a)+\Delta x\} \}\,.
	\end{align*}
	Now set
	\begin{align*}
		\beta_1(x_1) & = \inf\{ \HStat(-\alpha(x_0)) + \Delta x \,|\, x_0 \in \Omega \}, &
		\beta_2(x_1) & = \inf\{ \HStatAux(-\alpha(x_0)+\Delta x) \,|\, x_0 \in \Omega \}\,.
	\end{align*}
	Using that $\HStatAux$ is Lipschitz, it is easy to show that $\beta_1$ and $\beta_2$ are Lipschitz, and consequently so is $\tilde{\beta} : x \mapsto \min\{\beta_1(x),\beta_2(x) \}$. Moreover, $\tilde{\beta} \geq \beta$, so $(\alpha,\tilde{\beta})$ is still feasible for \eqref{eq:SCDual} with a score not lower than the one for $(\alpha,\beta)$. Now, one performs the analogous modification for $\alpha$, yielding a pair $(\tilde{\alpha},\tilde{\beta})$ which does not decrease the score. Then $(\tilde{\alpha},\tilde{\beta}) \in \Lip(\Omega)^2$, and from the constraint $(\tilde{\alpha}(x),\tilde{\beta}(x)) \in \SCB(x,x)$ for $x \in \Omega$ we find that $(\tilde{\alpha},\tilde{\beta})$ is feasible for \eqref{eq:StaticDual}. Hence, $W_{SC} \leq W_{S}$ and finally $W_{SC}=W_S$.
\end{proof}

The next lemma establishes the corresponding primal semi-coupling formulation.
\begin{lemma}
	For $\SCB$ and $\SCc$ as given in \thref{prop:SCW1TEquivalence} we have $\iota_{\SCB(x_0,x_1)}=\allowbreak \SCc(x_0,\cdot,\allowbreak x_1,\cdot)^\ast$, and $\SCc$ is a valid cost function in the sense of Definition \ref{def:SemiCouplingPrimal}.
\end{lemma}
\begin{proof}
	Again, we write $\Delta x = d(x_0,x_1)$.
	For fixed $(x_0,x_1) \in \Omega\times\Omega$ the map $\R^2 \ni (m_0,m_1) \mapsto \cPre(x_0,m_0,x_1,m_1)$ is 1-homogeneous and lower semi-continuous. Therefore, by Carath\'eodory's Theorem \cite[Corollary 17.1.6]{Rockafellar1972Convex} its convex envelope can be written as
	\begin{multline}
		\label{eq:CPreConvHull}
		\cPre^{\ast \ast}(x_0,m_0,x_1,m_1) = \min\{\SimLocC(a_0,a_1) + \Delta x \cdot \min\{a_0,a_1\}\\ + \SimLocC(b_0,b_1) + \Delta x \cdot \min\{b_0,b_1\}\,|\,
				a_0,a_1,b_0,b_1 \in \R,
				a_0+b_0=m_0,
				a_1+b_1=m_1
				\}\,.
	\end{multline}
	Let us show that this is equivalent to
	\begin{equation}
		\label{eq:CPreInfConv}
			\min\{\SimLocC(a_0,a_1) + a_0 \Delta x + \SimLocC(b_0,b_1) + b_1 \Delta x\,|\,
				a_0,a_1,b_0,b_1 \in \R,
				a_0+b_0=m_0,
				a_1+b_1=m_1
				\}\,.
	\end{equation}
	It trivially holds \eqref{eq:CPreConvHull} $\leq$ \eqref{eq:CPreInfConv}, so now we consider the opposite inequality.
	For a feasible candidate $(a_0,a_1,b_0,b_1)$ of \eqref{eq:CPreConvHull} we distinguish the following cases.
	\begin{itemize}
		\item $[a_0 \leq a_1] \text{ and } [b_1 \leq b_0]$: The candidate is also feasible for \eqref{eq:CPreInfConv} with the same score.
		\item $[a_0 > a_1] \text{ and } [b_1 > b_0]$: The candidate $(\tilde{a}_0,\tilde{a}_1,\tilde{b}_0,\tilde{b}_1) = (b_0,b_1,a_0,a_1)$ has the same score in \eqref{eq:CPreConvHull} and satisfies $\tilde{a}_0 \leq \tilde{a}_1$, $\tilde{b}_1 \leq \tilde{b}_0$. Thus it also has the same score in \eqref{eq:CPreInfConv}.
		\item $[a_0 > a_1] \text{ and } [b_1 \leq b_0$]: The candidate $(\tilde{a}_0,\tilde{a}_1,\tilde{b}_0,\tilde{b}_1) = (0,0,a_0+b_0,a_1+b_1)$ has at least an equally good score in \eqref{eq:CPreConvHull} (using the subadditivity of $\SimLocC$). Due to $\tilde{a}_0 \leq \tilde{a}_1$, $\tilde{b}_1 \leq \tilde{b}_0$ it also has the same score in \eqref{eq:CPreInfConv}.
		The case $[a_0 \leq a_1] \text{ and } [b_1 > b_0]$ works analogously.
	\end{itemize}
	Thus \eqref{eq:CPreConvHull} $=$ \eqref{eq:CPreInfConv}.	
	Note that \eqref{eq:CPreInfConv} is the infimal convolution of $(m_0,m_1) \mapsto \SimLocC(m_0,m_1) +\Delta x \cdot m_0$ and $(m_0,m_1) \mapsto \SimLocC(m_0,m_1) + \Delta x \cdot m_1$.
	
	Now abbreviate $\BStat(a,b)=\BStat + \{(a,b)\}$ for $(a,b) \in \R^2$. Then one finds (cf.~\thref{prop:StaticEquivalence}) $\iota_{\BStat(a,b)}^\ast(m_0,m_1) = \SimLocC(m_0,m_1) + m_0 \cdot a + m_1 \cdot b$
	so that $\cPre^{\ast \ast}(x_0,\cdot,x_1,\cdot) = \iota_{\BStat(\Delta x,0)}^\ast \,\square\,\iota_{\BStat(0,\Delta x)}^\ast$, where $\square$ denotes the infimal convolution. Therefore,
	\begin{align*}
		\SCc^{\ast}(x_0,\cdot,x_1,\cdot) =
			\cPre^{\ast \ast \ast}(x_0,\cdot,x_1,\cdot) =
			\left(\iota_{\BStat(\Delta x,0)}^\ast \,\square\,\iota_{\BStat(0,\Delta x)}^\ast \right)^\ast 
			= \iota_{\BStat(\Delta x,0)} + \iota_{\BStat(0,\Delta x)} = \iota_{\SCB(x_0,x_1)}\,.
	\end{align*}

	For fixed $(x_0,x_1)$, convexity, lower semi-continuity, and 1-homogeneity of $\SCc$ are properties of the Fenchel--Legendre conjugate, while $\SCc(x,m_0,y,m_1) = \infty$ for $\min\{m_0,m_1\} < 0$ follows from $(-\infty,0]^2 \subset \SCB(x,y)$.
	For bounded, non-negative sequences $\{(m_{0,k},m_{1,k})\}_{k}$, the family of functions $(x_0,x_1) \mapsto \SCc(x_0,\allowbreak m_{0,k},x_1,m_{1,k})$ is uniformly Lipschitz continuous. Therefore, the function $\SCc$ is jointly lower semi-continuous in all four arguments.
\end{proof}

\begin{remark}[Dirac cost interpretation]
\thref{thm:semiCoupling} suggests the interpretation of $\SCc(x_0,m_0,\allowbreak x_1,m_1)$ as the dynamic cost for unbalanced transport from a Dirac mass $m_0$ at $x_0$ to a Dirac mass $m_1$ at $x_1$.
The explicit characterization of $\SCc=\cPre^{\ast \ast}$ via \eqref{eq:CPreInfConv} now reveals the same transportation structure as obtained in \thref{thm:transportChar} and Figure~\ref{fig:dynamicSteps}:
first some mass $a_0$ is transported from $x_0$ to $x_1$, then the remaining mass $b_0$ at $x_0$ is changed to $b_1$ while the mass $a_0$ at $x_1$ changes to $a_1$, and finally the remaining $b_1$ is transported from $x_0$ to $x_1$.
\thref{exm:twoDiracs} will give a more detailed analysis of transport between two Dirac masses.
\end{remark}

\subsection{Relation between different formulations of unbalanced transport}
\label{sec:FormulationRelation}
We now turn to a more detailed discussion of Figure~\ref{fig:Diagram} and the relation between the various dynamic and static formulations of transport problems.
In the following we write (A) $\rightarrow$ (B) for
\begin{align*}
\text{`A problem of class (A) induces a corresponding problem of class (B)'.}
\end{align*}
Throughout this article we consider six different formulations, three pairs of primal and dual problems.
We use the shorthands (Dyn) for dynamic formulations, (SC) for the semi-coupling formulations, and (W1T) for the $W_1$-type formulations; suffixes (P) and (D) indicate the primal or dual form.

In \cite{ChizatDynamicStatic2015} a family of dynamic problems is defined, and corresponding dual formulations are given, (Dyn-P) $\leftrightarrow$ (Dyn-D). The Wasserstein-1-type dynamic problems defined in Section \ref{sec:ReminderDynamic} are a subset of this family.
Moreover, \cite{ChizatDynamicStatic2015} introduces the primal and dual semi-coupling formulations and establishes their equivalence (SC-P) $\leftrightarrow$ (SC-D) (cf.~Section~\ref{sec:SemiCouplingFormulations}).
One of the main results in \cite{ChizatDynamicStatic2015} then is the equivalence (Dyn) $\rightarrow$ (SC), which was proven using the primal formulations, involving the integration of the flow of the (smoothed) momentum field $\omega$.

In \cite{SchmitzerWirthUnbalancedW1-2017} families of primal and dual static Wasserstein-1-type problems are proposed, and their equivalence is shown, (W1T-P) $\leftrightarrow$ (W1T-D). The static Wasserstein-1-type problems defined in Section~\ref{sec:ReminderStatic} are a subset of these families.

The primal and dual Kantorovich formulation of standard optimal transport with fixed marginals can be interpreted as special cases of (SC-P) and (SC-D).
Similarly, the celebrated Benamou--Brenier formulation of the Wasserstein-2 distance \cite{BenamouBrenier2000} is a special instance of (Dyn).
For this setting (Dyn) $\rightarrow$ (SC) is shown in \cite{BrenierExtendedMoKa2003} by establishing the cyclic inequality (SC-P) $\geq$ (Dyn-P) $=$ (Dyn-D) $\geq$ (SC-D) $=$ (SC-P).
For the crucial step (Dyn-D) $\geq$ (SC-D) the Hopf--Lax solution for the Hamilton--Jacobi equation corresponding to (Dyn-D) is used, hence the link in Figure~\ref{fig:Diagram} is marked as (Dyn-D) $\rightarrow$ (SC-D).

Analogously, for the Wasserstein--Fisher--Rao/Hellinger--Kantorovich distance one can give formulations in terms of (SC) and (Dyn).
For this setting the link (Dyn-D) $\rightarrow$ (SC-D) is established in \cite{LieroMielkeSavare-HellingerKantorovich-2015a} by a generalized Hopf--Lax formula.
Note that the primal formulation paired with (SC-D) in \cite{LieroMielkeSavare-HellingerKantorovich-2015a} is a soft-marginal formulation and differs from the semi-coupling formulation of \cite{ChizatDynamicStatic2015}, but is equivalent \cite[Corollary 5.9]{ChizatDynamicStatic2015}.
In \cite{LieroMielkeSavare-HellingerKantorovich-2015a} the formulation of (Dyn-D) is given by Theorem\,8.13 and (SC-D) by Theorem\,6.3, Equation\,(6.14).
Different forms of the Hopf--Lax solution are given in Equation\,(8.41) and Theorem\,8.12.

Finally, the main result of Section \ref{sec:DynStatEquivalence} is the relation (Dyn-D) $\rightarrow$ (W1T-D), \thref{prop:W1DynamicStaticEquivalence}.
The equivalence proof involved constructing a feasible candidate for the dynamic dual problem from feasible candidates of the static dual problem.
In the next Section we will study when it is possible to derive a dynamic formulation from a static problem, (W1T-D) $\rightarrow$ (Dyn-D).

\subsection{Which static models have an equivalent dynamic formulation?}\label{sec:staticAsDynamic}
We already know that any dynamic model of type \eqref{eq:DynamicProblem} can equivalently be expressed as a static model of type \eqref{eq:StaticPrimalProblem} (\thref{prop:W1DynamicStaticEquivalence}).
In this section we answer the question under what conditions the reverse statement is true, that is, which static models have an equivalent dynamic formulation.
We will see that the class of static models is richer than the dynamic class: there are static models that cannot be expressed in a dynamic form (see \thref{rem:NoDynamicCounterexample} later).
At first glance this is surprising, since unlike the static formulation the dynamic one models the full time evolution of the measures and thus might be expected to have more design freedom.
In detail, the main result of this section is the following.

\begin{proposition}[Characterization of static models admitting an equivalent dynamic one]\thlabel{thm:staticAsDynamic}
Consider the functions $\HStat$ and $\HStatInv=-\HStat^{-1}(-\cdot)$ from \thref{prop:StaticEquivalence}, which by \thref{prop:StaticEquivalence} uniquely specify a static model.
Abbreviate
\begin{align*}
\underline d&=-\min\{z\in\R\ |\ \HStatInv(z)=\sup\HStatInv\}\in[-\infty,0)\,,&
\overline d&=\min\{z\in\R\ |\ \HStat(z)=\sup\HStat\}\in(0,\infty]\,,\\
\zKinkN&=-\max\{z\in\R\ |\ \HStatInv(z)=z\}\in[-\infty,0)\,,&
\zKink&=\max\{z\in\R\ |\ \HStat(z)=z\}\in(0,\infty]\,\\
\underline m&=\lim_{\veps\searrow0}(\HStatInv(\zKinkN+\veps)-\zKinkN)/\veps\in[0,1]\,,&
\overline m&=\lim_{\veps\searrow0}(\HStat(\zKink+\veps)-\zKink)/\veps\in[0,1]\,.
\end{align*}
Denoting the left-sided derivative of a function $g$ by $g'$ and the $j$-fold composition of $g$ with itself by $g^j$, the function
\begin{equation*}
q[\HStat](z)=\begin{cases}
-\infty&\text{ if }z<\underline d\,,\\
\underline c\lim_{j\to\infty}\frac{\HStatInv^j(-z)-\HStatInv^{j-1}(-z)}{(\HStatInv^j)'(-z)}&\text{ if }z\in(\underline d,0)\,,\\
\overline c\lim_{j\to\infty}\frac{\HStat^j(z)-\HStat^{j-1}(z)}{(\HStat^j)'(z)}&\text{ if }z\in[0,\overline d)\,,\\
-\infty&\text{ if }z>\overline d\,,
\end{cases}
\qquad
\begin{array}{c}
\underline c=\begin{cases}
1&\text{if }\underline m=1\,,\\
\frac{\log\underline m}{1-1/\underline m}&\text{else,}
\end{cases}\\
\\
\overline c=\begin{cases}
1&\text{if }\overline m=1\,,\\
\frac{\log\overline m}{1-1/\overline m}&\text{else,}
\end{cases}
\end{array}
\end{equation*}
is well-defined (at $\underline d$ and $\overline d$ we simply extend $q[\HStat]$ upper semi-continuously), and the following holds.
\begin{enumerate}
\item \emph{Existence.}\label{enm:existence}
The static model \eqref{eq:StaticPrimalProblem} has an equivalent dynamic formulation \eqref{eq:DynamicProblem} (that is, $W_S(\rho_0,\rho_1)=W_D(\rho_0,\rho_1)$ for all measures $\rho_0,\rho_1$)
if and only if $q[\HStat]$ is concave.
\item \emph{Necessary condition.}\label{enm:necessary}
A necessary condition for the above is that $\HStat$ is locally Lipschitz differentiable on $(\zKink,\overline d)$ and $\HStatInv$ is so on $(\underline d,\zKinkN)$.
\item \emph{Sufficient condition.}\label{enm:sufficient}
A sufficient condition for the above is that $\frac1{\HStat'}$ and $\frac1{\HStatInv'}$ are convex on $[0,\infty)$.
\item \emph{Uniqueness and reconstruction formula.}\label{enm:uniqueness}
If the static model \eqref{eq:StaticPrimalProblem} has an equivalent dynamic formulation \eqref{eq:DynamicProblem}, then necessarily $\BDynMH=q[\HStat]$.
(Recall that any dynamic model is uniquely specified by the function $\BDynMH$ from \thref{thm:DynConjugateSet}.
In that case, $\underline d=-\CDynM(0,-1)$, $\overline d=\CDynM(0,1)$.)
\end{enumerate}
\end{proposition}

\begin{figure}
\centering
\setlength\unitlength\linewidth
\begin{picture}(.7,.28)
\put(0,0){\includegraphics[width=.7\unitlength]{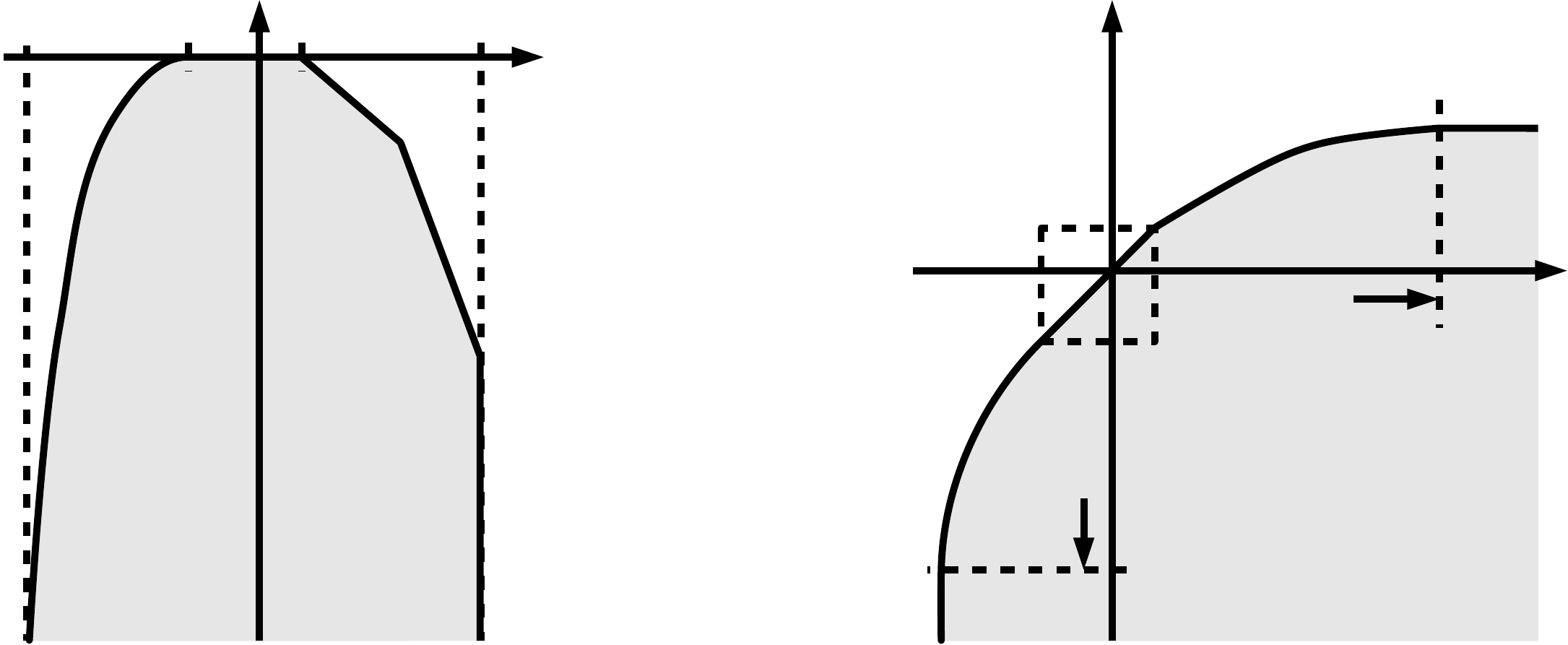}}
\put(.69,.15){\small$z$}
\put(.67,.24){\small$\HStat(z)$}
\put(.23,.245){\small$z$}
\put(.135,.17){\small$\BDynMH(z)$}
\put(-.07,.28){\small$\underline d\!=\!-\!\CDynM\!(0,\!-\!1)$}
\put(.19,.28){\small$\overline d\!=\!\CDynM\!(0,\!1)$}
\put(.51,.03){\small$\underline d\!=\!-\!\CDynM\!(0,\!-\!1)$}
\put(.6,.13){\small$\overline d\!=\!\CDynM(0,1)$}
\put(.08,.24){\small$\zKinkN$}
\put(.13,.24){\small$\zKink$}
\put(.43,.19){\small$[\zKinkN,\zKink]^2$}
\put(.07,.1){\small\color{gray}$\BDynM$}
\put(.55,.09){\small\color{gray}$\BStatFlip$}
\end{picture}
\caption{Sketch of the functions $\BDynMH$ and $\HStat$ from \thref{thm:DynConjugateSet} and \thref{prop:StaticEquivalence}. $\BStatFlip$ denotes the set $\BStat$ flipped along the vertical axis.}
\label{fig:involvedFunctions}
\end{figure}

For a better intuition, in Figure~\ref{fig:involvedFunctions} we illustrate two corresponding functions $\BDynMH$ and $\HStat$, specifying a dynamic model and its equivalent static formulation, respectively.
As an example, \thref{thm:staticAsDynamic}\eqref{enm:necessary} implies that the static model with the piecewise linear $\HStat(z)=\min(z,1+\frac z2,2+\frac z3)$ does not admit a dynamic formulation,
while by \thref{thm:staticAsDynamic}\eqref{enm:sufficient} the static model with $\HStat(z)=\frac z{1+z}$ does.
Note that \thref{thm:staticAsDynamic}\eqref{enm:sufficient} is only sufficient: for instance, it is readily checked that the function $\HStat$ belonging to a piecewise linear $\BDynMH$ typically has multiple linear segments of different slope so that $\frac1{\HStat'}$ or $\frac1{\HStatInv'}$ cannot be convex.

Before proving the proposition, let us provide some motivation for the definition of $q[\HStat]$, which will be made rigorous later.
First observe that $\HStat$ as induced by $\BDynMH$ is given by \thref{cor:HStatFromDyn} where $\Flow_t$ is the solution to \eqref{eqn:FlowIVP} and by \thref{lem:Flow}(\ref{item:FlowTimeConvex}) one has $\sup \dom F_1 = \sup \dom \BDynMH$.
Thus we can read off
\begin{equation}\label{eqn:staticFunctionFlow}
\overline d=\sup\dom\BDynMH=\CDynM(0,1)
\qquad\text{as well as}\quad
\HStat^j(z)=\Flow_j(z)
\text{ for all }z\in\dom\BDynMH\,.
\end{equation}%
If the static model with $\HStat$ indeed comes from a dynamic model with (some as yet unknown) $\BDynMH$,
then for small $z$ we can approximate $\BDynMH(z)$ as follows.
For $\veps \in \R$ of small magnitude, with \eqref{eq:InducedSimLocC} we estimate
\begin{align*}
	\SimLocC(1,1+\veps) = \inf_{(\rho,\zeta) \in \mc{CES}(1,1+\veps)} \int_{0}^1
		\CDynM(\rho,\zeta)\,\d t \approx \CDynM(1,\veps)\,.
\end{align*}
Now by \thref{prop:DynamicDual,prop:StaticEquivalence,thm:DynConjugateSet} we have
\begin{align}
\CDynM(1,\veps) & = \sup\{\alpha + \veps \cdot \beta \ |\ \alpha,\beta \in \R,\, \alpha \leq \BDynMH(\beta)\}\,, \nonumber \\
\SimLocC(1,1+\veps) & = \sup\{\alpha + (1+\veps) \cdot \beta\ |\ \alpha,\beta \in \R,\, \alpha \leq \HStatInv(-\beta)\} \label{eq:CStatDynConjugateFormPre} \\
&= \sup\{\tilde\alpha + \veps \cdot \tilde\beta\ |\ \tilde{\alpha},\tilde{\beta} \in \R,\, \tilde{\alpha} \leq \HStatInv(-\tilde{\beta})+\tilde{\beta}\}\,, \nonumber
\end{align}
where in the last line we substituted $(\tilde{\alpha},\tilde{\beta})=(\alpha+\beta,\beta)$.
Since for small $\veps$ we have $\SimLocC(1,1+\veps)\approx\CDynM(1,\veps)$, for small $\beta$ we expect
\begin{equation}
\label{eq:HDynStatLimit}
\BDynMH(\beta) \approx \HStatInv(-\beta)+\beta = -\HStat^{-1}(\beta)+\beta\,.
\end{equation}
For large $z$ on the other hand (analogously for large negative $z$) we can now obtain an estimate of $\BDynMH(z)$ by exploiting that $\HStat$ is the flow of $\BDynMH$.%
The flow can be differentiated, yielding $(\HStat^j)'(z)=\Flow_j'(z)=\frac{\BDynMH(\HStat^j(z))}{\BDynMH(z)}$, so that finally we arrive at
\begin{equation*}
\BDynMH(z)=\left.\begin{cases}
-\infty&\text{if }z>\overline d\,,\\
\frac{\BDynMH(\HStat^j(z))}{(\HStat^j)'(z)}&\text{if }z\in[0,\overline d)
\end{cases}\right\}
\approx\begin{cases}
-\infty&\text{if }z>\overline d\,,\\
\frac{\HStat^j(z)-\HStat^{j-1}(z)}{(\HStat^j)'(z)}&\text{if }z\in[0,\overline d)\,,
\end{cases}
\end{equation*}
where we have used our approximation of $\BDynMH(z)$ for small $z$ and that $\HStat^j(z)$ is small for $j$ large enough.
This is exactly the expression used in the definition of $q[\HStat]$.

\begin{remark}[Separate treatment on positive and negative real line]\thlabel{rem:negReals}
Analogously to \eqref{eqn:staticFunctionFlow} one obtains
\begin{equation*}
\underline d=-\CDynM(0,-1)=\inf\dom\BDynMH
\qquad\text{as well as}\quad
\HStatInv^j(z)=\FlowInv_j(z)
\text{ for all }z\in\dom\BDynMH\,.
\end{equation*}
For statements about $\BDynMH(z)$ with positive $z$ we will therefore always work with $\HStat$ and $\Flow_j$,
while for negative $z$ we shall use $\HStatInv$ as well as $\FlowInv_j$.
\end{remark}

We shall now begin proving \thref{thm:staticAsDynamic} in multiple steps, not necessarily in the order of its different statements.

\begin{proof}[Proof of \thref{thm:staticAsDynamic}\eqref{enm:necessary}]
Assume that the static model with $\HStat$ has an equivalent dynamic formulation with $\BDynMH$.
By \eqref{eqn:staticFunctionFlow} we have $\HStat(z)=\Flow_1(z)$ which is locally Lipschitz differentiable on $(\zKink,\overline d)$ by \thref{lem:Flow}(\ref{item:FlowDerivative}).
The statement for $\HStatInv$ follows analogously with \thref{rem:negReals}.
\end{proof}
The following lemma rigorously establishes the heuristic relation \eqref{eq:HDynStatLimit} `for small arguments', as required for the proof of \thref{thm:staticAsDynamic}.
\begin{lemma}[Approximation of $\BDynMH$ for small arguments]\thlabel{thm:BDynMHSmall}
Let the static model with $\HStat$ have an equivalent dynamic formulation with $\BDynMH$. Then
\begin{equation*}
\lim_{z\nearrow\zKinkN}\frac{\BDynMH(z)}{-\HStatInv^{-1}(-z)-z}=
\lim_{z\nearrow\zKinkN}\frac{\BDynMH(z)}{\HStat(z)-z}=\underline c\,,\qquad
\lim_{z\searrow\zKink}\frac{\BDynMH(z)}{z-\HStat^{-1}(z)}=\overline c\,.
\end{equation*}
\end{lemma}
\begin{proof}
We only prove the second limit, the first one following analogously via \thref{rem:negReals}.
First consider the case that $\HStat$ is differentiable in $\zKink$, that is, $\overline c=\overline m=1$ and $\HStat'(z)\to\overline m=1$ as $z\to\zKink$.
Using \eqref{eq:InducedSimLocC}, for all $\veps>0$ we have
\begin{multline*}
\SimLocC(1,1+\veps)
=\min \left\{ \int_0^1 \CDynM(\RadNik{\rho}{\mu},\RadNik{\xi}{\mu})\,\d\mu\,\middle|\,(\rho,\xi) \in \mc{CES}(1,1+\veps) \right\} \\
\leq\int_0^1\CDynM(1+\veps t,\veps)\,\d t
\leq\int_0^1\CDynM(1,\veps)\,\d t
=\CDynM(1,\veps)\,,
\end{multline*}
where we exploited $1+t\,\veps \geq 1$ and the convexity of $\CDynM$ via
\begin{equation*}
\CDynM(\rho,\veps)
=\rho\CDynM(1,\tfrac\veps\rho)
\leq\rho\left[\tfrac{\veps/\rho}\veps\CDynM(1,\veps)+\tfrac{\veps-\veps/\rho}\veps\CDynM(1,0)\right]
=\CDynM(1,\veps)
\end{equation*}
for $\rho \geq 1$.
Note also that (see \eqref{eq:CStatDynConjugateFormPre})
\begin{align*}
\CDynM(1,\veps)
&=\sup\left\{\alpha+\veps\beta\ \middle|\ \alpha\leq\BDynMH(\beta)\right\}
=(-\BDynMH)^\ast(\veps)\,,\\
\SimLocC(1,1+\veps)
&=\sup\left\{\alpha+(1+\veps)\beta\ \middle|\ \alpha\leq-\HStat^{-1}(\beta)\right\}
=(\HStat^{-1}-\id)^\ast(\veps)\,.
\end{align*}
Therefore, $\SimLocC(1,1+\veps)\leq\CDynM(1,\veps)$ for $\veps>0$ implies $\BDynMH(z)\geq z-\HStat^{-1}(z)$ for $z\geq0$ and thus
\begin{equation*}
\limsup_{z\searrow\zKink}\frac{\BDynMH(z)}{z-\HStat^{-1}(z)}\leq1\,.
\end{equation*}
Furthermore, by \eqref{eqn:staticFunctionFlow}, for $z>0$ sufficiently close to $\zKink$ we have
\begin{multline*}
\HStat(z)
=\Flow_1(z)
=\sup\left\{\phi(1)\ \middle|\ \partial_t \phi(t)\leq\BDynMH(\phi(t))\text{ for all }t\in[0,1],\,\phi(0)=z\right\}\\
\geq\sup\left\{\phi(1)\ \middle|\ \phi(t)=\phi(0)\exp\left(\tfrac{\BDynMH(\phi(0))}{\phi(0)}t\right),\,\phi(0)=z\right\}
=z\exp\tfrac{\BDynMH(z)}z\,,
\end{multline*}
where we used the concavity of $\BDynMH$.
This implies $\BDynMH(z)\leq z\log\frac{\HStat(z)}z$.
Since $\HStat$ is differentiable in $\zKink$ and thus $\HStat'(\HStat^{-1}(z))\to1$ as $z\to\zKink$,
using $\HStat'(\HStat^{-1}(z)) \cdot (\HStat^{-1}(z)-z)\leq z-\HStat(z)$ due to the concavity of $\HStat$ as well as $\lim_{z\to\zKink}\HStat(z)=\zKink$ we thus obtain
\begin{multline*}
\liminf_{z\searrow\zKink}\frac{\BDynMH(z)}{z-\HStat^{-1}(z)}
\geq\liminf_{z\searrow\zKink}\frac{z\log\frac{\HStat(z)}z}{z-\HStat^{-1}(z)}\\
\geq\liminf_{z\searrow\zKink}\frac{z\log\frac{\HStat(z)}z}{\HStat(z)-z}\HStat'(\HStat^{-1}(z))
=\liminf_{z\searrow\zKink}\frac{-\log\frac{\HStat(z)}z}{1-\frac{\HStat(z)}z}\HStat'(\HStat^{-1}(z))
=1\,.
\end{multline*}

The case $\overline m<1$ is treated more explicitly.
Let $\tilde m=\lim_{\veps\searrow0}\BDynMH(\zKink+\veps)/\veps$ denote the right derivative of $\BDynMH$ in $\zKink$ so that
\begin{equation*}
\BDynMH(z)=\tilde m(z-\zKink)-o(z-\zKink)\qquad \text{for }z\geq\zKink\,,
\end{equation*}
where $o(z-\zKink)$ denotes a nonnegative term decreasing to zero faster than $z-\zKink$.
The flow of $\BDynMH$ can readily be approximated:
$\BDynMH(z)\leq\tilde m(z-\zKink)$ implies $\Flow_1(z)\leq ze^{\tilde m}$, and the same calculation as above yields $\Flow_1(z)\geq z\exp\tfrac{\BDynMH(z)}z=ze^{\tilde m-o(1)}$,
where $o(1)$ is a nonnegative term decreasing to zero as $z\to\zKink$.
Now \eqref{eqn:staticFunctionFlow} implies that the right derivative of $\HStat$ in $\zKink$ is given by $\overline m=\Flow_1'(\zKink)=e^{\tilde m}$.
Furthermore, by the inverse function theorem, the right derivative of $\HStat^{-1}$ in $\zKink$ is given by $\frac1{\overline m}$.
Thus,
\begin{equation*}
\lim_{z\searrow\zKink}\frac{\BDynMH(z)}{z-\HStat^{-1}(z)}
=\lim_{z\searrow\zKink}\frac{\tilde m(z-\zKink)-o(z-\zKink)}{z-\zKink-\frac1{\overline m}(z-\zKink)-o(z-\zKink)}
=\lim_{z\searrow\zKink}\frac{\log\overline m(z-\zKink)-o(z-\zKink)}{(1-\frac1{\overline m})(z-\zKink)-o(z-\zKink)}
=\overline c\,.\qedhere
\end{equation*}
\end{proof}

\begin{proof}[Proof of \thref{thm:staticAsDynamic}\eqref{enm:uniqueness}]
We first show that the limits in the definition of $q[\HStat]$ are well-defined and finite.
In fact, defining
\begin{equation*}
q_j(z)
=\frac{\HStat^j(z)-\HStat^{j-1}(z)}{(\HStat^j)'(z)}\,,
\end{equation*}
for all $z\in[0,\overline d)$ we can show the monotonicity property
\begin{equation}
\label{eqn:qSeriesMonotonicity}
-\infty<q_1(z)\leq q_{j-1}(z)\leq q_{j}(z)\leq0\quad\text{for all }j>1\,.
\end{equation}
Indeed, we have $q_1(z)>-\infty$ since $\HStat'(z)>0$ (the strict positivity follows by the concavity of $\HStat$ and the definition of $\overline d$).
Likewise, $q_j(z)\leq0$ follows from $\HStat^j\leq \HStat^{j-1}$ (which is true since $\HStat(z)\leq z$, see \thref{prop:StaticEquivalence}) and $(\HStat^j)'>0$.
Finally, the monotonicity of the sequence $q_j$ is obtained as follows:
The composition of a concave, monotonically increasing function with a concave function is again concave, thus $\HStat^j$ is concave and increasing for all $j$. Therefore,
\begin{multline*}
\HStat(\HStat^{j-2}(z))\leq \HStat(\HStat^{j-1}(z))+\HStat'(\HStat^{j-1})(\HStat^{j-2}(z)-\HStat^{j-1}(z))\\
\text{ so that }
q_j(z)=\frac{\HStat(\HStat^{j-1}(z))-\HStat(\HStat^{j-2}(z))}{\HStat'(\HStat^{j-1}(z))(\HStat^{j-1})'(z)}\geq\frac{\HStat^{j-1}(z)-\HStat^{j-2}(z)}{(\HStat^{j-1})'(z)}=q_{j-1}(z)\,.
\end{multline*}
This monotonicity property then implies the existence of the finite limit $\lim_{j\to\infty}q_j(z)$ (the case of negative $z$ is treated analogously).

Note further that $q_j(0)=q_j'(0)=0$ and $q_j$ is nonincreasing on $(0,\infty)$. Indeed, the latter follows from
\begin{equation*}
q_j'(z)
=\frac{(\HStat^j)'(z)-(\HStat^{j-1})'(z)}{(\HStat^j)'(z)}-q_j(z)\frac{(\HStat^j)''(z)}{(\HStat^j)'(z)}
=1-\frac1{\HStat'(\HStat^{j-1}(z))}-q_j(z)\frac{(\HStat^j)''(z)}{(\HStat^j)'(z)}\leq0\,,
\end{equation*}
for $z \in [0,\ol d)$, where $(\HStat^j)''$ is well-defined as a weak derivative due to \thref{thm:staticAsDynamic}\eqref{enm:necessary}.
Analogously one can show that $q_j$ is nondecreasing on $(-\infty,0)$ (\thref{rem:negReals}).
By the monotone convergence $\overline c \cdot q_j\nearrow q[\HStat]$ the limit $q[\HStat]$ shares the same properties.

Now we prove $\BDynMH(z)=q[\HStat](z)$ for positive $z$ (the case of negative $z$ follows analogously via \thref{rem:negReals}).
By \eqref{eqn:staticFunctionFlow}, $\BDynMH(z)=q[\HStat](z)=-\infty$ for $z>\overline d$.
On the other hand, for $z\in[0,\zKink)$ we have $\HStat(z)=z$ so that $q[\HStat](z)=0$ as well as $\BDynMH(z)=0$ by \eqref{eqn:staticFunctionFlow}.
Finally, for $z\in(\zKink,\overline d)$ \thref{lem:Flow}(\ref{item:FlowDerivative}) together with \eqref{eqn:staticFunctionFlow} implies
\begin{equation*}
(\HStat^j)'(z)
=\Flow_j'(z)
=\frac{\BDynMH(\Flow_j(z))}{\BDynMH(z)}
=\frac{\BDynMH(\HStat^j(z))}{\BDynMH(z)}
\end{equation*}
so that application of \thref{thm:BDynMHSmall} yields
\begin{equation*}
\frac{\BDynMH(z)}{q[\HStat](z)}
=\lim_{j\to\infty}
	\frac{\BDynMH(z)\,(\HStat^j)'(z)}{\overline c \cdot (\HStat^j(z)-\HStat^{j-1}(z))}
=\lim_{j\to\infty}
	\frac{\BDynMH(\HStat^j(z))}{\overline c \cdot (\HStat^j(z)-\HStat^{j-1}(z))}
=\lim_{\tilde z\searrow\zKink}
	\frac{\BDynMH(\tilde z)}{\overline c \cdot (\tilde z-\HStat^{-1}(\tilde z))}
=1\,.\qedhere
\end{equation*}
\end{proof}

\begin{proof}[Proof of \thref{thm:staticAsDynamic}\eqref{enm:sufficient}]
We prove by induction that $q_j$ is concave on $[0,\infty)$ for all $j$, which automatically implies the desired concavity of the limit $q[\HStat]$ on $[0,\infty)$ (the proof for $(-\infty,0]$ is analogous by \thref{rem:negReals}).
First note that $q_0(z)=z-\HStat^{-1}(z)$ is indeed concave.
Furthermore, due to $q_{j+1}=\frac{q_j\circ\HStat}{\HStat'}$ we have
\begin{equation*}
q_{j+1}'(z)=q_j'(\HStat(z))-q_j(\HStat(z))\frac{\HStat''(z)}{(\HStat'(z))^2}\,.
\end{equation*}
We need to show that $q_{j+1}'$ is non-increasing.
Now $q_j'$ and thus also $q_j'\circ\HStat$ is non-increasing on $[0,\infty)$.
Likewise, $|q_j|$ and $|q_j\circ\HStat|$ are non-decreasing,
so it remains to show that $|\HStat''(z)|/(\HStat'(z))^2=-\HStat''(z)/(\HStat'(z))^2=(1/\HStat'(z))'$ is non-decreasing.
However, this is equivalent to the convexity of $\frac1{\HStat'}$.
\end{proof}

\begin{proof}[Proof of \thref{thm:staticAsDynamic}\eqref{enm:existence}]
If the static model has an equivalent dynamic formulation, then by \thref{thm:staticAsDynamic}\eqref{enm:uniqueness} we must have $\BDynMH=q[\HStat]$.
Since $\BDynMH$ is concave by \thref{thm:DynConjugateSet}, $q[\HStat]$ is so as well.

It remains to show that concavity of $q[\HStat]$ implies the existence of an equivalent dynamic formulation.
We have already shown $q[\HStat]\leq0$ with $q[\HStat](0)=q[\HStat]'(0)=0$.
Together with the concavity this means that $q[\HStat]$ specifies an admissible dynamic model of type \eqref{eq:DynamicProblem} via
\begin{equation*}
\CDynM(m_0,m_1)
=\sup\{m_0\alpha+m_1\beta\ |\ \alpha\leq q[\HStat](\beta)\}\,.
\end{equation*}
Denote by $\tilde\HStat$ the function specifying the corresponding static problem.
We need to show $\HStat=\tilde\HStat$ (which we do on the positive real line; the negative case follows analogously with \thref{rem:negReals}).
However, we have $\HStat(z)=z=\tilde\HStat(z)$ for all $z\in[0,\zKink]$ and both $\HStat$ and $\tilde\HStat$ constant on $[\overline d,\infty)$.
Furthermore, for $z\in(\zKink,\overline d)$, by the monotone convergence theorem (recall \eqref{eqn:qSeriesMonotonicity}) and a change of variables,
\begin{equation*}
\int_z^{\HStat(z)}\frac1{q[\HStat](x)}\,\d x
=\frac1{\overline c}\lim_{j\to\infty}\int_z^{\HStat(z)}\frac1{q_j(x)}\,\d x
=\frac1{\overline c}\lim_{j\to\infty}\int_{\HStat^j(z)}^{\HStat^{j+1}(z)}\frac1{y-\HStat^{-1}(y)}\,\d y\,.
\end{equation*}
Now first assume that $\HStat$ is differentiable in $\zKink$ and thus $\overline m=\overline c=1$ as well as $\HStat'(\HStat^{j-1}(z))\to1$ as $j\to\infty$.
Note that for a negative increasing function $f$ and $a \leq b$ one has $(a-b)\,f(b) \leq \int_b^a f(y) \d y \leq (a-b)\,f(a)$.
Choosing $a=\HStat^{j+1}(z)$, $b=\HStat^{j}(z)$, and $f(y)=\frac1{y-\HStat^{-1}(y)}$ (which is indeed negative and increasing on $(\zKink,\overline d)$) and using the concavity $\HStat(\HStat^{j-1}(z))\geq\HStat(\HStat^j(z))+\HStat'(\HStat^{j-1}(z)) \cdot (\HStat^{j-1}(z)-\HStat^j(z))$, we have
\begin{equation*}
\HStat'(\HStat^{j-1}(z))
\leq\frac{\HStat^{j}(z)-\HStat^{j+1}(z)}{\HStat^{j-1}(z)-\HStat^j(z)}
\leq\int_{\HStat^j(z)}^{\HStat^{j+1}(z)}\frac1{y-\HStat^{-1}(y)}\,\d y
\leq\frac{\HStat^{j}(z)-\HStat^{j+1}(z)}{\HStat^{j}(z)-\HStat^{j+1}(z)}
=1
\end{equation*}
so that $\int_z^{\HStat(z)}\frac1{q[\HStat](x)}\,\d x=1$.
In the case that $\overline m<1$, by the inverse function theorem the right derivative of $\HStat^{-1}$ in $\zKink$ is $\frac1{\overline m}$. Further, for $z\in(\zKink,\overline d)$ we have $\HStat^j(z) \searrow \zKink$ as $j \to \infty$ so that (using little o-notation)
\begin{multline*}
\frac1{\overline c}\lim_{j\to\infty}\int_{\HStat^j(z)}^{\HStat^{j+1}(z)}\frac1{y-\HStat^{-1}(y)}\,\d y\\
=\frac1{\overline c}\lim_{z\searrow\zKink}\int_{z}^{\HStat(z)}\frac1{y-\HStat^{-1}(y)}\,\d y
=\frac1{\overline c}\lim_{z\searrow\zKink}\int_{z}^{\HStat(z)}\frac1{y-\zKink-\frac1{\overline m}(y-\zKink)+o(z-\zKink)}\,\d y\\
=\frac1{\overline c}\lim_{z\searrow\zKink}\int_{z}^{\HStat(z)}\frac1{(1-\frac1{\overline m})(y-\zKink)}\,\d y
=\frac1{\overline c}\lim_{z\searrow\zKink}\frac1{1-\frac1{\overline m}}\log\frac{\HStat(z)-\zKink}{z-\zKink}
=\frac1{\overline c}\frac1{1-\frac1{\overline m}}\log\overline m
=1
\end{multline*}
and thus again $\int_z^{\HStat(z)}\frac1{q[\HStat](x)}\,\d x=1$.
Now this implies that $\HStat$ is the flow of $q[\HStat]$ at $t=1$ (cf.~proof of \thref{lem:Flow}(\ref{item:FlowIVP})).
However, $\tilde\HStat$ is so as well by construction so that $\tilde\HStat$ and $\HStat$ coincide.
\end{proof}
% !TEX root = W1TypeDynamic.tex

\section{Examples}

This section will provide some examples of possible static and dynamic models, some of which are well-known.
We shall also calculate in detail the optimal solution to some simple unbalanced mass transport problems,
which will give some additional insight into the solution structure beyond Section~\ref{sec:characterization}.

\subsection{Unbalanced transport between two Dirac masses}\label{sec:twoDiracs}
Here we consider the simple problem of unbalanced transport between two measures $\rho_0$ and $\rho_1$, where both consist of just two Dirac masses at the same positions.
Since the unbalanced transport problem is invariant under translation and rotation of the coordinate system
and since according to \thref{thm:transportChar} transport only happens between the measure supports $\spt\rho_0$ and $\spt\rho_1$,
we may without loss of generality assume the Dirac masses to lie on the one-dimensional real line, one at the origin and the other at some distance $L>0$
(in fact, our example covers the behaviour of unbalanced transport between two Dirac masses at distance $L$ in any metric space).
Due to the one-homogeneity of the unbalanced transport cost $W_S(\rho_0,\rho_1)$ we might even restrict one of the measures $\rho_0$ and $\rho_1$ to be a probability measure,
however, this does not simplify the exposition substantially.
We shall consider the static formulation \eqref{eq:StaticPrimalProblem}, since by \thref{prop:DynamicPrimalOptimizers} the dynamic model behaviour can be directly inferred from the static behaviour.

\begin{example}[Unbalanced transport between two Dirac masses]\thlabel{exm:twoDiracs}
For simplicity we shall assume $\SimLocC$ to be differentiable
(the exact same calculation can also be performed in the non-differentiable case if one simply replaces derivatives by subderivatives;
however, the rigorous justification of the calculation then involves a few more technical arguments that we avoid here for an easier exposition).
We set
\begin{equation*}
\rho_0=m_0^0\delta_0+m_0^L\delta_L\,,\qquad
\rho_1=m_1^0\delta_0+m_1^L\delta_L\qquad
\text{with }m_1^0<m_0^0\,,\; m_1^L>m_0^L\,,
\end{equation*}
that is, mass is removed at the origin and added at $L$.
We seek the optimal transport couplings $\pi_0$ and $\pi_1$ in \eqref{eq:StaticPrimalProblem}.
To this end, we may parameterize them as
\begin{gather}
\pi_0=a\delta_{(0,L)}+(m_0^0-a)\delta_{(0,0)}+m_0^L\delta_{(L,L)}\,,\qquad
\pi_1=b\delta_{(0,L)}+m_1^0\delta_{(0,0)}+(m_1^L-b)\delta_{(L,L)}\label{eqn:couplingsTwoDiracs}\\
\text{with }(a,b)\in S=[0,m_0^0]\times[0,m_1^L]\,,\nonumber
\end{gather}
that is, $\pi_0$ transports mass $a$ and $\pi_1$ transports mass $b$ from the origin to $L$.
That $\pi_0$ and $\pi_1$ must have the above form follows from \thref{thm:transportChar},
which implies $\spt\pi_0,\spt\pi_1\subset\spt(\rho_0+\rho_1)^2=\{0,L\}^2$ and $\pi_0(\{(L,0)\})=\pi_1(\{(L,0)\})=0$.
Indeed, suppose $\pi_0(\{(L,0)\})>0$ (and thus $\pi_0(\{(0,L)\})=0$ by optimality of $\pi_0$), then by \thref{thm:transportChar}
\begin{enumerate}
\item either $\{0,L\}=\Omega_=$ so that no mass change happens at all (thus both $\pi_0$ and $\pi_1$ must transport from $0$ to $L$, leading to a contradiction),
\item or $\{0\}\subset\Omega_+$ so that necessarily $\pi_1(\{(0,L)\})>0$ in order to achieve a net mass removal at the origin.
This again leads to a contradiction since by \thref{prop:GrowthShrinkDecomposition} $\pi_1$ may not transport any mass from $\Omega_+$.
\end{enumerate}
Analogously one obtains $\pi_1(\{(L,0)\})=0$.
The cost \eqref{eq:StaticPrimal} associated with $\pi_0$, $\pi_1$ can be calculated as
\begin{align}
P(a,b)
&=L(a+b)+\SimLocC(m_0^0-a,m_1^0+b)+\SimLocC(m_0^L+a,m_1^L-b)\label{eqn:costTwoDiracs}\\
&=L(a+b)+(m_0^0-a)\SimLocC(1,\alpha)+(m_0^L+a)\SimLocC(1,\beta)\nonumber
\end{align}
for the relative mass changes
\begin{equation}\label{eqn:relativeMassChanges}
\alpha=\frac{m_1^0+b}{m_0^0-a}\geq0
\quad\text{ and }\quad
\beta=\frac{m_1^L-b}{m_0^L+a}\geq0\,.
\end{equation}
Since $P$ is convex,
its optimum is well-defined.
To find the optimal $(a,b)\in S$ there are thus only two possible cases.
\begin{enumerate}
\item \emph{Case $0\neq\nabla P(a,b)$ for all $(a,b)\in S$.}
In that case the optimum satisfies $a=0$ or $b=0$.
Indeed, the optimum must lie in $\partial S$, and if $b=m_1^L$ then $a=0$ is optimal (analogously one shows that $a=m_0^0$ implies $b=0$):
using the notation of \thref{prop:GrowthShrinkDecomposition} we have $\rho_1'(\{L\})=0$ and thus necessarily $L\notin\Omega_+$.
If $L\in\Omega_-$, then by \thref{thm:massTptChng} the structure of $\pi_1$ implies $0\in\Omega_-$ as well and thus $a=0$ since by \thref{prop:GrowthShrinkDecomposition} $\pi_0$ is diagonal on $\Omega_-\times\Omega_-$.
If on the other hand $L\in\Omega_=$, then $\rho_0'(\{L\})=\rho_1'(\{L\})=0$ and thus again $a=0$.
\item \emph{Case $0=\nabla P(a,b)$ for some $(a,b)\in S$.}
In that case the optimum $(a,b)$ has to satisfy
\begin{align*}
0=\frac{\partial P}{\partial a}
&=L-\SimLocC(1,\alpha)+\alpha{\SimLocC}_{,2}(1,\alpha)+\SimLocC(1,\beta)-\beta{\SimLocC}_{,2}(1,\beta)\,,\\
0=\frac{\partial P}{\partial b}
&=L+{\SimLocC}_{,2}(1,\alpha)-{\SimLocC}_{,2}(1,\beta)\,,
\end{align*}
where subscript $,2$ denotes the derivative with respect to the second argument.
These equations can be transformed into
\begin{gather}
L={\SimLocC}_{,2}(1,\beta)-{\SimLocC}_{,2}(1,\alpha)
=T[\beta]'-T[\alpha]'\,,\label{eqn:DiracOptCond1}\\
\begin{aligned}
0=\frac{\partial P}{\partial a}-\frac{\partial P}{\partial b}
&=[\SimLocC(1,\beta)-(\beta-1){\SimLocC}_{,2}(1,\beta)]-[\SimLocC(1,\alpha)+(1-\alpha){\SimLocC}_{,2}(1,\alpha)]\\
&=T[\beta](1)-T[\alpha](1)\,,
\end{aligned}\label{eqn:DiracOptCond2}
\end{gather}
where $T[\gamma]$ denotes the tangent
to $\SimLocC(1,\cdot)$ at $\gamma$.
Note that \eqref{eqn:DiracOptCond1} implies $\beta>\alpha$ and \eqref{eqn:DiracOptCond2} thus implies $T[\alpha]'<0$ and $T[\beta]'>0$ or equivalently $$\alpha\leq1\quad\text{and}\quad\beta\geq1\,.$$
Thus the minimizer $(a,b)$ can be constructed geometrically as follows (compare Figure~\ref{fig:twoDiracs}):
For $s<0$ let
\begin{align*}
T_s^l&:[0,1]\to\R,&x&\mapsto\sup\{s+t(x-1)\ |\ t<0,\,s+t(y-1)\leq\SimLocC(1,y)\,\forall y\in[0,1]\}\,,\\
T_s^r&:[1,\infty)\to\R,&x&\mapsto\sup\{s+t(x-1)\ |\ t>0,\,s+t(y-1)\leq\SimLocC(1,y)\,\forall y\in[1,\infty)\}
\end{align*}
denote the left and the right tangent line to $\SimLocC$ running through the point $(1,s)$, and set
\begin{equation*}\label{eqn:DiracParameterization}
L(s)=(T_s^r)'-(T_s^l)'>0\,.
\end{equation*}
Obviously, $L(s)$ is continuous and strictly decreasing in $s$.
Note that $s$ parameterizes the solutions to \eqref{eqn:DiracOptCond2} and that by \eqref{eqn:DiracOptCond1} there must be one $s$ with
\begin{equation}\label{eqn:DiracOptCondParameterization}
L(s)=L\,.
\end{equation}
The optimal relative mass decrease $\alpha$ and mass increase $\beta$ can now be identified via the condition
\begin{equation}\label{eqn:DiracTangents}
T_s^l=T[\alpha]\,,\qquad T_s^r=T[\beta]
\end{equation}
(note that there may be a closed interval of $\alpha$s or $\beta$s satisfying this condition).
From these we can solve for $(a,b)$ as
\begin{equation}\label{eqn:DiracResolveRelMassChange}
\left(\begin{array}{c}a\\b\end{array}\right)
=\frac1{\beta-\alpha}\left(\begin{array}{cc}-1&-1\\\beta&\alpha\end{array}\right)
\left(\begin{array}{c}\alpha m_0^0-m_1^0\\\beta m_0^L-m_1^L\end{array}\right)\,.
\end{equation}
Obviously, in addition to the obvious constraints
\begin{equation}\label{eqn:twoDiracsConstraints1}
\alpha\geq m_1^0/m_0^0\text{ and }\beta\leq m_1^L/m_0^L\,,
\end{equation}
the condition $(a,b)\in S$ imposes four constraints on $(\alpha,\beta)$,
\begin{equation}\label{eqn:twoDiracsConstraints2}
\beta \leq \frac{\rho_1(\R)-\alpha m_0^0}{m_0^L}\,,\qquad
\beta \geq \frac{\alpha m_1^L}{\alpha\rho_0(\R)-m_1^0}\,,\qquad
\beta\geq\frac{\rho_1(\R)}{\rho_0(\R)}\,,\qquad
\alpha\leq\frac{\rho_1(\R)}{\rho_0(\R)}\,.
\end{equation}
\end{enumerate}
\end{example}

\begin{figure}
\centering
\setlength\unitlength{.3\linewidth}
\begin{picture}(1,.8)
\put(0,0){\includegraphics[width=\unitlength]{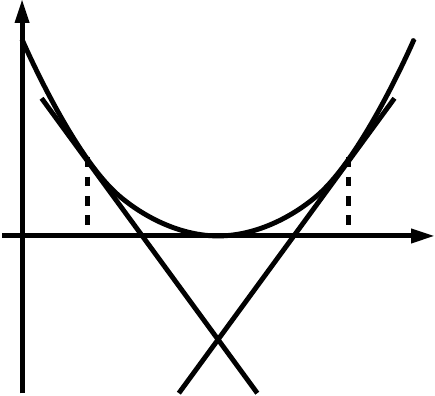}}
\put(.96,.76){\small$\SimLocC(1,\cdot)$}
\put(.18,.3){\small$\alpha$}
\put(.78,.3){\small$\beta$}
\put(.53,.1){\small$(1,s)$}
\put(.13,.17){\small$T_s^l\!=\!T[\alpha]$}
\put(.87,.57){\small$T_s^r\!=\!T[\beta]$}
\end{picture}
\caption{Sketch for the construction of the optimal relative mass changes $\alpha$ and $\beta$ from \thref{exm:twoDiracs}.}
\label{fig:twoDiracs}
\end{figure}

We learn several interesting facts from the above example and particularly optimality conditions \eqref{eqn:DiracOptCondParameterization}-\eqref{eqn:DiracResolveRelMassChange} for $s$, $(\alpha,\beta)$, and $(a,b)$:
\begin{enumerate}
\item A minimizer with $a,b>0$ has to satisfy \eqref{eqn:DiracOptCondParameterization}-\eqref{eqn:DiracResolveRelMassChange}. Consequently there is a minimum distance $L_{\min}=\lim_{s\nearrow0}L(s)\in[0,\infty)$ between both Dirac masses, only depending on $\SimLocC$, below which mass is either only transported before or after the mass change (that is, $a=0$ or $b=0$). $L_{\min}$ is the difference between left and right derivative of $\SimLocC(1,\cdot)$ at $1$, so $L_{\min}>0$ only if $\SimLocC(1,\cdot)$ is not differentiable at $1$.
\item For the same reason and consistent with \thref{thm:maxTransportDistance} there is a maximum distance $L_{\max}=\lim_{s\to-\infty}L(s)\in(0,\infty]$ between both Dirac masses, only depending on $\SimLocC$, beyond which there is no mass transport before or after the mass change ($a=0$ or $b=0$).
\item If there is transport before and after the mass change, then by optimality conditions \eqref{eqn:DiracOptCond1} and \eqref{eqn:DiracOptCond2} the relative mass change $\alpha$ and $\beta$ at each Dirac mass only depends on $\SimLocC$ and the distance between both Dirac masses.
\item If the distance between both Dirac masses satisfies $L\in[L_{\min},L_{\max}]$ and $\alpha$, $\beta$ are the corresponding relative mass changes,
then mass transport happens before as well as after the mass change (that is, $a,b>0$) unless $\alpha,\beta$ violate \eqref{eqn:twoDiracsConstraints1} and \eqref{eqn:twoDiracsConstraints2}.
\item For any $\SimLocC$ we can easily engineer $\rho_0$ and $\rho_1$, each consisting of two Diracs at the same locations, such that transport optimally happens only before, only after, or before as well as after mass change.
Indeed, choose arbitrary $a,b\geq0$ and $s<0$ and let $\alpha,\beta$ be the relative mass changes corresponding to $s$ via \eqref{eqn:DiracTangents}.
Then set $L=L(s)$ and choose any $m_0^0,m_1^0,m_0^L,m_1^L$ satisfying $\alpha=\frac{m_1^0+b}{m_0^0-a}$ and $\beta=\frac{m_1^L-b}{m_0^L+a}$.
Optimality conditions \eqref{eqn:DiracOptCondParameterization}-\eqref{eqn:DiracResolveRelMassChange} are now satisfied by construction.
\item Even in the simple setting with just two Dirac masses nonuniqueness of minimizers can occur in various ways.
First, if $\SimLocC(1,\cdot)$ is not strictly convex, then the tangents $T[\gamma]$ to $\SimLocC(1,\cdot)$ for multiple different values of $\gamma$ coincide
so that \eqref{eqn:DiracOptCond1} and \eqref{eqn:DiracOptCond2} may have multiple solutions.
Second, if $\SimLocC(1,\cdot)$ has nonzero left or right derivative in $1$ and $L$ is small enough, \eqref{eqn:DiracOptCond1} and \eqref{eqn:DiracOptCond2} are solved by $\alpha=\beta=1$
so that \eqref{eqn:relativeMassChanges} has infinitely many solutions $(a,b)$ as long as $\rho_0(\Omega)=\rho_1(\Omega)$.
Third, there may exist optimal $(\pi_0,\pi_1)$ which are not of the form \eqref{eqn:couplingsTwoDiracs}.
\thref{rem:uniquenessTwoDiracs} below examines when uniqueness can be expected.
\end{enumerate}

\begin{remark}[Uniqueness for unbalanced transport between two Dirac masses]\thlabel{rem:uniquenessTwoDiracs}
If $\SimLocC(1,\cdot)$ is strictly convex and differentiable at $1$, then the minimizer $(\pi_0,\pi_1)$ of \eqref{eq:StaticPrimalProblem} for \thref{exm:twoDiracs} is unique.
Indeed, we show below that any minimizer $(\pi_0,\pi_1)$ satisfies $\spt\pi_0,\spt\pi_1\subset\{0,L\}^2$,
and then following the same argument as given in \thref{exm:twoDiracs}, any minimizer must have the form \eqref{eqn:couplingsTwoDiracs}.
Furthermore, for strictly convex $\SimLocC(1,\cdot)$, \eqref{eqn:DiracOptCond1} and \eqref{eqn:DiracOptCond2} can have at most one solution $(\alpha,\beta)$,
which necessarily satisfies $\alpha,\beta\neq1$ (due to the differentiability of $\SimLocC(1,\cdot)$ in $1$) and thus results in at most one solution $(a,b)$.
Thus, there is at most one minimizer $(a,b)$ in the interior of $S$.
Likewise, there can be at most one minimizer on each line segment of $\partial S$ due to the strict convexity of the cost \eqref{eqn:costTwoDiracs} if either $a$ or $b$ are held fixed.
Together with the convexity of the cost \eqref{eqn:costTwoDiracs} these two statements imply that there is a unique minimizer $(a,b)$ and thus a unique minimizer $(\pi_0,\pi_1)$.

It remains to show $\spt\pi_0,\spt\pi_1\subset\{0,L\}^2$ or equivalently $\rho_0'(\Omega\setminus\{0,L\})=\rho_1'(\Omega\setminus\{0,L\})=0$.
For a contradiction, assume the converse and note that by \thref{prop:GrowthShrinkDecomposition} we have $\Omega\setminus\{0,L\}\subset\Omega_=$ so that $\rho_0'\restr(\Omega\setminus\{0,L\})=\rho_1'\restr(\Omega\setminus\{0,L\})$.
The coupling $\pi_0$ must have transported this mass to $\Omega\setminus\{0,L\}$ from $\{0,L\}$ (without loss of generality assume that part of it comes from $0$).
Now note $\pi_0((\Omega_+\cup\Omega_-)\times\Omega_=)=0$ (due to \thref{prop:GrowthShrinkDecomposition} and \thref{thm:transportChar},
whose proof in the case of strictly convex $\SimLocC(1,\cdot)$ actually implies $\pi_0(\Omega_-\times\Omega_=)=\pi_1(\Omega_=\times\Omega_+)=0$ for all minimizers $(\pi_0,\pi_1)$ and not just a particular one)
so that we must have $0\in\Omega_=$.
The optimality of $(\pi_0,\pi_1)$ now implies that $\pi_1$ does not transport mass back from $\Omega_=\setminus\{0\}$ to $0$.
Thus, $\pi_1$ must transport the mass $\rho_0'\restr(\Omega\setminus\{0,L\})=\rho_1'\restr(\Omega\setminus\{0,L\})$ from $\Omega\setminus\{0,L\}=\Omega_=\setminus\{0,L\}$ to $L$.
The fact $\pi_1(\Omega_=\times(\Omega_+\cup\Omega_-))=0$ (again due to \thref{prop:GrowthShrinkDecomposition} and \thref{thm:transportChar}) now implies $L\in\Omega_=$ and thus $\Omega_==\Omega$.
Hence there is no mass change at all, however, this contradicts the differentiability of $\SimLocC(1,\cdot)$ at $1$ with vanishing derivative,
since one can readily reduce the cost of pure transport by introducing a sufficiently small amount of mass change.
\end{remark}

As a further illustration of unbalanced transport between two Dirac masses we specify the previous example to the special case
in which $\rho_0$ and $\rho_1$ each only consist of a single Dirac mass and in which mass change is penalized by the so-called squared Hellinger distance
\begin{equation*}
\SimLocC(m_0,m_1)=(\sqrt{m_0}-\sqrt{m_1})^2\,.
\end{equation*}

\begin{example}[Unbalanced transport between Dirac masses with squared Hellinger metric]\thlabel{exm:twoDiracsHellinger}
Here we pick
\begin{equation*}
\rho_0=m_0^0\delta_0\,,\qquad\rho_1=m_1^L\delta_L
\end{equation*}
as well as the squared Hellinger cost $\SimLocC(m_0,m_1)=(\sqrt{m_0}-\sqrt{m_1})^2$.
As in the previous example we can parameterize $\pi_0$ and $\pi_1$ by $(a,b)\in S=[0,m_0^0] \times [0,m_1^L]$. This can be further restricted to $S=[0,\min(m_0^0,m_1^L)]$. Indeed, if $m_1^L < a \leq m_0^0$, then $L \in \Omega_-$ which must be suboptimal by virtue of \thref{prop:GrowthShrinkDecomposition} and one argues analogously for $b$.
The optimality conditions \eqref{eqn:DiracOptCond1} and \eqref{eqn:DiracOptCond2} turn into
\begin{equation*}
L=\frac1{\sqrt\alpha}-\frac1{\sqrt\beta}\,,\qquad
0=\sqrt\alpha-\sqrt\beta+\frac1{\sqrt\alpha}-\frac1{\sqrt\beta}\,,
\end{equation*}
which can be solved to yield
\begin{equation*}
\alpha=\frac1{S(L)}
\quad\text{and}\quad
\beta=S(L)
\quad\text{for}\quad
S(L)=\left(\frac L2+\sqrt{\frac{L^2}4+1}\right)^2>1
\end{equation*}
and thus
\begin{equation*}
a=\frac{S(L)m_1^L-m_0^0}{S(L)^2-1}\,,\qquad
b=\frac{S(L)m_0^0-m_1^L}{S(L)^2-1}\,.
\end{equation*}
This is admissible (that is, $(a,b)\in S$) as long as
\begin{equation*}
\frac{m_1^L}{m_0^0}\geq\frac1{S(L)}
\quad\text{and}\quad
\frac{m_1^L}{m_0^0}\leq S(L)\,.
\end{equation*}
Otherwise, from the previous example we know $a=0$ or $b=0$ with cost
\begin{align*}
P(0,b)
&=Lb-2\sqrt{bm_0^0}+m_0^0+m_1^L\,,\\
P(a,0)
&=La-2\sqrt{am_1^L}+m_0^0+m_1^L\,,
\end{align*}
from which we obtain $b=\min(m_0^0,m_1^L,\frac{m_0^0}{L^2})$ or $a=\min(m_0^0,m_1^L,\frac{m_1^L}{L^2})$, respectively, by optimization.
Comparing the respective costs, we obtain
\begin{equation*}
a=\begin{cases}
0&\text{if }\frac{m_1^L}{m_0^0}\leq\frac1{S(L)}\,,\\
\frac{S(L)m_1^L-m_0^0}{S(L)^2-1}&\text{if }\frac1{S(L)}<\frac{m_1^L}{m_0^0}<S(L)\,,\\
m_0^0&\text{if }S(L)\leq\frac{m_1^L}{m_0^0}\,,
\end{cases}\qquad
b=\begin{cases}
m_1^L&\text{if }\frac{m_1^L}{m_0^0}\leq\frac1{S(L)}\,,\\
\frac{S(L)m_0^0-m_1^L}{S(L)^2-1}&\text{if }\frac1{S(L)}<\frac{m_1^L}{m_0^0}<S(L)\,,\\
0&\text{if }S(L)\leq\frac{m_1^L}{m_0^0}\,,
\end{cases}
\end{equation*}
which is illustrated via the phase diagram in Figure~\ref{fig:phaseDiagram}.

As can be seen from Figure~\ref{fig:phaseDiagram}, $L_{\min}=0$ and $L_{\max}=\infty$
(it is straightforward to check $T_s^{l/r}=T[1-\frac s2\mp\sqrt{s^2/4-s}]$ with $L(s)=(T_s^r)'-(T_s^l)'=\sqrt{s^2-4s}\to\infty$ as $s\to-\infty$).
Likewise, $L_0=L_1=\infty$ for the maximal transport distances from \thref{thm:maxTransportDistance}.
Thus there will always occur some mass transport, no matter how far the Dirac masses are apart.
This behaviour is fundamentally different from the corresponding unbalanced transport with the Wasserstein-2 transport metric
(the Wasserstein--Fisher--Rao or Hellinger--Kantorovich distance \cite{LieroMielkeSavare-HellingerKantorovich-2015a,ChizatDynamicStatic2015})
in which mass transport is known to cease for distances larger than $\pi$.
\end{example}

\begin{figure}
\centering
\setlength\unitlength{.28\linewidth}
\begin{picture}(1,.8)
\put(0,0){\includegraphics[width=\unitlength]{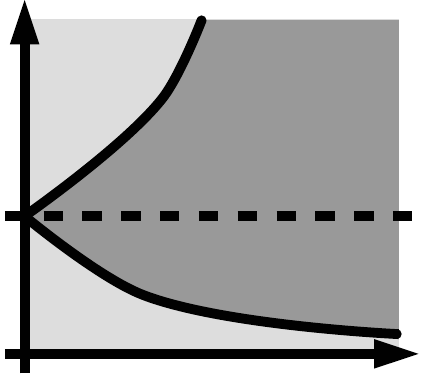}}
\put(.95,-.03){\small$L$}
\put(-.1,.85){\small$m_1^L$}
\put(.4,.9){\small$\frac{m_1^L}{m_0^0}\!=\!S(L)$}
\put(.95,.1){\small$\frac{m_1^L}{m_0^0}\!=\!\frac1{S(L)}$}
\put(-.1,.37){\small$m_0^0$}
\put(.1,.66){\small$b\!=\!0$}
\put(.1,.76){\small$a\!=\!m_0^0$}
\put(.3,.52){\small$a,b>0$}
\put(.3,.43){\small$a\!+\!b\!<\!\min(m_0^0,\!m_1^L\!)$}
\put(.08,.19){\small$a\!=\!0$}
\put(.08,.09){\small$b\!=\!m_1^L$}
\end{picture}
\caption{Phase diagram for the unbalanced transport of one Dirac mass onto another one from \thref{exm:twoDiracsHellinger}.
In the bottom parameter region, the first Dirac mass $m_0^0$ is first shrunk to size $m_1^L$ and then completely moved to the new position;
in the top parameter region it is first completely moved and then grown to $m_1^L$.
In the intermediate region, only part $a$ of the mass is moved initially, then the first Dirac mass is shrunk while the second one is simultaneously grown, and finally the rest $b$ of the mass is moved.}
\label{fig:phaseDiagram}
\end{figure}

\subsection{Examples of different static and corresponding dynamic models}\label{sec:staticAsDynamicExamples}
In the literature there are several examples of static mass change penalties, a few of which are described in \cite{SchmitzerWirthUnbalancedW1-2017} and summarized in Table~\ref{tab:discrepancies}.
By \thref{thm:staticAsDynamic}\eqref{enm:sufficient}, all of these models admit a corresponding dynamic formulation.
For some of them we calculate the corresponding dynamic functions $\BDynMH$ and $\CDynM$ below (they are also provided in Table~\ref{tab:discrepancies} for reference),
for others (for which a closed formula is not straightforward to obtain) we just numerically approximate $\BDynMH$ and $\CDynM$ using \thref{thm:staticAsDynamic}\eqref{enm:uniqueness}.
The numerical approximations are displayed in Table~\ref{tab:discrepancyGraphs}.

\begin{sidewaystable}
\def\arraystretch{1.5}%
\centering
\begin{tabular}{l|c|c|c|c}
&$\SimLocC(m_0,m_1)$ (for $m_0,m_1\geq0$)
&$\HStat(z)$
&$\BDynMH(z)$
&$\CDynM(\rho,\zeta)$ (for $\rho\geq0$)
\\\hline
$\SimDisc$
&$\begin{cases}0&\text{if }m_0=m_1\\\infty&\text{else}\end{cases}$
&$z$
&$0$
&$\begin{cases}0&\text{if }\zeta=0\\\infty&\text{else}\end{cases}$
\\\hline
$\SimTV$
&$|m_1-m_0|$
&$\min(z,1)-\iota_{[-1,\infty)}(z)$
&$-\iota_{[-1,1]}$
&$|\zeta|$
\\\hline
$\SimLoc^{\tn{pwl}}$
&
$\begin{cases}
m_1b(\overline d\!-\!\overline s)\!+\!\\\quad\overline sm_1\!-\!\overline dm_0&\text{if }\frac{m_1}{m_0}\!\leq\! b,\\
(m_1-m_0)\overline s&\text{if }b\!<\!\frac{m_1}{m_0}\!\leq\!1,\\
(m_1-m_0)\underline s&\text{if }1\!<\!\frac{m_1}{m_0}\!\leq\! a,\\
m_1a(\underline d\!-\!\underline s)\!+\!\\\quad\underline sm_1\!-\!\underline dm_0&\text{if }a\!<\!\frac{m_1}{m_0}.
\end{cases}$
&
$\begin{cases}
-\infty&\text{if }z<\underline d,\\
\underline s+a(z-\underline s)&\text{if }z\in[\underline d,\underline s),\\
z&\text{if }z\in[\underline s,\overline s),\\
\overline s+b(z-\overline s)&\text{if }z\in[\overline s,\overline d),\\
\overline s+b(\overline d-\overline s)&\text{else.}
\end{cases}$
&$\begin{cases}
-\infty&\text{if }z\notin[\underline d,\overline d],\\
(z-\underline s)\log a&\text{if }z\in[\underline d,\underline s],\\
0&\text{if }z\in[\underline s,\overline s],\\
(z-\overline s)\log b&\text{if }z\in[\overline s,\overline d].\\
\end{cases}$
&$\begin{cases}
\underline d\zeta+\\\;\rho\log a(\underline d\!-\!\underline s)&\text{if }\zeta\!\leq\!\rho\log\frac1a,\\
\underline s\zeta&\text{if }\rho\log\frac1a\!\leq\!\zeta\!\leq\!0,\\
\overline s\zeta&\text{if }0\!\leq\!\zeta\!\leq\!\rho\log\frac1b,\\
\overline d\zeta+\\\;\rho\log b(\overline d\!-\!\overline s)&\text{if }\zeta\!\geq\!\rho\log\frac1b.
\end{cases}$
\\\hline
$\SimFR$
&$(\sqrt{m_1}-\sqrt{m_0})^2$
&$\frac{z}{1+z}-\iota_{(-1,\infty)}(z)$
&$-z^2$
&$\frac{\zeta^2}{4\rho}$
\\\hline
$\SimJS$
&$m_0\log_2\frac{2m_0}{m_0+m_1}+m_1\log_2\frac{2m_1}{m_0+m_1}$
&$\log_2(2-\frac1{2^z})-\iota_{(-1,\infty)}(z)$
&$\frac{(2^{z/2}-2^{-z/2})^2}{-\log2}$
&$\begin{array}{l}\frac{(\sqrt g-1/\sqrt g)^2}{-\log 2}\rho+\zeta\log_2g\\\text{with }g=\frac\zeta{2\rho}+\sqrt{\frac{\zeta^2}{4\rho^2}+1}\end{array}$
\\\hline
$\SimChi$
&$\tfrac{(m_1-m_0)^2}{m_1+m_0}$
&$\begin{cases}1&\text{if }z>3\\1\!-\!(2\!-\!\sqrt{1\!+\!z})^2&\text{if }z\in[-1,3]\\-\infty&\text{else}\end{cases}$
&num.\ approx.\  in Table~\ref{tab:discrepancyGraphs}
&num.\ approx.\  in Table~\ref{tab:discrepancyGraphs}
\\\hline
$\SimE{0}$
&$m_1-m_0-m_0\log\tfrac{m_1}{m_0}$
&$\log(1+z)-\iota_{(-1,\infty)}(z)$
&num.\ approx.\  in Table~\ref{tab:discrepancyGraphs}
&num.\ approx.\  in Table~\ref{tab:discrepancyGraphs}
\\\hline
$\SimE{1}$
&$m_1\log\tfrac{m_1}{m_0}-m_1+m_0$
&$1-\exp(-z)$
&num.\ approx.\  in Table~\ref{tab:discrepancyGraphs}
&num.\ approx.\  in Table~\ref{tab:discrepancyGraphs}
\\\hline
$\SimEp$
&$m_0\tfrac1{p(p-1)}((\tfrac{m_1}{m_0})^p-p(\tfrac{m_1}{m_0}-1)-1)$
&$\frac{1-(1+(1-p)z)^{\frac p{p-1}}}p-\iota_{[-1,\infty)}((1\!-\!p)z)$
&num.\ approx.\  in Table~\ref{tab:discrepancyGraphs}
&num.\ approx.\  in Table~\ref{tab:discrepancyGraphs}
\end{tabular}
\caption{Formulas of different static and dynamic model functions.}
\label{tab:discrepancies}
\end{sidewaystable}

\begin{table}
\centering
\setlength\unitlength{.15\linewidth}
\vspace*{-.6\unitlength}
\begin{tabular}{lcccc}
&$\SimLocC(1,\cdot)$
&$\HStat$
&$\BDynMH$
&$\CDynM(1,\cdot)$
\\
\raisebox{.45\unitlength}{$\SimDisc$}&\includegraphics[height=\unitlength]{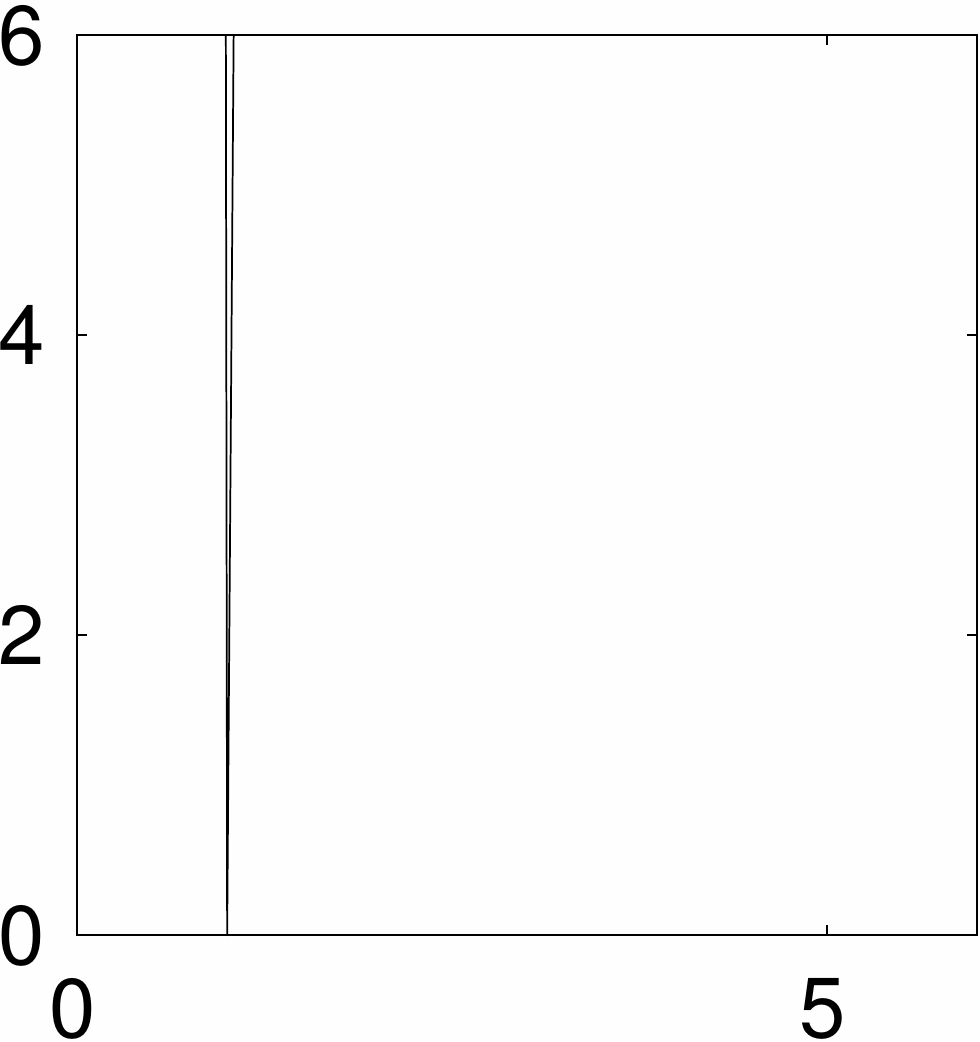}&\includegraphics[height=\unitlength]{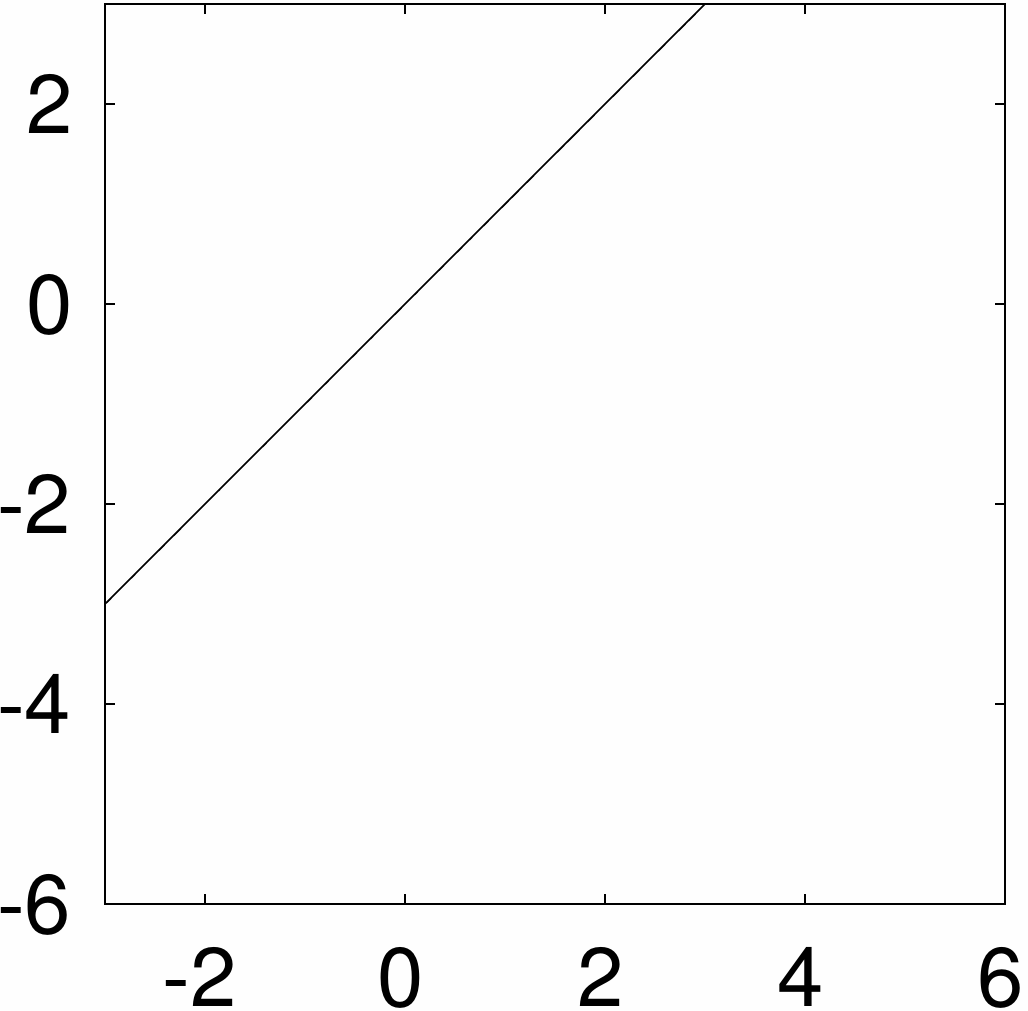}&\includegraphics[height=\unitlength]{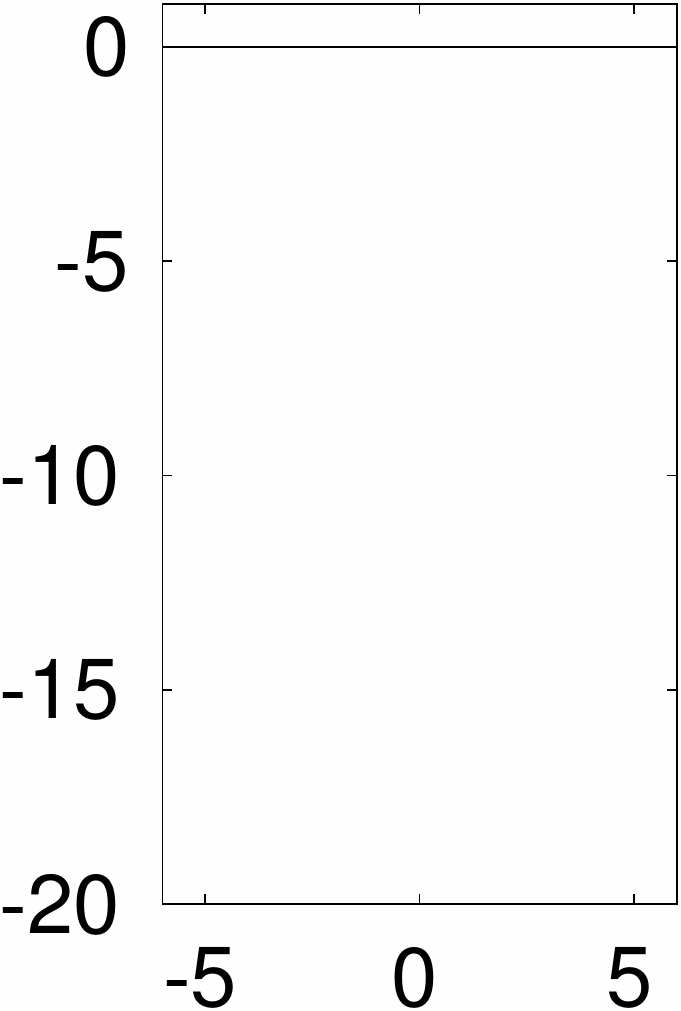}&\raisebox{.15\unitlength}{\includegraphics[height=.7\unitlength]{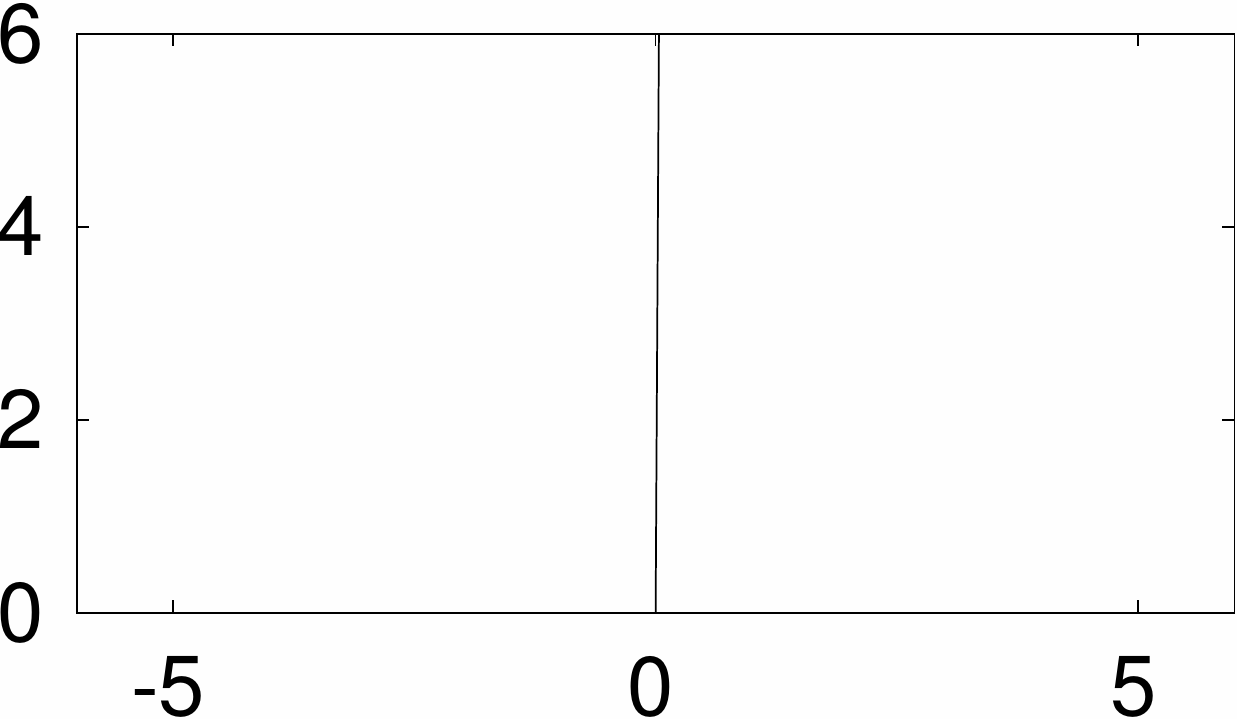}}\\
\raisebox{.45\unitlength}{$\SimTV$}&\includegraphics[height=\unitlength]{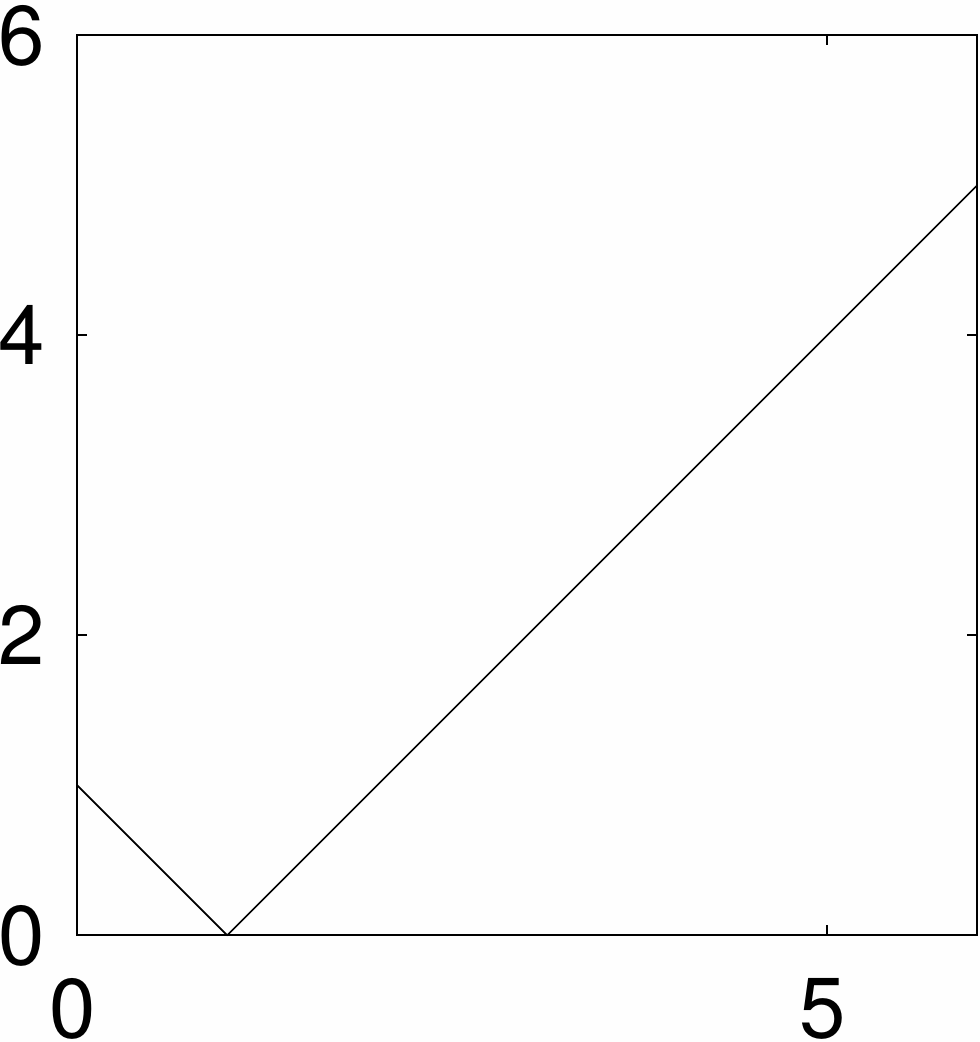}&\includegraphics[height=\unitlength]{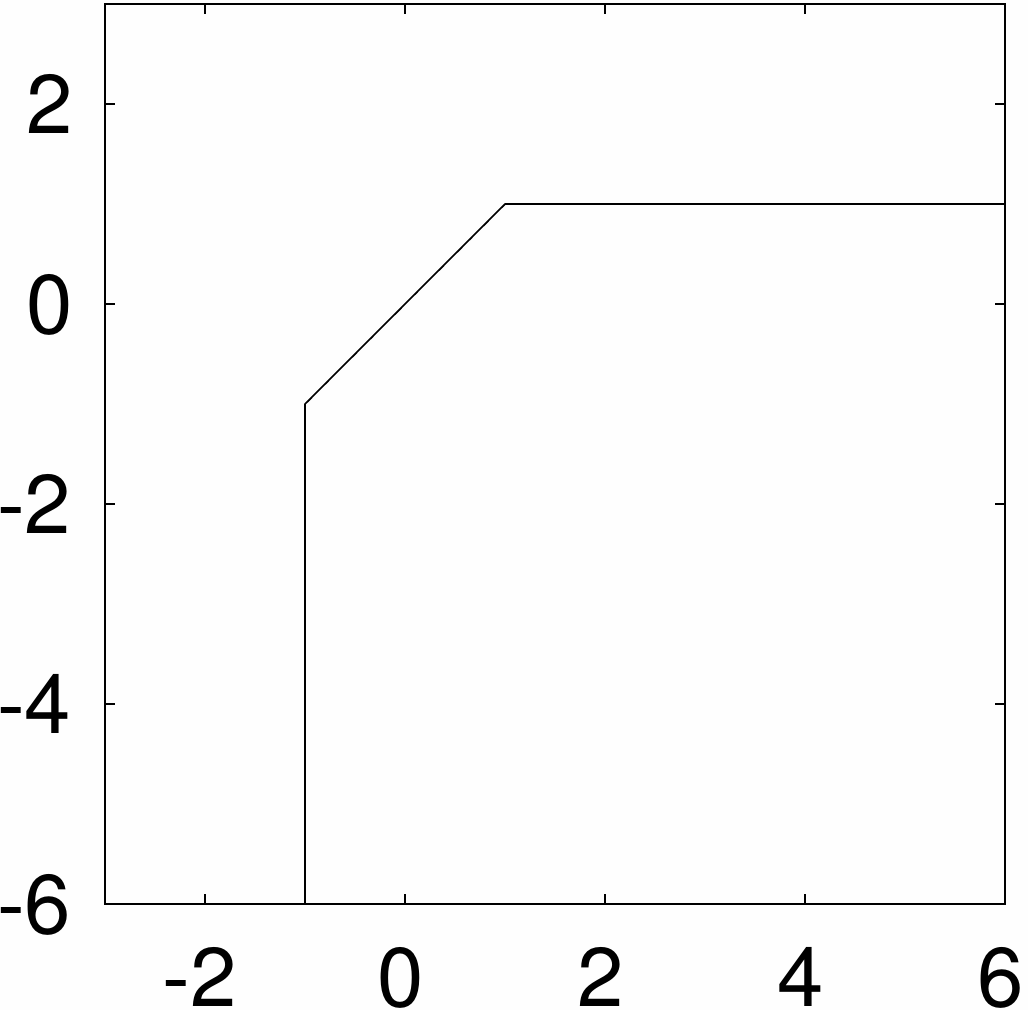}&\includegraphics[height=\unitlength]{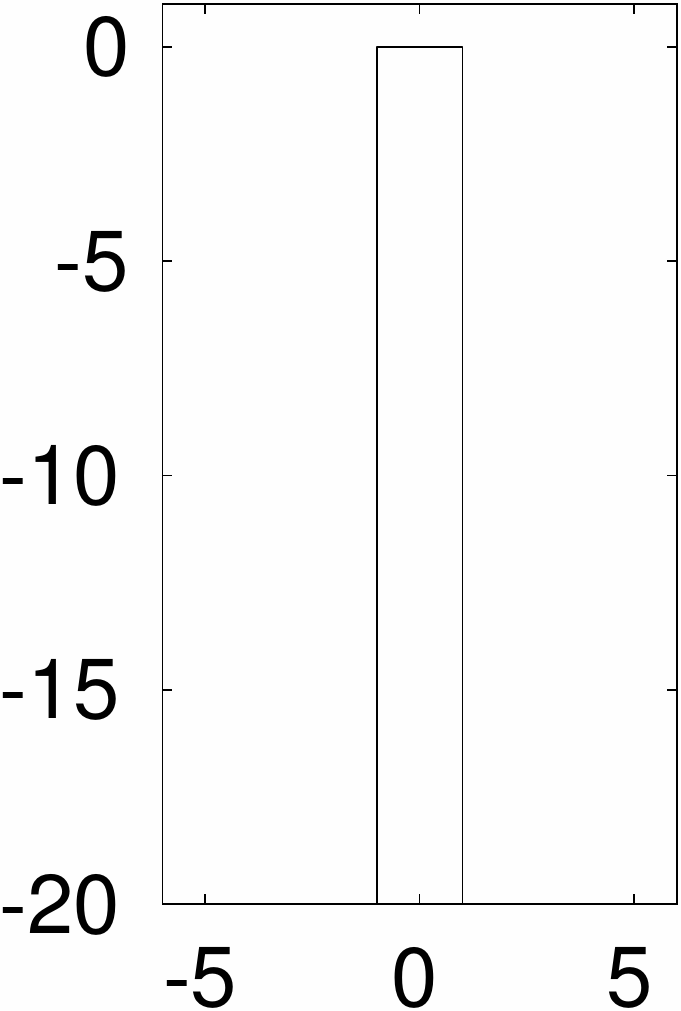}&\raisebox{.15\unitlength}{\includegraphics[height=.7\unitlength]{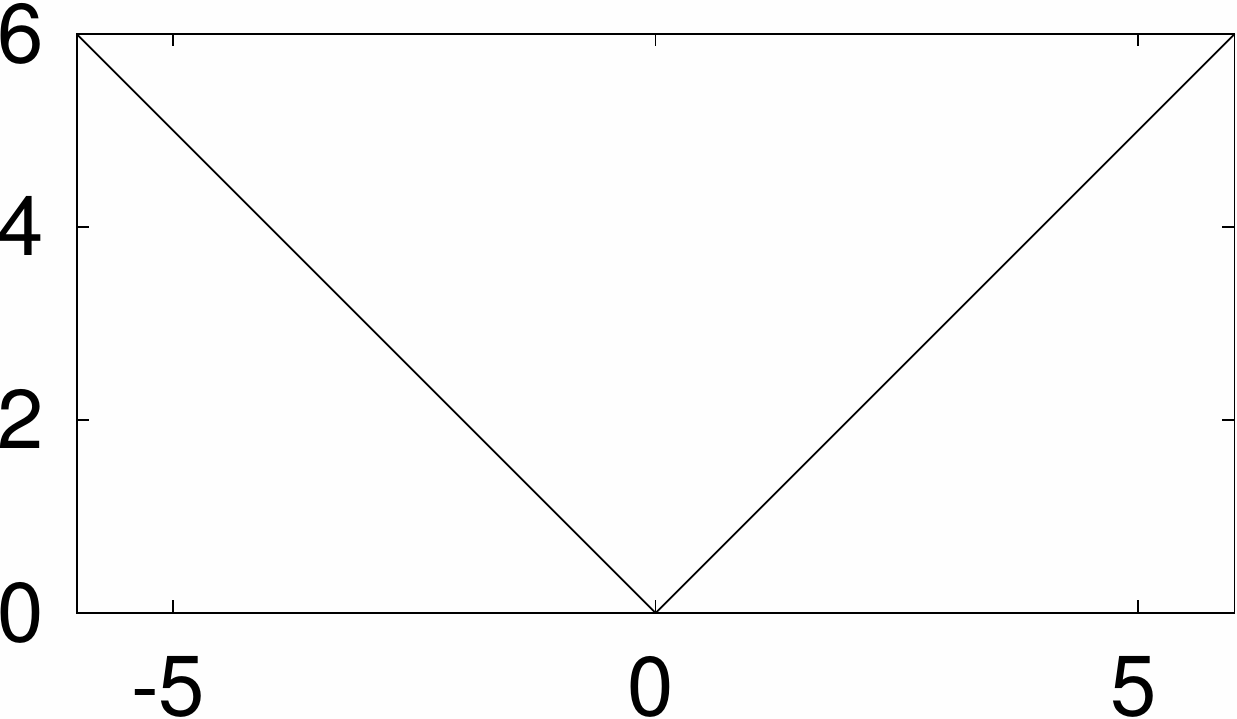}}\\
\raisebox{.45\unitlength}{$\SimLoc^{\tn{pwl}}$}&\includegraphics[height=\unitlength]{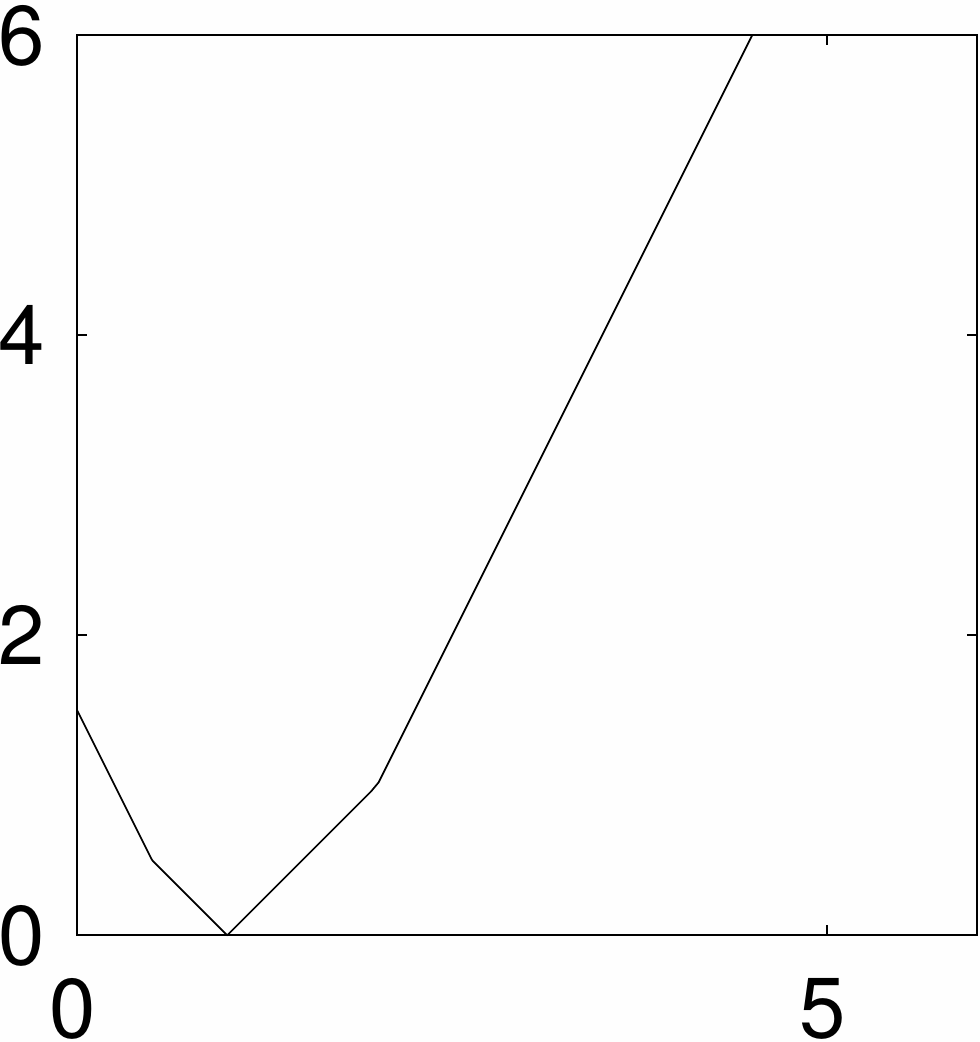}&\includegraphics[height=\unitlength]{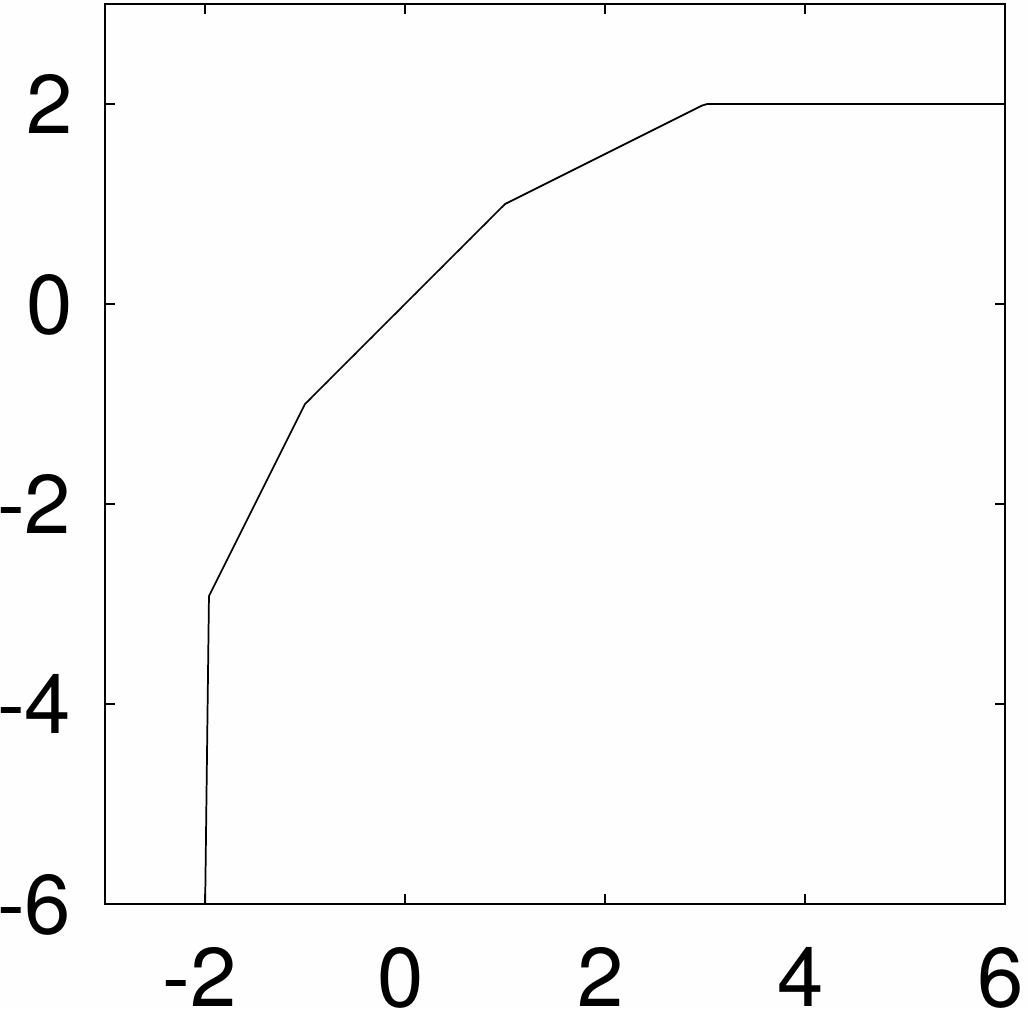}&\includegraphics[height=\unitlength]{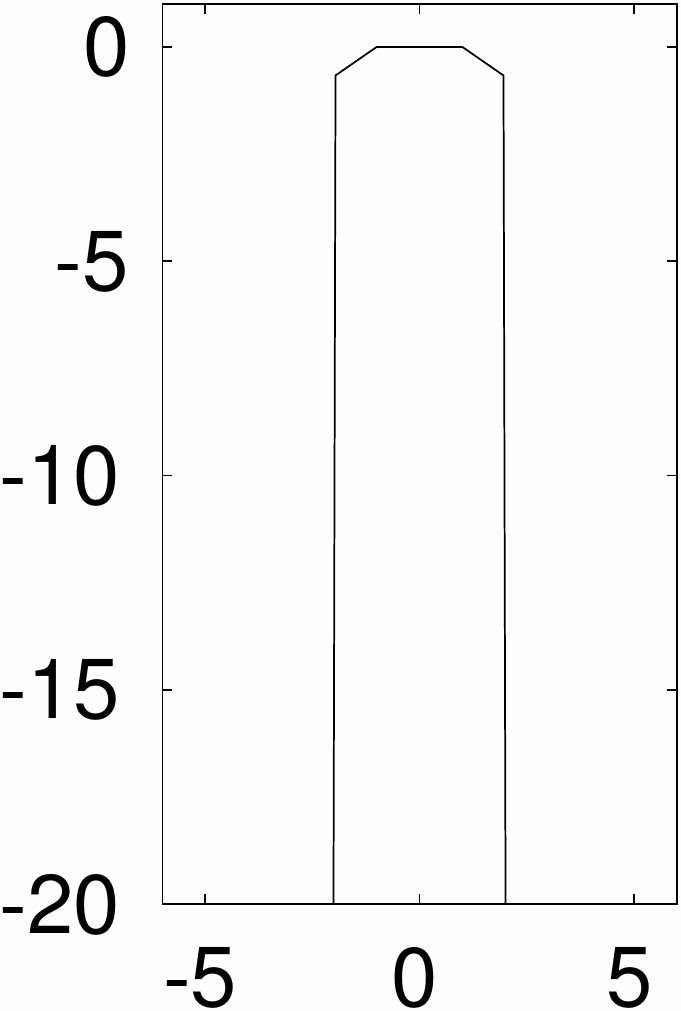}&\raisebox{.15\unitlength}{\includegraphics[height=.7\unitlength]{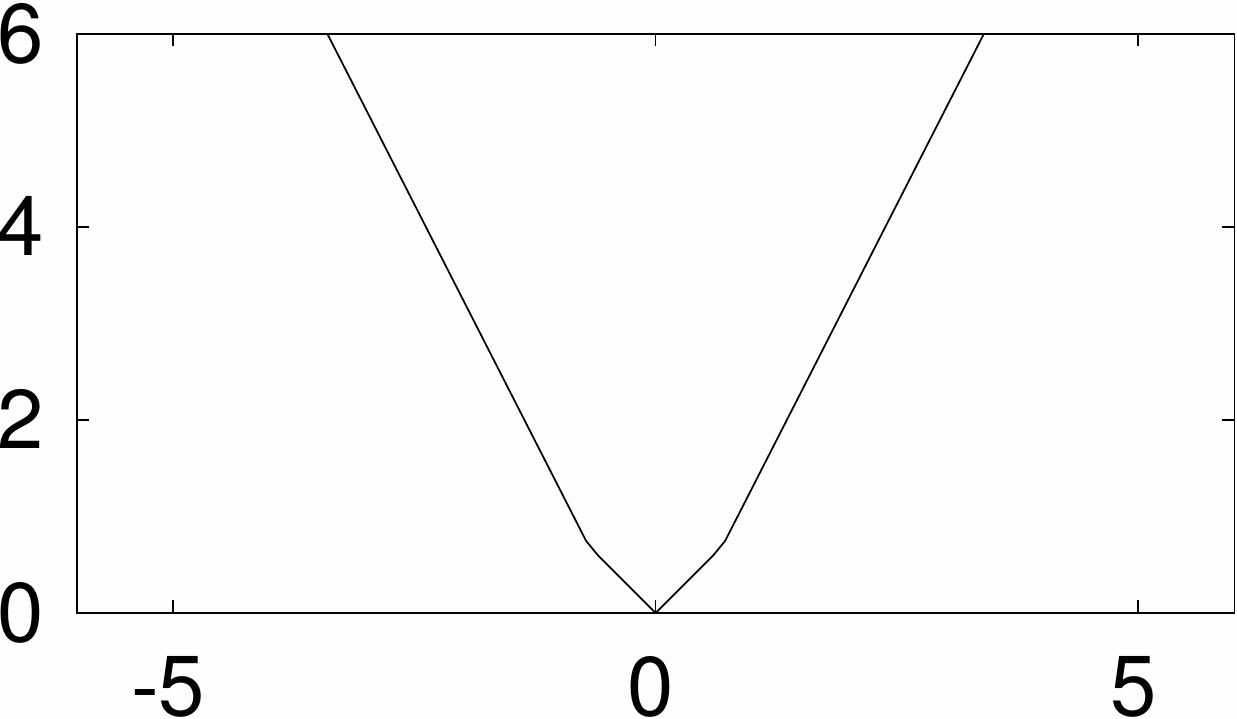}}\\
\raisebox{.45\unitlength}{$\SimFR$}&\includegraphics[height=\unitlength]{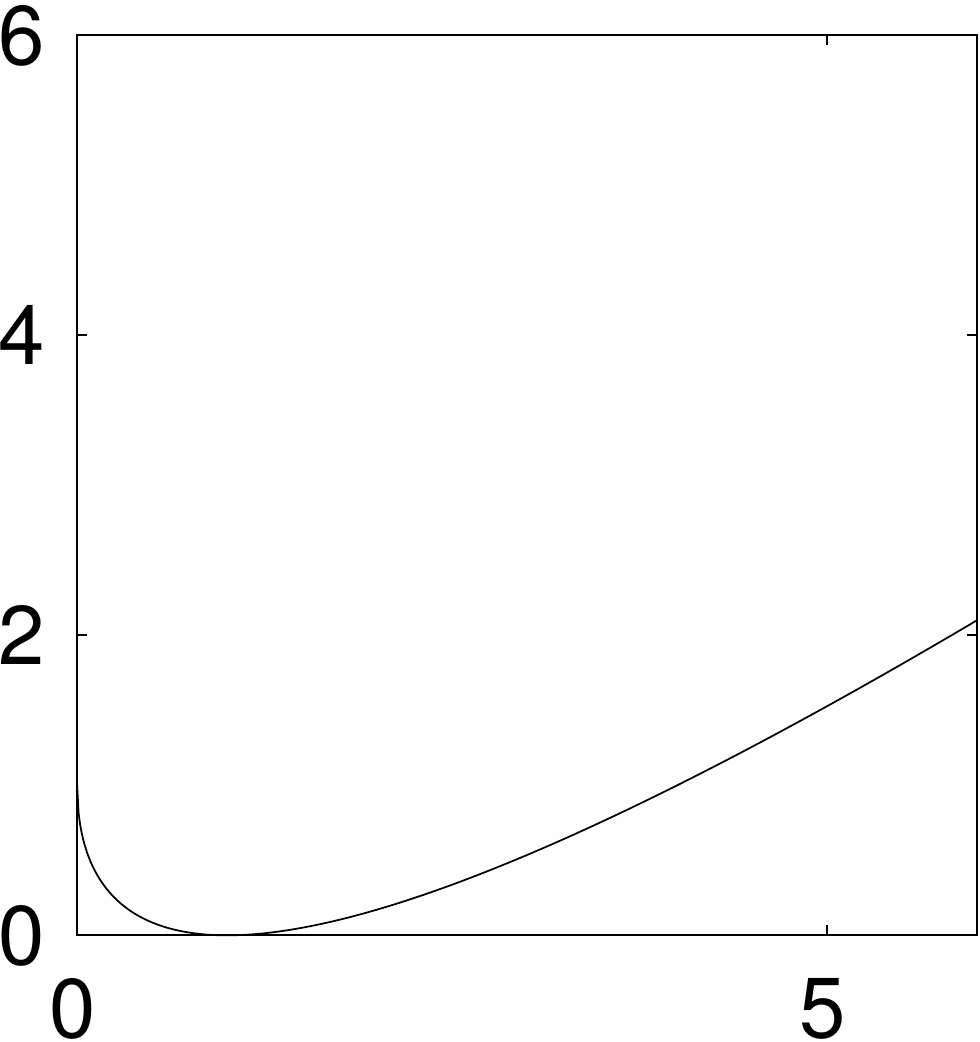}&\includegraphics[height=\unitlength]{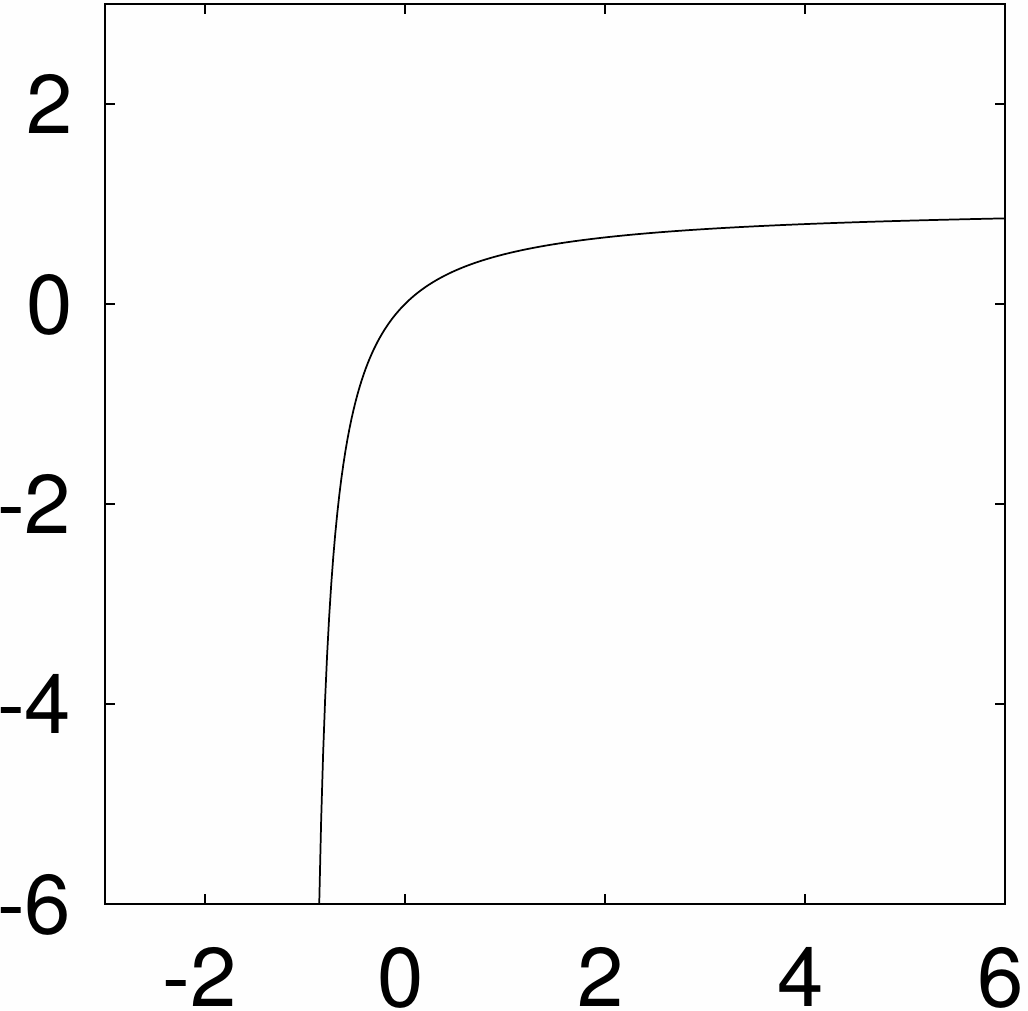}&\includegraphics[height=\unitlength]{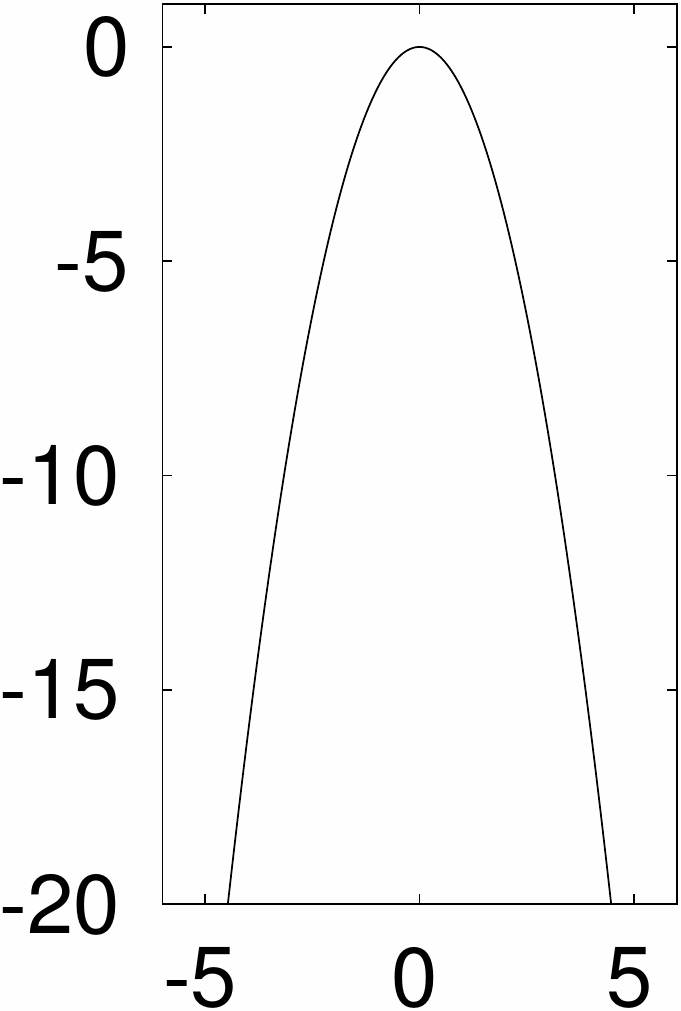}&\raisebox{.15\unitlength}{\includegraphics[height=.7\unitlength]{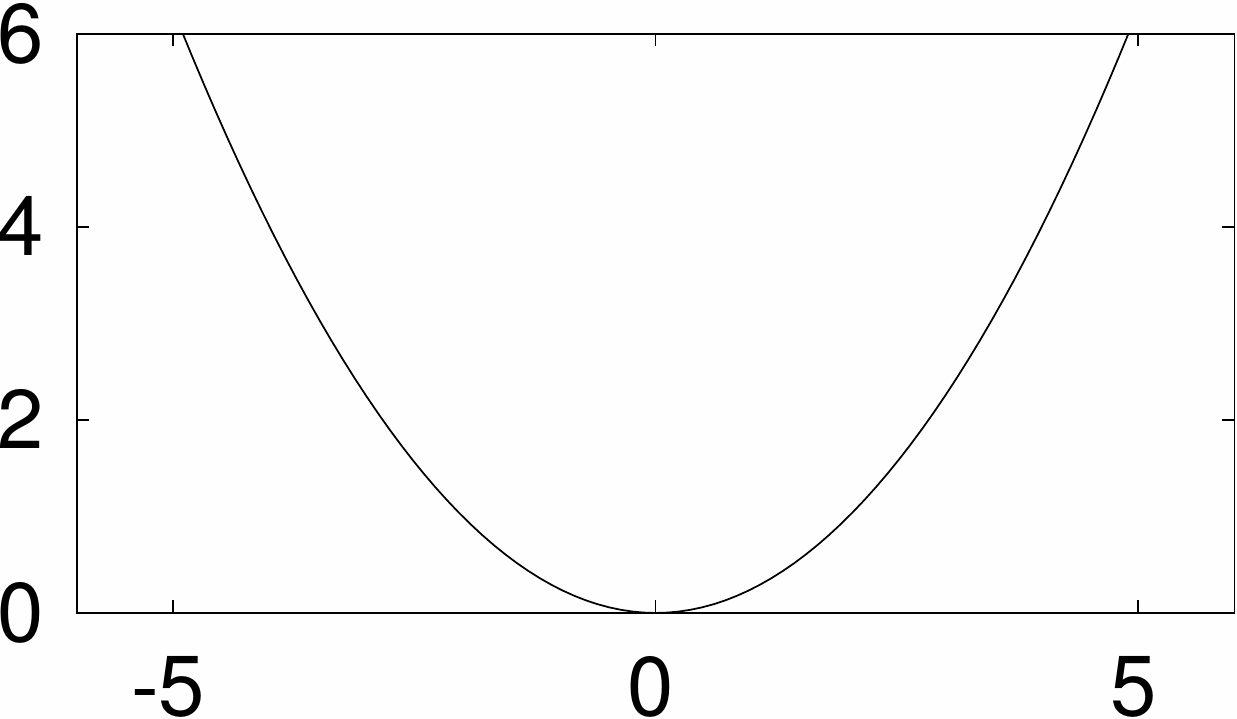}}\\
\raisebox{.45\unitlength}{$\SimJS$}&\includegraphics[height=\unitlength]{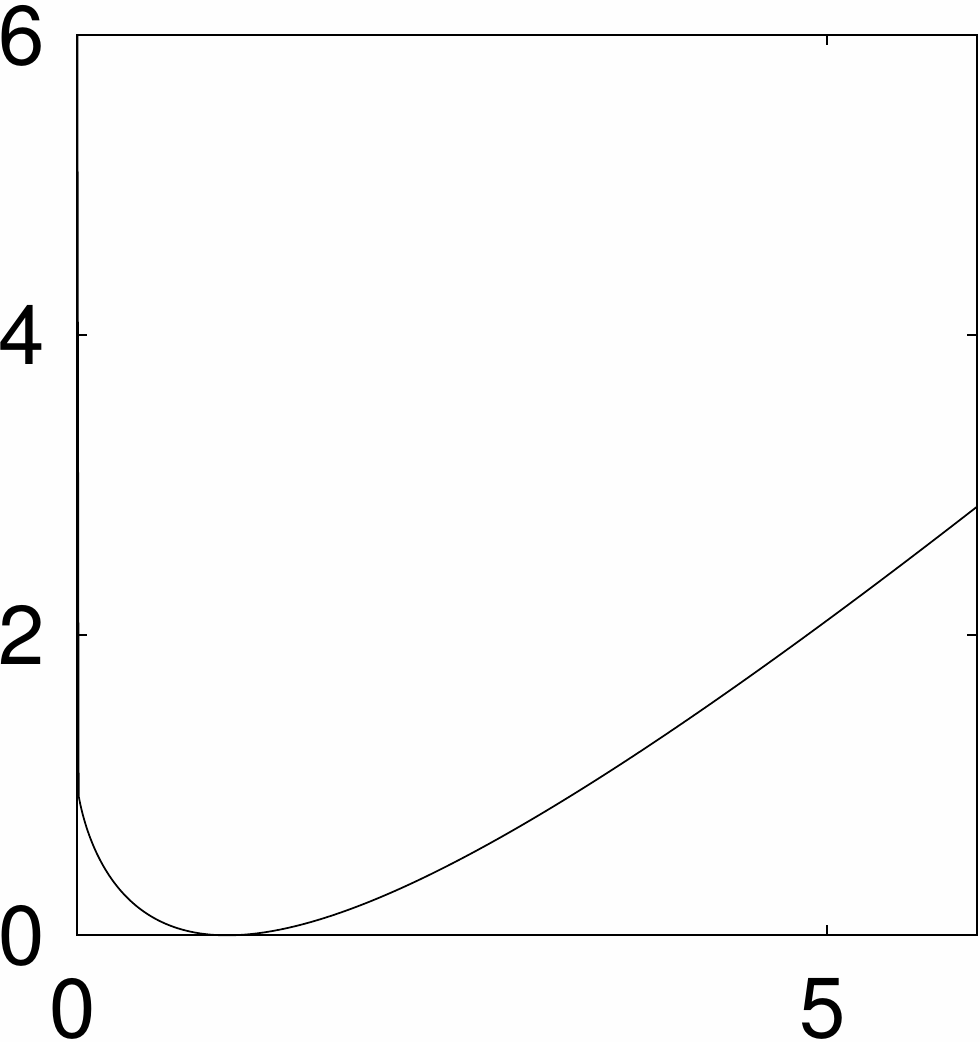}&\includegraphics[height=\unitlength]{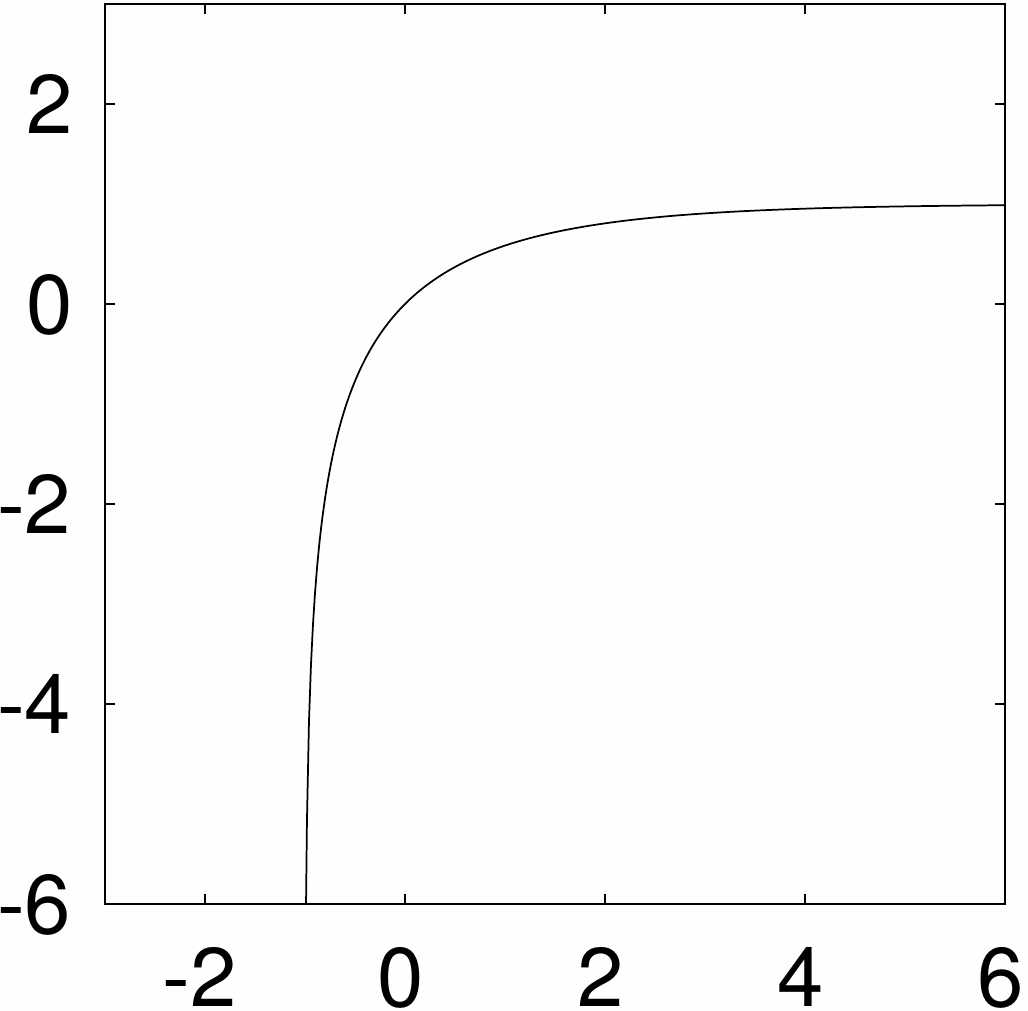}&\includegraphics[height=\unitlength]{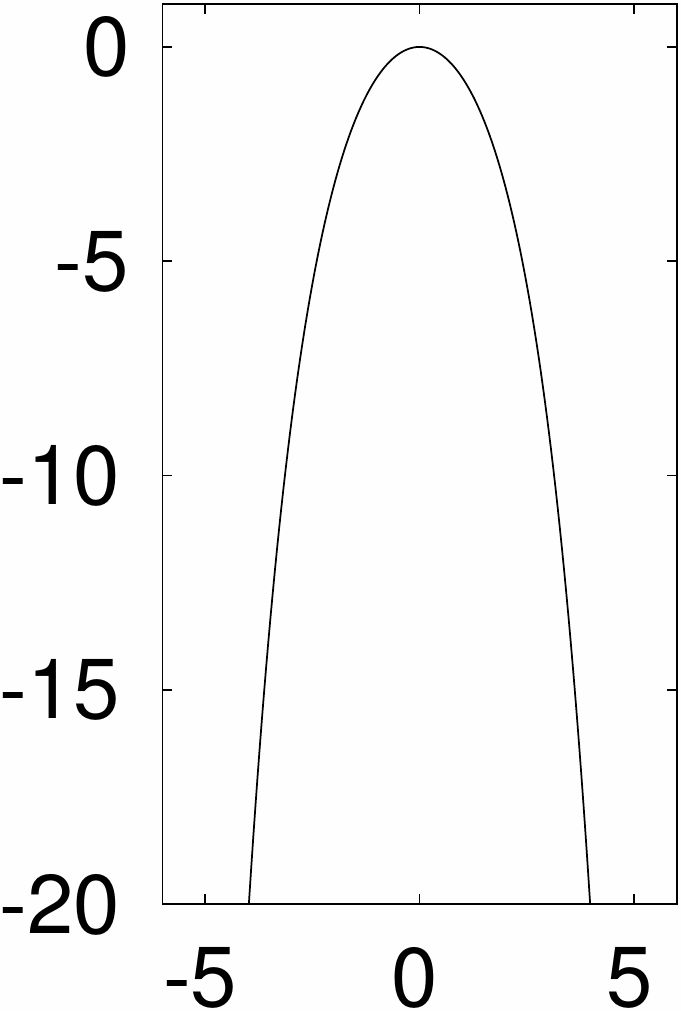}&\raisebox{.15\unitlength}{\includegraphics[height=.7\unitlength]{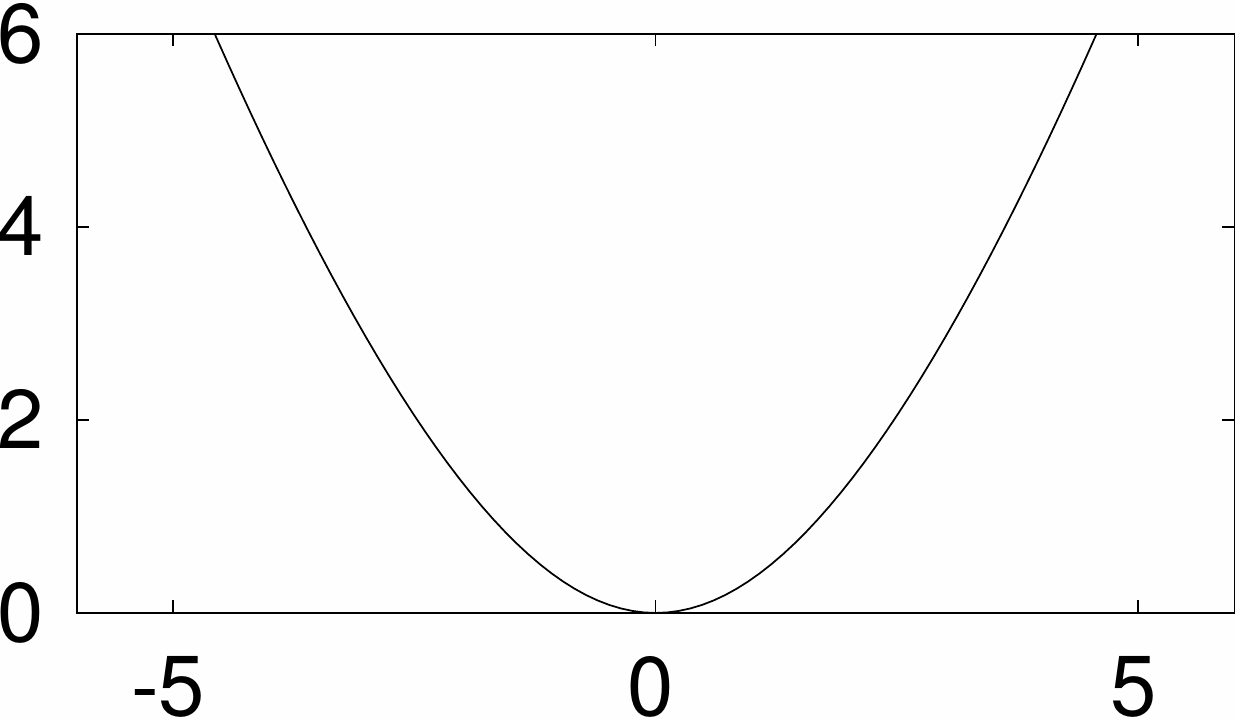}}\\
\raisebox{.45\unitlength}{$\SimChi$}&\includegraphics[height=\unitlength]{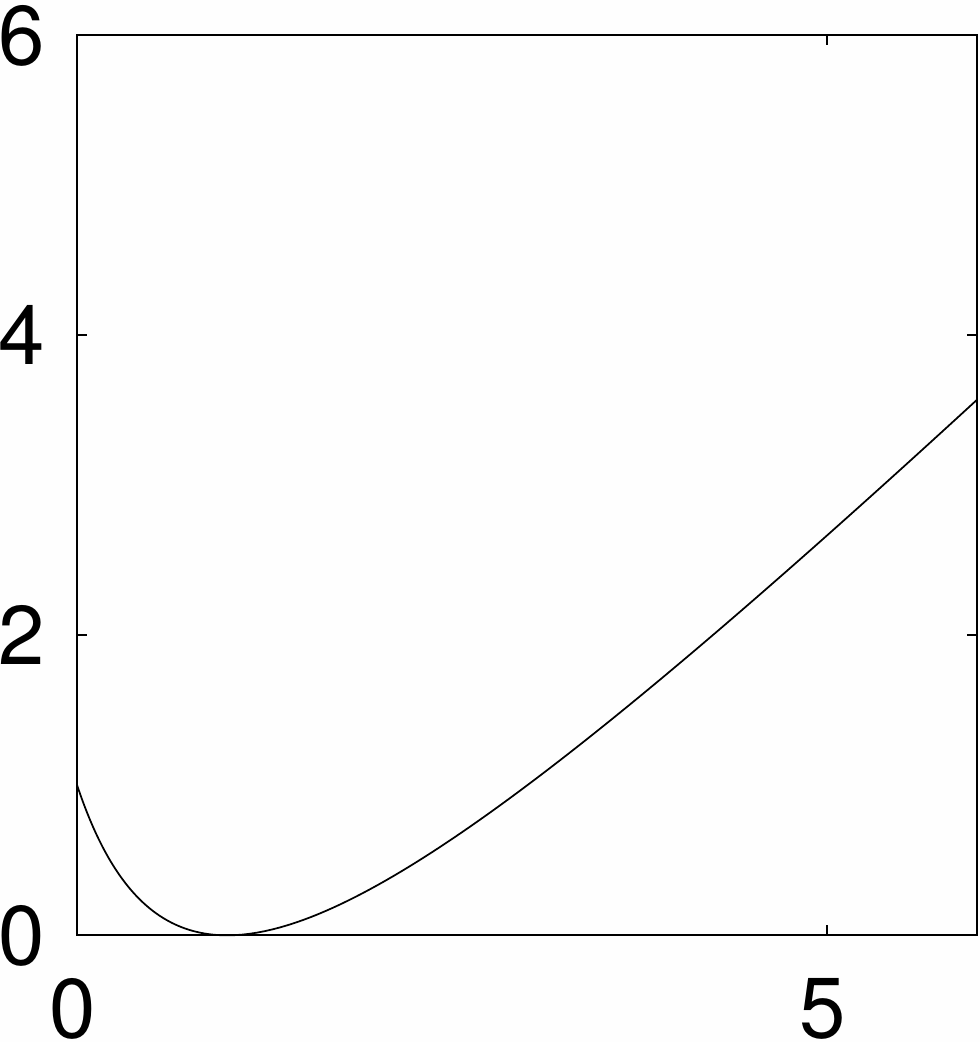}&\includegraphics[height=\unitlength]{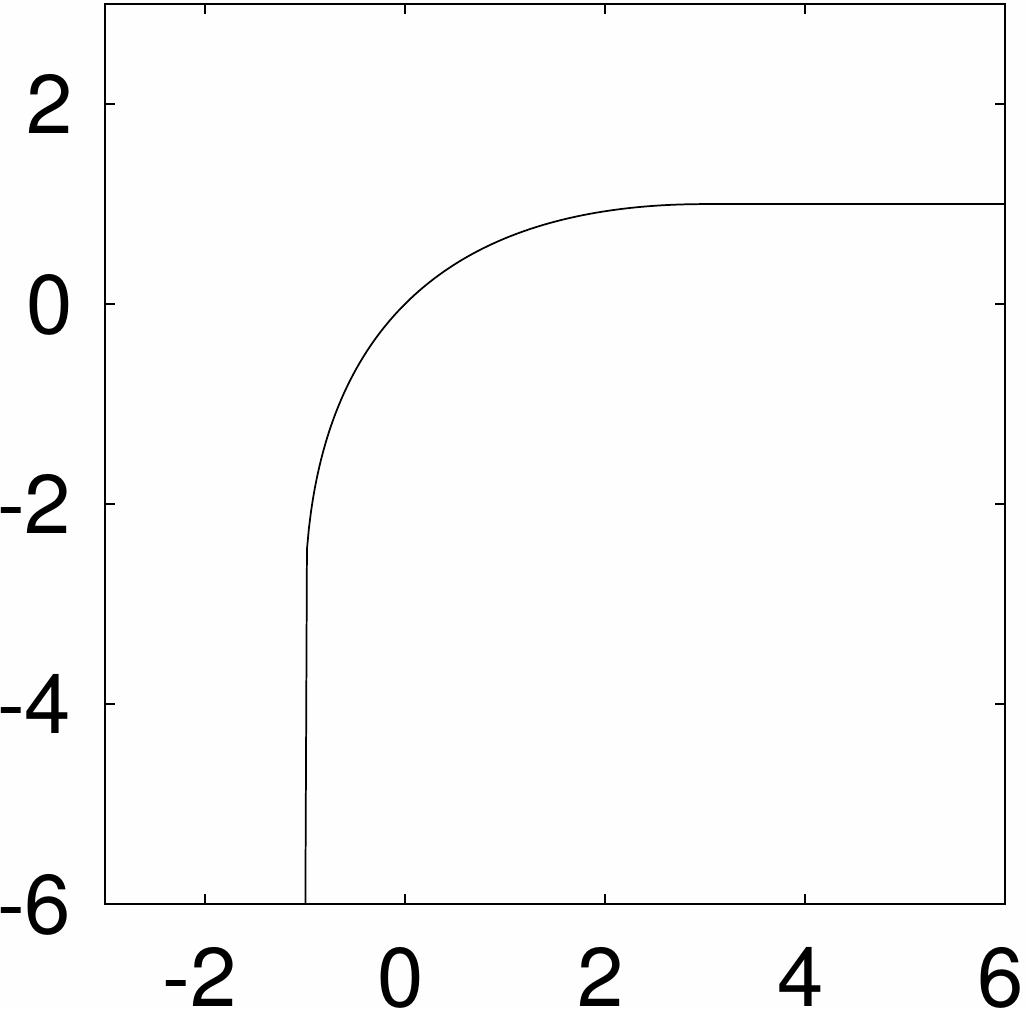}&\includegraphics[height=\unitlength]{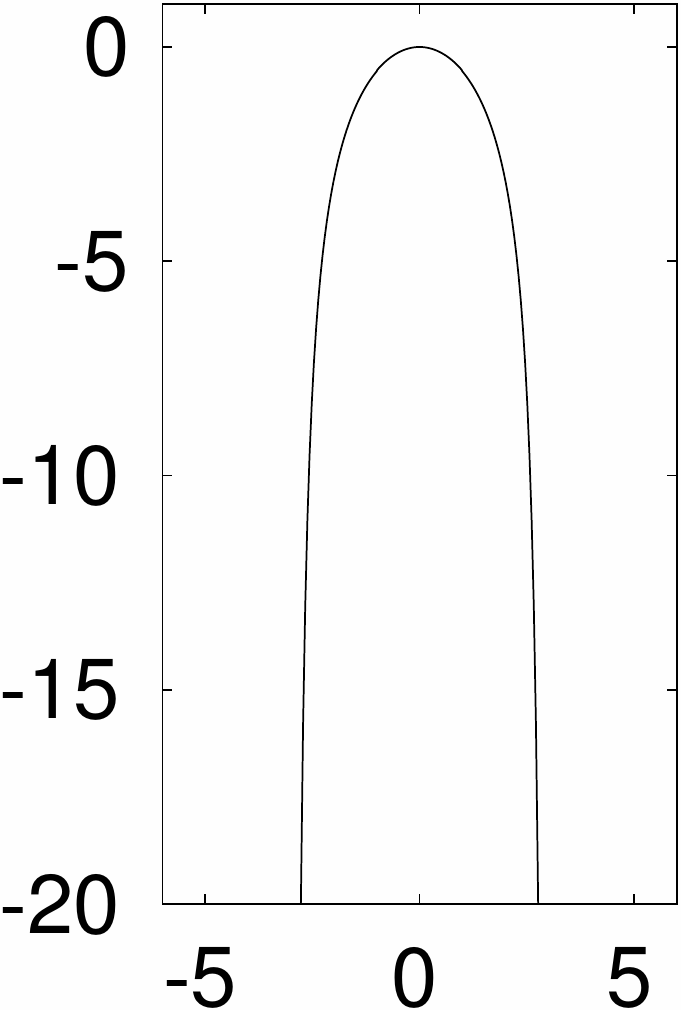}&\raisebox{.15\unitlength}{\includegraphics[height=.7\unitlength]{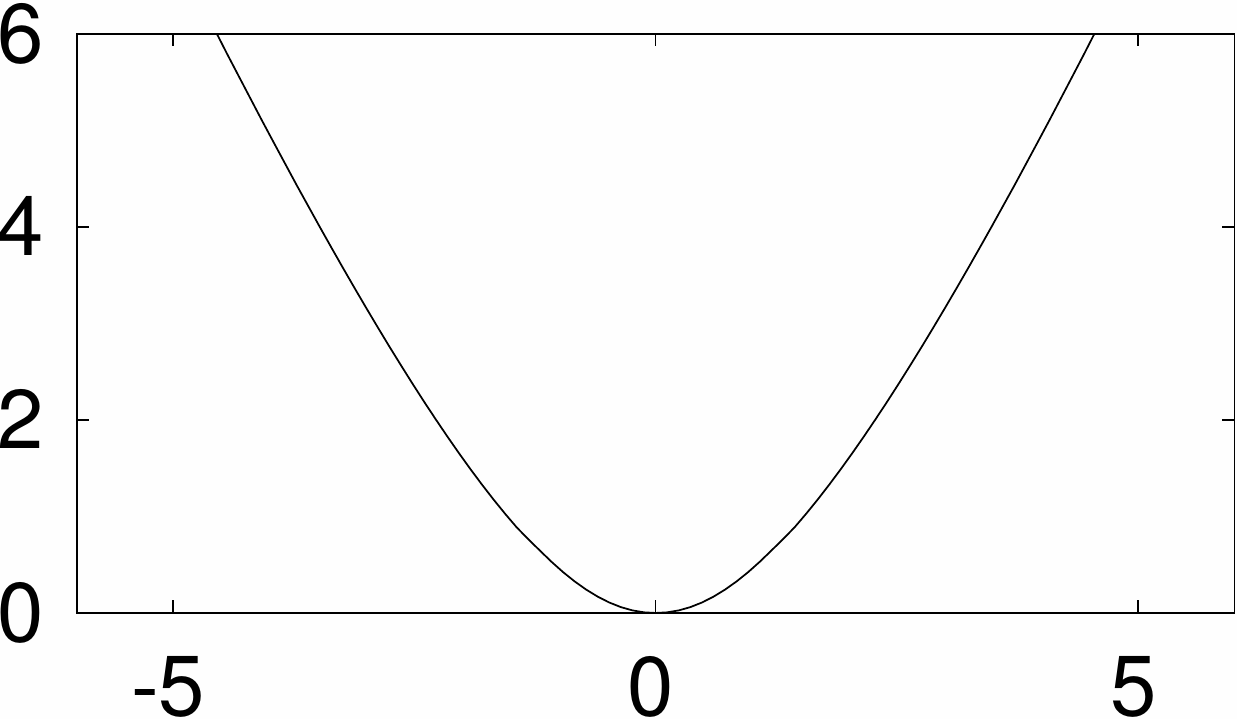}}\\
\raisebox{.45\unitlength}{$\SimE{0}$}&\includegraphics[height=\unitlength]{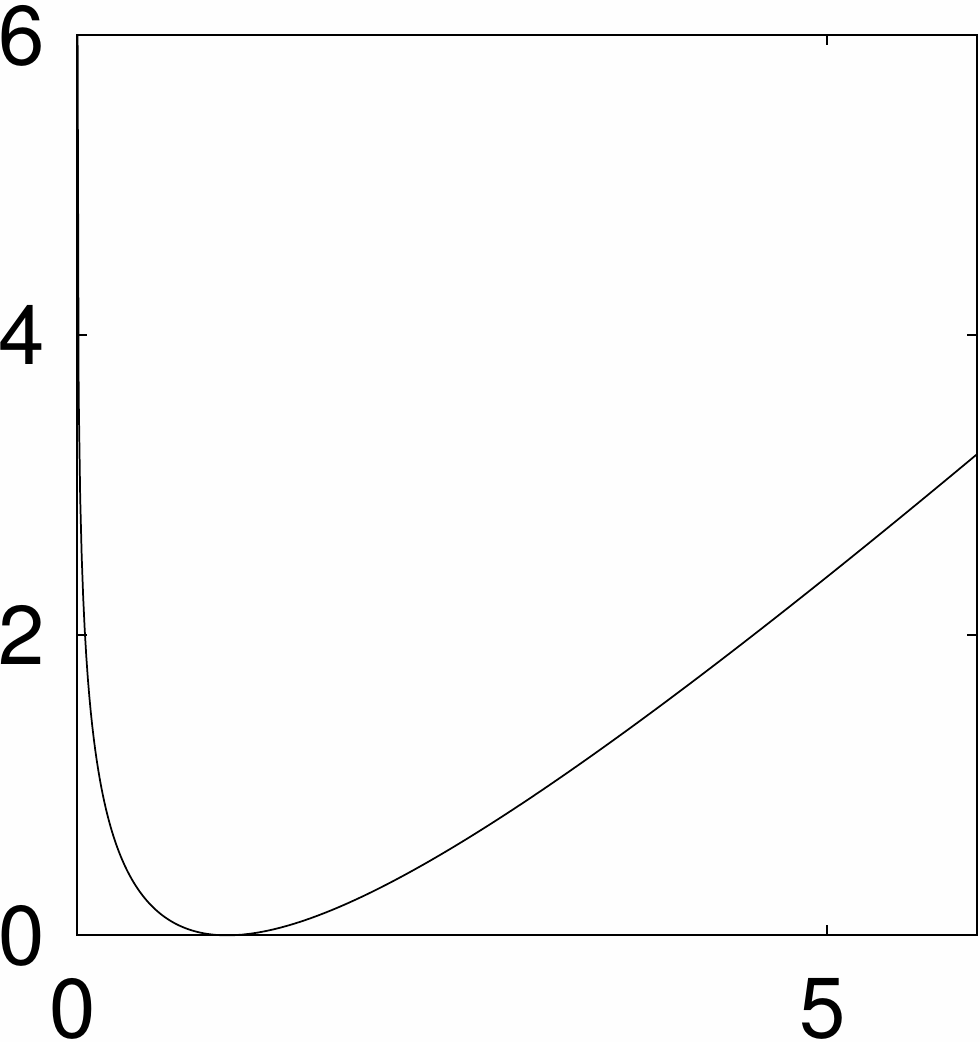}&\includegraphics[height=\unitlength]{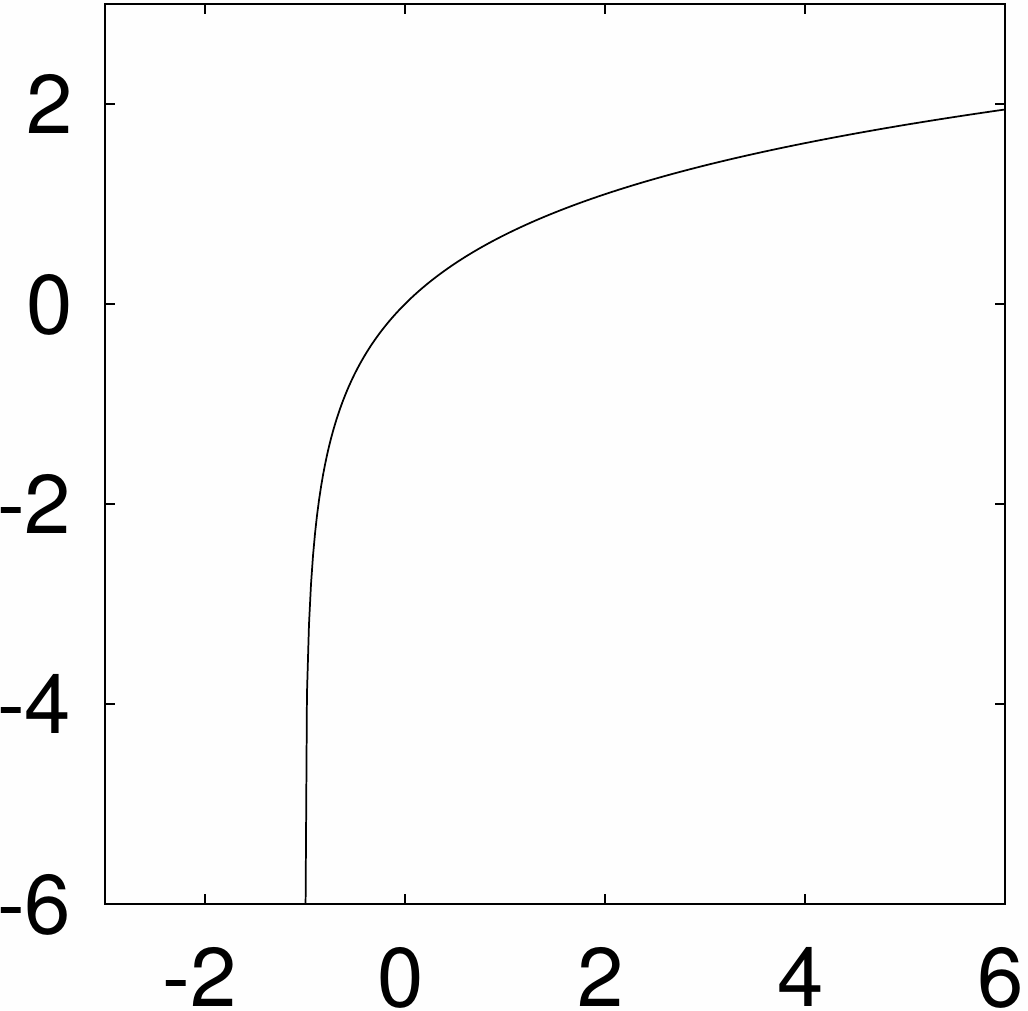}&\includegraphics[height=\unitlength]{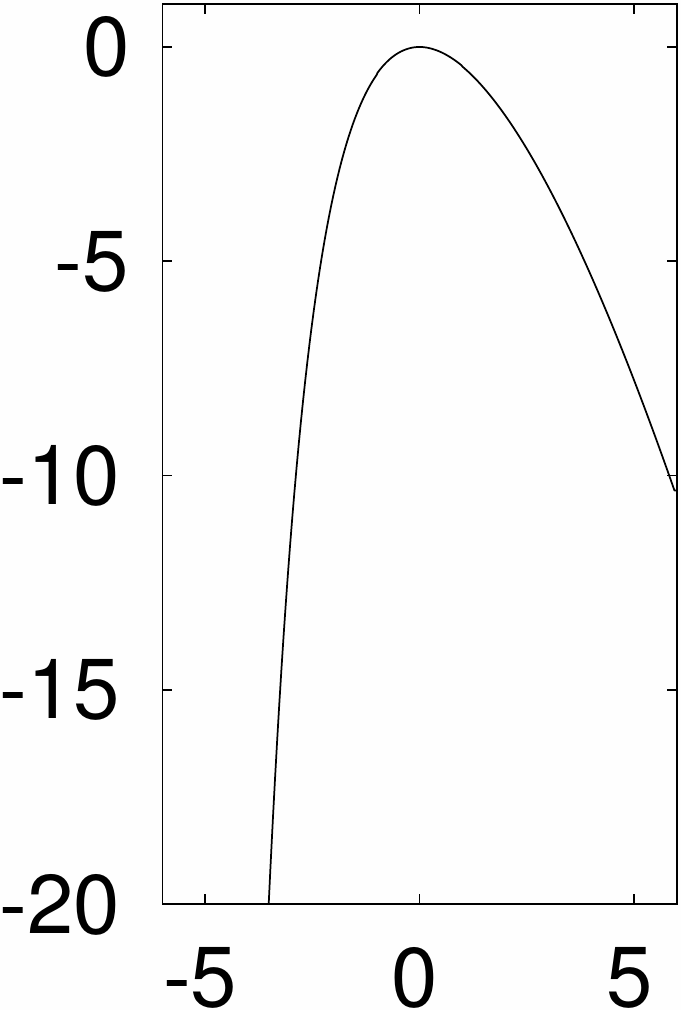}&\raisebox{.15\unitlength}{\includegraphics[height=.7\unitlength]{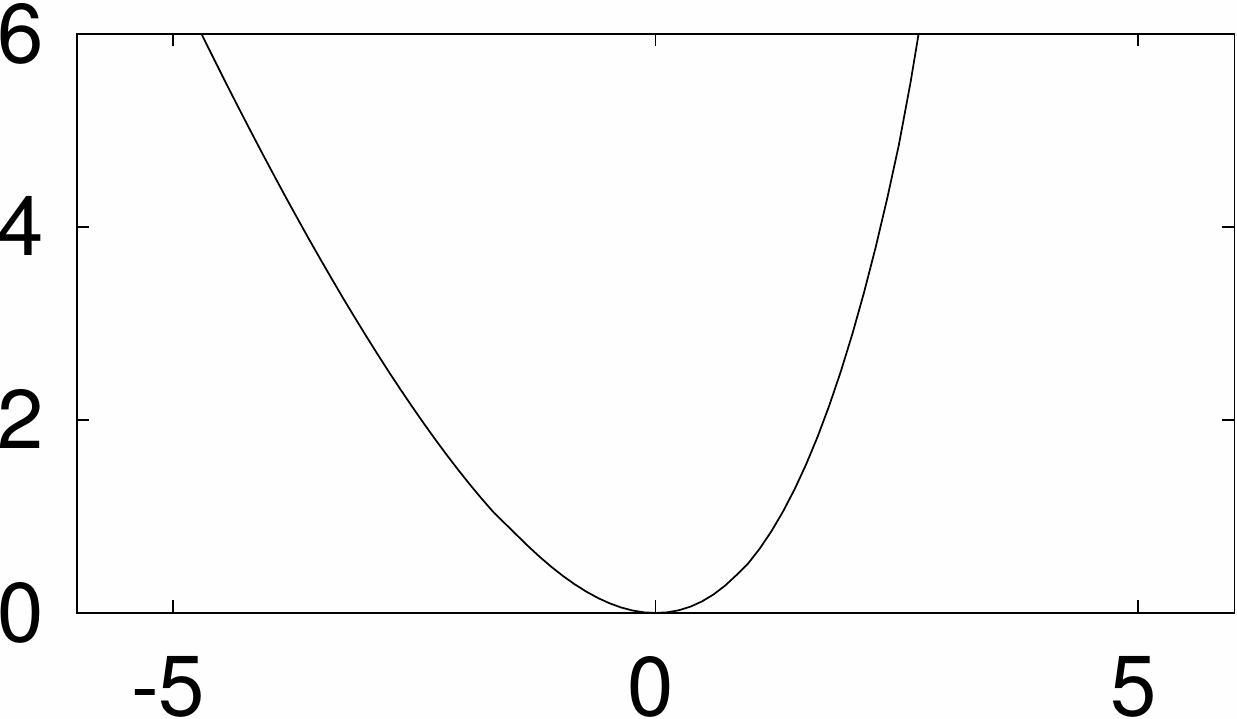}}\\
\raisebox{.45\unitlength}{$\SimE{1}$}&\includegraphics[height=\unitlength]{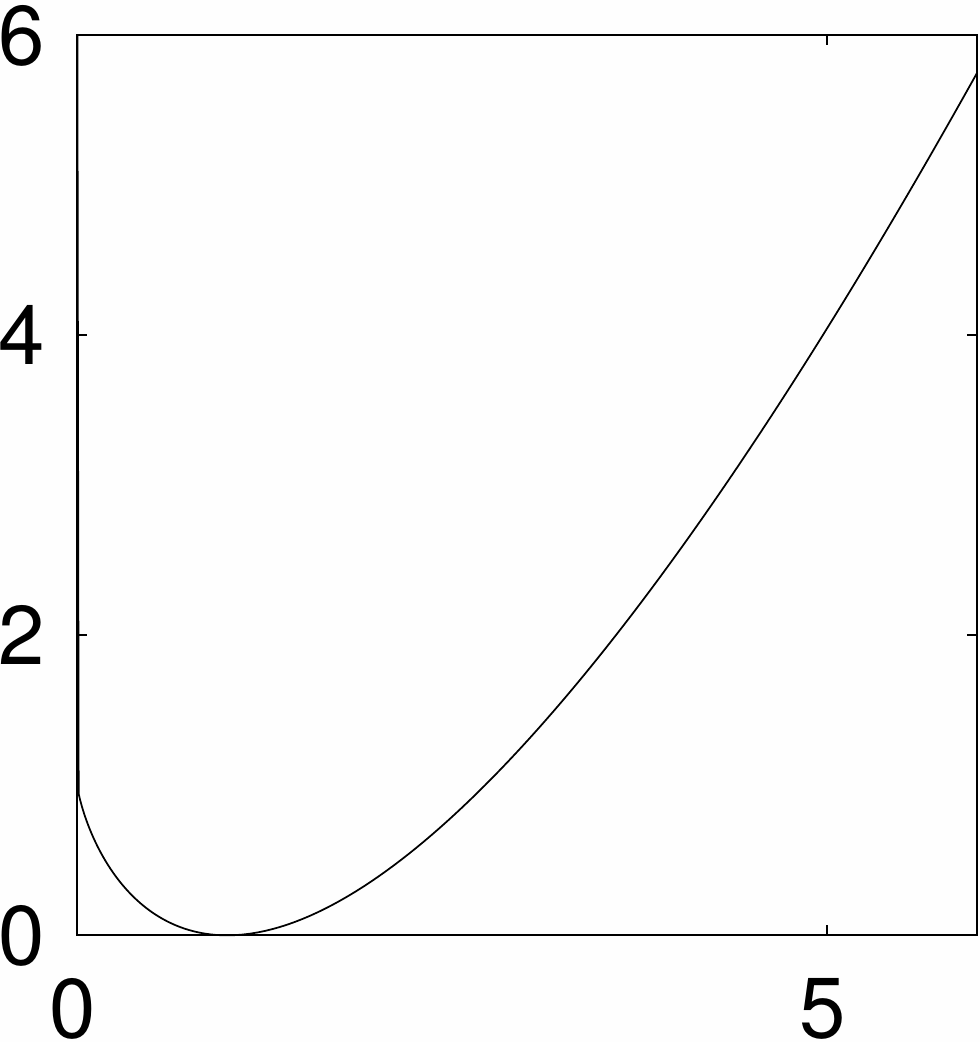}&\includegraphics[height=\unitlength]{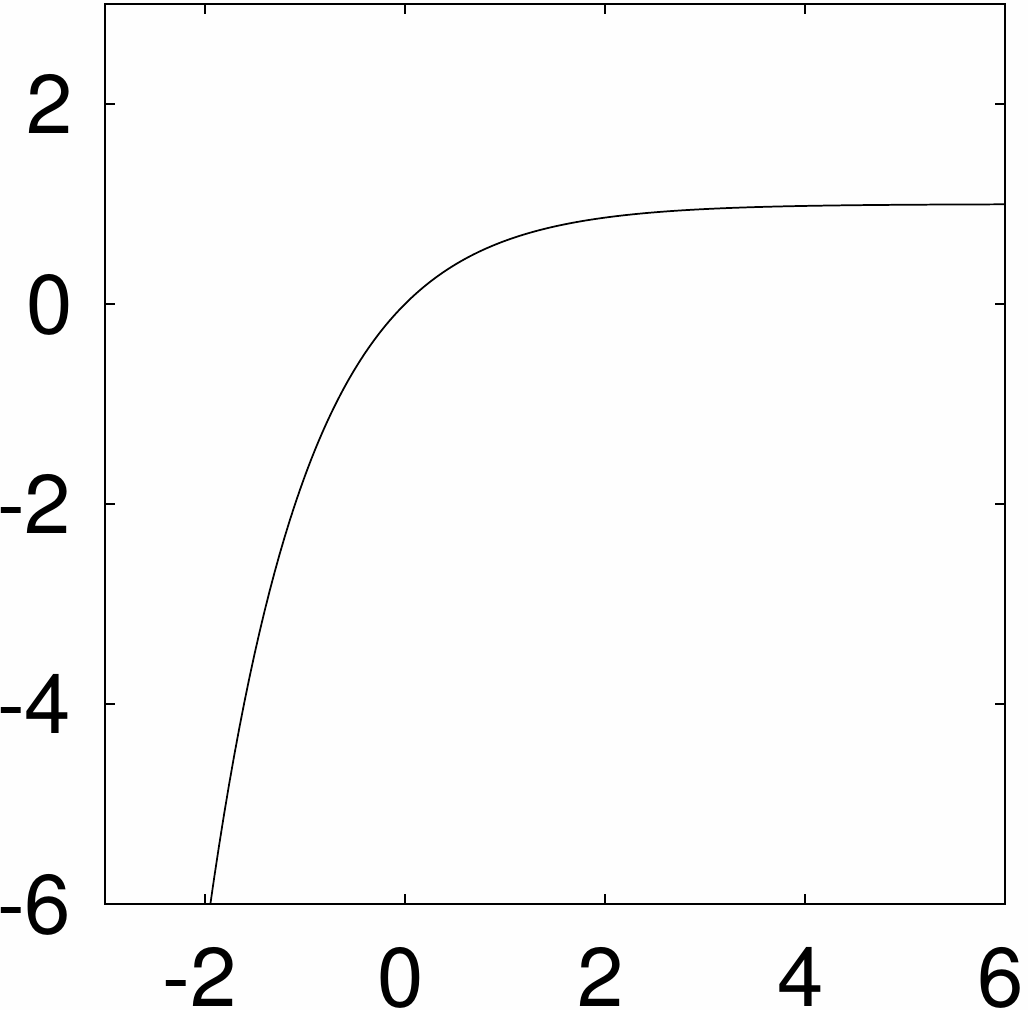}&\includegraphics[height=\unitlength]{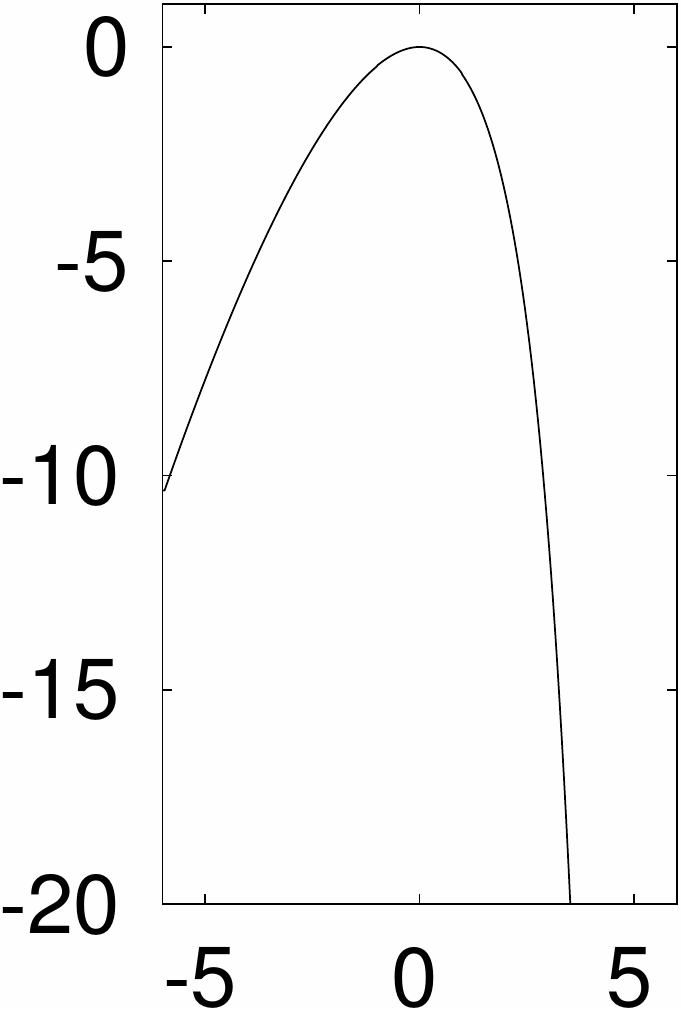}&\raisebox{.15\unitlength}{\includegraphics[height=.7\unitlength]{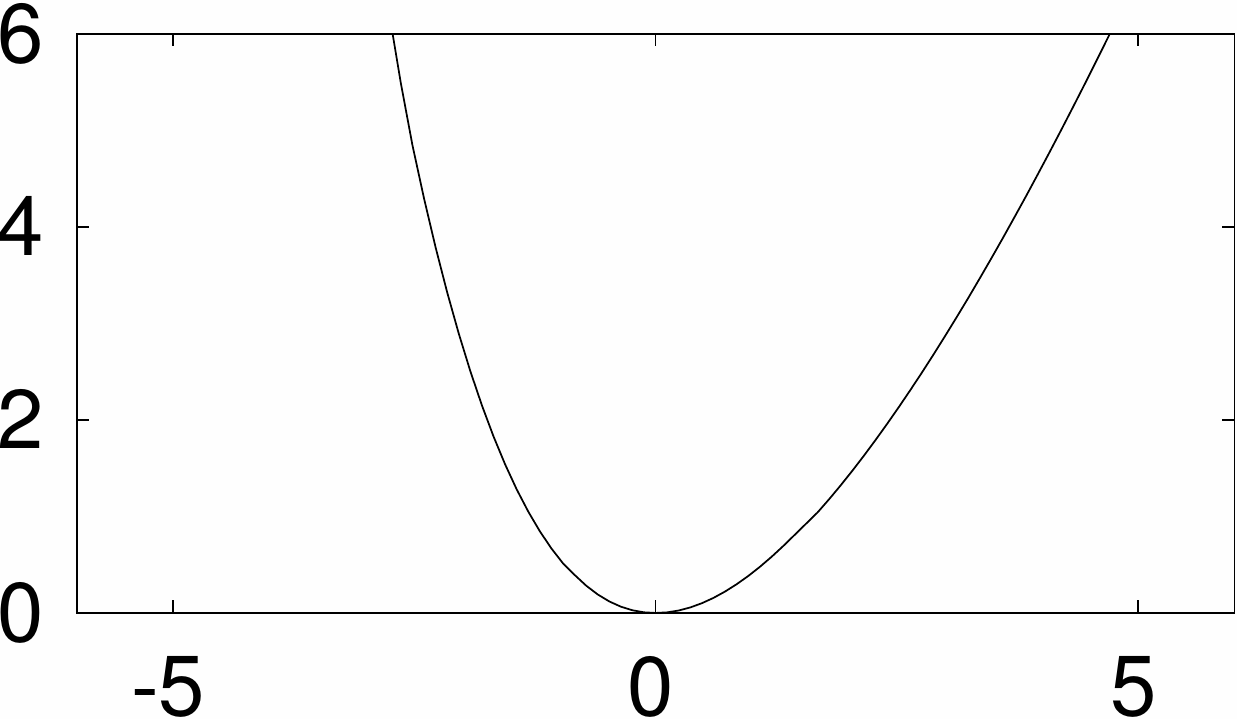}}\\
\raisebox{.45\unitlength}{$\SimLoc^{\tn{nd}}$}&\includegraphics[height=\unitlength]{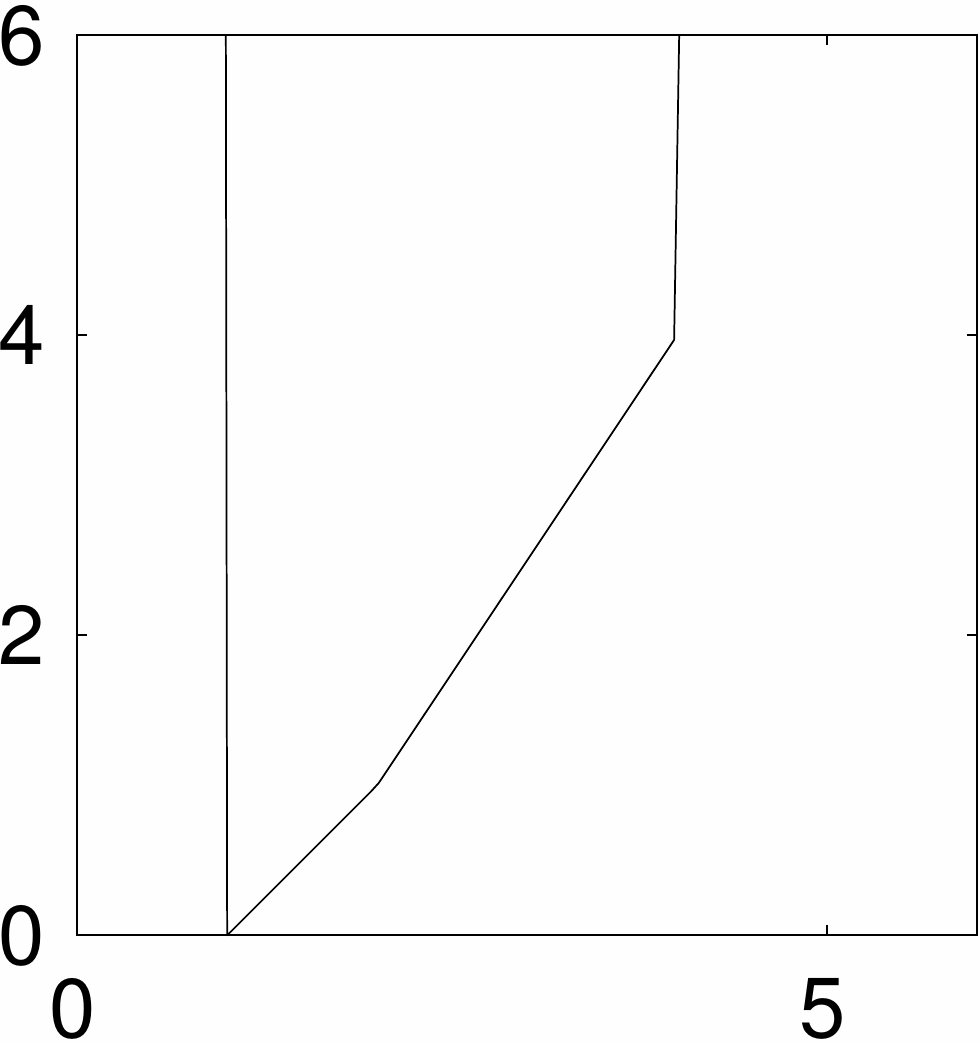}&\includegraphics[height=\unitlength]{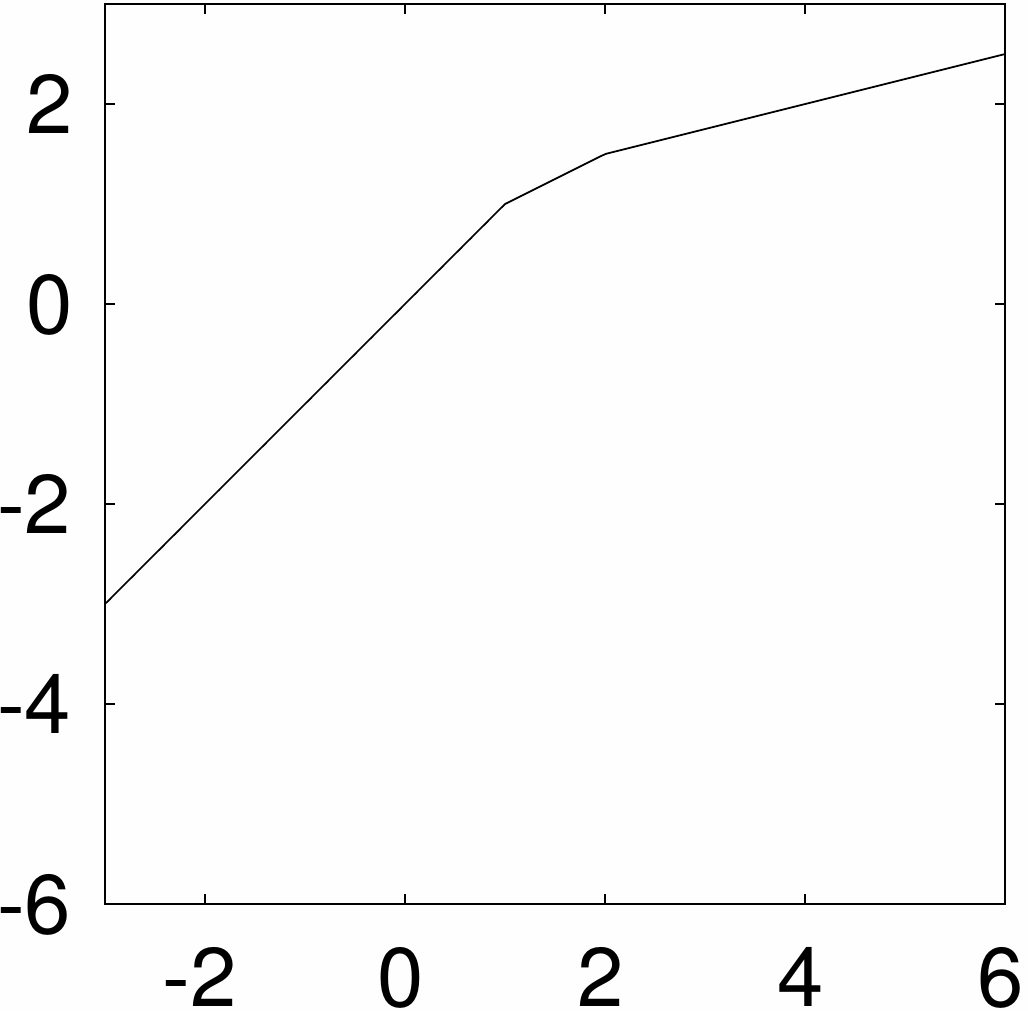}&\includegraphics[height=\unitlength]{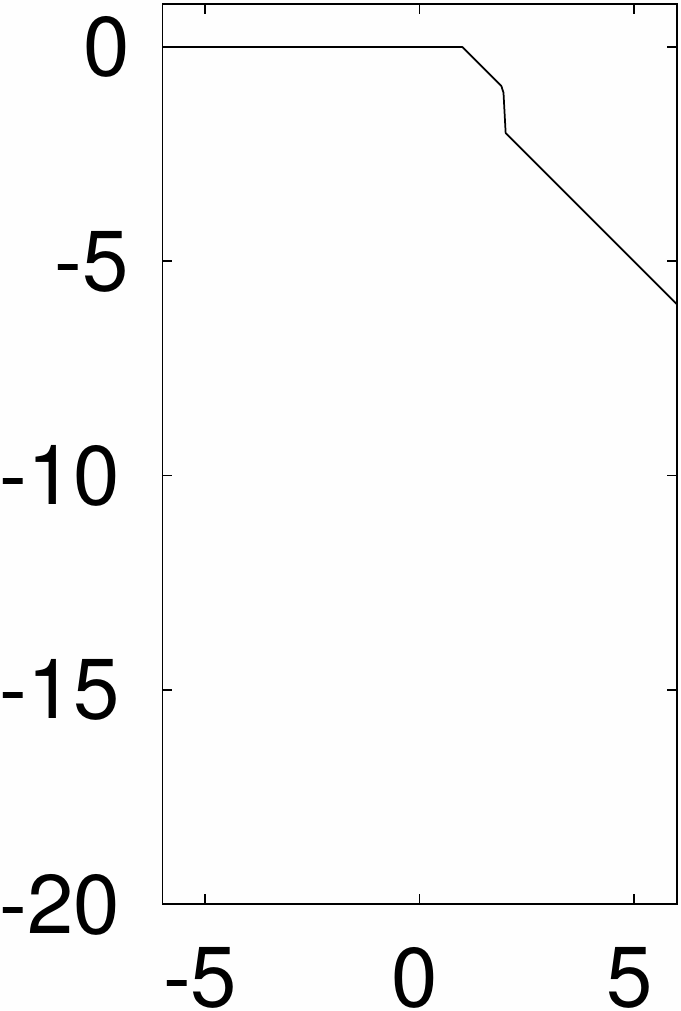}&
\end{tabular}
\caption{Graphs of the functions $\SimLocC$, $\HStat$, $\BDynMH$, and $\CDynM$ for the dissimilarities from Table~\ref{tab:discrepancies}.
The last row shows a dissimilarity $\SimLoc^{\tn{nd}}$ for which $q[\HStat]$ is not convex so that there is no corresponding dynamic model.}
\label{tab:discrepancyGraphs}
\end{table}

\begin{example}[Dynamic formulation for piecewise linear metric $\SimLoc^{\tn{pwl}}$]
The mass change penalty of $\SimLoc^{\tn{pwl}}$ from Table~\ref{tab:discrepancies} has four linear segments:
$\SimLocC(1,\cdot)$ is zero at $1$,
and it increases to the right first with slope $-\underline s$ and then with slope $-\underline d$
while it increases to the left first with slope $\overline s$ and then with slope $\overline d$.
The resulting $\HStat$ is piecewise linear as well and satisfies
\begin{equation*}
\HStatInv(z)=-\underline s+\tfrac1a(z+\underline s)
\quad\text{for }z\in[-\underline s,-\underline d]\,,\qquad
\HStat(z)=\overline s+b(z-\overline s)
\quad\text{for }z\in[\overline s,\overline d]\,.
\end{equation*}
Note that $\HStatInv^j(z)=\tfrac1{a^j}z-(1-\tfrac1{a^j})\underline s$ for $z\in[\underline d,\underline s]$ and $\HStat^j(z)=b^jz+(1-b^j)\overline s$ for $z\in[\overline s,\overline d]$.
\thref{thm:staticAsDynamic}\eqref{enm:uniqueness} thus implies
\begin{equation*}
\BDynMH(z)
=\begin{cases}
-\infty&\text{if }z\notin[\underline d,\overline d],\\
(\underline s-z)\log\frac1a&\text{if }z\in[\underline d,\underline s],\\
0&\text{if }z\in[\underline s,\overline s],\\
(z-\overline s)\log b&\text{if }z\in[\overline s,\overline d]
\end{cases}
\end{equation*}
so that $\CDynM$ can readily be computed via $\CDynM(\rho,\zeta)=\sup\{\rho\BDynMH(z)+\zeta z\ |\ z\in\R\}$.
Note that the mass change penalties $\SimDisc$ and $\SimTV$ can be seen as special cases
with $-\underline d=-\underline s=\overline s=\overline d=\infty$ and $-\underline d=-\underline s=\overline s=\overline d=1$, respectively.
\end{example}

\begin{example}[Dynamic formulation for Hellinger metric $\SimFR$]
The squared Hellinger metric $\SimLocC(m_0,m_1)=(\sqrt{m_0}-\sqrt{m_1})^2$ readily results in $\HStat(z)=\HStatInv(z)=\frac z{1+z}$ for $z\in(-1,\infty)$.
Note $\HStat^j(z)=\HStatInv^j(z)=\frac z{1+jz}$, as can easily be shown by induction.
\thref{thm:staticAsDynamic}\eqref{enm:uniqueness} thus implies
\begin{equation*}
\BDynMH(z)
=\lim_{j\to\infty}\frac{\HStat^j(z)-\HStat^{j-1}(z)}{(\HStat^j)'(z)}
=\lim_{j\to\infty}\frac{z/(1+jz)-z/(1+(j-1)z)}{1/(1+jz)^2}
=-z^2
\end{equation*}
for $z\geq0$ and likewise $\BDynMH(z)=-z^2$ for $z\leq0$.
From this we can now readily compute
\begin{equation*}
\CDynM(\rho,\zeta)
=\sup\{\rho\BDynMH(z)+\zeta z\ |\ z\in\R\}
=\begin{cases}
\infty&\text{if }\rho<0\,,\\
\frac{\zeta^2}{4\rho}&\text{else.}
\end{cases}
\end{equation*}
\end{example}

\begin{example}[Dynamic formulation for Jensen--Shannon divergence $\SimJS$]
The Jensen--Shannon divergence $\SimLocC(m_0,m_1)=m_0\log_2\frac{2m_0}{m_0+m_1}+m_1\log_2\frac{2m_1}{m_0+m_1}$ results in $\HStat(z)=\HStatInv(z)=\log_2(2-\frac1{2^z})$ on $(-1,\infty)$.
Note that $\HStat^j(z)=\HStatInv^j(z)=\log_2\frac{(j+1)2^z-j}{j2^z-(j-1)}$.
Thus \thref{thm:staticAsDynamic}\eqref{enm:uniqueness} implies
\begin{equation*}
\BDynMH(z)
=\lim_{j\to\infty}\frac{\HStat^j(z)-\HStat^{j-1}(z)}{(\HStat^j)'(z)}
=\frac{(2^{z/2}-2^{-z/2})^2}{-\log2}
\end{equation*}
for $z\geq0$ and likewise $\BDynMH(z)=\frac{(2^{z/2}-2^{-z/2})^2}{-\log2}$ for $z\leq0$,
from which we obtain
\begin{equation*}
\CDynM(\rho,\zeta)
=\sup\{\rho\BDynMH(z)+\zeta z\ |\ z\in\R\}
=\begin{cases}
\infty&\text{if }\rho<0\,,\\
\frac{(\sqrt g-1/\sqrt g)^2}{-\log 2}\rho+\zeta\log_2g&\text{else}
\end{cases}
\end{equation*}
with $g=\frac\zeta{2\rho}+\sqrt{\frac{\zeta^2}{4\rho^2}+1}$.
\end{example}

The graphs of the functions $\SimLocC$, $\HStat$, $\BDynMH$ and $\CDynM$ for the above examples and a few further ones are provided in Table~\ref{tab:discrepancyGraphs}.
\begin{remark}[Static model without dynamic correspondence]%
\thlabel{rem:NoDynamicCounterexample}%
There are static models which are not induced by a corresponding dynamic model.
Indeed, the last row of Table~\ref{tab:discrepancyGraphs} shows a static model with
\begin{align*}
	\SimLocC(1,z) & = \begin{cases}
		+\infty & \tn{if } z \notin [1,4], \\
		\max\{z-1,\tfrac{3}{2}z-2\} & \tn{if } z \in [1,4],
		\end{cases} &
	\HStat(z) & = \min\{z,\tfrac12 z+\tfrac12,\tfrac14 z+1\}\,.
\end{align*}
By \thref{thm:staticAsDynamic}\eqref{enm:necessary} it does not admit a corresponding dynamic model;
and indeed, the calculated $q[\HStat]$ shown in Table~\ref{tab:discrepancyGraphs} is nonconvex. Note that the non-existence of a dynamic model in this example is unrelated to the behaviour of $\SimLocC(1,z)$ for $z \leq 1$.
\end{remark}

As a final remark, let us mention that the cost \eqref{eq:StaticPrimalProblem} with $\SimTV$ is also known as the flat norm difference between two zero-dimensional currents or flat chains $\rho_0,\rho_1$ \cite[Sec.\,4.1-2]{Fe69}.
Hence, the unbalanced transport cost with $\SimTV$ metrizes flat convergence.
Since flat convergence of uniformly bounded measures coincides with weak convergence,
the following proposition implies that weak convergence is equivalent to convergence of the unbalanced transport cost
for all static models with dissimilarities from Table~\ref{tab:discrepancies} except for $\SimDisc$, $\SimE{0}$, and $\SimE{1}$.

\begin{proposition}[Topological equivalence]\thlabel{thm:topology}
The cost $W_S$ from \eqref{eq:StaticPrimalProblem} for the dissimilarities $\SimTV$, $\SimLoc^{\tn{pwl}}$ with $(\underline d,\underline s,a)=(-\overline d,-\overline s,\frac1b)$, $\SimFR$, $\SimJS$, or $\SimChi$ is a semimetric
on the space $\{\rho\in\measp(\Omega)\ |\ \rho(\Omega)\leq M\}$ of nonnegative measures with mass uniformly bounded by $M>1$.
Convergence in this semimetric is equivalent among all those dissimilarities.
\end{proposition}
\begin{proof}
It is straightforward to see that the cost $W_S(\rho_0,\rho_1)$ is positive unless $\rho_0=\rho_1$, and that $W_S(\rho_0,\rho_1)=W_S(\rho_1,\rho_0)$.
Thus, $W_S$ is a semimetric.

Consider two sequences of nonnegative measures $\rho_j,\sigma_j$, $j=1,2,\ldots$.
Since all listed dissimilarities can be bounded above by a multiple of $\SimTV$,
convergence $W_S(\rho_j,\sigma_j)\to0$ for $\SimTV$ also implies $W_S(\rho_j,\sigma_j)\to0$ for any of the other dissimilarities.
On the other hand, all dissimilarities can be bounded below by a multiple of the dissimilarity
$\SimLoc(\rho_0,\rho_1)=\int_\Omega\SimLocC(\RadNik{\rho_0}{\gamma},\RadNik{\rho_1}{\gamma})\,\d\gamma$ for
\begin{equation*}
\SimLocC(m_0,m_1)
=\begin{cases}
\infty&\text{if }m_0<0\text{ or }m_1<0,\\
2(m_0-m_1)-m_1&\text{if }\frac{m_1}{m_0}<\frac12,\\
\frac{(m_0-m_1)^2}{m_1}&\text{if }\frac{m_1}{m_0}\in[\frac12,1),\\
\frac{(m_1-m_0)^2}{m_0}&\text{if }\frac{m_1}{m_0}\in[1,2),\\
2(m_1-m_0)-m_0&\text{if }\frac{m_1}{m_0}\geq2.\\
\end{cases}
\end{equation*}
Thus it remains to show that $W_S(\rho_j,\sigma_j)\to0$ for $\SimLoc$ implies $W_S(\rho_j,\sigma_j)\to0$ also for $\SimTV$.
For $\varepsilon\in(0,1)$ let $N$ be large enough such that $W_S(\rho_j,\sigma_j)<\varepsilon$ for $\SimLoc$ and all $j>N$.
Let $\rho'_j=P_Y\pi_0$ and $\sigma'_j=P_X\pi_1$ for the optimal $\pi_0$ and $\pi_1$ in \eqref{eq:StaticPrimalProblem} with $\SimLoc$
and abbreviate $\gamma=\rho'_j+\sigma'_j$ and $\tilde\Omega=\{x\in\Omega\ |\ \RadNik{\rho'_j}\gamma/\RadNik{\sigma'_j}\gamma\in(\frac12,2)\}$, then
\begin{multline*}
\SimLoc(\rho'_j,\sigma'_j)=
\int_\Omega\SimLocC(\RadNik{\rho'_j}\gamma,\RadNik{\sigma'_j}\gamma)\,\d\gamma
=\int_{\Omega\setminus\tilde\Omega}2|\RadNik{\rho'_j}\gamma-\RadNik{\sigma'_j}\gamma|-\min(\RadNik{\rho'_j}\gamma,\RadNik{\sigma'_j}\gamma)\,\d\gamma
+\int_{\tilde\Omega}\frac{|\RadNik{\rho'_j}\gamma-\RadNik{\sigma'_j}\gamma|^2}{\min(\RadNik{\rho'_j}\gamma,\RadNik{\sigma'_j}\gamma)}\,\d\gamma\\
\geq\int_{\Omega\setminus\tilde\Omega}|\RadNik{\rho'_j}\gamma-\RadNik{\sigma'_j}\gamma|\,\d\gamma
+\int_{\tilde\Omega}|\RadNik{\rho'_j}\gamma-\RadNik{\sigma'_j}\gamma|^2\,\d\gamma
\geq\int_{\Omega\setminus\tilde\Omega}|\RadNik{\rho'_j}\gamma-\RadNik{\sigma'_j}\gamma|\,\d\gamma
+\frac1{\gamma(\tilde\Omega)}\left(\int_{\tilde\Omega}|\RadNik{\rho'_j}\gamma-\RadNik{\sigma'_j}\gamma|\,\d\gamma\right)^2\\
\geq\int_{\Omega\setminus\tilde\Omega}|\RadNik{\rho'_j}\gamma-\RadNik{\sigma'_j}\gamma|\,\d\gamma
+\frac1{2M}\left(\int_{\tilde\Omega}|\RadNik{\rho'_j}\gamma-\RadNik{\sigma'_j}\gamma|\,\d\gamma\right)^2
\geq\frac1{4M}\left(\int_{\Omega}|\RadNik{\rho'_j}\gamma-\RadNik{\sigma'_j}\gamma|\,\d\gamma\right)^2
=\frac1{4M}\SimTV(\rho'_j,\sigma'_j)^2\,.
\end{multline*}
Thus, the cost with $\SimTV$ satisfies $W_S(\rho_j,\sigma_j)\leq\varepsilon+\sqrt{4M\varepsilon}$ for all $j>N$.
The result follows now from the arbitrariness of $\varepsilon$.
\end{proof}

The same argument shows that weak convergence is also equivalent to convergence with respect to the cost $W_S$ using the dissimilarities $\SimE{p}$ with $p\in(0,1)$ or a general $\SimLoc^{\tn{pwl}}$.
However, those costs are not symmetric and thus not a semimetric.

% !TEX root = W1TypeDynamic.tex
\section{Conclusion}
We considered a dynamic formulation of general unbalanced transport models with $1$-homo\-gen\-eous transport costs as in the Wasserstein-$1$ distance.
We showed its equivalence to a corresponding static formulation,
and we showed how the dynamic and static model parameters (the functions $\CDynM$ and $\SimLocC$) as well as the dynamic and static optimizers relate to each other.
This is particularly relevant from an application viewpoint,
since the static models have the advantage of being much lower-dimensional (in particular in their dual formulation in \thref{prop:StaticEquivalence}) and thus easier to compute.

We also answered the question which static models allow a dynamic formulation, arriving at a necessary and sufficient condition that can for instance be checked numerically.
Due to our focus on $W_1$-type models, this characterization becomes much simpler than a characterization of what dynamic and static unbalanced transport models correspond to each other in the general case as in \cite[Thm.\,4.3]{ChizatOTFR2015}.
An interesting, non-obvious result of this analysis is that the class of static models is richer than the class of dynamic models,
even though one might have thought that modelling the full time dynamics allows more modelling freedom.
This phenomenon is connected to the fact that the optimal dynamics are controlled by an autonomous ordinary differential equation (ode) \eqref{eqn:FlowIVP} describing the mass change,
which imposes some a priori regularity on the resulting mass change.
In other settings without our particular $W_1$ structure, a similar situation should be expected.
Allowing the mass change penalty to change over time (thereby turning the above autonomous into a nonautonomous ode)
might remove this asymmetry between dynamic and static models and could be further investigated.

The analysis of which static models allow a dynamic formulation turns out to be rather decoupled from the actual transport problem.
Thus the question arises whether it might have a broader applicability.
Essentially it answers the question what mass change penalties (that is, convex, $1$-homogeneous premetrics)
are generated by geodesic paths with respect to a convex $1$-homogeneous infinitesimal cost in the space of positive measures.
A completely different, dynamical systems or inverse problems perspective is the following:
If the static model has an equivalent dynamic formulation, then we showed that the static model function $\HStat$ is the time-one map of the autonomous scalar ode governed by the dynamic model function $\BDynMH$.
Of course, for any time-one map there are many possible compatible right-hand sides of the ode.
Thus we essentially answered the question which monotone maps $\HStat:\R\to\R$ with a single interval of fixed points are the time-one map of an autonomous ode with concave right-hand side $\BDynMH$,
and we provided an explicit formula to recover this $\BDynMH$.

We furthermore gave a detailed characterization of the structure of dynamic optimizers.
This characterization heavily exploits the $W_1$-type nature of our problem,
and it will be fundamentally different for other dynamic unbalanced transport models, in which typically mass change and transport occur simultaneously.
In such cases no clean separation of dynamic processes into transport and mass change can be expected.
This clean separation is actually advantageous from the application perspective,
since it avoids that even a simple, mere transport task will be accompanied by mass loss throughout the first half of the transport and subsequent mass gain throughout the second half,
just to reduce transport costs.
A similar situation of clean separation between both processes is expected if not the transport costs, but rather the mass change costs are restricted to being $1$-homogeneous, a case which seems worth investigating.

\appendix
% !TEX root = W1TypeDynamic.tex
\section{A version of the Stone--Weierstra\ss\ theorem}

Throughout, let $\Omega\subset\R^n$ be the closure of an open bounded connected set, and let the metric $d$ on $\Omega$ be induced by shortest curves inside $\Omega$ (those exist, since $\Omega$ is geodesically complete).
Let $B_\delta(x)\subset\R^n$ denote the closed ball of radius $\delta>0$ around $x\in\R^n$.

\begin{lemma}[Approximation property of distance]
For $\delta>0$ denote the $\delta$-dilation of $\Omega$ by $\Omega_\delta=\{x\in\R^n\ |\ \exists y\in\Omega:\|x-y\|\leq\delta\}$, and denote the metric induced by shortest curves on $\Omega_\delta$ by $d_\delta$.
For any $x,y\in\Omega$ we have $d_\delta(x,y)\to d(x,y)$ monotonically from below as $\delta\to0$.
\end{lemma}
\begin{proof}
Obviously, $d_\delta(x,y)$ is monotonically increasing as $\delta\to0$ and satisfies $d_\delta(x,y)\leq d(x,y)$.
Let $\gamma_\delta\subset\Omega_\delta$ be a curve connecting $x$ and $y$ with minimal length, that is, $\mathcal H^1(\gamma_\delta)=d_\delta(x,y)$.
By the Blaschke Compactness Theorem \cite[Thm.\,4.4.15]{AmTi04} there exists a subsequence $\gamma_{\delta_i}$ converging in the Hausdorff distance to a compact set $\gamma$ connecting $x$ and $y$.
By Go\l\c{a}b's Theorem \cite[Thm.\,4.4.17]{AmTi04}, $\gamma$ is connected with $\mathcal H^1(\gamma)\leq\liminf_{i\to\infty}\mathcal H^1(\gamma_{\delta_i})=\liminf_{i\to\infty}d_{\delta_i}(x,y)$.
Furthermore, $\gamma\subset\Omega$ so that $d(x,y)\leq\liminf_{i\to\infty}d_{\delta_i}(x,y)$, concluding the proof.
Indeed, if $z\in\gamma\setminus\Omega$, let $D=\mathrm{dist}(z,\Omega)$.
Then for all $i$ large enough there is some $z_i\in\gamma_{\delta_i}$ with $\|z-z_i\|\leq\frac D4$, which contradicts $\gamma_{\delta_i}\subset\Omega_{\delta_i}$ for $\delta_i<\frac D4$.
\end{proof}

\begin{lemma}[Point separation]\thlabel{thm:pointSeparation}
For any nonidentical $x,y\in\Omega$ and $a,b\in\R$ with $|a-b|<d(x,y)$ there is a function $f\in C^1(\Omega)\cap\Lip(\Omega)$ with $f(x)=a$ and $f(y)=b$.
\end{lemma}
\begin{proof}
Without loss of generality $b>a$.
For $\delta>0$ define $g_\delta:\Omega_\delta\to\R$, $g_\delta(z)=d_\delta(x,z)$, and extend $g_\delta$ by zero outside $\Omega_\delta$.
Then convolve $g_\delta$ with a smooth mollifier of support in $B_{\frac\delta2}(0)$ to obtain $h_\delta\in C^1(\R^n)$.
The value of $\nabla h_\delta$ at each point $z\in\Omega$ is the weighted mean of $\nabla g_\delta$ over $B_{\frac\delta2}(z)$,
which lies completely inside $\Omega_\delta$ so that $g_\delta\in\Lip(B_{\frac\delta2}(z))$ and thus $\|\nabla h_\delta(z)\|\leq1$ (since $g_\delta$ satisfied the same condition).
Thus $h_\delta|_{\Omega}\in C^1(\Omega)\cap\Lip(\Omega)$.
Finally, define $f_\delta:\Omega\to\R$, $f_\delta(z)=a+(h_\delta(z)-h_\delta(x))\frac{b-a}{h_\delta(y)-h_\delta(x)}$.
By definition, $f_\delta(x)=a$ and $f_\delta(y)=b$.
Furthermore, as $\delta\to0$, $h_\delta\to g_\delta$ uniformly on $\Omega$ and $g_\delta(y)\to d(x,y)$ by the previous lemma.
Thus, $\frac{b-a}{h_\delta(y)-h_\delta(x)}\leq1$ for small enough $\delta$ so that $f_\delta\in C^1(\Omega)\cap\Lip(\Omega)$.
\end{proof}

\begin{lemma}[Smoothed Heaviside function]
For any $\varepsilon,\delta>0$ there exists a function $H_{\varepsilon,\delta}\in C^\infty(\R)$ with $|xH_{\varepsilon,\delta}'(x)|\leq\delta$ for all $x\in\R$
and $H_{\varepsilon,\delta}(x)=0$ for all $x\leq-\varepsilon$ and $H_{\varepsilon,\delta}(x)=1$ for all $x\geq\varepsilon$.
\end{lemma}
\begin{proof}
Since we can always mollify $H_{\varepsilon,\delta}$, it suffices to show the existence of a piecewise differentiable $H_{\varepsilon,\delta}$ satisfying the above conditions.
Set $\eta=\varepsilon/(\exp(\frac1{2\delta})-1)$ and define
\begin{equation*}
H_{\varepsilon,\delta}(x)=\begin{cases}
0&\text{if }x\leq-\varepsilon,\\
\frac12-\delta\log\frac{|x|+\eta}\eta&\text{if }x\in(-\varepsilon,0],\\
\frac12+\delta\log\frac{|x|+\eta}\eta&\text{if }x\in(0,\varepsilon],\\
1&\text{if }x>\varepsilon.
\end{cases}
\end{equation*}
It is straightforward to check that $H_{\varepsilon,\delta}$ is continuous and satisfies the required conditions almost everywhere.
\end{proof}

\begin{lemma}[Smooth min and max operation]
For $f,g\in C^1(\Omega)\cap\Lip(\Omega)$ and $\delta>0$ define
\begin{align*}
f\wedge_\delta g&=gH_{\delta,\delta}(f-g)+f(1-H_{\delta,\delta}(f-g))\,,\\
f\vee_\delta g&=fH_{\delta,\delta}(f-g)+g(1-H_{\delta,\delta}(f-g))\,.
\end{align*}
Then $f\wedge_\delta g,f\vee_\delta g\in C^1(\Omega)\cap(1+2\delta)\Lip(\Omega)$, and
\begin{align*}
\min(f,g)-\delta&\leq f\wedge_\delta g\leq\min(f,g)+\delta\,,\\
\max(f,g)-\delta&\leq f\vee_\delta g\leq\max(f,g)+\delta\,.
\end{align*}
\end{lemma}
\begin{proof}
The continuous differentiability follows from the continuous differentiability of $f$, $g$, and $H_{\delta,\delta}$. Furthermore,
\begin{multline*}
\|\nabla(f\wedge_\delta g)\|
=\|\nabla gH_{\delta,\delta}(f-g)+\nabla f(1-H_{\delta,\delta}(f-g))+(g-f)H_{\delta,\delta}'(f-g)\nabla(f-g)\|\\
\leq\|\nabla gH_{\delta,\delta}(f-g)+\nabla f(1-H_{\delta,\delta}(f-g))\|+|(g-f)H_{\delta,\delta}'(f-g)|(\|\nabla f\|+\|\nabla g\|)
\leq1+\delta(1+1)\,.
\end{multline*}
Analogously one shows $f\vee_\delta g\in(1+2\delta)\Lip(\Omega)$.
Finally, if $|f-g|\geq\delta$, then $\min(f,g)=f\wedge_\delta g$ and $\max(f,g)=f\vee_\delta g$, while otherwise we obtain a convex combination of $f$ and $g$, which implies the rest of the statement.
\end{proof}

The following is a simple variant of the classical proof by Stone.

\begin{proposition}[Stone--Weierstra\ss]\thlabel{thm:StoneWeierstrass}
$C^1(\Omega)\cap\Lip(\Omega)$ is dense in $\Lip(\Omega)$ with respect to the supremum norm.
\end{proposition}
\begin{proof}
Given $g\in\Lip(\Omega)$ and $\varepsilon>0$ we seek some $f\in C^1(\Omega)\cap\Lip(\Omega)$ with $\|f-g\|\leq\varepsilon$.
Actually, since $\lambda g\to g$ uniformly as $\lambda\to1$, it suffices to consider $g\in\lambda\Lip(\Omega)$ for an arbitrary $\lambda\in(0,1)$.

For given $x\in\Omega$, by \thref{thm:pointSeparation} there exists for each $y\in\Omega$ a function $f_y\in C^1(\Omega)\cap\Lip(\Omega)$ with $f_y(x)=g(x)$ and $f_y(y)=g(y)$.
Introduce the open set $V_y=\{z\in\Omega\ |\ f_y(z)<g(z)+\frac\varepsilon4\}$.
Since the family $V_y$ for $y\in\Omega$ forms an open cover of $\Omega$ (note that $x,y\in V_y$) and $\Omega$ is compact,
we can extract a finite subcover $V_{y_1},\ldots,V_{y_k}$.
For $\delta>0$ define
\begin{equation*}
\tilde F_x=(\ldots((f_{y_1}\wedge_\delta f_{y_2})\wedge_\delta f_{y_3})\ldots\wedge_\delta f_{y_k})\,.
\end{equation*}
By the previous lemma, $\tilde F_x\in C^1(\Omega)\cap(1+2\delta)^k\Lip(\Omega)$ with $|\tilde F_x-\min(f_{y_1},\ldots,f_{y_k})|\leq k\delta$.
Thus, by choosing $\delta$ small enough (depending on $x$) we have $F_x=\frac1{(1+2\delta)^k}\tilde F_x\in C^1(\Omega)\cap\Lip(\Omega)$ with $F_x\leq g+\frac\varepsilon2$.

Now introduce the open set $W_x=\{z\in\Omega\ |\ F_x(z)>g(z)-\frac\varepsilon4\}$, which contains $x$.
Again the $W_x$ form an open cover of $\Omega$ from which we can extract a subcover $W_{x_1},\ldots,W_{x_l}$, and analogously to before we set
\begin{equation*}
\tilde f=(\ldots((F_{x_1}\vee_\delta F_{x_2})\vee_\delta F_{x_3})\ldots\vee_\delta F_{x_l})\,.
\end{equation*}
Again, by choosing $\delta$ small enough we have $f=\frac1{(1+2\delta)^l}\tilde f\in C^1(\Omega)\cap\Lip(\Omega)$ with $g-\varepsilon\leq f\leq g+\varepsilon$.
\end{proof}

\begin{remark}[Higher smoothness]
One may replace $C^1$ by $C^\infty$.
\end{remark}

\bibliography{references}{}
\bibliographystyle{plain}

\end{document}